\documentclass[11pt,reqno]{amsart}
\usepackage{mathrsfs}
\usepackage{url}
\usepackage{mathtools}
\usepackage{latexsym,epsfig,amssymb,amsmath,amsthm,color,url,bm}
\usepackage[inline,shortlabels]{enumitem}
\usepackage{hyperref}
\usepackage[foot]{amsaddr}
\usepackage{amsmath,amsbsy}
\usepackage{mwe}
\RequirePackage{natbib}
\usepackage{mathptmx}
\usepackage[text={16cm,24cm}]{geometry}
\usepackage{soul}

\allowdisplaybreaks 
 \numberwithin{equation}{section}
\newtheorem{theorem}{Theorem}[section]
\newtheorem{prop}[theorem]{Proposition}
\newtheorem{lemma}[theorem]{Lemma}

\newcommand\Item[1][]{%
  \ifx\relax#1\relax  \item \else \item[#1] \fi
  \abovedisplayskip=0pt\abovedisplayshortskip=0pt~\vspace*{-\baselineskip}}

\theoremstyle{definition}
\newtheorem{defn}[theorem]{Definition}

\theoremstyle{definition}
\newtheorem{remark}[theorem]{Remark}

\DeclareMathOperator{\SG}{SG}
\DeclareMathOperator{\mex}{mex}

\title[]{The Game of Graph Nim on graphs with four edges}
\date{}
\author{Sayar Karmakar, Moumanti Podder, Souvik Roy, Soumyarup Sadhukhan}
\address{Sayar Karmakar, University of Florida, 230 Newell Drive, Gainesville, Florida 32605, USA.}
\address{Moumanti Podder, Indian Institute of Science Education and Research (IISER) Pune, Dr.\ Homi Bhabha Road, Pashan, Pune 411008, Maharashtra, India.}
\address{Souvik Roy, Indian Statistical Institute, 203 Barrackpore Trunk Road, Kolkata 700108, West Bengal, India.}
\address{Soumyarup Sadhukhan, Indian Institute of Technology, Kalyanpur, Kanpur, Uttar Pradesh 208016, India.}
\email{sayarkarmakar@ufl.edu}
\email{moumanti@iiserpune.ac.in}
\email{souvik.2004@gmail.com}
\email{soumyarup.sadhukhan@gmail.com}

\begin{document}

\begin{abstract}
This work is concerned with the study of the \emph{game of Graph Nim} -- a class of two-player combinatorial games -- on graphs with $4$ edges. To each edge of such a graph is assigned a positive-integer-valued edge-weight, and during each round of the game, the player whose turn it is to make a move selects a vertex, and removes a non-negative integer edge-weight from each of the edges incident on that vertex, such that
\begin{enumerate*}
\item the remaining edge-weight on each of these edges is a non-negative integer,
\item and the total edge-weight removed during a round is strictly positive.
\end{enumerate*}
The game continues for as long as the sum of the edge-weights remaining on all edges of the graph is strictly positive, and the player who plays the last round wins. An initial configuration of edge-weights is considered \emph{winning} if the player who plays the first round wins the game, whereas it is defined as \emph{losing} if the player who plays the second round wins. In this paper, we characterize, \emph{almost entirely}, all winning and losing configurations for this game on all graphs with precisely $4$ edges each. Only one such graph defies our attempt to \emph{fully} characterize the winning and losing configurations of edge-weights on its edges -- we are still able to provide a set of partial results pertaining to this graph.  
\end{abstract}

\subjclass[2020]{}%60C05, 68Q87, 05C05, 05C80, 05C65, 05D40, 05C57, 37B15, 37A25, 68Q80


\maketitle

\section{Introduction and a formal description of the game of Graph Nim}\label{sec:intro}
The \emph{game of Nim} is one of the earliest examples of \emph{two-player combinatorial games}, purported to have originated in ancient China and mentioned in European references that date back to as early as the 16th century. While several variants of this game are in existence, we describe here one of the most commonly studied versions, as follows:
\begin{defn}[The game of Nim]\label{defn:nim}
The game begins with $k$ piles or heaps of chips, with the $i$-th pile containing $x_{i}$ chips for $i \in \{1,2,\ldots,k\}$, where $x_{1}, x_{2}, \ldots, x_{k}$ are positive integers. Two players take turns to make moves, where a \emph{move} involves choosing a pile (that is not already empty) and subsequently removing at least one chip from that pile. The game continues until all the piles have been rendered empty, and the player to remove the very last chip is declared the winner.
\end{defn}
The set $\{x_{1},x_{2},\ldots,x_{k}\}$ is referred to as the \emph{initial configuration} or \emph{initial position} of the game, and it is referred to as a \emph{winning position} if the player who plays the first round is destined to win the game, whereas it is referred to as a \emph{losing position} if the player who plays the first round is destined to lose the game. This is a \emph{perfect information game}, as the piles of chips are revealed in their entirety to the two players \emph{before} the game begins.

This game has been studied extensively, and its winning and losing positions have been characterized fully (see, for instance, the seminal work done in \cite{bouton1901nim}, as well as \cite{conway2000numbers} and \cite{ferguson2020course} for a more detailed literature on these games). In this paper, in some of our proofs, we have made use of the necessary and sufficient condition for an initial configuration to be winning for the game of Nim defined in Definition~\ref{defn:nim} (we refer the reader to Theorem~\ref{thm:Nim} for a description of the same).

The game we study in this paper is inspired by the game of Nim, and played on a graph. As far as we are aware, this game, henceforth referred to as the \emph{game of Graph Nim}, was introduced and studied in \cite{williams2017combinatorial}, and our work in this paper serves as both a significant extension to the results obtained in \cite{williams2017combinatorial} and to show that the analysis required herein is far more complicated and \emph{ad hoc} (in the sense that it depends heavily on the underlying graph on which the game is being played) compared to the game of Nim in Definition~\ref{defn:nim}. In what follows, we denote by $G=(V,E)$ an undirected graph whose set of vertices is given by $V$ and whose set of edges is given by $E$. For $u, v \in V$, we let $\{u,v\}$ denote the edge between these two vertices if they are adjacent (i.e.\ if $\{u,v\} \in E$). Henceforth, we let $\mathbb{N}$ denote the set of all positive integers and $\mathbb{N}_{0}$ the set of all non-negative integers.
\begin{defn}[The game of Graph Nim]\label{defn:graph_nim}
Consider a graph $G=(V,E)$, with an edge-weight $w_{0}(e)$, where $w_{0}(e) \in \mathbb{N}_{0}$, assigned to each $e \in E$. Two players take turns (each such turn is, henceforth, referred to as a \emph{round}) to make moves, where a \emph{move} involves 
\begin{enumerate}
\item selecting a vertex $u \in V$,
\item removing a non-negative integer weight from each edge $\{u,v\} \in E$ (i.e.\ each edge in $E$ that is incident on the vertex $u$), so that the total weight removed is a strictly positive integer, and the remaining edge-weight of each $\{u,v\} \in E$ is a non-negative integer.
\end{enumerate}
The game continues for as long as the edge-weight of each edge in $E$ has not been reduced to $0$. The player to play the last round is declared the winner.
\end{defn}
Note that this game can be thought of as being played on the complete graph $K_{|V|}$ with $|V|$ many vertices, since for each pair of vertices $u,v \in V$ with no edge present between $u$ and $v$ (i.e.\ $\{u,v\} \notin E$), it is as if the initial edge-weight $w_{0}(\{u,v\})$ equals $0$. Henceforth, we let the players be referred to as $P_{1}$ and $P_{2}$, with $P_{1}$ assumed to play the first round.

Given the underlying graph $G=(V,E)$, we call $\left(w_{0}(e): e \in E\right)$ a \emph{winning initial weight configuration} (or, simply, a \emph{winning configuration}) if the player who plays the first round of the game on $G$, starting with these initial edge-weights, is destined to win. We call $\left(w_{0}(e): e \in E\right)$ a \emph{losing initial weight configuration} (or, simply, a \emph{losing configuration}) if the player who plays the first round is destined to lose the game. 

In this paper, our objective is to consider all graphs comprising precisely $4$ edges, i.e.\ $|E|=4$, and characterize the winning and losing configurations on each of them. It suffices to consider all \emph{unlabeled} distinct graphs comprising $4$ edges each, and barring one, we are able to provide a \emph{complete} characterization of the winning and losing configurations on each of them. For the graph for which we are not able to obtain a full understanding of the winning and losing configurations, we provide partial results  that, we believe, are adequate in demonstrating to the reader the challenges faced when analysing the game of Graph Nim on certain underlying graphs.

\subsection{Organization of the rest of the paper}
Our paper is organized as follows. In \S\ref{sec:prelim_useful_results}, we explain how the game of Graph Nim can be seen as a generalization (that can be played on graphs) of the game of Nim. This is followed by \S\ref{subsec:lit_review}, where we include a brief survey of the existing literature pertinent to the topic this paper is concerned with -- in particular, we discuss in detail the Hypergraph Nim games that generalize a broad class of impartial games. In \S\ref{sec:graphs_main_results}, we describe, in detail, the graphs on which we study the game of Graph Nim in this paper, and the main results we establish pertaining to these underlying graphs. We draw the reader's attention to \S\ref{subsubsec:G_{4}_main_results} and \S\ref{subsubsec:H_{1}_main_results}, as these two subsections are dedicated to the statements of our main results concerning the two most challenging graphs we study in this paper, namely $G_{4}$ and $H_{1}$ of Figure~\ref{Consolidating_all_graphs}.  We dedicate \S\ref{sec:hypergraph} to a detailed discussion regarding the connection between our game of Graph Nim and the theory of Hypergraph Nim games developed by Boros et al., including two small characterization results (Propositions~\ref{prop:characterizing_graphs} and \ref{prop:hypergraph_to_graph}). The proofs of our main results concerning the graph $G_{4}$ have been provided in \S\ref{subsec:proofs_G_{4}}, and in \S\ref{sec:proof_6_vertices}, we outline the detailed proofs of \emph{most} of our findings regarding the graph $H_{1}$. This is followed by the Appendix, \S\ref{sec:appendix}, which begins with the detailed proofs of Theorem~\ref{thm:galaxy}, Proposition~\ref{prop:characterizing_graphs} and Proposition~\ref{prop:hypergraph_to_graph}. Next, the proofs of our main results concerning the graphs $F_{2}$, $G_{2}$ and $G_{3}$ shown in Figure~\ref{Consolidating_all_graphs} have been provided in \S\ref{sec:proof_4_vertices}, and finally, the rest of the Appendix has been dedicated to laying down the details omitted from the proofs provided in \S\ref{subsec:proofs_G_{4}} and \S\ref{sec:proof_6_vertices}. For instance, some steps of the proof of Theorem~\ref{thm:main_G_{4}} have been elaborated upon in \S\ref{subsec:appendix_G_{4}_losing_1} and \S\ref{subsec:appendix_G_{4}_losing_2}, Lemma~\ref{lem:H_{1}_winning_galaxy} has been proved in \S\ref{subsec:proof_lem_H_{1}_winning_galaxy}, and several of the claims made in Theorem~\ref{thm:H_{1}_losing}, concerning losing initial weight configurations on $H_{1}$, have been proved in \S\ref{subsec:H_{1}_losing_1_proof}, \S\ref{subsec:H_{1}_losing_3_proof}, \S\ref{subsec:H_{1}_losing_4_1,2_proof}, \S\ref{subsec:H_{1}_losing_4_3,4_proof}, \S\ref{subsec:H_{1}_losing_4_5,6_proof} and \S\ref{subsec:H_{1}_losing_4_7_proof}.

\section{The game of Graph Nim generalizes the game of Nim}\label{sec:prelim_useful_results}
We begin \S\ref{sec:prelim_useful_results} by introducing a notation that will be used throughout this paper: henceforth, given any $x\in\mathbb{N}$, we write $x=(a_{r}a_{r-1}\ldots a_{1}a_{0})_{2}$ to mean a base-$2$ representation of $x$, i.e.\ $a_{r}, a_{r-1}, \ldots, a_{1}, a_{0}$ are values in $\{0,1\}$ such that $x=\sum_{i=0}^{r}a_{i}2^{i}$. Notice that this includes the possibility that there exists some $r'<r$ such that $a_{r'}=1$ and $a_{i}=0$ for each $i\in\{r'+1,\ldots,r\}$ -- this simply means that we have placed additional $0$s to the left of the `minimal' base-$2$ representation of $x$ (i.e.\ the base-$2$ representation of $x$ that consists of the minimum number of digits), and this is done in order to render $x$ comparable with the base-$2$ representations of other positive integers.

We now state a well-known result (see, for instance, \cite{brualdi1977introductory}) that gives a complete description of the winning and losing configurations for the game of Nim defined in Definition~\ref{defn:nim}:
\begin{theorem}\label{thm:Nim}
Consider a game of Nim that begins with $k$ piles of chips, with the $i$-th pile containing $p_{i}$ chips, where $p_{1}, p_{2}, \ldots, p_{k}$ are positive integers. Let $p_{i}=\left(b_{i,s}b_{i,s-1}\ldots b_{i,1}b_{i,0}\right)_{2}$ for each $i \in \{1,2,\ldots,k\}$, where $s$ is chosen such that $b_{i,s}=1$ for at least one $i \in \{1,2,\ldots,k\}$. This configuration is losing if and only if it is \emph{balanced}, i.e.\ the sum $\sum_{i=1}^{k}b_{i,t}$ is even for each $t \in \{0,1,\ldots,s\}$.
\end{theorem}
Henceforth, whenever we talk about a $k$-tuple $(p_{1},p_{2},\ldots,p_{k})$, for any $k \in \mathbb{N}$ and any $p_{1},p_{2},\ldots,p_{k} \in \mathbb{N}_{0}$, being `balanced' or `unbalanced', we interpret $(p_{1},p_{2},\ldots,p_{k})$ as the initial configuration for a game of Nim involving $k$ piles of chips, with the $i$-th pile containing $p_{i}$ chips for each $i \in \{1,2,\ldots,k\}$.  

Theorem~\ref{thm:Nim} has consequences when it comes to the game of Graph Nim, as defined in Definition~\ref{defn:graph_nim}, played on a certain class of graphs that we henceforth refer to as \emph{galaxy graphs}. Before we are able to define the notion of galaxy graphs, we must introduce the notion of \emph{star graphs}, which is what we begin Definition~\ref{defn:star_galaxy} with. Furthermore, given a finite collection of finite graphs, $\{G_{1},G_{2},\ldots,G_{n}\}$, with $V_{i}$ denoting the set of vertices in $G_{i}$ for each $i \in \{1,2,\ldots,n\}$, we call this collection \emph{pairwise-vertex-disjoint} if $V_{i}\cap V_{j}=\emptyset$ for all $i, j \in \{1,2,\ldots,n\}$ with $i \neq j$. 
\begin{defn}\label{defn:star_galaxy}
A finite graph $G = (V,E)$ is called a \emph{star graph} if all its edges are incident on a common vertex $v_{0}$, i.e.\ there exists a $v_{0} \in V$ such that $\{v,v_{0}\} \in E$ for each $v \in V \setminus \{v_{0}\}$, and $\{u,v\} \notin E$ for each pair of distinct vertices $u,v \in V \setminus \{v_{0}\}$. We refer to $v_{0}$ as the \emph{centre} of the star graph (the centre is non-unique if and only if the star graph is just a single edge), and the edge $\{v,v_{0}\}$ as a \emph{ray} of the star graph for each $v \in V \setminus \{v_{0}\}$. 

We call a graph $G = (V,E)$ a \emph{galaxy graph} if it consists of a finite collection of pairwise-vertex-disjoint star graphs, i.e.\ $V$ can be partitioned into subsets $V_{1}, V_{2}, \ldots, V_{k}$, such that the induced subgraph of $G$ on $V_{i}$ is a star graph, and there exists no edge between $V_{i}$ and $V_{j}$ for any pair of distinct $i, j \in \{1,2,\ldots,k\}$. In other words, all connected components of a galaxy graph are star graphs.
\end{defn}
Note that, in particular, a graph consisting of a single edge is a star graph, and a graph consisting of a finite collection of pairwise-vertex-disjoint edges is a galaxy graph. Some examples of galaxy graphs have been illustrated in Figure~\ref{galaxy_graph}.
\begin{figure}[h!]
  \centering
    \includegraphics[width=0.5\textwidth]{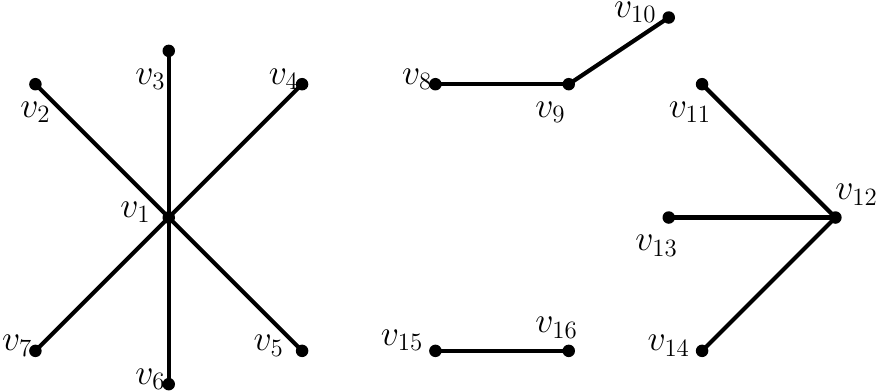}
\caption{An example of a galaxy graph}
  \label{galaxy_graph}
\end{figure}
We now state a result, pertaining to the game of Graph Nim played on galaxy graphs, that follows from Theorem~\ref{thm:Nim}:
\begin{theorem}\label{thm:galaxy}
Consider a galaxy graph $G$ that comprises $k$ pairwise-vertex-disjoint star graphs, the centres of which are $u_{1}, u_{2}, \ldots, u_{k}$. Let $u_{i,1}, \ldots, u_{i,\ell_{i}}$ denote the vertices that are adjacent to $u_{i}$, for each $i \in \{1,2,\ldots,k\}$. The initial weight configuration, $\mathbf{w}_{0} = \left(w_{0}\left(\left\{u_{i},u_{i,t}\right\}\right): t \in \{1,2,\ldots,\ell_{i}\}, i \in \{1,2,\ldots,k\}\right)$, is losing for the game of Graph Nim played on $G$ if and only if the $k$-tuple $(p_{1},p_{2},\ldots,p_{k})$, where $p_{i}=\sum_{t=1}^{\ell_{i}}w_{0}\left(\left\{u_{i},u_{i,t}\right\}\right)$ for each $i \in \{1,2,\ldots,k\}$, is balanced (as defined in Theorem~\ref{thm:Nim}).
\end{theorem}
The proof of this result has been postponed to the Appendix (\S\ref{sec:appendix}). Theorem~\ref{thm:galaxy} also makes it evident that a $k$-pile game of Nim can be viewed as a game of Graph Nim played on a galaxy graph consisting of $k$ pairwise-vertex-disjoint star graphs (or, in other words, consisting of $k$ connected components).

We state here yet another result, this one from \cite{williams2017combinatorial}, that is going to be of use to us repeatedly in the sequel:
\begin{theorem}[Page 13 of \cite{williams2017combinatorial}]\label{triangle_losing}
When $G$ is a triangle (i.e.\ $|V| = 3$ and $|E| = 3$), an initial weight configuration on $G$ is losing if and only if the edge-weight assigned to each of the three edges in $G$ is the same.
\end{theorem}

\section{Literature review}\label{subsec:lit_review}
The game of Nim has a history spanning centuries, with its origins possibly tracing back to China. It was also documented in European countries during the 16th century. 
Many variations and generalizations of the classic game of Nim have been studied since the foundational work \cite{bouton1901nim}, which introduced the nim-sum and provided a complete mathematical solution. \cite{moore1910generalization} extended this game to $\text{Nim}_k$, where a player may draw from up to $k$ piles out of $n$ piles of objects. Gurvich et al.\ \cite{gurvich2015slow} analyzes a class of nim games where players remove one token each from up to (or exactly) $k$ non-empty piles.  \cite{sprague1935mathematische} and \cite{grundy1939mathematics} independently showed that every impartial game under normal play is equivalent to a Nim heap, leading to the Sprague–Grundy theorem. 

The works most closely related to ours are \cite{calkin2010computing,low2016notes,brown2019graphnim}, which study Graph Nim games under the restriction that all edge weights are exactly one. These correspond to the unweighted variant of the Graph Nim games considered in our paper. Specifically, \cite{calkin2010computing} analyzes such games on paths and caterpillars, \cite{low2016notes} extends the analysis to a broader class of graphs, and \cite{brown2019graphnim} provides a complete characterization of winning strategies and losing positions for all graphs with four vertices.

\cite{fukuyama2003nim,fukuyamanim2003,erickson2011game} study a variant called Nim on graphs, where a token is moved along adjacent edges and positive integer weights are removed from those edges. Like the game of Nim, the player making the last move wins. \cite{fukuyama2003nim} analyzes strategies for various graphs, including bipartite graphs and multigraphs, while \cite{erickson2011game} considers both unit and arbitrary weights on graphs such as complete graphs, the Petersen graph, hypercubes, and even cycles. 

Other variants include Bounded Nim  (\cite{schwartz1971some}), where moves are limited by an upper bound; CHOMP  (\cite{gale1974curious}), a 2-player grid game; $n$-player Nim extensions (\cite{li1978n}); poisoned chocolate bar games (\cite{robin1989poisoned}); Greedy Nim (\cite{albert2004nim}), where players choose the largest heap; and $k$-bounded Greedy Nim (\cite{lv2018greedy}), which restricts the number of tokens that can be removed from a heap per turn.
Recent works have introduced further variants of Nim.  \cite{duchene2016building} proposed a two-stage version where players first build stacks by placing tokens, then play Nim on the resulting configuration.  \cite{XU20181} introduced Bounded Greedy Nim, combining constraints from both Bounded and Greedy Nim, and provided a complete solution.  \cite{van2022nim} studied Nim under imperfect information, computing Nash equilibria for several setups. 

An important class of two-player combinatorial games played on graphs, that serves as a common generalization for several families of impartial games, including the original Nim games, is the Hypergraph Nim (often denoted NIM$_{\mathcal{H}}$ when the underlying hypergraph is $\mathcal{H}$). The well-known Sprague-Grundy (SG) theory, that reduces any sum of impartial games to the classical Nim game, is generalized in \cite{boros2017tetris} by replacing the notion of `sum' with the notion of $\mathcal{H}$-combinations of impartial games for any hypergraph $\mathcal{H}$ -- in particular, \cite{boros2017tetris} introduces NIM$_{\mathcal{H}}$ as the $\mathcal{H}$-combination of single-pile classical Nim games. Given a hypergraph $\mathcal{H}$ on the set of vertices $V$, and a positive integer weight assigned to each vertex, a move in this game involves selecting a hyperedge and removing a positive integer weight from \emph{each} of the vertices contained in that hyperedge. In addition, \cite{boros2017tetris} reduces the $\mathcal{H}$-combination of SG-decreasing games to NIM$_{\mathcal{H}}$, and provides necessary and sufficient conditions on $\mathcal{H}$ for which NIM$_{\mathcal{H}}$ is SG-decreasing. A special case of the Hypergraph Nim, obtained when the underlying hypergraph $\mathcal{H}$ is $k$-uniform (i.e.\ where each hyperedge consists of $k$ vertices) has been studied in \cite{boros2018sprague}, and an explicit formula for the SG-function for the case of $2k\geqslant n$ has been proposed (where $n$ indicates the number of vertices in $\mathcal{H}$). In \cite{boros2019sprague}, hypergraphs $\mathcal{H}$ (referred to as JM hypergraphs) for which a certain explicit, closed form expression for the SG-function corresponding to NIM$_{\mathcal{H}}$, referred to as the JM-formula, is true, have been characterized. In \cite{boros2019symmetric}, symmetric JM hypergraphs (where $\mathcal{H}$ being \emph{symmetric} indicates that for every hyperedge $H$ of $\mathcal{H}$, every subset of size $|H|$ of the vertex set $V$ of $\mathcal{H}$ must also form a hyperedge of $\mathcal{H}$). Combinatorial properties of JM hypergraphs have been investigated and explicit formulae for the SG-functions for the conjunctive and selective compounds of the corresponding Hypergraph Nim games have been obtained in \cite{boros2019compound}. Finally, \cite{boros2024impartial} reduces hypergraph compounds of SG-decreasing games to hypergraph compounds of single-pile Nim games, and provide some necessary and some sufficient conditions for a hypergraph to be SG-decreasing.

A detailed discussion of the connection between Hypergraph Nim and the Graph Nim games studied in this paper is presented in \S\ref{sec:hypergraph}. Briefly, we show that every Graph Nim game studied in this paper can be naturally interpreted as a Hypergraph Nim game on a suitably constructed hypergraph. This interpretation enables a direct comparison between our results and the general theory developed in~\cite{boros2019sprague,boros2024impartial}. In this context, we show that the hypergraphs arising from the graph classes considered in this paper generally do not satisfy the structural assumptions required by the results of~\cite{boros2019sprague,boros2024impartial}, such as the property of being minimal transversal-free, total transversality, and the property of being SG-decreasing. Consequently, the results established in~\cite{boros2019sprague,boros2024impartial} are not directly applicable to the Graph Nim games studied here. %We also discuss the conceptual differences between the objectives and techniques of these works and identify a natural direction for future research at the interface of Hypergraph Nim and Graph Nim.

Alongside SG-functions, the iteratively defined Smith's remoteness function $R(x)$ (where $x$ indicates a position of the impartial game under consideration), introduced in \cite{smith1966graphs}, measures how quickly the player who moves first can force a win starting from $x$ when $x$ is an N-position, and how long the player who moves first can resist their ultimate defeat starting from $x$ when $x$ is a P-position. Polynomial algorithms for computing $R(x)$ for Euclid games and generalized Wythoff games have been provided in \cite{boros2024computing}, along with 
\begin{enumerate*}
\item an explicit formula for $R(x)$ when it is even, 
\item and a proof that the computation of $R(x)$ is NP-hard when it is odd,
\end{enumerate*}
both in case of Moore's Nim game. The remoteness function has also been studied in \cite{gurvich2024remoteness} (for the mis\`{e}re version of the exact slow $k$-nim game with $(k+1)$ piles) and \cite{gurvich2024screw} (for the normal version of the exact slow $k$-Nim), though neither of these specifically considers the Hypergraph Nim game.

\section{The graphs studied, and the main results of this paper}\label{sec:graphs_main_results}
We begin by listing all distinct unlabeled graphs comprising $4$ edges with no isolated vertices (i.e.\ a vertex with no neighbour), as for a graph with an isolated vertex, our game on such a graph reduces to the analysis on the corresponding subgraph with all isolated vertices removed. This, in turn, implies that the number of vertices is between $4$ and $8$ vertices (i.e.\ where $|E|=4$ and $4 \leqslant |V| \leqslant 8$). 

% A graph consisting of more than $8$ vertices and precisely $4$ edges is bound to have at least one isolated vertex (i.e.\ a vertex with no neighbour), and hence, the analysis of our game on such a graph reduces to the analysis on the corresponding subgraph with all isolated vertices removed.

Although we have stated above that it suffices to consider all unlabeled distinct graphs, the vertices of each graph in Figures~\ref{Consolidating_all_graphs} and \ref{fig:main_graphs} have been labeled for ease of exposition (when we state our results pertaining to these graphs). We exclude all graphs that contain at least one isolated vertex, since such graphs can be reduced to graphs with smaller number of vertices by eliminating all isolated vertices anyway. Likewise, in the entirety of this paper, we assume that each edge of the graph under consideration receives an initial weight that is a \emph{strictly} positive integer, i.e.\ each initial weight configuration we consider is such that no edge of the graph we are concerned with is left with weight $0$. This is an important assumption that is to be borne in mind while reading and interpreting the main results.

The graphs that we study in this paper have all been illustrated in Figure~\ref{Consolidating_all_graphs}. 
\begin{figure}[h!]
  \centering
    \includegraphics[width=0.7\textwidth]{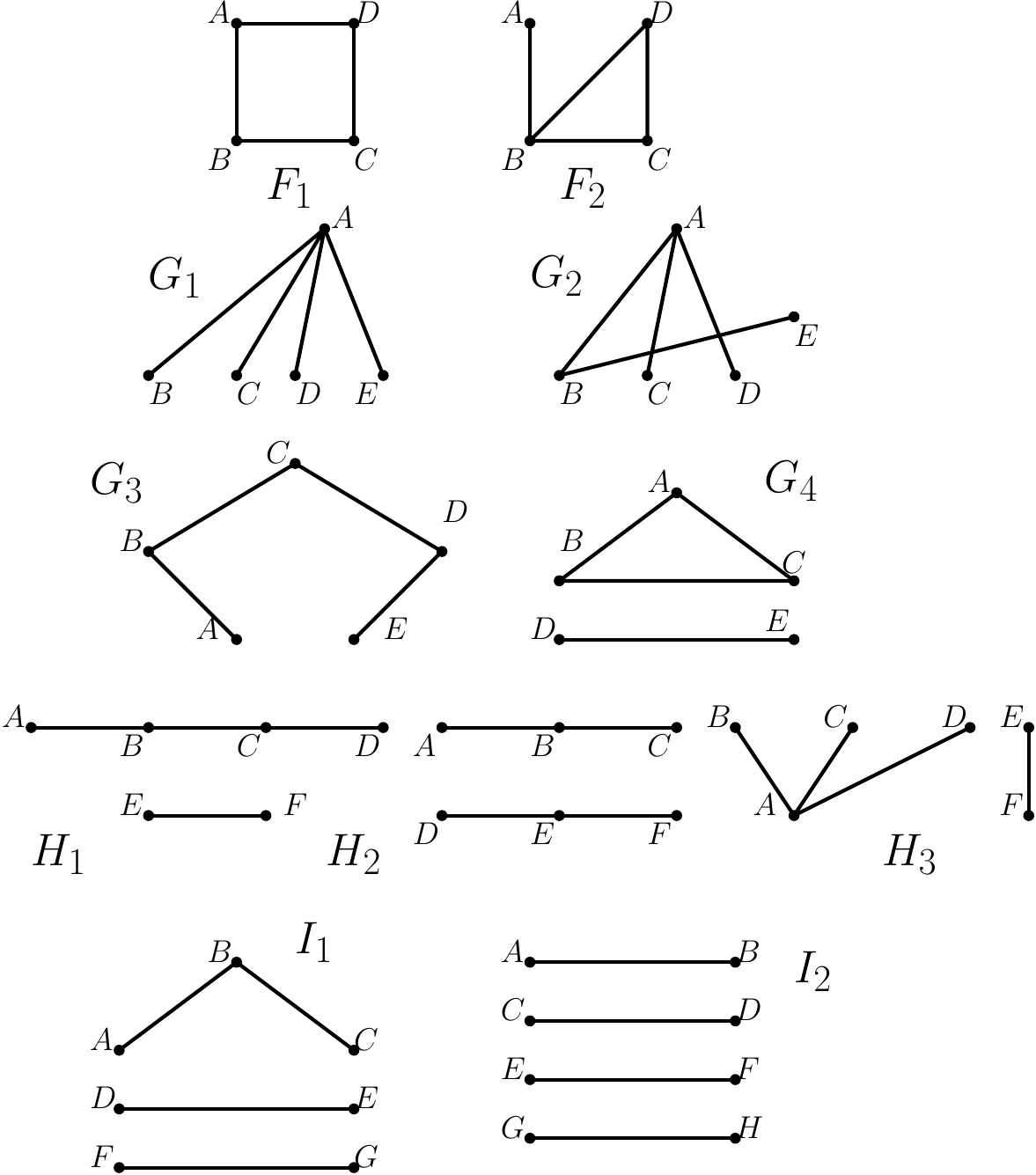}
\caption{All graphs with precisely $4$ edges and no isolated vertices}
  \label{Consolidating_all_graphs}
\end{figure}
Of these, the graphs $G_{1}$, $H_{2}$, $H_{3}$, $I_{1}$ and $I_{2}$ are all galaxy graphs, and the characterization the winning and losing initial weight configurations on these graphs follows directly from Theorem~\ref{thm:galaxy}. Moreover, the result characterizing the winning and losing initial weight configurations on $F_{1}$ was proved in \cite{williams2017combinatorial}, and we have included it in our paper for the sake of completeness. 
\begin{theorem}\label{thm:all_galaxy_graphs}
We enumerate here our results concerning all of the galaxy graphs illustrated in Figure~\ref{Consolidating_all_graphs}:
\begin{enumerate}
\item \emph{Every} initial weight configuration on $G_{1}$, with the sum of all edge-weights being strictly positive, is winning.
\item An initial weight configuration $(w_{0}(AB),w_{0}(BC),w_{0}(DE),w_{0}(EF))$ on $H_{2}$ is losing if and only if $w_{0}(AB)+w_{0}(BC)=w_{0}(DE)+w_{0}(EF)$.
\item An initial weight configuration $(w_{0}(AB),w_{0}(AC),w_{0}(AD),w_{0}(EF))$ on $H_{3}$ is losing if and only if $w_{0}(AB)+w_{0}(AC)+w_{0}(AD)=w_{0}(EF)$.
\item An initial weight configuration $(w_{0}(AB),w_{0}(BC),w_{0}(DE),w_{0}(FG))$ on $I_{1}$ is losing if and only if the triple $(w_{0}(AB)+w_{0}(BC),w_{0}(DE),w_{0}(FG))$ is balanced.
\item An initial weight configuration $(w_{0}(AB),w_{0}(CD),w_{0}(EF),w_{0}(GH))$ on $I_{2}$ is losing if and only if the tuple $(w_{0}(AB),w_{0}(CD),w_{0}(EF),w_{0}(GH))$ is balanced. 
\end{enumerate}
\end{theorem}
\begin{theorem}[Theorem 2 of \cite{williams2017combinatorial}]\label{thm:F_{1}}
\sloppy An initial weight configuration $\left(w_{0}(AB),w_{0}(BC),w_{0}(CD),w_{0}(DA)\right)$ on the graph $F_{1}$, illustrated in Figure~\ref{Consolidating_all_graphs}, is losing if and only if $w_{0}(AB)=w_{0}(CD)$ and $w_{0}(BC)=w_{0}(DA)$.
\end{theorem}

Having stated Theorem~\ref{thm:all_galaxy_graphs} and Theorem~\ref{thm:F_{1}}, we now come to the statements of the main results of this paper, most of which necessitate far more involved proofs. Before we state these, we include here Figure~\ref{fig:main_graphs}, which illustrate, for the reader's convenience, only those graphs (namely, $F_{2}$, $G_{2}$, $G_{3}$, $G_{4}$ and $H_{1}$) with $4$ edges, and up to $8$ vertices, that are neither galaxy graphs nor those already explored in \cite{williams2017combinatorial}.
\begin{figure}[h!]
  \centering
    \includegraphics[width=0.7\textwidth]{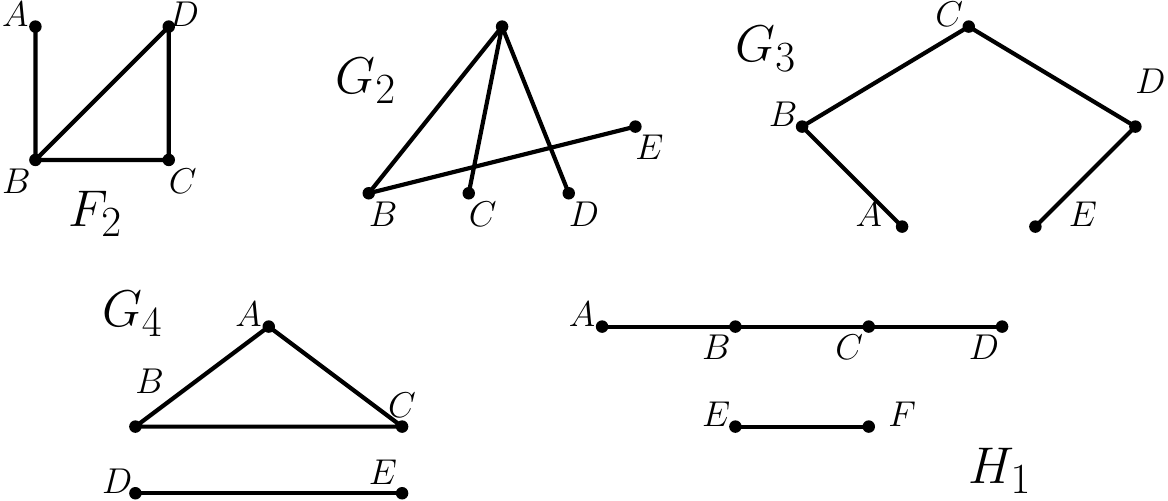}
\caption{The main non-galaxy graphs studied in this paper, with precisely $4$ edges and at most $8$ vertices}
  \label{fig:main_graphs}
\end{figure}
We continue by first stating our results pertaining to the graphs $F_{2}$, $G_{2}$ and $G_{3}$, as follows:
\begin{theorem}\label{thm:F_{2}}
An initial weight configuration $\left(w_{0}(AB),w_{0}(BC),w_{0}(CD),w_{0}(DB)\right)$ on the graph $F_{2}$, illustrated in Figure~\ref{fig:main_graphs}, is losing if and only if $w_{0}(BC)=w_{0}(DB)$ and $w_{0}(CD)=w_{0}(AB)+w_{0}(BC)$.
\end{theorem}
\begin{theorem}\label{thm:G_{2},G_{3}}
\emph{Every} initial weight configuration on each of the graphs $G_{2}$ and $G_{3}$, illustrated in Figure~\ref{fig:main_graphs}, is winning.
\end{theorem}
The two graphs that prove to be the most challenging when it comes to the task of characterizing the winning and losing initial weight configurations on them, are $G_{4}$ and $H_{1}$. We dedicate \S\ref{subsubsec:G_{4}_main_results} to the statements of all of our results pertaining to the graph $G_{4}$, and we dedicate \S\ref{subsubsec:H_{1}_main_results} to the statements of all of our results pertaining to the graph $H_{1}$. While we have been able to provide a complete characterization of the winning and losing initial weight configurations on $G_{4}$, our findings regarding $H_{1}$ provide a partial description of the winning and losing initial weight configurations on $H_{1}$.

\subsection{A complete characterization of the winning and losing configurations on $G_{4}$}\label{subsubsec:G_{4}_main_results}
We begin \S\ref{subsubsec:G_{4}_main_results} with some discussions on a couple of notions, as these are going to aid us in stating our results pertaining to the graph $G_{4}$ in a more concise manner than otherwise possible. We recall for the reader that by a \emph{multiset}, we mean an unordered collection of elements with repetitions allowed (in other words, a multiset would be reduced to a set if precisely one copy of each distinct element is retained). With a slight abuse of notation, which is nonetheless quite common in the literature, we express a multiset in the same manner as a set: by listing its (not necessarily distinct) elements within curly brackets. For example, $\{a, a, a, b, b, c\}$ represents a multiset in which the element $a$ appears three times, the element $b$ appears twice, and the element $c$ appears once. Given any element belonging to a multiset, the number of copies of this element that appear in that multiset is referred to as its \emph{repetition number}. Thus, the repetition number of $a$ in the multiset $\{a, a, a, b, b, c\}$ equals $3$, that of $b$ equals $2$ and that of $c$ equals $1$. We emphasize here that the ordering of the elements in a multiset is irrelevant: for instance, continuing with the same example as in the previous sentence, $\{a, a, a, b, b, c\}$ and $\{c, a, b, a, b, a\}$ represent the same multiset. The \emph{set obtained from a multiset} is simply the (unordered) collection of all distinct elements appearing in that multiset -- thus, the set obtained from the multiset $\{a,a,a,b,b,c\}$ is, simply, $\{a,b,c\}$. Two multisets are defined to be equal if and only if 
\begin{enumerate}
\item the sets obtained from these two multisets are the same,
\item for each distinct element, its repetition number in one multiset equals its repetition number in the other.
\end{enumerate}
% Thus, the multisets $\{1,2,2,2,3,4,4,5,5,5,5\}$ is the same as the multiset $\{2,5,2,5,2,5,1,5,4,3,4\}$, since the set obtained from each of these multisets is $\{1,2,3,4,5\}$, and in either of these multisets, the repetition number of the elements $1$, $2$, $3$, $4$ and $5$ equal $1$, $3$, $1$, $2$ and $4$ respectively. On the other hand, the multiset $\{1,2,2,2,3,4,4,5,5,5,5\}$ does \emph{not} equal the multiset $\{1,2,2,3,3,4,4,5,5,\}$, even though the set obtained from each of these multisets is, once again, $\{1,2,3,4,5\}$, since the repetition number of each of the elements $2$, $3$ and $5$ in the latter equals $2$.
 
Next, we introduce the following definition:
\begin{defn}\label{defn:special}
Given $k, m, i \in \mathbb{N}$ and $\ell \in \mathbb{N}_{0}$ such that
\begin{equation}\label{M_G_{4}}
i \in \{1,2,\ldots,m+1\} \quad \text{and} \quad \frac{m(m+1)}{2} \leqslant k \leqslant \frac{m(m+3)}{2},
\end{equation}
we say that a multiset $S$, consisting of $3$ elements in total, is \emph{$(k,\ell,m,i)$-special} if $S$, as a multiset, equals $\{k+1+(m+1)\ell,k+i+(m+1)\ell,k+m+2-i+(m+1)\ell\}$. We call a multiset $S$ simply \emph{special} if $S$ is $(k,\ell,m,i)$-special for \emph{some} $k, m, i \in \mathbb{N}$ and $\ell \in \mathbb{N}_{0}$ satisfying \eqref{M_G_{4}}.
\end{defn}
Some examples may help elucidate this definition further. If $m$ is even, and $i\in\{2,\ldots,m/2,m/2+2,\ldots,m+1\}$, then the elements $k+1+(m+1)\ell$, $k+i+(m+1)\ell$ and $k+m+2-i+(m+1)\ell$ are all distinct, and the multiset $S$ containing these elements becomes simply a set; on the other hand, for $m$ even, we have $k+1+(m+1)\ell=k+i+(m+1)\ell$ if $i=1$ and we have $k+i+(m+1)\ell=k+m+2-i+(m+1)\ell$ if $i=m/2+1$, so that the multiset $S$ either becomes $\{k+1+(m+1)\ell,k+1+(m+1)\ell,k+m+1+(m+1)\ell\}$, or it becomes $\{k+1+(m+1)\ell,k+m+1+(m+1)\ell,k+m+1+(m+1)\ell\}$. If $m$ is odd, and $i>1$, once again the multiset $S$ becomes simply a set, whereas if $i=1$, the multiset $S$ becomes $\{k+1+(m+1)\ell,k+1+(m+1)\ell,k+m+1+(m+1)\ell\}$.

It is worthwhile to note here that, since $(m+1)(m+2)/2-1=m(m+3)/2$ for each $m \in \mathbb{N}$, for each $k \in \mathbb{N}$, there exists a \emph{unique} $m \in \mathbb{N}$ such that \eqref{M_G_{4}} holds, and the sets $\{k \in \mathbb{N}: m(m+1) \leqslant 2k \leqslant m(m+3)\}$, for all $m \in \mathbb{N}$, yields a \emph{partition} of the set $\mathbb{N}$. 

Our result, characterizing \emph{fully} the winning and losing initial weight configurations on $G_{4}$, is as follows:
\begin{theorem}\label{thm:main_G_{4}}
An initial weight configuration $(w_{0}(AB),w_{0}(BC),w_{0}(CA),w_{0}(DE))$ on $G_{4}$ is losing if and only if precisely one of the following two conditions holds:
\begin{enumerate}[label=(A\arabic*), ref=A\arabic*]
\item \label{G_{4}_losing_cond_1} \sloppy letting $w_{0}(DE) = k$, there exist $m, i \in \mathbb{N}$ and $\ell \in \mathbb{N}_{0}$ such that \eqref{M_G_{4}} is satisfied and $\{w_{0}(AB),w_{0}(BC),w_{0}(CA)\}$, as a multiset, is $(k,\ell,m,i)$-special;
\item \label{G_{4}_losing_cond_2} $w_{0}(DE)=w_{0}(AB)+w_{0}(BC)+w_{0}(CA)$, the edge-weights $w_{0}(AB)$, $w_{0}(BC)$ and $w_{0}(CA)$ are not all equal, and the multiset $\{w_{0}(AB),w_{0}(BC),w_{0}(CA)\}$ is \emph{not} special.
\end{enumerate}
\end{theorem} 

\subsection{A partial characterization of the winning and losing configurations on $H_{1}$}\label{subsubsec:H_{1}_main_results}
The most challenging of all the graphs we have dealt with in this paper is $H_{1}$, and we are able to obtain only a partial characterization of the winning and losing initial weight configurations on $H_{1}$. The nature of our findings itself speaks for the difficulty that one faces as one attempts to generalize the patterns we have observed so far, and we include a brief discussion following Theorem~\ref{thm:H_{1}_losing} to further reinforce our belief that a \emph{full} characterization of the winning and losing configurations on $H_{1}$ would be a commendable, and possibly quite arduously accomplished, feat. Since our findings encompass \emph{several} different possible scenarios, we have included, for the convenience of the reader, a flowchart that summarizes these findings, in Figure~\ref{fig:flowchart_H_{1}} at the end of \S\ref{subsubsec:H_{1}_main_results}.

We begin by enumerating a few results pertaining to the winning initial weight configurations on $H_{1}$:
\begin{lemma}\label{lem:H_{1}_winning_1}
An initial weight configuration on $H_{1}$ is winning whenever it satisfies at least one of 
\begin{align}
{}&w_{0}(AB)\leqslant w_{0}(EF)\leqslant w_{0}(AB)+w_{0}(BC),\label{eq:H_{1}_winning_1}\\
{}&w_{0}(CD)\leqslant w_{0}(EF)\leqslant w_{0}(BC)+w_{0}(CD).\label{eq:H_{1}_winning_2}
\end{align}
\end{lemma}
A less obvious criterion that guarantees that a weight configuration on $H_{1}$ is winning is captured by our next result. Before we state it, we recall, for the reader, that any base-$2$ representation of $x \in \mathbb{N}$ is written as $x=(a_{r}a_{r-1}\ldots a_{1}a_{0})_{2}$, so that $x=\sum_{i=0}^{r}a_{i}2^{i}$, and wherever needed, we allow for $a_{r}$ to be equal to $0$ (for instance, $7=(111)_{2}$, but if needed, we may also write $7=(0111)_{2}$).
\begin{theorem}\label{thm:H_{1}_winning_1}
Let $w_{0}(EF)=(a_{s}a_{s-1}\ldots a_{1}a_{0})_{2}$, $w_{0}(AB)=(b_{s}b_{s-1}\ldots b_{1}b_{0})_{2}$ and $w_{0}(CD)=(c_{s}c_{s-1}\ldots c_{1}c_{0})_{2}$, where $s$ is chosen such that at least one of $a_{s}$, $b_{s}$ and $c_{s}$ equals $1$. We define $I$ to be the maximum element of the set $S$, where $S=\{i \in \{0,1,\ldots,s\}:a_{i}+b_{i}+c_{i} \text{ is odd}\}$, provided $S$ is non-empty. The initial weight configuration $(w_{0}(AB),w_{0}(BC),w_{0}(CD),w_{0}(EF))$ is winning whenever 
\begin{enumerate}[label=(B\arabic*), ref=B\arabic*]
\item $w_{0}(BC)\geqslant 1$ and $S$ is empty, 
\item \label{A2} or $w_{0}(BC) \geqslant 0$ and $S$ is non-empty and at least one of $b_{I}$ and $c_{I}$ equals $1$.
\end{enumerate}
\end{theorem}
Before we proceed with our next results, we set down the following notation which will be used repeatedly in the sequel: given $k \in \mathbb{N}$, we let $f(k)$ denote the unique non-negative integer that satisfies the inequalities $2^{f(k)}\leqslant k < 2^{f(k)+1}$. 
\begin{lemma}\label{prop:H_{1}_winning_1}
Any initial weight configuration $(w_{0}(AB),w_{0}(BC),w_{0}(CD),w_{0}(EF))$ on $H_{1}$ is winning as long as $w_{0}(EF)=k$, $w_{0}(AB)=2^{f(k)+1}m_{1}+\ell_{1}$ and $w_{0}(CD)=2^{f(k)+1}m_{2}+\ell_{2}$, for any $k \in \mathbb{N}$, any $m_{1}, m_{2} \in \mathbb{N}_{0}$, and any $\ell_{1}, \ell_{2} \in \{0,1,\ldots,2^{f(k)+1}-1\}$, such that 
\begin{enumerate*}
\item either $m_{1}\neq m_{2}$ 
\item or $\min\{\ell_{1},\ell_{2}\} \geqslant k$ 
\item or $k \in \{\ell_{1},\ell_{2}\}$ and either $\min\{\ell_{1},\ell_{2}\}>0$ or $w_{0}(BC)=0$. 
\end{enumerate*}
\end{lemma}
So far, we have only enumerated results pertaining to winning configurations on $H_{1}$. Our next result is a lengthy one, summarizing the various losing configurations on $H_{1}$ that we have, so far, been able to identify:
\begin{theorem}\label{thm:H_{1}_losing}
In what follows, we set $w_{0}(EF)=k$ for some $k \in \mathbb{N}$, and we let $m \in \mathbb{N}_{0}$ (making sure that none of the initial edge-weights equals $0$). Each of the following is true:
\begin{enumerate}%[label=(C\arabic*), ref=C\arabic*]
\item \label{B1} Any initial weight configuration on $H_{1}$ that is of the form
\begin{equation}
\{w_{0}(AB),w_{0}(CD)\}=\{2^{f(k)+1}m+r,2^{f(k)+1}m+k-r-s\} \quad \text{and} \quad w_{0}(BC)=s,\label{eq:H_{1}_losing_0,1,3}
\end{equation}
for $r \in \{0,1,3\}$ and $s \in \{1,2,\ldots,k-2r\}$, is losing on $H_{1}$. 
\item \label{B2} Any initial weight configuration on $H_{1}$ that is of the form
\begin{equation}
\{w_{0}(AB),w_{0}(CD)\}=\{2^{f(k)+1}m+r,2^{f(k)+1}m+k+r-s\} \quad \text{and} \quad w_{0}(BC)=s,\label{eq:H_{1}_losing_2,4_case_1}
\end{equation}
with $r \in \{2,4\}$, $s \in \{1,\ldots,r-1\}$ and $k \equiv j \bmod 2^{f(r)+1}$ for some $j \in \{s,s+1,\ldots,r-1\}$, is losing on $H_{1}$.
\item \label{B3} \sloppy Any initial weight configuration on $H_{1}$ that is of the form
\begin{equation}
\{w_{0}(AB),w_{0}(CD)\}=\{2^{f(k)+1}m+r,2^{f(k)+1}m+k-r-s\} \quad \text{and} \quad w_{0}(BC)=s,\label{eq:H_{1}_losing_2,4_case_2}
\end{equation}
with $r \in \{2,4\}$, $s \in \{1,\ldots,r-1\}$ and $k \geqslant 3r$ with $k \equiv j \bmod 2^{f(r)+1}$ for some $j \in \left\{0,1,\ldots,2^{f(r)+1}-1\right\}\setminus\{s,s+1,\ldots,r-1\}$, is losing on $H_{1}$.
\item \label{B4} Any initial weight configuration on $H_{1}$ that is of the form
\begin{equation}
\{w_{0}(AB),w_{0}(CD)\}=\{2^{f(k)+1}m+r,2^{f(k)+1}m+k-r-s\} \quad \text{and} \quad w_{0}(BC)=s,\label{eq:H_{1}_losing_2,4_case_3}
\end{equation}
with $r \in \{2,4\}$, $s \in \{r,r+1,\ldots,k-2r\}$ and $k\geqslant 3r$, is losing on $H_{1}$.
\end{enumerate}
\end{theorem}
Theorem~\ref{thm:H_{1}_losing} does seem to reveal a pattern -- however, this pattern does not seem to extend any further. For instance, as we explored configurations on $H_{1}$ that are of the form
\begin{equation}
w_{0}(AB)=2^{f(k)+1}m+5,\ w_{0}(BC)=s,\ w_{0}(CD)=2^{f(k)+1}m+t,\ w_{0}(EF)=k,\nonumber
\end{equation}
for $t\geqslant 5$ and $s\geqslant 1$, we found that setting $s=1$, $t=6$ and $k=11$ yields a losing configuration, even though this does not conform to the patterns suggested by Theorem~\ref{thm:H_{1}_losing}. Despite repeated attempts, we have not been able to extend Theorem~\ref{thm:H_{1}_losing} to encompass a broader class of losing configurations on $H_{1}$.

We include here two small results, of which Lemma~\ref{lem:H_{1}_winning_2} is a corollary of our claim, stated in Theorem~\ref{thm:H_{1}_losing}, that any configuration on $H_{1}$ of the form given by \eqref{eq:H_{1}_losing_0,1,3}, with $r=0$, is losing, whereas Lemma~\ref{lem:H_{1}_winning_galaxy} is a corollary of Theorem~\ref{thm:Nim} and Theorem~\ref{thm:galaxy} -- the motivation for stating these lies in being able to directly use their conclusions in various steps of the proof of Theorem~\ref{thm:H_{1}_losing}.
\begin{lemma}\label{lem:H_{1}_winning_2}
An initial weight configuration on $H_{1}$ is winning whenever we have $w_{0}(EF)=k$, $w_{0}(AB)=2^{f(k)+1}m_{1}+\ell_{1}$, $w_{0}(CD)=2^{f(k)+1}m_{2}+\ell_{2}$ and $w_{0}(BC)>k-\min\{\ell_{1},\ell_{2}\}$, for any $k \in \mathbb{N}$, any $m_{1}, m_{2} \in \mathbb{N}_{0}$, and any $\ell_{1}, \ell_{2} \in \{0,1,\ldots,2^{f(k)+1}-1\}$ with $\min\{\ell_{1},\ell_{2}\}<k$. Moreover, any initial weight configuration on $H_{1}$ with $w_{0}(BC)>w_{0}(EF)$ is also winning. 
\end{lemma}
\begin{remark}\label{rem:EF_edgeweight_0}
If $w_{0}(EF)=0$, then the corresponding initial weight configuration on $H_{1}$ is winning. This can be argued as follows: assuming, without loss of generality, that $w_{0}(AB)\geqslant w_{0}(CD)$, $P_{1}$ removes, in the first round, weight $w_{0}(CD)-w_{0}(AB)$ from the edge $CD$, and the entire edge $BC$, leaving $P_{2}$ with a galaxy graph consisting of the edges $AB$ and $CD$, with $w_{1}(AB)=w_{1}(CD)$. Consequently, $P_{2}$ loses by Theorem~\ref{thm:galaxy}.
\end{remark}
Because of the last assertion made in Remark~\ref{rem:EF_edgeweight_0}, we need not consider, in the proof of Theorem~\ref{thm:H_{1}_losing}, the scenario where $P_{1}$, in the first round, removes the entire edge $EF$ (since, in such cases, we already know that $P_{2}$ would win).
\begin{lemma}\label{lem:H_{1}_winning_galaxy}
Consider a galaxy graph consisting of three pairwise-vertex-disjoint star graphs, with the sum of the initial edge-weights assigned to the edges of the $i$-th star graph being equal to $w_{i}$, for $i \in \{1,2,3\}$. If $w_{3}=k$ for some $k \in \mathbb{N}$, and $w_{i}=2^{f(k)+1}m+\ell_{i}$ for some $m \in \mathbb{N}_{0}$ and $\ell_{i} \in \left\{0,1,\ldots,2^{f(k)+1}-1\right\}$ for each $i \in \{1,2\}$, then this configuration is winning whenever $\ell_{1}+\ell_{2}<k$.
\end{lemma}

As mentioned at the start of \S\ref{subsubsec:H_{1}_main_results}, for a quick summary of our main results regarding the winning and losing initial weight configurations on the graph $H_{1}$, we urge the reader to vide Figure~\ref{fig:flowchart_H_{1}}.
\begin{figure}[h!]
  \centering
    \includegraphics[width=0.85\textwidth]{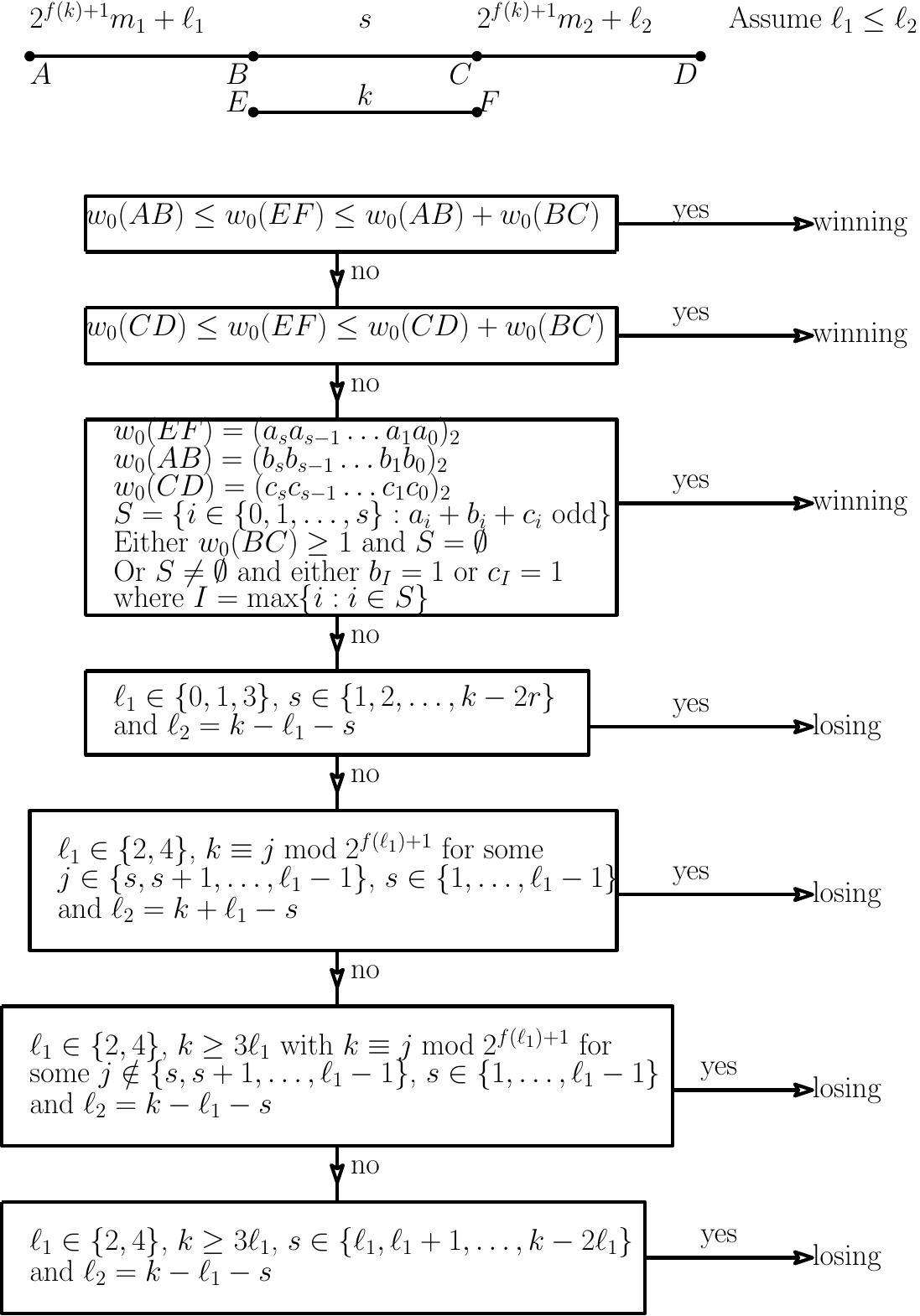}
\caption{A flowchart summarizing our findings regarding $H_{1}$}
  \label{fig:flowchart_H_{1}}
\end{figure}

\section{Connections with the Hypergraph Nim Theory of Boros et al.}\label{sec:hypergraph}
The game studied in this paper bears some connection to the Hypergraph Nim framework discussed briefly in \S\ref{subsec:lit_review}. Although every graph can be viewed as a $2$-uniform hypergraph, our game is \emph{not} obtained by simply restricting Hypergraph Nim to the class of $2$-uniform hypergraphs, since the admissible moves in the two games are different.

The Hypergraph Nim game is played on a hypergraph $\mathcal{H}$, with the set of vertices (or the ground set) given by $V$ and the set of hyperedges given by $\mathcal{E}\subset 2^{V}$, where each vertex represents a heap. A legal move consists of selecting a hyperedge $H\in\mathcal{E}$ and strictly decreasing every heap contained in $H$. Thus, every heap belonging to the selected hyperedge must be modified simultaneously. In contrast, the game studied in this paper is played on a graph whose \emph{edges} serve as the heaps. A move here consists of selecting a vertex of the graph and lowering, strictly, the weight of each edge belonging to \emph{some} arbitrary non-empty subset of the edges incident on that vertex. In particular, upon selecting a vertex, a player is not compelled to lower, strictly, the weight of each and every edge incident on that vertex -- it suffices for them to ensure that at least one edge incident on that vertex has its edge-weight strictly lessened. Although our game differs from Hypergraph Nim played on the same graph (interpreted as a $2$-uniform hypergraph), it can naturally be interpreted as a Hypergraph Nim game on a \emph{different}, suitably defined hypergraph: for graph $G=(V,E)$, let
\[
N(v)=\{e\in E:e \text{ is incident on }v\}.
\]
Then the Graph Nim game studied in this paper is precisely the Hypergraph Nim game on the hypergraph
\begin{equation}\label{hypergame}
\mathcal{H}_{G}
=
\bigcup_{v\in V}
\left(
2^{N(v)}
\setminus\{\emptyset\}
\right).
\end{equation}

This observation naturally raises the question of whether the general results of Boros et al.~\cite{boros2019sprague,boros2024impartial} can be applied to the graph-induced hypergraphs defined by~\eqref{hypergame}. Before addressing this question, we briefly summarize the principal results of these two papers.

\subsection{Main results of Boros et al.}
%The paper~\cite{boros2019sprague} studies Hypergraph Nim, in which a move consists of selecting a hyperedge and decreasing the size of the heap corresponding to \emph{each} of the vertices that form the chosen hyperedge. 
The main objective of \cite{boros2019sprague} is to identify classes of hypergraphs for which the Sprague--Grundy function admits an explicit closed-form expression, known as the \emph{JM formula}. To this end, the authors introduce the class of \emph{JM hypergraphs}, namely those hypergraphs for which the Sprague--Grundy function of the corresponding Hypergraph Nim is given by the JM formula. They identify the following three structural conditions on a hypergraph:
\begin{enumerate}
    \item[(A1)]\label{A1} the hypergraph is \emph{minimal transversal-free};
    \item[(D1)]\label{D1} every pair of hyperedges is connected by a chain of pairwise-intersecting hyperedges satisfying an exchange property; and
    \item[(D2)] every position admits a move that decreases the height by one while reducing each modified heap by at most one.
\end{enumerate}
These conditions are shown to be sufficient, while Conditions \eqref{A1} and \eqref{D1} are shown to be necessary for the hypergraph to be a JM hypergraph; \cite{boros2019sprague} also establishes several equivalent characterizations and important subclasses of JM hypergraphs.

Subsequently, \cite{boros2024impartial} develops a general theory of \emph{hypergraph compounds} of impartial games. Given a collection of impartial games together with a hypergraph specifying which component games are played simultaneously, the paper investigates when the Sprague--Grundy function of the compound game can be determined solely from the Sprague--Grundy values of the component games. The central concept introduced is that of an \emph{SG-decreasing game}. The authors prove that if each component game is SG-decreasing and the underlying hypergraph is totally transversal, then the Sprague--Grundy function of the hypergraph compound depends only on the Sprague--Grundy values of the component games. This framework unifies several classical constructions, including conjunctive, selective, and disjunctive compounds, within a common hypergraph setting.

\subsection{Applicability of the Results of Boros et al.}
The hypergraphs associated, via \eqref{hypergame}, with the graphs studied in this paper, apart from the one associated with the graph $G_{1}$ illustrated in Figure~\ref{Consolidating_all_graphs}, are neither minimal transversal-free nor SG-decreasing. Consequently, the explicit Sprague--Grundy characterization for JM hypergraphs established in~\cite{boros2019sprague} and the results on SG-decreasing hypergraph compounds in~\cite{boros2024impartial} are not directly applicable to the Graph Nim games studied in this paper.

%Since the results for nim games on galaxy graphs follow from the existing results (Theorem~\ref{thm:Nim}) with the observation that these games are basically games of Nim with $k$ piles, it suffices to consider the remaining graph classes studied in this paper. These graphs are shown in Figure~\ref{fig:main_graphs}. 

For the sake of completeness, we first recall the relevant definitions from~\cite{boros2019sprague,boros2024impartial}. Let $\mathcal{H}\subseteq 2^V$ be a hypergraph (in other words, the set of vertices of $\mathcal{H}$ is given by $V$, while $\mathcal{H}$ itself indicates the hyperedges that it is made up of, so that $\mathcal{H}$ can also be visualized as a subset of the power set of $V$, with each element $H$ of $\mathcal{H}$ indicating a subset of $V$), and let $S\subseteq V$. The induced subhypergraph on $S$ is given by the collection of all those hyperedges of $\mathcal{H}$ each of which consists of vertices belonging to $S$, i.e.\ 
\[
\mathcal{H}\big|_S=\{H\in\mathcal{H}:H\subseteq S\}.
\]
A set $T\subseteq V$ is called a \emph{transversal} of $\mathcal{H}$ if $T\cap H\neq\emptyset$ for every hyperedge $H\in\mathcal{H}$. The hypergraph $\mathcal{H}$ is said to be \emph{transversal-free} if no hyperedge of $\mathcal{H}$ is itself a transversal of $\mathcal{H}$. It is called \emph{minimal transversal-free} if it is transversal-free and every non-empty proper induced subhypergraph is \emph{not} transversal-free. For instance, the hypergraph $\mathcal{H}_{F_{2}}$ corresponding to the graph $F_2$ shown in Figure~\ref{fig:main_graphs} is
\[
\{\{AB\},\{BC\},\{CD\},\{BD\},\{AB,BC\},\{BC,BD\},\{BC,CD\},\{CD,BD\},\{AB,BC,BD\}\}.
\]
A routine case-by-case examination reveals that no hyperedge of $\mathcal{H}_{F_{2}}$ intersects every other hyperedge of $\mathcal{H}_{F_{2}}$. Hence, no hyperedge of $\mathcal{H}_{F_{2}}$ is a transversal of $\mathcal{H}_{F_{2}}$, making it transversal-free. However, as is evident from Proposition~\ref{prop:characterizing_graphs}, $\mathcal{H}_{F_{2}}$ is not minimal transversal-free.

A hypergraph $\mathcal{H}$ is said to be \emph{totally transversal} if every non-empty induced subhypergraph contains a hyperedge that intersects every other hyperedge. Equivalently,
\[
\forall\,S\subseteq V,\ \mathcal{H}\big|_S\neq\emptyset
\quad\Longrightarrow\quad
\exists\,H\in\mathcal{H}\big|_S
\text{ such that }
H\cap H'\neq\emptyset
\text{ for every }
H'\in\mathcal{H}\big|_S.
\]
A hypergraph $\mathcal{H}$ is said to be \emph{SG-decreasing} if every single legal move of the Hypergraph Nim played on $\mathcal{H}$ strictly lowers the Sprague-Grundy value -- in other words, if $\overline{x}=(x_{1},\ldots,x_{n})$ and $\overline{y}=(y_{1},\ldots,y_{n})$ indicate positions of this game (so that $x_{i}$ (respectively, $y_{i}$) is the size of the heap represented by the $i$-th vertex of $\mathcal{H}$) such that a single legal move can take the game from $\overline{x}$ to $\overline{y}$, then $\SG(\overline{x})>\SG(\overline{y})$, where $\SG(\overline{x})$ indicates the Sprague-Grundy value of the position $\overline{x}$. 

%Table~\ref{tab:hypergraph-properties} summarizes, for each non-galaxy graph considered in this paper, whether the corresponding hypergraph satisfies the above two properties. 

At this point, instead of individually investigating the hypergraphs corresponding to the non-galaxy graphs studied in this paper, we propose and prove a characterization of the graphs for which the corresponding hypergraphs are minimal transversal-free or SG-decreasing. For the reader's convenience, we recall here that, given any simple, finite graph $G$ whose set of vertices is indicated by $V(G)$ and whose set of edges is indicated by $E(G)$, we define the hypergraph $\mathcal{H}_{G}$ whose set of vertices is given by $E(G)$ and whose set of hyperedges (which, by a slight abuse of notation, we also indicate by $\mathcal{H}_{G}$) is given by
\begin{equation}
\mathcal{H}_{G}=\bigcup_{v\in V(G)}\left(2^{N(v)}\setminus\emptyset\right), \text{ where }N(v)=\{e\in E(G):e \text{ incident on }v\},\label{N(v)_defn}
\end{equation}
as described above in \eqref{hypergame}. Our characterization can now be stated as follows:
\begin{prop}\label{prop:characterizing_graphs}
Up to the presence of isolated vertices, the only finite, simple graphs $G$ corresponding to which the hypergraphs $\mathcal{H}_{G}$ are minimal transversal-free are
\begin{enumerate*}
\item the graph consisting of two vertex-disjoint edges, and
\item the triangle.
\end{enumerate*}
Up to the presence of isolated vertices, the only finite, simple graphs $G$ corresponding to which the hypergraphs $\mathcal{H}_{G}$ are SG-decreasing are the star graphs as described in Definition~\ref{defn:star_galaxy}.
\end{prop}
We include here yet another result that addresses, in some sense, the reverse transformation (i.e.\ going from a hypergraph to a graph via \eqref{hypergame}). To this end, for any finite hypergraph $\mathcal{H}$ defined on the ground set or vertex set $V$ and \emph{not} containing the empty set as a hyperedge, we call a hyperedge $H$ \emph{maximal} if there exists no proper superset of $H$ that is also a hyperedge of $\mathcal{H}$, i.e.\
\begin{equation}
\text{there exists no }H' \text{ such that }H\subsetneq H'\in 2^{V} \text{ and }H'\in\mathcal{H}.\nonumber
\end{equation}
\begin{prop}\label{prop:hypergraph_to_graph}
Given any finite hypergraph $\mathcal{H}$, there exists a simple, finite graph $G$ such that $\mathcal{H}=\mathcal{H}_{G}$ (in the sense of \eqref{hypergame}) if and only if
\begin{enumerate}[label=(E\arabic*), ref=E\arabic*]
\item \label{E1} for every hyperedge $H\in\mathcal{H}$ and every non-empty subset $H'$ of $H$, we have $H'\in\mathcal{H}$ as well,
\item \label{E2} any two maximal hyperedges of $\mathcal{H}$ can have at most one vertex in common,
\item \label{E3} and every vertex belonging to the ground set on which $\mathcal{H}$ is defined belongs to either one or two maximal hyperedge(s) of $\mathcal{H}$.
\end{enumerate} 
\end{prop}
The proof of each of Proposition~\ref{prop:characterizing_graphs} and Proposition~\ref{prop:hypergraph_to_graph} has been postponed to the Appendix (\S\ref{sec:appendix}).

With regard to the SG-decreasing property, we make one further observation: the Graph Nim games played on $F_2$, $G_4$, and $H_1$ admit both $\mathcal{P}$-positions and $\mathcal{N}$-positions among their intermediate positions (weight configurations), and hence cannot be SG-decreasing. Although every weight configuration considered in this paper for $G_3$, namely those with strictly positive edge weights, is an $\mathcal{N}$-position, this property does not extend to all intermediate positions. In particular, when zero edge weights are permitted, configurations such as
\[
(w_0(AB),w_0(BC),w_0(CD),w_0(DE))=(1,0,1,0)
\]
and
\[
(w_0(AB),w_0(BC),w_0(CD),w_0(DE))=(2,0,1,1)
\]
are $\mathcal{P}$-positions. Consequently, the Graph Nim game on $G_3$, when extended to allow zero edge weights, is also not SG-decreasing.

\subsection{Further remarks and future directions}
We further note that the objectives of the papers are also different. Boros et al.~\cite{boros2024impartial} develop a general theory of hypergraph compounds of impartial games and introduce the notion of SG-decreasing games, while Boros et al.~\cite{boros2019sprague} study Hypergraph Nim on important classes of hypergraphs, including JM hypergraphs, for which explicit Sprague--Grundy formulas can be obtained. In contrast, our objective is to introduce a new graph-based Nim game and characterize its $\mathcal{P}$-positions for several natural graph classes. Consequently, our proofs rely primarily on graph-theoretic arguments and the combinatorial interaction among the edge weights, rather than on structural properties of hypergraphs.

The comparison with Hypergraph Nim also leads to an interesting question. Instead of interpreting Graph Nim through the hypergraph $\mathcal{H}_{\mathrm{GN}}$ defined in~\eqref{hypergame}, suppose that the underlying graph itself is viewed as a $2$-uniform hypergraph and the Hypergraph Nim move rule of~\cite{boros2019sprague,boros2024impartial} is applied. In other words, a player selects a vertex of the graph and is required to decrease the weights on \emph{all} edges incident on that vertex. A natural problem is then to characterize the $\mathcal{P}$-positions for the graph classes studied in this paper under this alternative game rule.

Our preliminary investigation suggests that the existing results of~\cite{boros2019sprague,boros2024impartial} do not directly resolve this problem. Indeed, for several of the graph classes considered in this paper, the corresponding $2$-uniform hypergraphs do not appear to satisfy the structural assumptions underlying the available characterizations, such as those defining JM or SG-decreasing hypergraphs. Consequently, these graph-induced Hypergraph Nim games seem to require additional analysis. We believe that this constitutes an interesting direction for future research.

\section{Proofs of our results pertaining to the graph $G_{4}$}\label{subsec:proofs_G_{4}}
We establish Theorem~\ref{thm:main_G_{4}} via the proofs of several results that have been enumerated below as a list for the reader's convenience:
\begin{enumerate}[label=(D\arabic*), ref=D\arabic*]
\item\label{lem:G_{4}_winning_1}
An initial weight configuration $(w_{0}(AB), w_{0}(BC), w_{0}(CA), w_{0}(DE))$ on $G_{4}$ is winning whenever it satisfies the inequalities $\min\{w_{0}(AB),w_{0}(BC),w_{0}(CA)\} \leqslant w_{0}(DE) \leqslant w_{0}(AB)+w_{0}(BC)+w_{0}(CA)-\min\{w_{0}(AB),w_{0}(BC),w_{0}(CA)\}$.

\item\label{thm:G_{4}_losing_1}
An initial weight configuration $\left(w_{0}(AB),w_{0}(BC),w_{0}(CA),w_{0}(DE)\right)$ on the graph $G_{4}$ is losing whenever the criterion stated in \eqref{G_{4}_losing_cond_1} is satisfied.

\item\label{thm:G_{4}_winning_1}
If $0 < w_{0}(DE) < \min\left\{w_{0}(AB), w_{0}(BC), w_{0}(CA)\right\}$ but the weight configuration does not satisfy the criterion stated in part \eqref{G_{4}_losing_cond_1} of Theorem~\ref{thm:main_G_{4}}, then $P_{1}$ wins.

\item\label{thm:G_{4}_losing_2}
Any initial weight configuration $\left(w_{0}(AB),w_{0}(BC),w_{0}(CA),w_{0}(DE)\right)$ on the graph $G_{4}$, that satisfies the criterion stated in part \eqref{G_{4}_losing_cond_2} of Theorem~\ref{thm:main_G_{4}}, is losing.

\item\label{thm:G_{4}_winning_2}
\sloppy If $w_{0}(DE)>w_{0}(AB)+w_{0}(BC)+w_{0}(CA)-\min\{w_{0}(AB),w_{0}(BC),w_{0}(CA)\}$, but $w_{0}(DE) \neq w_{0}(AB)+w_{0}(BC)+w_{0}(CA)$, and the multiset $\{w_{0}(AB),w_{0}(BC),w_{0}(CA)\}$ is \emph{not} special, the initial weight configuration $\left(w_{0}(AB),w_{0}(BC),w_{0}(CA),w_{0}(DE)\right)$ is winning on $G_{4}$.
\end{enumerate}
It is fairly straightforward to see that, once each of \eqref{lem:G_{4}_winning_1}, \eqref{thm:G_{4}_losing_1}, \eqref{thm:G_{4}_winning_1}, \eqref{thm:G_{4}_losing_2} and \eqref{thm:G_{4}_winning_2} has been proved, the proof of Theorem~\ref{thm:main_G_{4}} would be complete.

\subsection{Proof of Step~\ref{lem:G_{4}_winning_1} for proving Theorem~\ref{thm:main_G_{4}}}
Without loss of generality (since the edges $AB$, $BC$ and $CA$ play symmetric roles in the graph $G_{4}$), let us assume that $w_{0}(AB)\leqslant w_{0}(BC)\leqslant w_{0}(CA)$, so that the inequalities stated in the hypothesis of \eqref{lem:G_{4}_winning_1} boil down to: $w_{0}(AB) \leqslant w_{0}(DE) \leqslant w_{0}(BC)+w_{0}(CA)$. We split the analysis into two cases. 

If $w_{0}(AB)\leqslant w_{0}(DE)\leqslant w_{0}(AB)+w_{0}(CA)$, then $P_{1}$ selects the vertex $C$, removes the entire edge $BC$, and removes weight $w_{0}(AB)+w_{0}(CA)-w_{0}(DE)$ from the edge $CA$. This leaves $P_{2}$ with a galaxy graph comprising two components: the first of which consists of the edges $AB$ and $CA$, the second of which consists of only the edge $DE$, and $w_{1}(CA)=w_{0}(CA)-\{w_{0}(AB)+w_{0}(CA)-w_{0}(DE)\}=w_{0}(DE)-w_{0}(AB)=w_{1}(DE)-w_{1}(AB)$. Consequently, $P_{2}$ loses by Theorem~\ref{thm:galaxy}.

If $w_{0}(AB)+w_{0}(CA) < w_{0}(DE) \leqslant w_{0}(BC)+w_{0}(CA)$, then $P_{1}$ removes weight $w_{0}(BC)+w_{0}(CA)-w_{0}(DE)$ from the edge $BC$, and the entire edge $AB$, in the first round, leaving $P_{2}$ with a galaxy graph comprising two components: one is a star graph with the edges $BC$ and $CA$, the other is the single edge $DE$, such that $w_{1}(BC)+w_{1}(CA)=w_{0}(DE)=w_{1}(DE)$. Once again, $P_{2}$ loses by Theorem~\ref{thm:galaxy}. \qed

\subsection{Proof of Step~\ref{thm:G_{4}_losing_1} for proving Theorem~\ref{thm:main_G_{4}}}
The proof of Step~\eqref{thm:G_{4}_losing_1} is accomplished via induction, which happens with respect to the parameters $m$, $k$ and $\ell$ involved in the statement of \eqref{G_{4}_losing_cond_1}. The base case, corresponding to $m=k=1$ and all $\ell \in \mathbb{N}_{0}$, has been proved in \S\ref{subsec:appendix_G_{4}_losing_1} of \S\ref{sec:appendix}. Suppose, now, that for some $M \in \mathbb{N}$, some $K \in \mathbb{N}$ satisfying the inequalities 
\begin{equation}
\frac{(M+1)(M+2)}{2} \leqslant K \leqslant \frac{(M+1)(M+4)}{2},\label{K_M+1_inequalities}
\end{equation}
and some $L \in \mathbb{N}_{0}$, we have already shown that
\begin{enumerate}[label=(I\arabic*), ref=I\arabic*]
\item \label{G_{4}_ind_hyp_1} each configuration satisfying the hypothesis of \eqref{G_{4}_losing_cond_1}, with $m \leqslant M$ and $m$, $k$ and $i$ satisfying \eqref{M_G_{4}}, is losing;
\item \label{G_{4}_ind_hyp_2} each configuration satisfying the hypothesis of \eqref{G_{4}_losing_cond_1}, with $m=M+1$, all $k \in \mathbb{N}$ such that $(M+1)(M+2)/2 \leqslant k < K$, all $i \in \{1,2,\ldots,M+2\}$, and all $\ell \in \mathbb{N}_{0}$, is losing;
\item \label{G_{4}_ind_hyp_3} the configuration satisfying the hypothesis of \eqref{G_{4}_losing_cond_1}, with $k=K$, $m=M+1$, all $i \in \{1,2,\ldots,m+1\}$, and all $\ell \in \mathbb{N}_{0}$ with $\ell < L$, is losing.
\end{enumerate}
Note that the induction hypothesis stated in \eqref{G_{4}_ind_hyp_3} becomes vacuous (i.e.\ no longer needs to be considered) when $L=0$. In the proof that follows (of showing that the configuration given by \eqref{G_{4}_losing_inductive_eq_1} is losing), we incorporate the possibility of $L=0$. In other words, we do not separately prove that \eqref{G_{4}_losing_inductive_eq_1} is losing for $L=0$ -- rather, we show that the proof for $L=0$ is a subset of the proof for $L \in \mathbb{N}$, obtained by simply \emph{not} considering the case where $P_{1}$ removes weights from at most two of $AB$, $BC$ and $CA$ in the first round such that $K+1+(M+2)L > \min\{w_{1}(AB),w_{1}(BC),w_{1}(CA)\} > K$. 

Likewise, the induction hypothesis in \eqref{G_{4}_ind_hyp_2} becomes vacuous when $K=(M+1)(M+2)/2$ -- this, too, has been taken into account in the proof that follows, and has not been treated separately. Specifically, in this case, if, after the second round, the edge-weight of $DE$ decreases strictly from what it was at the beginning, i.e.\ $w_{2}(DE)=k$ for some $k < K$, we have $k \leqslant (M+1)(M+2)/2-1=M(M+3)/2$. Consequently, $k$ satisfies \eqref{M_G_{4}} for some $m \leqslant M$, and hence, the induction hypothesis in \eqref{G_{4}_ind_hyp_1} can be applied.

We now come to the inductive step of our proof, and focus on the initial weight configuration
\begin{align}
{}&w_{0}(AB)=K+1+(M+2)L,\ w_{0}(BC)=K+i+(M+2)L,\nonumber\\
{}&w_{0}(CA)=K+M+3-i+(M+2)L,\ w_{0}(DE)=K,\label{G_{4}_losing_inductive_eq_1}
\end{align}
where $i \in \{1,2,\ldots,M+2\}$. Here, we have assumed, without any loss of generality, that $w_{0}(AB) \leqslant w_{0}(BC) \leqslant w_{0}(CA)$, which yields $2i \leqslant (M+3)$. We consider the first round of the game played on this configuration:
\begin{enumerate}
\item Suppose $P_{1}$ removes some weight from at most two of the edges $AB$, $BC$ and $CA$ in the first round, such that $\min\{w_{1}(AB),w_{1}(BC),w_{1}(CA)\}=K+1+(M+2)L$. This happens if and only if $P_{1}$ keeps the weight of $AB$ unchanged, while she removes a positive integer weight from at least one of $BC$ and $CA$. Consequently, we can write $w_{1}(BC)=K+j_{1}+(M+2)L$ and $w_{1}(CA)=K+j_{2}+(M+2)L$, for some $1 \leqslant j_{1} \leqslant i$ and $1 \leqslant j_{2} \leqslant M+3-i$ with $j_{1}+j_{2} \leqslant M+2$. Defining $m=j_{1}+j_{2}-2$, we thus have $m \leqslant M$. Moreover, the lower bounds of $j_{1} \geqslant 1$ and $j_{2} \geqslant 1$, along with $j_{1}+j_{2}=(m+2)$, implies that each of $j_{1}$ and $j_{2}$ is bounded above by $(m+1)$. Let $k$ be the unique positive integer such that
\begin{equation}
\frac{m(m+1)}{2} \leqslant k \leqslant \frac{m(m+3)}{2} \quad \text{and} \quad K+(M+2)L \equiv k \bmod (m+1).\label{k_m_inequalities}
\end{equation}
Using the above-mentioned assertion that $m \leqslant M$, the former of the two criteria stated in \eqref{k_m_inequalities}, and \eqref{K_M+1_inequalities}, we conclude that 
\begin{equation}
k \leqslant \frac{m(m+3)}{2} \leqslant \frac{M(M+3)}{2} < \frac{(M+1)(M+2)}{2} \leqslant K.\nonumber
\end{equation}
The second criterion of \eqref{k_m_inequalities} implies the existence of some $\ell \in \mathbb{N}_{0}$ such that $K+(M+2)L=k+(m+1)\ell$. In the second round, $P_{2}$ removes the strictly positive edge-weight $(K-k)$ from the edge $DE$, leaving $P_{1}$ with a configuration where $w_{2}(DE)=k$ and 
\begin{align}
{}&w_{2}(AB)=K+1+(M+2)L=k+1+(m+1)\ell, \nonumber\\
{}&w_{2}(BC)=K+j_{1}+(M+2)L=k+j_{1}+(m+1)\ell,\nonumber\\
{}&w_{2}(CA)=K+j_{2}+(M+2)L=k+j_{2}+(m+1)\ell=k+m+2-j_{1}+(m+1)\ell,\nonumber
\end{align} 
which satisfies the hypothesis of \eqref{G_{4}_losing_cond_1} (with \eqref{M_G_{4}} ensured by the first condition stated in \eqref{k_m_inequalities}). As $m \leqslant M$, hence $P_{1}$ loses by our induction hypothesis in \eqref{G_{4}_ind_hyp_1}. 

\item The second scenario, which needs to be considered only when $L$ is a \emph{positive} integer, is as follows: suppose $P_{1}$ removes some weight from at most two of the edges $AB$, $BC$ and $CA$ in the first round, such that 
\begin{equation}
K+1+(M+2)L > \min\{w_{1}(AB),w_{1}(BC),w_{1}(CA)\} > K.\label{triangle_edgeweight_>K}
\end{equation}
\sloppy Let $r$ be the unique element in $\{1,2,\ldots,M+2\}$ such that $\min\{w_{1}(AB),w_{1}(BC),w_{1}(CA)\}-K \equiv r \bmod (M+2)$, which implies the existence of some $\ell \in \mathbb{N}_{0}$ such that $\min\{w_{1}(AB),w_{1}(BC),w_{1}(CA)\}=K+r+(M+2)\ell$. Note, using the first inequality of \eqref{triangle_edgeweight_>K}, that
\begin{equation}
K+1+(M+2)L > K+r+(M+2)\ell \quad \text{and} \quad r \geqslant 1 \implies \ell \leqslant (L-1).\label{ell_less_L_eq_1}
\end{equation}
This, in turn, shows that 
\begin{align}
K+M+3-r+(M+2)\ell \leqslant{}& K+M+3-r+(M+2)(L-1)=K+1-r+(M+2)L\nonumber\\
{}&< K+1+(M+2)L=\min\{w_{0}(AB),w_{0}(BC),w_{0}(CA)\}.\label{less_than_minimum_triangle_edgeweight}
\end{align}
For ease of exposition, let us denote by $e_{1}$ the edge out of $AB$, $BC$ and $CA$ whose edge-weight equals $\min\{w_{1}(AB),w_{1}(BC),w_{1}(CA)\}$ -- it is evident from the first inequality of \eqref{triangle_edgeweight_>K} (and from the observation that $\min\{w_{0}(AB),w_{0}(BC),w_{0}(CA)\}=K+1+(M+2)L$) that $w_{1}(e_{1})<w_{0}(e_{1})$. Let us denote by $e_{2}$ and $e_{3}$ the remaining two edges out of $AB$, $BC$ and $CA$. Since $P_{1}$ could have modified the edge-weights of at most two edges out of $AB$, $BC$ and $CA$ in the first round, we may assume, without loss of generality, that $w_{0}(e_{3})=w_{1}(e_{3})$. Note that, since $r \geqslant 1$, we have
\begin{align}
K+1+(M+2)\ell \leqslant K+r+(M+2)\ell=\min\{w_{1}(AB),w_{1}(BC),w_{1}(CA)\}\leqslant w_{1}(e_{2}),\nonumber
\end{align}
and from \eqref{less_than_minimum_triangle_edgeweight}, we know that $K+M+3-r+(M+2)\ell < w_{0}(e_{3})=w_{1}(e_{3})$. Hence, in the second round, $P_{2}$ selects the vertex that is in common between $e_{2}$ and $e_{3}$, removes weight $w_{1}(e_{2})-\{K+1+(M+2)\ell\}$ from the edge $e_{2}$, and removes weight $w_{1}(e_{3})-\{K+M+3-r+(M+2)\ell\}$ from the edge $e_{3}$, leaving $P_{1}$ with a configuration where $w_{2}(DE)=K$, and  
\begin{align}
w_{2}(e_{1})=K+r+(M+2)\ell,\ w_{2}(e_{2})=K+1+(M+2)\ell,\ w_{2}(e_{3})=K+M+3-r+(M+2)\ell.\nonumber
\end{align} 
This configuration satisfies the hypothesis of \eqref{G_{4}_losing_cond_1}. As $\ell \leqslant (L-1)$ (as justified in \eqref{ell_less_L_eq_1}), hence $P_{1}$ loses by our induction hypothesis in \eqref{G_{4}_ind_hyp_3} (once again, we emphasize that this entire case need not be considered at all when $L=0$).

\item Suppose $P_{1}$ removes some weight from at most two of the edges $AB$, $BC$ and $CA$ in the first round, such that $\min\{w_{1}(AB),w_{1}(BC),w_{1}(CA)\}\leqslant K$. Let, once again, $e_{1}$ denote the edge, out of $AB$, $BC$ and $CA$, whose edge-weight attains this minimum, and let $e_{2}$ and $e_{3}$ be the remaining two edges, with $w_{1}(e_{3})=w_{0}(e_{3})$ since $P_{1}$ cannot alter the edge-weights of all three of these edges in a single round. In the second round, $P_{2}$ selects the vertex that is in common between $e_{2}$ and $e_{3}$, removes from $e_{3}$ the weight $w_{1}(e_{3})-\{K-w_{1}(e_{1})\}$ (note that $w_{1}(e_{3})=w_{0}(e_{3}) \geqslant K+1$), and removes the entire edge $e_{2}$. This leaves $P_{1}$ with a galaxy graph, with one component consisting of the edges $e_{1}$ and $e_{3}$, the other comprising the edge $DE$, and $w_{2}(e_{1})+w_{2}(e_{3})=K=w_{2}(DE)$. Consequently, $P_{1}$ loses by Theorem~\ref{thm:galaxy}.

\item The final possibility is where $P_{1}$ removes a positive integer weight from $DE$ in the first round, so that $w_{1}(DE)=k < K$. If $k=0$, then $P_{2}$ is left with the edges $AB$, $BC$ and $CA$ of a triangle, with their edge-weights \emph{not} all equal, and she wins by Theorem~\ref{triangle_losing}. Therefore, in what follows, it suffices to assume that $w_{1}(DE)=k$ with $k > 0$.

As mentioned right after \eqref{G_{4}_losing_inductive_eq_1}, we assume, without loss of generality, that $w_{0}(AB)\leqslant w_{0}(BC)\leqslant w_{0}(CA)$, so that $2i \leqslant (M+3)$. Now, let $m$ be the unique positive integer such that $m(m+1) \leqslant 2k \leqslant m(m+3)$, and let $r$ be the unique element in $\{0,1,\ldots,m\}$ such that $K+(M+2)L-k \equiv r \bmod (m+1)$. Thus, there exists some $\ell \in \mathbb{N}_{0}$ such that $K+(M+2)L=k+r+(m+1)\ell$, so that
\begin{align}
{}&w_{1}(AB)=w_{0}(AB)=K+1+(M+2)L=k+r+1+(m+1)\ell,\nonumber\\
{}&w_{1}(BC)=w_{0}(BC)=K+i+(M+2)L=k+r+i+(m+1)\ell,\nonumber\\
{}&w_{1}(CA)=w_{0}(CA)=K+M+3-i+(M+2)L=k+M+3-i+r+(m+1)\ell.\nonumber
\end{align}
There are two subcases to consider. The first is where $2(r+1) > m+i-M$, in which case $P_{2}$ removes weight $(r+i-1)$ from the edge $BC$ and weight $\{2(1+r)+M-i-m\}$ (which is strictly positive) from the edge $CA$ in the second round, leaving $P_{1}$ with
\begin{align}
{}&w_{2}(AB)=k+(r+1)+(m+1)\ell,\ w_{2}(BC)=k+1+(m+1)\ell,\nonumber\\
{}&w_{2}(CA)=k+m+2-(r+1)+(m+1)\ell,\ w_{2}(DE)=k.\nonumber
\end{align}
This configuration satisfies the hypothesis of \eqref{G_{4}_losing_cond_1}, and since $k < K$, $P_{1}$ loses by our induction hypothesis in \eqref{G_{4}_ind_hyp_2} when $(M+1)(M+2)/2 \leqslant k$, and by our induction hypothesis in \eqref{G_{4}_ind_hyp_1} when $k \leqslant M(M+3)/2$ (which, in turn, implies that $m \leqslant M$). Once again, we recall for the reader one of the remarks made right before \eqref{G_{4}_losing_inductive_eq_1}: when $K=(M+1)(M+2)/2$, we have $k < K \implies k \leqslant M(M+3)/2$, which in turn implies that $k$ satisfies the second criterion of \eqref{M_G_{4}} for some $m \leqslant M$, and hence, our induction hypothesis in \eqref{G_{4}_ind_hyp_1} becomes applicable.

The second subcase is where 
\begin{equation}
2(r+1) \leqslant m+i-M,\label{r_m_i_M_case_2}
\end{equation}
which, along with the inequality $2i \leqslant M+3$ (as mentioned in the first sentence of the previous paragraph), yields:
\begin{align}
m+i-M \geqslant 2(r+1)\geqslant 2 \implies m+\frac{M+3}{2}-M \geqslant 2 \implies m-\frac{M}{2}\geqslant\frac{1}{2} \implies m\geqslant\frac{M+1}{2}.\label{r_m_i_M_case_2_consequence}
\end{align}
Starting with \eqref{r_m_i_M_case_2} and applying the inequality $2i \leqslant M+3$, as well as the inequality deduced in \eqref{r_m_i_M_case_2_consequence}, we obtain: 
\begin{align}
(r+i) <{}& \frac{m+i-M}{2}+i-1 = \frac{m-M}{2}+\frac{3i}{2}-1\nonumber\\
\leqslant{}& \frac{m-M}{2}+\frac{3(M+3)}{4}-1=\frac{m}{2}+\frac{M+1}{4}+1\leqslant m+1,\nonumber
\end{align} 
which is what we need for the remaining argument. In the second round, $P_{2}$ removes weight $r$ from the edge $AB$ and weight $(M-m+1+2r)$ (which is strictly positive) from the edge $CA$, leaving $P_{1}$ with
\begin{align}
{}&w_{2}(AB)=k+1+(m+1)\ell,\ w_{2}(BC)=k+(r+i)+(m+1)\ell,\nonumber\\
{}&w_{2}(CA)=k+m+2-(r+i)+(m+1)\ell,\ w_{2}(DE)=k.\nonumber
\end{align}
This configuration satisfies the hypothesis of \eqref{G_{4}_losing_cond_1}, and since $k < K$, once again, $P_{1}$ loses by our induction hypothesis in \eqref{G_{4}_ind_hyp_2} when $(M+1)(M+2)/2 \leqslant k$, and by our induction hypothesis in \eqref{G_{4}_ind_hyp_1} when $k \leqslant M(M+3)/2$ (as before, we need only apply the induction hypothesis of \eqref{G_{4}_ind_hyp_1} when $K=(M+1)(M+2)/2$).
\end{enumerate}

This completes the proof of the fact that the configuration in \eqref{G_{4}_losing_inductive_eq_1} is losing. This also concludes the inductive proof of Step~\eqref{thm:G_{4}_losing_1} in the proof of Theorem~\ref{thm:H_{1}_losing}. \qed

\subsection{Proof of Step~\eqref{thm:G_{4}_winning_1} for proving Theorem~\ref{thm:main_G_{4}}}
We assume, without loss of generality, that $w_{0}(AB) \leqslant w_{0}(BC) \leqslant w_{0}(CA)$. Letting $w_{0}(DE)=k$ for some $k \in \mathbb{N}$, let $m$ be the unique positive integer for which $k$ satisfies the second criterion of \eqref{M_G_{4}}. If $w_{0}(AB)-k\equiv r_{1}\bmod (m+1)$, $w_{0}(BC)-k\equiv r_{2}\bmod(m+1)$ and $w_{0}(CA)-k\equiv r_{3}\bmod(m+1)$, where $r_{1}, r_{2}, r_{3} \in \{1,2,\ldots,m+1\}$, then for some $\ell_{1}$, $\ell_{2}$ and $\ell_{3}$ in $\mathbb{N}_{0}$, we have
\begin{equation}
w_{0}(AB)=k+r_{1}+(m+1)\ell_{1},\ w_{0}(BC)=k+r_{2}+(m+1)\ell_{2},\ w_{0}(CA)=k+r_{3}+(m+1)\ell_{3}.\label{G_{4}_winning_1_eq_1}
\end{equation}

The first case to consider is where $\ell_{1}=\ell_{2}=\ell_{3}=\ell$. In this case, under our assumption that $w_{0}(AB) \leqslant w_{0}(BC) \leqslant w_{0}(CA)$, we must have $r_{1} \leqslant r_{2} \leqslant r_{3}$. If $r_{1}=r_{2}=r_{3}$, then we have $w_{0}(AB)=w_{0}(BC)=w_{0}(CA)$, in which case $P_{1}$ simply removes the edge $DE$ in the first round and wins by Theorem~\ref{triangle_losing}.

If $r_{2}$ and $r_{3}$ are such that $m+2-r_{2} \leqslant r_{3}$, then $P_{1}$ chooses the vertex $A$, and removes from the edges $AB$ and $CA$ the weights $(r_{1}-1)$ and $\{r_{3}-(m+2-r_{2})\}$ respectively (note that the total weight removed is strictly positive since our initial configuration does \emph{not} satisfy the hypothesis of \eqref{G_{4}_losing_cond_1}). This leaves $P_{2}$ with $w_{1}(DE)=k$ and 
\begin{equation}
w_{1}(AB)=k+1+(m+1)\ell,\ w_{1}(BC)=k+r_{2}+(m+1)\ell,\ w_{2}(CA)=k+m+2-r_{2}+(m+1)\ell,\nonumber
\end{equation}
which now satisfies the hypothesis of \eqref{G_{4}_losing_cond_1}, and $P_{2}$ loses by Step~\eqref{thm:G_{4}_losing_1} proved earlier.

Suppose, now, that $m+2-r_{2}>r_{3}$, so that, if we set $m'=r_{2}+r_{3}-2r_{1}$, we have, using $r_{1} \geqslant 1$,
\begin{equation}
m+2-r_{2}>r_{3} \implies r_{2}+r_{3}<m+2 \implies m'=r_{2}+r_{3}-2r_{1}<m+2-2r_{1}\leqslant m.\label{m'<m}
\end{equation}
Let $k'$ be the unique positive integer satisfying
\begin{equation}
\frac{m'(m'+1)}{2}\leqslant k'\leqslant\frac{m'(m'+3)}{2} \quad \text{and} \quad k+r_{1}-1+(m+1)\ell \equiv k' \bmod (m'+1).\label{k'_defn}
\end{equation}
The first of the two criteria stated in \eqref{k'_defn}, along with the inequality $m'<m$ that we have deduced in \eqref{m'<m}, and the fact that $k$ and $m$ satisfy the second criterion of \eqref{M_G_{4}}, yields $k' < k$, which, in turn, combined with the second of the two criteria stated in \eqref{k'_defn}, implies the existence of some $\ell' \in \mathbb{N}_{0}$ such that $k+r_{1}+(m+1)\ell=k'+1+(m'+1)\ell'$. Note that this yields, using the definition of $m'$:
\begin{align}
{}&w_{0}(AB)=k+r_{1}+(m+1)\ell=k'+1+(m'+1)\ell',\nonumber\\
{}&w_{0}(BC)=k+r_{2}+(m+1)\ell=k'+(1+r_{2}-r_{1})+(m'+1)\ell',\nonumber\\
{}&w_{0}(CA)=k+r_{3}+(m+1)\ell=k'+1+r_{3}-r_{1}+(m'+1)\ell'=k'+m'+2-(1+r_{2}-r_{1})+(m'+1)\ell'.\nonumber
\end{align}
Therefore, $P_{1}$ simply removes weight $(k-k')$ (which is strictly positive, as argued above) from the edge $DE$ in the first round, leaving $P_{2}$ with a configuration that satisfies the hypothesis of \eqref{G_{4}_losing_cond_1}, and hence, $P_{2}$ loses by Step~\eqref{thm:G_{4}_losing_1} proved earlier.

Now, let us consider the second case, i.e.\ where $\ell_{1}$, $\ell_{2}$ and $\ell_{3}$ are \emph{not} all equal. This, along with our assumption that $w_{0}(AB)\leqslant w_{0}(BC)\leqslant w_{0}(CA)$, implies that $\ell_{1}<\ell_{3}$. This, in turn, implies (using the inequalities $r_{1}+r_{3}\geqslant 2$ and $\ell_{3}-\ell_{1}\geqslant 1$) that
\begin{align}
{}&\{k+r_{3}+(m+1)\ell_{3}\}-\{k+m+2-r_{1}+(m+1)\ell_{1}\}=r_{3}+r_{1}-(m+2)+(m+1)(\ell_{3}-\ell_{1}) \geqslant 1.\nonumber
\end{align}
$P_{1}$, therefore, removes weight $\{k+r_{3}+(m+1)\ell_{3}\}-\{k+m+2-r_{1}+(m+1)\ell_{1}\}$ from the edge $CA$, and weight $r_{2}-1+(m+1)(\ell_{2}-\ell_{1})$ from the edge $BC$, in the first round, leaving $P_{2}$ with
\begin{align}
{}&w_{1}(AB)=k+r_{1}+(m+1)\ell_{1},\ w_{1}(BC)=k+1+(m+1)\ell_{1},\nonumber\\
{}&w_{1}(CA)=k+m+2-r_{1}+(m+1)\ell_{1},\ w_{1}(DE)=k,\nonumber
\end{align}
which satisfies the hypothesis of \eqref{G_{4}_losing_cond_1}, and hence, $P_{2}$ loses by Step~\eqref{thm:G_{4}_losing_1} proved earlier. This completes the proof of Step~\eqref{thm:G_{4}_winning_1}. \qed

Before we begin the proof of Step~\eqref{thm:G_{4}_losing_2}, we state and prove a crucial lemma: 
\begin{lemma}\label{lem:nec_suff}
Given $w_{0}(AB) \leqslant w_{0}(BC) \leqslant w_{0}(CA)$, with $w_{0}(AB)$, $w_{0}(BC)$ and $w_{0}(CA)$ not all equal, the multiset $\{w_{0}(AB),w_{0}(BC),w_{0}(CA)\}$ is special (see Definition~\ref{defn:special}) if and only if we have
\begin{equation}
2w_{0}(AB)>\left(w_{0}(BC)+w_{0}(CA)-2w_{0}(AB)\right)\left(w_{0}(BC)+w_{0}(CA)-2w_{0}(AB)+1\right).\label{eq:nec_suff}
\end{equation}
\end{lemma}
\begin{proof}
When \eqref{eq:nec_suff} holds, we set $m=w_{0}(BC)+w_{0}(CA)-2w_{0}(AB)$, and we let $k$ be the unique positive integer that, along with this choice of $m$, satisfies the second criterion of \eqref{M_G_{4}} as well as the relation $w_{0}(AB)-1\equiv k\bmod(m+1)$. Note that the inequality in \eqref{eq:nec_suff} ensures that 
\begin{equation}
w_{0}(AB)>\frac{m(m+1)}{2} \implies w_{0}(AB)\geqslant \frac{m(m+1)}{2}+1 \implies w_{0}(AB)-1\geqslant \frac{m(m+1)}{2},\nonumber
\end{equation}
so that, from the way we have defined $k$ above, we know that there exists $\ell \in \mathbb{N}_{0}$ such that $w_{0}(AB)=k+1+(m+1)\ell$. Next, from our definition of $m$, we have
\begin{align}
w_{0}(BC)+w_{0}(CA)=2w_{0}(AB)+m=2\{k+1+(m+1)\ell\}+m=2k+m+2+2(m+1)\ell.\label{nec_suff_eq_1}
\end{align}
Let $w_{0}(BC)-k\equiv i \bmod (m+1)$ for some $i \in \{1,2,\ldots,m+1\}$. Note that, since $w_{0}(BC)\geqslant w_{0}(AB)=k+1+(m+1)\ell$, we can find $\ell_{1} \in \mathbb{N}_{0}$ such that $w_{0}(BC)=k+i+(m+1)\ell_{1}$. Likewise, setting $w_{0}(CA)-k\equiv j \bmod(m+1)$ for some $j \in \{1,2,\ldots,m+1\}$, we see that there exists some $\ell_{2}\in\mathbb{N}_{0}$ such that $w_{0}(CA)=k+j+(m+1)\ell_{2}$. From \eqref{nec_suff_eq_1}, we then obtain
\begin{align}
{}&\{k+i+(m+1)\ell_{1}\}+\{k+j+(m+1)\ell_{2}\}=2k+m+2+2(m+1)\ell \nonumber\\
{}&\Longleftrightarrow i+j+(m+1)(\ell_{1}+\ell_{2})=m+2+2(m+1)\ell.\nonumber
\end{align}
The first conclusion to draw from this is that $(i+j) \equiv 1 \bmod (m+1)$, and since $2 \leqslant (i+j) \leqslant 2(m+1)$, the only possibility we are left with is that $(i+j)=(m+2)$. This, in turn, yields $(\ell_{1}+\ell_{2})=2\ell$. If possible, let us assume that $\ell_{1}\neq \ell_{2}$, which implies that $\ell_{1}<\ell_{2}$ (since $w_{0}(BC)\leqslant w_{0}(CA)$). This implies that $\ell_{1}\leqslant (\ell-1)$, which yields
\begin{align}
w_{0}(BC)=k+i+(m+1)\ell_{1}\leqslant k+i+(m+1)(\ell-1)=k+i-(m+1)+(m+1)\ell<k+1+(m+1)\ell=w_{0}(AB),\nonumber
\end{align}
leading to a contradiction. Therefore, our assumption that $\ell_{1}\neq\ell_{2}$ must be erroneous, leading to $\ell_{1}=\ell_{2}=\ell$. We have, therefore, deduced the existence of $k, m, i \in \mathbb{N}$ and $\ell \in \mathbb{N}_{0}$ such that \eqref{M_G_{4}} holds and $w_{0}(AB)=k+1+(m+1)\ell$, $w_{0}(BC)=k+i+(m+1)\ell$ and $w_{0}(CA)=k+j+(m+1)\ell=k+m+2-i+(m+1)\ell$. This proves one side of the implication stated in Lemma~\ref{lem:nec_suff}. The proof that the converse is also true is much more straightforward and hence omitted from this paper.
\end{proof}

\subsection{Proof of Step~\eqref{thm:G_{4}_losing_2} for proving Theorem~\ref{thm:main_G_{4}}}
The proof of Step~\eqref{thm:G_{4}_losing_2} happens via induction on the value of $w_{0}(DE)$. The base case for this inductive argument has been addressed in \S\ref{subsec:appendix_G_{4}_losing_2}. Suppose, for some $K \in \mathbb{N}$ with $K \geqslant 4$, we have shown that every initial weight configuration on $G_{4}$ that satisfies \emph{all} of the criteria mentioned in \eqref{G_{4}_losing_cond_2}, along with $w_{0}(DE)\leqslant K$, is losing. We now consider an initial weight configuration $\left(w_{0}(AB),w_{0}(BC),w_{0}(CA),w_{0}(DE)\right)$, with $w_{0}(DE)=(K+1)$, that satisfies all of the criteria stated in \eqref{G_{4}_losing_cond_2}. The first round of the game played on this initial configuration can unfold in one of the following ways:
\begin{enumerate}
\item Suppose $P_{1}$ removes some positive integer weight from at most two of the edges out of $AB$, $BC$ and $CA$. If the resulting multiset $\{w_{1}(AB),w_{1}(BC),w_{1}(CA)\}$ is special, thus guaranteeing the existence of $k$, $m$, $i$ and $\ell$ as in Definition~\ref{defn:special}, then $w_{1}(AB)+w_{1}(BC)+w_{1}(CA)<w_{0}(DE)=(K+1)$ implies that $k<(K+1)$. Thus, $P_{2}$ is allowed to remove weight $(K+1)-k$ from the edge $DE$ in the second round, which she does, leaving $P_{1}$ with a configuration that satisfies the hypothesis of \eqref{G_{4}_losing_cond_1}, so that $P_{1}$ loses by Step~\eqref{thm:G_{4}_losing_1}.

\sloppy Suppose the multiset $\{w_{1}(AB),w_{1}(BC),w_{1}(CA)\}$ is not special. If $w_{1}(AB)=w_{1}(BC)=w_{1}(CA)$, then $P_{2}$ removes the edge $DE$ in the second round, defeating $P_{1}$ by Theorem~\ref{triangle_losing}. If not all of $w_{1}(AB)$, $w_{1}(BC)$ and $w_{1}(CA)$ are equal, then $P_{2}$ removes from $DE$ the weight $(K+1)-\{w_{1}(AB)+w_{1}(BC)+w_{1}(CA)\}$ in the second round, leaving $P_{1}$ with a configuration $(w_{2}(AB),w_{2}(BC),w_{2}(CA),w_{2}(DE))$, with $w_{2}(DE)\leqslant K$, that satisfies all of the criteria stated in \eqref{G_{4}_losing_cond_2}. Consequently, $P_{1}$ loses by our induction hypothesis. 

\item Suppose $P_{1}$ removes a positive integer weight from $DE$ in the first round. If $w_{1}(DE)=0$, then $P_{2}$ wins by Theorem~\ref{triangle_losing}. So, we henceforth assume that $w_{1}(DE)>0$.

Without loss of generality, for the rest of this analysis, let us assume that $w_{0}(AB)\leqslant w_{0}(BC)\leqslant w_{0}(CA)$. Note that if $w_{0}(AB) \leqslant w_{1}(DE) \leqslant w_{0}(BC)+w_{0}(CA)$, then $P_{1}$ loses by Step~\eqref{lem:G_{4}_winning_1}. 

\sloppy When $w_{1}(DE)<w_{0}(AB)$, we have $w_{1}(DE)<\min\{w_{0}(AB),w_{0}(BC),w_{0}(CA)\}$. Since our initial weight configuration satisfies the criteria stated in \eqref{G_{4}_losing_cond_2} (in particular, the multiset $\{w_{0}(AB),w_{0}(BC),w_{0}(CA)\}$ is \emph{not} special), hence the weight configuration $(w_{1}(AB),w_{1}(BC),w_{1}(CA),w_{1}(DE))$, obtained after the first round of the game, does \emph{not} satisfy the hypothesis of \eqref{G_{4}_losing_cond_1}. Therefore, all the criteria stated in Step~\eqref{thm:G_{4}_winning_1} are satisfied by $(w_{1}(AB),w_{1}(BC),w_{1}(CA),w_{1}(DE))$, and hence, $P_{2}$ wins. 

Finally, we consider $w_{1}(DE)>w_{0}(BC)+w_{0}(CA)$. In this case, $P_{2}$ removes weight $w_{0}(DE)-w_{1}(DE)$ from the edge $AB$ in the second round, so that $w_{2}(AB)=w_{0}(AB)-w_{0}(DE)+w_{1}(DE)=w_{1}(DE)-w_{0}(BC)-w_{0}(CA)$, which implies that $w_{2}(AB)+w_{2}(BC)+w_{2}(CA)=w_{2}(DE)$. Since, to begin with, we had $w_{0}(AB)\leqslant w_{0}(BC)\leqslant w_{0}(CA)$, we certainly have $w_{2}(AB)<w_{2}(BC)=w_{0}(BC)\leqslant w_{2}(CA)=w_{0}(CA)$ -- i.e.\ $w_{2}(AB)$, $w_{2}(BC)$ and $w_{2}(CA)$ are \emph{not} all equal. By Lemma~\ref{lem:nec_suff}, the multiset $\{w_{2}(AB),w_{2}(BC),w_{2}(CA)\}$ is special iff
\begin{align}
2w_{2}(AB)>{}&\{w_{2}(BC)+w_{2}(CA)-2w_{2}(AB)\}\{w_{2}(BC)+w_{2}(CA)-2w_{2}(AB)+1\}\nonumber\\
={}&\{w_{0}(BC)+w_{0}(CA)-2w_{2}(AB)\}\{w_{0}(BC)+w_{0}(CA)-2w_{2}(AB)+1\}.\label{nec_suff_ineq}
\end{align}
Considering, for any real $y>0$, the function $f_{y}(x)=(y-x)(y-x+1)-x$ for $x \in [0,y]$, we see that 
\begin{align}
f'_{y}(x)=-2y+2x-2=-2(y+1-x),\nonumber
\end{align}
which is strictly negative for as long as we have $y \geqslant x$. Therefore, $f_{y}(x)$ is strictly decreasing over $x \in [0,y]$, which implies that if for some $x \in [0,y)$, we have $x>(y-x)(y-x+1)$, then for any $x' \in (x,y]$, we would have $x'>(y-x')(y-x'+1)$. Setting $y=w_{0}(BC)+w_{0}(CA)$, $x=2w_{2}(AB)$ and $x'=2w_{0}(AB)$ and applying this observation, we find that \eqref{nec_suff_ineq} holds if and only if \eqref{eq:nec_suff} holds, and if \eqref{eq:nec_suff} holds, then, by Lemma~\ref{lem:nec_suff}, we conclude that the multiset $\{w_{0}(AB),w_{0}(BC),w_{0}(CA)\}$ must be special, contradicting the assumption that our initial weight configuration satisfies all constraints of \eqref{G_{4}_losing_cond_2}. Therefore, the multiset $\{w_{2}(AB),w_{2}(BC),w_{2}(CA)\}$ \emph{cannot} be special, and the configuration $(w_{2}(AB),w_{2}(BC),w_{2}(CA),w_{2}(DE))$ satisfies all the criteria stated in \eqref{G_{4}_losing_cond_2}, along with $w_{2}(DE)\leqslant K$. Hence, $P_{1}$ loses by our induction hypothesis.
\end{enumerate}
This concludes the inductive argument, and thereby, the proof of Step~\eqref{thm:G_{4}_losing_2}. \qed

\subsection{Proof of Step~\eqref{thm:G_{4}_winning_2} for proving Theorem~\ref{thm:main_G_{4}}}
\sloppy We assume, without loss of generality, that $w_{0}(AB)\leqslant w_{0}(BC) \leqslant w_{0}(CA)$, so that the first criterion stated in Step~\eqref{thm:G_{4}_winning_2} boils down to $w_{0}(DE) > w_{0}(BC)+w_{0}(CA)$. 

If $w_{0}(AB)=w_{0}(BC)=w_{0}(CA)$, then $P_{1}$ simply removes $DE$ in the first round, and $P_{2}$ loses by Theorem~\ref{triangle_losing}. Henceforth, therefore, we assume that $w_{0}(AB)$, $w_{0}(BC)$ and $w_{0}(CA)$ are not all equal.

If $w_{0}(DE)>w_{0}(AB)+w_{0}(BC)+w_{0}(CA)$, then $P_{1}$ removes weight $w_{0}(DE)-\{w_{0}(AB)+w_{0}(BC)+w_{0}(CA)\}$ from the edge $DE$ in the first round, so that $w_{1}(DE)=w_{1}(AB)+w_{1}(BC)+w_{1}(CA)$. Since the edge-weights on $AB$, $BC$ and $CA$ remain unchanged, and since $\{w_{0}(AB),w_{0}(BC),w_{0}(CA)\}$ is not special to begin with, neither is $\{w_{1}(AB),w_{1}(BC),w_{1}(CA)\}$. Thus, $P_{2}$ is left with a configuration that satisfies all the constraints stated in \eqref{G_{4}_losing_cond_2}, and hence, by Theorem~\ref{thm:G_{4}_losing_2}, $P_{2}$ loses.

Let us now consider $w_{0}(BC)+w_{0}(CA)<w_{0}(DE)<w_{0}(AB)+w_{0}(BC)+w_{0}(CA)$. In this case, $P_{1}$ removes weight $w_{0}(AB)+w_{0}(BC)+w_{0}(CA)-w_{0}(DE)$ from the edge $AB$ in the first round, so that $P_{2}$ is left with $w_{1}(AB)+w_{1}(BC)+w_{1}(CA)=w_{1}(DE)$. Evidently, $w_{1}(AB)<w_{0}(AB)\leqslant w_{0}(BC)=w_{1}(BC)\leqslant w_{0}(CA)=w_{1}(CA)$, so that $w_{1}(AB)$, $w_{1}(BC)$ and $w_{1}(CA)$ are not all equal. Next, we note that, by Lemma~\ref{lem:nec_suff}, the multiset $\{w_{1}(AB),w_{1}(BC),w_{1}(CA)\}$ would be special if and only if we have
\begin{align}
2w_{1}(AB)>{}&\{w_{1}(BC)+w_{1}(CA)-2w_{1}(AB)\}\{w_{1}(BC)+w_{1}(CA)-2w_{1}(AB)+1\}\nonumber\\
={}&\{w_{0}(BC)+w_{0}(CA)-2w_{1}(AB)\}\{w_{0}(BC)+w_{0}(CA)-2w_{1}(AB)+1\},\label{nec_suff_2}
\end{align}
and, much as we argued towards the end of the proof of Step~\eqref{thm:G_{4}_losing_2}, we conclude that \eqref{nec_suff_2} holds if and only if \eqref{eq:nec_suff} holds, which, in turn, applying Lemma~\ref{lem:nec_suff}, is true if and only if the multiset $\{w_{0}(AB),w_{0}(BC),w_{0}(CA)\}$ is special. However, this contradicts one of the criteria mentioned in the statement of Step~\eqref{thm:G_{4}_winning_2}, which allows us to conclude that, in fact, the multiset $\{w_{1}(AB),w_{1}(BC),w_{1}(CA)\}$ cannot be special. Combining all these observations, we see that the configuration $(w_{1}(AB),w_{1}(BC),w_{1}(CA),w_{1}(DE))$ satisfies all of the criteria mentioned in \eqref{G_{4}_losing_cond_2}, and hence, by Step~\eqref{thm:G_{4}_losing_2} that we have already proved, $P_{2}$ loses. This completes the proof of Step~\eqref{thm:G_{4}_winning_2}.\qed

\section{Proof of our results regarding the graph $H_{1}$}\label{sec:proof_6_vertices}
This section is dedicated to the proofs of all of our results in \S\ref{subsubsec:H_{1}_main_results}. At the very outset of \S\ref{sec:proof_6_vertices}, we inform the reader that not all parts of the proof of Theorem~\ref{thm:H_{1}_losing} have been included in this section. In \S\ref{sec:proof_6_vertices}, we have included 
\begin{enumerate}
\item the proof of our claim that any configuration on $H_{1}$ that is of the form given by \eqref{eq:H_{1}_losing_0,1,3}, with $r=0$, is losing,
\item and the proofs of our claims that any configuration on $H_{1}$ that is of one of the forms given by \eqref{eq:H_{1}_losing_2,4_case_1}, \eqref{eq:H_{1}_losing_2,4_case_2} and \eqref{eq:H_{1}_losing_2,4_case_3}, with $r=2$, is losing.
\end{enumerate}
The claims that configurations on $H_{1}$ of the form given by \eqref{eq:H_{1}_losing_0,1,3}, with $r=1$ or $r=3$, are losing, have been established in \S\ref{subsec:H_{1}_losing_1_proof} (when $r=1$) and \S\ref{subsec:H_{1}_losing_3_proof} (when $r=3$) of the Appendix, \S\ref{sec:appendix}. Likewise, the claims that configurations on $H_{1}$ of the form given by one of \eqref{eq:H_{1}_losing_2,4_case_1}, \eqref{eq:H_{1}_losing_2,4_case_2} and \eqref{eq:H_{1}_losing_2,4_case_3}, with $r=4$, are losing, have been proved in \S\ref{subsec:H_{1}_losing_4_1,2_proof} (when $s=1$), \S\ref{subsec:H_{1}_losing_4_3,4_proof} (when $s=2$), \S\ref{subsec:H_{1}_losing_4_5,6_proof} (when $s=3$) and \S\ref{subsec:H_{1}_losing_4_7_proof} (when $s\geqslant 4$) of \S\ref{sec:appendix}.

\subsection{Proof of Lemma~\ref{lem:H_{1}_winning_1}} Since the edges $AB$ and $CD$ play symmetric roles in the graph $H_{1}$, it suffices for us to prove that an initial weight configuration on $H_{1}$ is winning whenever it satisfies the inequalities in \eqref{eq:H_{1}_winning_1}, and the proof would be analogous if, instead, it satisfies the inequalities in \eqref{eq:H_{1}_winning_2}. Assuming that \eqref{eq:H_{1}_winning_1} holds, $P_{1}$ chooses the vertex $C$, removes weight $w_{0}(AB)+w_{0}(BC)-w_{0}(EF)$ from the edge $BC$, and removes the entire edge $CD$, in the first round, leaving $P_{2}$ with a galaxy graph comprising two components: one of these is the edge $EF$, the other is a star graph consisting of the edges $AB$ and $BC$, and $w_{1}(AB)+w_{1}(BC)=w_{0}(EF)=w_{1}(EF)$. Hence, $P_{2}$ loses by Theorem~\ref{triangle_losing}.

\sloppy \subsection{Proof of Theorem~\ref{thm:H_{1}_winning_1}} If the set $S$ defined in the statement of Theorem~\ref{thm:H_{1}_winning_1} is empty, then $a_{i}+b_{i}+c_{i}$ is even for each $i \in \{0,1,\ldots,s\}$. In this case, $P_{1}$ removes the edge $BC$ in the first round, leaving $P_{2}$ with a galaxy graph comprising three components, namely, the edges $AB$, $CD$ and $EF$, such that the triple $(w_{0}(AB),w_{0}(CD),w_{0}(EF))$ is balanced, and $P_{2}$ loses by Theorem~\ref{thm:galaxy}.

Suppose, now, that $S$ is non-empty, so that the index $I$ is well-defined, and we assume, without loss of generality, that $b_{I}=1$. Let us choose $e_{i}\in \{0,1\}$, for $i \in \{0,1,\ldots,s\}$, such that $e_{i}+a_{i}+c_{i}$ is even -- note, immediately, that $e_{i}=b_{i}$ for all $i \in \{I+1,\ldots,s\}$ (if this set is non-empty), and $e_{I}=0$ (since $I$ itself is an element of $S$). Consequently,
\begin{align}
\sum_{i=0}^{s}e_{i}2^{i}\leqslant\sum_{i=I+1}^{s}b_{i}2^{i}+\sum_{i=0}^{I-1}2^{i}=\sum_{i=I+1}^{s}b_{i}2^{i}+2^{I}-1<\sum_{i=0}^{s}b_{i}2^{i}=w_{0}(AB),\nonumber
\end{align}
so that if $P_{1}$ chooses the vertex $B$, removes weight $w_{0}(AB)-\sum_{i=0}^{s}e_{i}2^{i}$ from the edge $AB$, and removes the entire edge $BC$, in the first round, it leaves $P_{2}$ with a galaxy graph comprising three components, namely, the edges $AB$, $CD$ and $EF$, such that the triple $(w_{0}(AB),w_{0}(CD),w_{0}(EF))$ is balanced (by our choice of the $e_{i}$s). Consequently, $P_{2}$ loses by Theorem~\ref{thm:galaxy}. This completes the proof of Theorem~\ref{thm:H_{1}_winning_1}.

\subsection{Proof of Lemma~\ref{prop:H_{1}_winning_1}}
Consider an initial weight configuration on $H_{1}$ that is of the form described in the statement of Lemma~\ref{prop:H_{1}_winning_1}, with $m_{1} \neq m_{2}$, and we assume, without loss of generality, that $m_{1} > m_{2}$. Writing $f(k)=r$, where $f(k)$ is as defined just before the statement of Proposition~\ref{prop:H_{1}_winning_1}, let us write $k=(a_{r}a_{r-1}\ldots a_{1}a_{0})_{2}$ (so that $a_{r}=1$), $\ell_{1}=(b_{r}b_{r-1}\ldots b_{1}b_{0})_{2}$ and $\ell_{2}=(c_{r}c_{r-1}\ldots c_{1}c_{0})_{2}$. We further write $m_{1}=(\alpha_{s}\alpha_{s-1}\ldots \alpha_{1}\alpha_{0})_{2}$ and $m_{2}=(\beta_{s}\beta_{s-1}\ldots \beta_{1}\beta_{0})_{2}$, such that at least one of $\alpha_{s}$ and $\beta_{s}$ equals $1$, and as $m_{1} > m_{2}$, there exists $J \in \{0,1,\ldots,s\}$ such that $\alpha_{i}=\beta_{i}$ for all $i \in \{J+1,\ldots,s\}$ (if this set is non-empty), $\alpha_{J}=1$ and $\beta_{J}=0$. Armed with these, we see that
\begin{align}
{}&w_{0}(AB)=(\alpha_{s}\alpha_{s-1}\ldots \alpha_{1}\alpha_{0}b_{r}b_{r-1}\ldots b_{1}b_{0})_{2},\nonumber\\
{}&w_{0}(CD)=(\beta_{s}\beta_{s-1}\ldots \beta_{1}\beta_{0}c_{r}c_{r-1}\ldots c_{1}c_{0})_{2},\nonumber\\
{}&w_{0}(EF)=(\underbrace{0\ 0\ 0 \ldots\ 0\ 0\ 0}_{(s+1) \text{ zeros}}a_{r}a_{r-1}\ldots a_{1}a_{0})_{2}.\nonumber
\end{align}
Consequently, the largest index $I$ for which the sum of the $I$-th coordinates of the base-$2$ representations of $w_{0}(AB)$, $w_{0}(CD)$ and $w_{0}(EF)$ is odd is given by $I=r+1+J$, and we have, as mentioned above, $\alpha_{J}=1$, which implies that the $I$-th coordinate in the base-$2$ representation of $w_{0}(AB)$ equals $1$. Thus, $P_{1}$ wins by Theorem~\ref{thm:H_{1}_winning_1}.

Now, let us consider an initial weight configuration on $H_{1}$ that is of the form described in Lemma~\ref{prop:H_{1}_winning_1}, with $m_{1}=m_{2}=m$ and $\min\{\ell_{1},\ell_{2}\}\geqslant k$. Let us, as above, set $r=f(k)$, and let $k=(a_{r}a_{r-1}\ldots a_{1}a_{0})_{2}$ (so that, by definition of $f(k)$, we must have $a_{r}=1$), $\ell_{1}=(b_{r}b_{r-1}\ldots b_{1}b_{0})_{2}$ and $\ell_{2}=(c_{r}c_{r-1}\ldots c_{1}c_{0})_{2}$ (keeping in mind that $0 \leqslant \ell_{1}, \ell_{2} \leqslant 2^{r+1}-1$). Since $\min\{\ell_{1},\ell_{2}\}\geqslant k$, we must have $b_{r}=c_{r}=1$ as well. Let us also set $m=(\alpha_{s}\alpha_{s-1}\ldots \alpha_{1}\alpha_{0})_{2}$, with $\alpha_{s}=1$. We then have
\begin{align}
{}&w_{0}(AB)=(\alpha_{s}\alpha_{s-1}\ldots \alpha_{1}\alpha_{0}b_{r}b_{r-1}\ldots b_{1}b_{0})_{2},\nonumber\\
{}&w_{0}(CD)=(\alpha_{s}\alpha_{s-1}\ldots \alpha_{1}\alpha_{0}c_{r}c_{r-1}\ldots c_{1}c_{0})_{2},\nonumber\\
{}&w_{0}(EF)=(\underbrace{0\ 0\ 0 \ldots\ 0\ 0\ 0}_{(s+1) \text{ zeros}}a_{r}a_{r-1}\ldots a_{1}a_{0})_{2},\nonumber
\end{align}
so that the largest index $I$ for which the sum of the $I$-th coordinates of the base-$2$ representations of $w_{0}(AB)$, $w_{0}(CD)$ and $w_{0}(EF)$ is odd is given by $I=r$ (since $a_{r}=b_{r}=c_{r}=1$, as mentioned above). Since $b_{r}=c_{r}=1$, $P_{1}$ wins by Theorem~\ref{thm:H_{1}_winning_1}. 

Finally, let us consider an initial weight configuration on $H_{1}$ that is of the form described in Lemma~\ref{prop:H_{1}_winning_1}, with $m_{1}=m_{2}=m$, $k \in \{\ell_{1},\ell_{2}\}$ and either $\min\{\ell_{1},\ell_{2}\}>0$ or $w_{0}(BC)>0$. Without loss of generality, let us assume that $k=\ell_{1}$. Then, in the first round, $P_{1}$ chooses the vertex $C$, removes weight $\ell_{2}$ from $CD$, and removes the entire edge $BC$ (our assumption ensures that the total weight removed is strictly positive), so that $P_{2}$ is left with a galaxy graph consisting of the edges $AB$, $CD$ and $EF$, where $w_{1}(AB)=2^{f(k)+1}m+\ell_{1}$, $w_{1}(CD)=2^{f(k)+1}m$ and $w_{1}(EF)=k=\ell_{1}$. As the triple $(2^{f(k)+1}m+\ell_{1},2^{f(k)+1}m,\ell_{1})$ is balanced, $P_{2}$ loses by Theorem~\ref{thm:galaxy}. This completes the proof of Lemma~\ref{prop:H_{1}_winning_1}.

\subsection{Proof that a configuration of the form given by \eqref{eq:H_{1}_losing_0,1,3}, with $r=0$, is losing}\label{subsec:proof_eq_H_{1}_losing_0}
We begin by proving that an initial weight configuration on $H_{1}$, that is of the form given by \eqref{eq:H_{1}_losing_0,1,3}, is losing for $r=0$ and $s=k$. First, we consider the configuration where $w_{0}(AB)=w_{0}(CD)=2m$ and $w_{0}(BC)=w_{0}(EF)=1$, for $m \in \mathbb{N}$:
\begin{enumerate}
\item Suppose $P_{1}$ removes a positive integer weight from $AB$ and a non-negative integer weight from $BC$ during the first round (an analogous situation would arise if, instead, $P_{1}$ removes a positive integer weight from $CD$ and a non-negative integer weight from $BC$). This yields $w_{1}(AB)=2n+\ell_{1}$ for some $n<m$ and some $\ell_{1}\in \{0,1\}$. Note that the resulting configuration is of the form described in Lemma~\ref{prop:H_{1}_winning_1}, with $m_{1}=n$, $m_{2}=m$ and $\ell_{2}=0$, and since $m_{1}\neq m_{2}$, $P_{1}$ loses.

\item Suppose $P_{1}$ removes the edge $EF$ in the first round. Then $P_{2}$ removes the edge $BC$ in the second round, leaving $P_{1}$ with a galaxy graph comprising the edges $AB$ and $CD$, where $w_{2}(AB)=w_{2}(CD)=2m$, and $P_{1}$ loses by Theorem~\ref{thm:galaxy}.

\item Suppose $P_{1}$ removes the edge $BC$ in the first round, without disturbing the edge-weights of $AB$ or $CD$. Then $P_{2}$ removes the edge $EF$ in the second round, leaving $P_{1}$ with a galaxy graph comprising the edges $AB$ and $CD$, where $w_{2}(AB)=w_{2}(CD)=2m$, and $P_{1}$ loses by Theorem~\ref{thm:galaxy}.
\end{enumerate}

Suppose we have shown, for some $K \in \mathbb{N}$, that the configuration given by \eqref{eq:H_{1}_losing_0,1,3}, with $r=0$ and $s=k$, with $k \leqslant K$, is losing. We now consider the first round of the game played on the initial configuration on $H_{1}$ where $w_{0}(AB)=w_{0}(CD)=2^{R+1}m$ and $w_{0}(BC)=w_{0}(EF)=(K+1)$, where $m \in \mathbb{N}$ and we set $R=f(K+1)$:
\begin{enumerate}
\item Suppose $P_{1}$ removes a positive integer weight from $AB$, and a non-negative integer weight from $BC$, during the first round (as before, an analogous situation would be if $P_{1}$ removes a positive integer weight from $CD$, and a non-negative integer weight from $BC$). Once again, this leads to $w_{1}(AB)=2^{R+1}n+\ell_{1}$ for some $n < m$ and some $\ell_{1}\in \{0,1,\ldots2^{R+1}-1\}$, so that the resulting configuration is of the form described in Lemma~\ref{prop:H_{1}_winning_1}, with $m_{1}=n$, $m_{2}=m$ and $\ell_{2}=0$, and since $m_{1}\neq m_{2}$, $P_{1}$ loses.

\item Suppose $P_{1}$ removes a positive integer weight from $EF$ in the first round, so that $w_{1}(EF)=k\leqslant K$. Then, $P_{2}$ removes weight $(K+1-k)$ from the edge $BC$ in the second round, leaving $P_{1}$ with 
\begin{align}
w_{2}(AB)=w_{2}(CD)=2^{R+1}m=2^{f(k)+1}2^{R-f(k)}m,\ w_{2}(BC)=w_{2}(EF)=k,\nonumber
\end{align}
where $k \leqslant K\implies f(k)\leqslant f(K+1)=R$. When $k\geqslant 1$, this is of the form given by \eqref{eq:H_{1}_losing_0,1,3} with $r=0$ and $w_{2}(BC)=w_{2}(EF)=k \leqslant K$, so that $P_{1}$ loses by our induction hypothesis. When $k=0$, $P_{1}$ loses by Theorem~\ref{thm:galaxy}.

\item Suppose $P_{1}$ removes a positive integer weight from $BC$, without disturbing the edge-weights of $AB$ and $CD$, during the first round. Then $P_{2}$ removes weight $(K+1-k)$ from $EF$ in the second round, and, as in the previous case, $P_{1}$ loses by either our induction hypothesis (when $1 \leqslant w_{1}(BC)\leqslant K$), or by Theorem~\ref{thm:galaxy} (when $w_{1}(BC)=0$).
\end{enumerate}
This completes the inductive proof of our claim that any configuration on $H_{1}$ of the form given by \eqref{eq:H_{1}_losing_0,1,3} is losing whenever $r=0$ and $s=k$ (equivalently, $r=0$ and $w_{0}(AB)=w_{0}(CD)$).

Let us assume, now, that for some $K \in \mathbb{N}$, we have proved that any initial configuration of the form given by \eqref{eq:H_{1}_losing_0,1,3}, with $r=0$, is losing whenever $k\leqslant K$. We fix any $L \in \{0,1,\ldots,K-1\}$, and consider the first round of the game played on the initial weight configuration on $H_{1}$ where $w_{0}(AB)=2^{R+1}m$, $w_{0}(CD)=2^{R+1}m+(L+1)$, $w_{0}(BC)=(K-L)$ and $w_{0}(EF)=(K+1)$, for any $m \in \mathbb{N}$, where we have set $R=f(K+1)$:
\begin{enumerate}
\item Suppose $P_{1}$ removes a positive integer weight from $CD$, and a non-negative integer weight from $BC$, in the first round. We consider the following few subcases:
\begin{enumerate}
\item Suppose $w_{1}(CD)=2^{R+1}m+\ell$ for some $\ell \in \{0,1,\ldots,L\}$, and $w_{1}(BC)=t \in \mathbb{N}$ with $t\leqslant (K-L)$. Note that $\ell \leqslant L$ and $t \leqslant (K-L)$ together imply $(\ell+t)\leqslant K$, which, in turn, implies $f(t+\ell)\leqslant f(K+1)=R$. In this case, $P_{2}$ removes weight $(K+1)-(t+\ell)$ from the edge $EF$ in the second round, leaving $P_{1}$ with 
\begin{align}
{}&w_{2}(AB)=2^{f(t+\ell)+1}2^{R-f(t+\ell)}m,\ w_{2}(BC)=t,\nonumber\\
{}&w_{2}(CD)=2^{f(t+\ell)+1}2^{R-f(t+\ell)}m+\ell,\ w_{2}(EF)=(t+\ell),\nonumber
\end{align}
which is of the form given by \eqref{eq:H_{1}_losing_0,1,3} with $r=0$ and $w_{2}(EF)=(t+\ell)\leqslant K$, and hence, $P_{1}$ loses by our induction hypothesis.

\item Suppose $w_{1}(CD)=2^{R+1}m+\ell$ for some $\ell \in \{0,1,\ldots,L\}$, but this time, $w_{1}(BC)=0$. Here, $P_{2}$ removes weight $(K+1-\ell)$ from the edge $EF$ in the second round, leaving $P_{1}$ with a galaxy graph consisting of the edges $AB$, $CD$ and $EF$, where $w_{2}(AB)=2^{R+1}m$, $w_{2}(CD)=2^{R+1}m+\ell$ and $w_{2}(EF)=\ell$, and $P_{1}$ loses by Theorem~\ref{thm:galaxy}.

\item Suppose $w_{1}(CD)=2^{R+1}n+\ell$ for some $n < m$ and $\ell \in \{0,1,\ldots,2^{R+1}-1\}$. The resulting configuration is of the form described in Lemma~\ref{prop:H_{1}_winning_1}, with $m_{1}=n$, $m_{2}=m$, $\ell_{1}=\ell$ and $\ell_{2}=0$, and since $m_{1}\neq m_{2}$, $P_{1}$ loses.
\end{enumerate}

\item Suppose $P_{1}$ removes a positive integer weight from $AB$, and a non-negative integer weight from $BC$, in the first round. Then we would have $w_{1}(AB)=2^{R+1}n+\ell$ for some $n < m$ and $\ell \in \{0,1,\ldots,2^{R+1}-1\}$, and once again, $P_{1}$ would lose by Lemma~\ref{prop:H_{1}_winning_1}.

\item Suppose $P_{1}$ removes a positive integer weight from $EF$ in the first round, so that $w_{1}(EF)=k\leqslant K$. If $k\in \{K-L,K-L+1,\ldots,K\}$, then $P_{2}$ removes weight $(L+1)-\{k-(K-L)\}$ from the edge $CD$ in the second round, leaving $P_{1}$ with \begin{align}
{}&w_{2}(AB)=2^{f(k)+1}2^{R-f(k)}m,\ w_{2}(BC)=(K-L),\nonumber\\
{}&w_{2}(CD)=2^{f(k)+1}2^{R-f(k)}m+k-(K-L),\ w_{2}(EF)=k,\nonumber
\end{align}
which is of the same form as \eqref{eq:H_{1}_losing_0,1,3} with $r=0$ and $w_{2}(EF)=k\leqslant K$, so that $P_{1}$ loses by our induction hypothesis. If $k \in \{1,2,\ldots,K-L-1\}$, then $P_{2}$ removes weight $(L+1)$ from the edge $CD$, and weight $(K-L)-k$ from the edge $BC$, in the second round, leaving $P_{1}$ with 
\begin{align}
{}&w_{2}(AB)=w_{2}(CD)=2^{f(k)+1}2^{R-f(k)}m,\ w_{2}(BC)=w_{2}(EF)=k,\nonumber
\end{align}
which is of the same form as \eqref{eq:H_{1}_losing_0,1,3} with $r=0$ and $k\leqslant K$, so that $P_{1}$ loses by our induction hypothesis (we could have also argued that the configuration obtained above, after the second round, is of the form given by \eqref{eq:H_{1}_losing_0,1,3} with $r=0$ and $w_{2}(AB)=w_{2}(CD)$, and we have already proved above that such a configuration is losing). Finally, if $k=0$, $P_{2}$ wins because of Remark~\ref{rem:EF_edgeweight_0}.

\item Finally, suppose $P_{1}$ removes a positive integer weight from $BC$ in the first round, but leaves the edge-weights of $AB$ and $CD$ undisturbed. If $w_{1}(BC)=t\in \{1,2,\ldots,K-L-1\}$, then $P_{2}$ removes weight $(K-L-t)$ from $EF$ in the second round, so that $w_{2}(EF)=L+t+1 \leqslant L+(K-L-1)+1=K$, which, in turn, implies that $f(L+t+1)\leqslant f(K+1)=R$. Thus, $P_{1}$ is left with 
\begin{align}
{}&w_{2}(AB)=2^{f(L+t+1)+1}2^{R-f(L+t+1)}m,\ w_{2}(BC)=t,\nonumber\\
{}&w_{2}(CD)=2^{f(L+t+1)+1}2^{R-f(L+t+1)}m+(L+1),\ w_{2}(EF)=L+t+1,\nonumber
\end{align}
which is of the same form as \eqref{eq:H_{1}_losing_0,1,3} with $r=0$ and $w_{2}(EF)=(L+t+1)\leqslant K$, so that $P_{1}$ loses by our induction hypothesis. If $w_{1}(BC)=0$, then $P_{2}$ removes weight $K-L$ from the edge $EF$ in the second round, leaving $P_{1}$ with a galaxy graph comprising the edges $AB$, $CD$ and $EF$, where $w_{2}(AB)=2^{R+1}m$, $w_{2}(CD)=2^{R+1}m+(L+1)$ and $w_{2}(EF)=(L+1)$, so that $P_{1}$ loses by Theorem~\ref{thm:galaxy}.
\end{enumerate}
This completes the proof of our claim that any configuration on $H_{1}$ of the form given by \eqref{eq:H_{1}_losing_0,1,3}, with $r=0$, is losing.

\subsection{Proof of Lemma~\ref{lem:H_{1}_winning_2}}\label{subsec:proof_H_{1}_winning_2_lemma}
Consider an initial weight configuration on $H_{1}$ that is of the form given by Lemma~\ref{lem:H_{1}_winning_2}. To begin with, if $m_{1}\neq m_{2}$, we know that $P_{1}$ wins by Lemma~\ref{prop:H_{1}_winning_1}. Assuming, therefore, that $m_{1}=m_{2}=m$, as well as that $\ell_{1} \leqslant \ell_{2}$, we obtain $w_{0}(AB)=2^{f(k)+1}m+\ell_{1}$, $w_{0}(CD)=2^{f(k)+1}m+\ell_{2}$, $w_{0}(EF)=k$ and $w_{0}(BC)>k-\min\{\ell_{1},\ell_{2}\}=k-\ell_{1}$. In the first round, $P_{1}$ removes weight $\ell_{2}$ from $CD$ and weight $w_{0}(BC)-(k-\ell_{1})$ from $BC$, so that $P_{2}$ is left with $w_{1}(AB)=2^{f(k)+1}m+\ell_{1}$, $w_{1}(CD)=2^{f(k)+1}m$, $w_{1}(EF)=k$ and $w_{1}(BC)=(k-\ell_{1})$, which is of the form given by \eqref{eq:H_{1}_losing_0,1,3} with $r=0$. Consequently, $P_{2}$ loses.

Consider, now, any initial weight configuration on $H_{1}$ that satisfies $w_{0}(BC)>w_{0}(EF)$. If $w_{0}(EF)\geqslant w_{0}(AB)$, then we have $w_{0}(AB)\leqslant w_{0}(EF)<w_{0}(AB)+w_{0}(BC)$, and if $w_{0}(EF)\geqslant w_{0}(CD)$, then we have $w_{0}(CD)\leqslant w_{0}(EF)<w_{0}(BC)+w_{0}(CD)$. In either of these cases, $P_{1}$ wins by Lemma~\ref{lem:H_{1}_winning_1}. Therefore, throughout the rest of the proof, let us assume that $w_{0}(EF)<\min\{w_{0}(AB),w_{0}(CD)\}$. 

Writing $w_{0}(EF)=k$ with $k \in \mathbb{N}$ (so that $w_{0}(BC)>k$), $w_{0}(AB)=2^{f(k)+1}m_{1}+\ell_{1}$ and $w_{0}(CD)=2^{f(k)+1}m_{2}+\ell_{2}$, where $m_{1}, m_{2} \in \mathbb{N}_{0}$ and $\ell_{1}, \ell_{2} \in \{0,1,\ldots,2^{f(k)+1}-1\}$, we note the following:
\begin{enumerate}
\item If $m_{1}\neq m_{2}$, then $P_{1}$ wins by Lemma~\ref{prop:H_{1}_winning_1}.
\item If $m_{1}=m_{2}=0$, then $w_{0}(EF)<\min\{w_{0}(AB),w_{0}(CD)\}$ implies that $k<\min\{\ell_{1},\ell_{2}\}$, and once again, $P_{1}$ wins by Lemma~\ref{prop:H_{1}_winning_1}.
\item Finally, if $m_{1}=m_{2}=m\in \mathbb{N}$ and, assuming (without loss of generality) that $\ell_{1}\leqslant \ell_{2}$, if $k>\ell_{1}$, then $P_{1}$ removes weight $\ell_{2}$ from the edge $CD$, and weight $w_{0}(BC)-(k-\ell_{1})$ from the edge $BC$, in the first round, leaving $P_{2}$ with
\begin{align}
w_{1}(AB)=2^{f(k)+1}m+\ell_{1},\ w_{1}(BC)=k-\ell_{1},\ w_{1}(CD)=2^{f(k)+1}m,\ w_{1}(EF)=k,\nonumber
\end{align}
which is of the same form as \eqref{eq:H_{1}_losing_0,1,3} with $r=0$, and by what we have already proved in \S\ref{subsec:proof_eq_H_{1}_losing_0}, we know that $P_{2}$ loses. \qed
\end{enumerate}

We recall for the reader here that the proof of Lemma~\ref{lem:H_{1}_winning_galaxy} has been included in \S\ref{subsec:proof_lem_H_{1}_winning_galaxy} of \S\ref{sec:appendix}.

\subsection{Proof that a configuration that is either of the form given by \eqref{eq:H_{1}_losing_2,4_case_1} or of the form given by \eqref{eq:H_{1}_losing_2,4_case_2}, for $r=2$, is losing}\label{subsec:proof_eq_H_{1}_losing_2_1,2}
Before we begin \S\ref{subsec:proof_eq_H_{1}_losing_2_1,2}, we remind the reader that our claim, that any configuration of the form given by \eqref{eq:H_{1}_losing_0,1,3} with $r=1$ is losing, has already been proved in \S\ref{subsec:H_{1}_losing_1_proof} of \S\ref{sec:appendix}. This is going to be of use to us in the proofs that follow. That configurations that are either of the form given by \eqref{eq:H_{1}_losing_2,4_case_1} or of the form given by \eqref{eq:H_{1}_losing_2,4_case_2}, with $r=2$, are losing, has been proved \emph{together}, via a single inductive argument in \S\ref{subsec:proof_eq_H_{1}_losing_2_1,2}. The base cases for this inductive argument have been addressed in \S\ref{subsec:H_{1}_losing_2_1,2_base_cases} of \S\ref{sec:appendix}.

Suppose, for some $K \in \mathbb{N}$, we have proved that 
\begin{enumerate}
\item any initial weight configuration on $H_{1}$, that is of the form given by \eqref{eq:H_{1}_losing_2,4_case_1} with $r=2$, and having $w_{0}(EF)=k \leqslant K$, is losing, 
\item any initial weight configuration on $H_{1}$, that is of the form given by \eqref{eq:H_{1}_losing_2,4_case_2} with $r=2$, and having $w_{0}(EF)=k \leqslant K$, is losing.
\end{enumerate}

The first scenario to consider is where $K\equiv 0 \bmod 4$, and the initial weight configuration is
\begin{equation}
w_{0}(AB)=2^{R+1}m+2,\ w_{0}(BC)=1,\ w_{0}(CD)=2^{R+1}m+K+2,\ w_{0}(EF)=K+1.\label{eq:H_{1}_losing_2_inductive_1}
\end{equation}
\begin{enumerate}
\item Suppose $P_{1}$ removes a positive integer weight from $CD$ and a non-negative integer weight from $BC$ in the first round. We consider a few subcases of this case:
\begin{enumerate}
\item Suppose $w_{1}(CD)=2^{R+1}m+\ell$ for some $\ell \leqslant K+1$ with $\ell\equiv 2\bmod 4$ (which actually tells us that we must have $\ell\leqslant K-2$), and $w_{1}(BC)=1$. Here, $P_{2}$ removes weight $(K+1)-(\ell-1)$ from the edge $EF$ in the second round, so that $P_{1}$ is left with
\begin{align}
{}&w_{2}(AB)=2^{f(\ell-1)+1}2^{R-f(\ell-1)}m+2,\ w_{2}(BC)=1,\nonumber\\
{}&w_{2}(CD)=2^{f(\ell-1)+1}2^{R-f(\ell-1)}m+\ell,\ w_{2}(EF)=(\ell-1).\nonumber
\end{align}
Since $(\ell-1) \equiv 1 \bmod 4$, this configuration is of the form given by \eqref{eq:H_{1}_losing_2,4_case_1} with $r=2$ and $w_{2}(EF)=(\ell-1)\leqslant (K-3)<K$, so that $P_{1}$ loses by our induction hypothesis. 

\item Suppose $w_{1}(CD)=2^{R+1}m+\ell$ for some $\ell \leqslant K-3$ with $\ell\equiv i \bmod 4$ for some $i \in \{0,1,3\}$, and $w_{1}(BC)=1$. Here, $P_{2}$ removes weight $(K+1)-(\ell+3)$ from the edge $EF$ in the second round, so that $P_{1}$ is left with 
\begin{align}
{}&w_{2}(AB)=2^{f(\ell+3)+1}2^{R-f(\ell+3)}m+2,\ w_{2}(BC)=1,\nonumber\\
{}&w_{2}(CD)=2^{f(\ell+3)+1}2^{R-f(\ell+3)}m+\ell,\ w_{2}(EF)=(\ell+3).\nonumber
\end{align}
Since $(\ell+3) \equiv j \bmod 4$ for some $j \in \{0,2,3\}$, this configuration is of the form given by \eqref{eq:H_{1}_losing_2,4_case_2} with $r=2$ and $w_{2}(EF)=(\ell+3)\leqslant K$, so that $P_{1}$ loses by our induction hypothesis. 

\item Suppose $w_{1}(CD)=2^{R+1}m+\ell$ for $\ell\in\{K-1,K\}$, and $w_{1}(BC)=1$. Then $P_{2}$ removes weight $2-(K-\ell)$ from the edge $AB$ in the second round, leaving $P_{1}$ with
\begin{align}
w_{2}(AB)=2^{R+1}m+(K-\ell),\ w_{2}(BC)=1,\ w_{2}(CD)=2^{R+1}m+\ell,\ w_{2}(EF)=(K+1),\nonumber
\end{align}
which is of the form given by \eqref{eq:H_{1}_losing_0,1,3} with $r=1$ when $\ell=K-1$, and of the form given by \eqref{eq:H_{1}_losing_0,1,3} with $r=0$ when $\ell=K$, and by what we have already proved in \S\ref{subsec:proof_eq_H_{1}_losing_0} and \S\ref{subsec:H_{1}_losing_1_proof}, we conclude that $P_{1}$ loses.

\item Suppose $w_{1}(CD)=2^{R+1}m+\ell$ for some $\ell \leqslant K$, and $w_{1}(BC)=0$. This leaves $P_{2}$ with a galaxy graph consisting of the edges $AB$, $CD$ and $EF$, where $w_{1}(AB)=2^{R+1}m+2$, $w_{1}(CD)=2^{R+1}m+\ell$ and $w_{1}(EF)=(K+1)$. If $\ell \leqslant K-2$, $P_{2}$ wins by Lemma~\ref{lem:H_{1}_winning_galaxy}. 

Suppose, now, that $\ell=K$. Since $K \equiv 0 \bmod 4$, the triple $(2^{R+1}m+1,2^{R+1}m+K,K+1)$ is balanced, which means that the triple $(2^{R+1}m+2,2^{R+1}m+K,K+1)$ is unbalanced, so that $P_{2}$ wins by Theorem~\ref{thm:galaxy}. Likewise, if $\ell=(K-1)$, since $K \equiv 0 \bmod 4$, the triple $(2^{R+1}m+2,2^{R+1}m+K-1,K-3)$ is balanced, which means that the triple $(2^{R+1}m+2,2^{R+1}m+K-1,K+1)$ is unbalanced, and $P_{2}$ wins by Theorem~\ref{thm:galaxy}. 

\item If $w_{1}(CD)=2^{R+1}m+(K+1)$, then the configuration after the first round is of the form stated in Lemma~\ref{prop:H_{1}_winning_1}, with $m_{1}=m_{2}=m$, $\ell_{1}=2$, $\ell_{2}=(K+1)$ and $k=(K+1)$. Since $k=\ell_{1}$ (so that $k \in \{\ell_{1},\ell_{2}\}$) and $\min\{\ell_{1},\ell_{2}\}\neq 0$, hence $P_{2}$ wins by Lemma~\ref{prop:H_{1}_winning_1}.

\item If $w_{1}(CD)=2^{R+1}n+\ell_{1}$ for some $n < m$ and $\ell\in\{0,1,\ldots,2^{R+1}-1\}$, the configuration is of the same form as in Lemma~\ref{prop:H_{1}_winning_1}, with $m_{1}=n$, $m_{2}=m$ and $\ell_{2}=2$, and as $n < m \implies m_{1}\neq m_{2}$, we know that $P_{2}$ wins by Lemma~\ref{prop:H_{1}_winning_1}.
\end{enumerate}

\item Suppose $P_{1}$ removes a positive integer weight from $AB$, and a non-negative integer weight from $BC$, in the first round. We consider the following subcases:
\begin{enumerate}
\item Suppose $w_{1}(AB)=2^{R+1}m+\ell$ for some $\ell \in \{0,1\}$, and $w_{1}(BC)=1$. Then $P_{2}$ removes weight $(2+\ell)$ from the edge $CD$ in the second round, leaving $P_{1}$ with
\begin{align}
w_{2}(AB)=2^{R+1}m+\ell,\ w_{2}(BC)=1,\ w_{2}(CD)=2^{R+1}m+(K-\ell),\ w_{2}(EF)=(K+1),\nonumber
\end{align}
which is of the form given by \eqref{eq:H_{1}_losing_0,1,3} with either $r=0$ (which happens when $\ell=0$) or with $r=1$ (which happens when $\ell=1$), and by what we have already proved in \S\ref{subsec:proof_eq_H_{1}_losing_0} and \S\ref{subsec:H_{1}_losing_1_proof}, we conclude that $P_{1}$ loses.

\item Suppose $w_{1}(AB)=2^{R+1}m+\ell$ for some $\ell \in \{0,1\}$, and $w_{1}(BC)=0$. This leaves $P_{2}$ with a galaxy graph consisting of the edges $AB$, $CD$ and $EF$, where $w_{1}(AB)=2^{R+1}m+\ell$, $w_{1}(CD)=2^{R+1}m+(K+2)$ and $w_{1}(EF)=(K+1)$. Since $K \equiv 0 \bmod 4$, the triple $(2^{R+1}m+1,2^{R+1}m+K,K+1)$ is balanced when $\ell=1$, and the triple $(2^{R+1}m,2^{R+1}m+K+1,K+1)$ is balanced when $\ell=0$. In either case, therefore, the triple $(2^{R+1}m+\ell,2^{R+1}m+(K+2),K+1)$ is not balanced, and therefore, $P_{2}$ wins by Theorem~\ref{thm:galaxy}.

\item If $w_{1}(AB)=2^{R+1}n+\ell_{1}$ for some $n < m$ and $\ell\in\{0,1,\ldots,2^{R+1}-1\}$, the configuration is of the same form as in Lemma~\ref{prop:H_{1}_winning_1}, with $m_{1}=n$, $m_{2}=m$ and $\ell_{2}=K+2$, and as $n < m \implies m_{1}\neq m_{2}$, we know that $P_{2}$ wins by Lemma~\ref{prop:H_{1}_winning_1}.
\end{enumerate}

\item Suppose $P_{1}$ removes a positive integer weight from $EF$ in the first round, so that $w_{1}(EF)=k\leqslant K$ (this implies $f(k)\leqslant f(K+1)=R$). If $k \geqslant 1$ and $k \equiv 1 \bmod 4$, then $P_{2}$ removes weight $(K+2)-(k+1)$ from the edge $CD$ in the second round, leaving $P_{1}$ with
\begin{align}
w_{2}(AB)=2^{f(k)+1}2^{R-f(k)}m+2,\ w_{2}(BC)=1,\ w_{2}(CD)=2^{f(k)+1}2^{R-f(k)}m+(k+1),\ w_{2}(EF)=k,\nonumber
\end{align}
which is of the same form as \eqref{eq:H_{1}_losing_2,4_case_1} with $r=2$ and $s=1$, and as $w_{2}(EF)=k \leqslant K$, we conclude that $P_{1}$ loses by our induction hypothesis. On the other hand, if $k \geqslant 6$ and $k \equiv i \bmod 4$ for some $i \in \{0,2,3\}$, then $P_{2}$ removes weight $(K+2)-(k-3)$ from the edge $CD$ in the second round, leaving $P_{1}$ with
\begin{align}
w_{2}(AB)=2^{f(k)+1}2^{R-f(k)}m+2,\ w_{2}(BC)=1,\ w_{2}(CD)=2^{f(k)+1}2^{R-f(k)}m+(k-3),\ w_{2}(EF)=k,\nonumber
\end{align}
which is of the same form as \eqref{eq:H_{1}_losing_2,4_case_2} with $r=2$ and $s=1$, and as $w_{2}(EF)=k \leqslant K$, we conclude that $P_{1}$ loses by our induction hypothesis. If $k=2$, then the configuration after the first round is of the form given by Lemma~\ref{prop:H_{1}_winning_1}, with $\ell_{1}=2$, $\ell_{2}=(K+2)$ and $k=\ell_{1}$, which implies that $k \in \{\ell_{1},\ell_{2}\}$. Hence, $P_{2}$ wins by Lemma~\ref{prop:H_{1}_winning_1}. If $k=3$, then, letting $K=4n$ for some $n\in\mathbb{N}$ (since $K \equiv 0 \bmod 4$), the configuration obtained after the first round can be written as follows:
\begin{align}
{}&w_{1}(AB)=2^{R+1}m+2=4m_{1}+2, \text{ where } m_{1}=2^{R-1}m,\nonumber\\  
{}&w_{1}(CD)=2^{R+1}m+(K+2)=2^{R+1}m+4n+2=4m_{2}+2, \text{ where } m_{2}=2^{R-1}m+n,\nonumber
\end{align}
along with $w_{1}(BC)=1$ and $w_{1}(EF)=3$. This is of the same form as the configuration described in Lemma~\ref{prop:H_{1}_winning_1}, with $m_{1} \neq m_{2}$, which leads to the conclusion that $P_{2}$ wins. If $k=4$, then $P_{2}$ removes weight $(K+1)$ from the edge $CD$ in the second round, leaving $P_{1}$ with 
\begin{equation}
w_{2}(AB)=2^{R+1}m+2,\ w_{2}(BC)=1,\ w_{2}(CD)=2^{R+1}m+1,\ w_{2}(EF)=4,\nonumber
\end{equation}
which is of the same form as \eqref{eq:H_{1}_losing_0,1,3} with $r=1$, and by what we have already proved in \S\ref{subsec:H_{1}_losing_1_proof}, we know that $P_{1}$ loses. Finally, if $k=0$, $P_{2}$ wins by Remark~\ref{rem:EF_edgeweight_0}.

\item Suppose $P_{1}$ removes the edge $BC$ in the first round, without disturbing the edge-weights of $AB$ and $CD$. This leaves $P_{2}$ with a galaxy graph consisting of the edges $AB$, $CD$ and $EF$, where $w_{1}(AB)=2^{R+1}m+2$, $w_{1}(CD)=2^{R+1}m+(K+2)$ and $w_{1}(EF)=(K+1)$. Since $K \equiv 0 \bmod 4$, the triple $(2^{R+1}m+2,2^{R+1}m+(K+2),K)$ is balanced, and consequently, the triple $(2^{R+1}m+2,2^{R+1}m+(K+2),(K+1))$, is unbalanced. Therefore, $P_{2}$ wins by Theorem~\ref{thm:galaxy}.
\end{enumerate}

This completes the proof of our claim that the configuration in \eqref{eq:H_{1}_losing_2_inductive_1} is losing when $K \equiv 0 \bmod 4$.

We now let $K \equiv i \bmod 4$ for some $i \in \{1,2,3\}$, with $K \geqslant 6$, and consider the configuration
\begin{align}
w_{0}(AB)=2^{R+1}m+2,\ w_{0}(BC)=1,\ w_{0}(CD)=2^{R+1}m+(K-2),\ w_{0}(EF)=(K+1),\label{eq:H_{1}_losing_2_inductive_2}
\end{align}
where we set $R=f(K+1)$. The first round of the game played on this configuration can unfold as follows:
\begin{enumerate}
\item Suppose $P_{1}$ removes a positive integer weight from $CD$ and a non-negative integer weight from $BC$ in the first round. This leads to a few possible subcases:
\begin{enumerate}
\item Suppose $w_{1}(CD)=2^{R+1}m+\ell$ for some $\ell \in \{2,3,\ldots,K-3\}$, and $w_{1}(BC)=1$. If $\ell \equiv j \bmod 4$ for some $j \in \{0,1,3\}$, then $P_{2}$ removes weight $(K+1)-(\ell+3)$ from the edge $EF$ in the second round, leaving $P_{1}$ with
\begin{align}
{}&w_{2}(AB)=2^{f(\ell+3)+1}2^{R-f(\ell+3)}m+2,\ w_{2}(BC)=1,\nonumber\\
{}&w_{2}(CD)=2^{f(\ell+3)+1}2^{R-f(\ell+3)}m+\ell,\ w_{2}(EF)=(\ell+3).\nonumber
\end{align}
This configuration is of the same form as \eqref{eq:H_{1}_losing_2,4_case_2} with $r=2$, $s=1$ and $w_{2}(EF)=(\ell+3)\leqslant K$, so that $P_{1}$ loses by our induction hypothesis. If $\ell\equiv 2 \bmod 4$, then $P_{2}$ removes weight $(K+1)-(\ell-1)$ from the edge $EF$ in the second round, leaving $P_{1}$ with
\begin{align}
{}&w_{2}(AB)=2^{f(\ell-1)+1}2^{R-f(\ell-1)}m+2,\ w_{2}(BC)=1,\nonumber\\
{}&w_{2}(CD)=2^{f(\ell-1)+1}2^{R-f(\ell-1)}m+\ell,\ w_{2}(EF)=(\ell-1).\nonumber
\end{align}
This configuration is of the same form as \eqref{eq:H_{1}_losing_2,4_case_1} with $r=2$, $s=1$ and $w_{2}(EF)=(\ell-1)\leqslant (K-4)<K$, so that $P_{1}$ loses by our induction hypothesis. If $\ell\in\{0,1\}$, then $P_{2}$ removes weight $(K+1)-(\ell+3)$ from the edge $EF$ in the second round, leaving $P_{1}$ with
\begin{equation}
w_{2}(AB)=2^{R+1}m+2,\ w_{2}(BC)=1,\ w_{2}(CD)=2^{R+1}m+\ell,\ w_{2}(EF)=(\ell+3),\nonumber
\end{equation}
which is of the form given by \eqref{eq:H_{1}_losing_0,1,3} with $r=0$ (when $\ell=0$) or $r=1$ (when $\ell=1$), and by what we have already proved in \S\ref{subsec:proof_eq_H_{1}_losing_0} and \S\ref{subsec:H_{1}_losing_1_proof}, we know that $P_{1}$ loses.

\item Suppose $w_{1}(CD)=2^{R+1}m+\ell$ for some $\ell \in \{0,1,\ldots,K-3\}$ and $w_{1}(BC)=0$. This leaves $P_{2}$ with a galaxy graph consisting of the edges $AB$, $CD$ and $EF$, where $w_{1}(AB)=2^{R+1}m+2$, $w_{1}(CD)=2^{R+1}m+\ell$ and $w_{1}(EF)=(K+1)$. Since $\ell\leqslant (K-3) \implies (2+\ell)\leqslant (K-1) < (K+1)$, $P_{2}$ wins by Lemma~\ref{lem:H_{1}_winning_galaxy}. 

\item If $w_{1}(CD)=2^{R+1}n+\ell_{1}$ for some $n < m$ and $\ell\in\{0,1,\ldots,2^{R+1}-1\}$, the configuration is of the same form as in Lemma~\ref{prop:H_{1}_winning_1}, with $m_{1}=n$, $m_{2}=m$ and $\ell_{2}=2$, and as $n < m \implies m_{1}\neq m_{2}$, we know that $P_{2}$ wins by Lemma~\ref{prop:H_{1}_winning_1}.
\end{enumerate}

\item Suppose $P_{1}$ removes a positive integer weight from $AB$, and a non-negative integer weight from $CD$, in the first round. Once again, we consider a few subcases:
\begin{enumerate}
\item Suppose $w_{1}(AB)=2^{R+1}m+\ell$ for some $\ell \in \{0,1\}$, and $w_{1}(BC)=1$. Then $P_{2}$ removes weight $(2-\ell)$ from the edge $EF$ in the second round, so that $P_{1}$ is left with 
\begin{equation}
w_{2}(AB)=2^{R+1}m+\ell,\ w_{2}(BC)=1,\ w_{2}(CD)=2^{R+1}m+(K-2),\ w_{2}(EF)=(K-1+\ell),\nonumber
\end{equation}
which is of the form given by \eqref{eq:H_{1}_losing_0,1,3} with either $r=0$ (when $\ell=0$) or $r=1$ (when $\ell=1$), and by what we have already proved in \S\ref{subsec:proof_eq_H_{1}_losing_0} and \S\ref{subsec:H_{1}_losing_1_proof}, we conclude that $P_{1}$ loses. 

\item Suppose $w_{1}(AB)=2^{R+1}m+\ell$ for some $\ell \in \{0,1\}$, and $w_{1}(BC)=0$. This leaves $P_{2}$ with a galaxy graph consisting of the edges $AB$, $CD$ and $EF$, where $w_{1}(AB)=2^{R+1}m+\ell$, $w_{1}(CD)=2^{R+1}m+(K-2)$ and $w_{2}(EF)=(K+1)$, and $P_{2}$ wins by Lemma~\ref{lem:H_{1}_winning_galaxy}.

\item If $w_{1}(AB)=2^{R+1}n+\ell_{1}$ for some $n < m$ and $\ell\in\{0,1,\ldots,2^{R+1}-1\}$, the configuration is of the same form as in Lemma~\ref{prop:H_{1}_winning_1}, with $m_{1}=n$, $m_{2}=m$ and $\ell_{2}=K-2$, and as $n < m \implies m_{1}\neq m_{2}$, we know that $P_{2}$ wins by Lemma~\ref{prop:H_{1}_winning_1}.
\end{enumerate}

\item Suppose $P_{1}$ removes a positive integer weight from $EF$ in the first round, so that $w_{1}(EF)=k\leqslant K$:
\begin{enumerate}
\item When $k \in \{K-1,K\}$, $P_{2}$ removes weight $(K+1-k)$ from the edge $AB$ in the second round, leaving $P_{1}$ with
\begin{align}
w_{2}(AB)=2^{R+1}m+k+1-K,\ w_{2}(BC)=1,\ w_{2}(CD)=2^{R+1}m+(K-2),\ w_{2}(EF)=k,\nonumber
\end{align}
which is of the form given by \eqref{eq:H_{1}_losing_0,1,3} with $r=0$ when $k=K-1$, and of the form given by \eqref{eq:H_{1}_losing_0,1,3} with $r=1$ when $k=K$, and by what we have already proved in \S\ref{subsec:proof_eq_H_{1}_losing_0} and \S\ref{subsec:H_{1}_losing_1_proof}, we know that $P_{1}$ loses. If $k=(K-2)$, the configuration after the first round is of the form given by Lemma~\ref{prop:H_{1}_winning_1}, with $m_{1}=m_{2}=2^{R-f(K-2)}m$, $\ell_{1}=2$ and $\ell_{2}=(K-2)$, so that $k \in \{\ell_{1},\ell_{2}\}$, and hence, $P_{2}$ wins by Lemma~\ref{prop:H_{1}_winning_1}.

\item Suppose $k\in \{1,2,\ldots,K-3\}$, such that $k \equiv 1 \bmod 4$. Since $K \equiv i \bmod 4$ for some $i \in \{1,2,3\}$, we have $(K-3) \equiv j \bmod 4$ for some $j \in \{0,2,3\}$. Thus, this case does not cover $k=(K-3)$, and we may as well consider all $k \in \{1,2,\ldots,K-4\}$ such that $k \equiv 1 \bmod 4$. Here, $P_{2}$ removes weight $(K-2)-(k+1)$ from the edge $CD$ in the second round, leaving $P_{1}$ with
\begin{align}
{}&w_{2}(AB)=2^{f(k)+1}2^{R-f(k)}m+2,\ w_{2}(BC)=1,\nonumber\\
{}&w_{2}(CD)=2^{f(k)+1}2^{R-f(k)}m+(k+1),\ w_{2}(EF)=k,\nonumber
\end{align} 
which is of the form given by \eqref{eq:H_{1}_losing_2,4_case_1} with $r=2$, $s=1$ and $w_{2}(EF)=k\leqslant K$, so that $P_{1}$ loses by our induction hypothesis.

\item Suppose $k\in\{6,7,\ldots,K-3\}$ such that $k \equiv j \bmod 4$ for some $j \in \{0,2,3\}$. In this case, $P_{2}$ removes weight $(K-2)-(k-3)$ from the edge $CD$ in the second round, leaving $P_{1}$ with
\begin{align}
{}&w_{2}(AB)=2^{f(k)+1}2^{R-f(k)}m+2,\ w_{2}(BC)=1,\nonumber\\
{}&w_{2}(CD)=2^{f(k)+1}2^{R-f(k)}m+(k-3),\ w_{2}(EF)=k,\nonumber
\end{align} 
which is of the form given by \eqref{eq:H_{1}_losing_2,4_case_2} with $r=2$, $s=1$ and $w_{2}(EF)=k\leqslant K$, so that $P_{1}$ loses by our induction hypothesis.

\item Suppose $k \in \{2,3,4\}$. If $k=2$, then the configuration after the first round is of the same form as in Lemma~\ref{prop:H_{1}_winning_1}, with $m_{1}=m_{2}=2^{R-1}m$ (since $2^{f(2)+1}=4$), $\ell_{1}=2$ and $\ell_{2}=(K-2)$, so that $k\in \{\ell_{1},\ell_{2}\}$, and hence, $P_{2}$ wins by Lemma~\ref{prop:H_{1}_winning_1}. If $k\in \{3,4\}$, $P_{2}$ removes weight $(K+1-k)$ from $CD$ during the second round, so that $P_{1}$ is left with 
\begin{align}
w_{2}(AB)=2^{R+1}m+2,\ w_{2}(BC)=1,\ w_{2}(CD)=2^{R+1}m+(k-3),\ w_{2}(EF)=k,\nonumber
\end{align}
which is of the form given by \eqref{eq:H_{1}_losing_0,1,3} with $r=0$ when $k=3$, and of the form given by \eqref{eq:H_{1}_losing_0,1,3} with $r=1$ when $k=4$, and by what we have already proved in \S\ref{subsec:proof_eq_H_{1}_losing_0} and \S\ref{subsec:H_{1}_losing_1_proof}, we conclude that $P_{1}$ loses.
\end{enumerate}

\item Suppose $P_{1}$ removes the edge $BC$ in the first round, without perturbing the edge-weights of $AB$ and $CD$. This leaves $P_{2}$ with a galaxy graph consisting of the edges $AB$, $CD$ and $EF$, where $w_{1}(AB)=2^{R+1}m+2$, $w_{1}(CD)=2^{R+1}m+(K-2)$ and $w_{1}(EF)=(K+1)$, and $P_{2}$ wins by Lemma~\ref{lem:H_{1}_winning_galaxy}. 
\end{enumerate}

This completes the proof of our claim that \eqref{eq:H_{1}_losing_2_inductive_2} is a losing configuration on $H_{1}$. This also brings us to the end of our inductive proof of the claim that each of \eqref{eq:H_{1}_losing_2,4_case_1} and \eqref{eq:H_{1}_losing_2,4_case_2}, with $r=2$, represents a losing configuration on $H_{1}$.

\subsection{Proof that a configuration that is of the form given by \eqref{eq:H_{1}_losing_2,4_case_3}, with $r=2$, is losing}\label{subsec:proof_eq_H_{1}_losing_2_3}
The base case for the inductive argument employed for this proof has been addressed in \S\ref{subsec:H_{1}_losing_2_3_base_case} of \S\ref{sec:appendix}. Suppose, now, that for some $K \in \mathbb{N}$ with $K \geqslant 6$, we have shown that a configuration of the form given by \eqref{eq:H_{1}_losing_2,4_case_3}, with $r=2$, is losing as long as $w_{0}(EF)=k \leqslant K$. We now consider the configuration (setting $s=(K-L-1)$):
\begin{align}
w_{0}(AB)=2^{R+1}m+2,\ w_{0}(BC)=K-L-1,\ w_{0}(CD)=2^{R+1}m+L,\ w_{0}(EF)=K+1,\label{eq:H_{1}_losing_2_3_inductive}
\end{align}
where $R=f(K+1)$ and $L\in\{2,3,\ldots,K-3\}$. The first round of the game played on this configuration unfolds as follows:
\begin{enumerate}
\item Suppose $P_{1}$ removes a positive integer weight from $CD$ and a non-negative integer weight from $BC$ in the first round. First, we consider the possibility that $w_{1}(CD)=2^{R+1}m+\ell$ for some $\ell \in \{0,1,\ldots,L-1\}$ and $w_{1}(BC)=t\in \{0,1,\ldots,K-L-1\}$. This can be divided into a few subcases:
\begin{enumerate}
\item Suppose $\ell \in \{0,1\}$ and $t \in \{1,2,\ldots,K-L-1\}$, or $\ell \in \{2,3,\ldots,L-1\}$ and $t \in \{2,3,\ldots,K-L-1\}$. In each of these cases, $P_{2}$ removes weight $(K+1)-(\ell+t+2)$ from the edge $EF$ in the second round. Note that $\ell\leqslant (L-1)$ and $t \leqslant (K-L-1)$ together imply $(\ell+t+2)\leqslant K$, which, in turn, implies $f(\ell+t+2)\leqslant f(K+1)=R$. This leaves $P_{1}$ with
\begin{align}
{}&w_{2}(AB)=2^{f(\ell+t+2)+1}2^{R-f(\ell+t+2)}m+2,\ w_{2}(BC)=t,\nonumber\\
{}&w_{2}(CD)=2^{f(\ell+t+2)+1}2^{R-f(\ell+t+2)}m+\ell,\ w_{2}(EF)=(\ell+t+2),\nonumber
\end{align}
which is of the form given by \eqref{eq:H_{1}_losing_0,1,3} with $r=0$ when $\ell=0$, of the form given by \eqref{eq:H_{1}_losing_0,1,3} with $r=1$ when $\ell=1$, and of the form given by \eqref{eq:H_{1}_losing_2,4_case_3} with $r=2$ and $w_{2}(EF)=(\ell+t+2)\leqslant K$ when $\ell \in \{2,3,\ldots,L-1\}$ and $t \in \{2,3,\ldots,K-L-1\}$. In the first two cases, $P_{1}$ loses by what we have already proved in \S\ref{subsec:proof_eq_H_{1}_losing_0} and \S\ref{subsec:H_{1}_losing_1_proof}, while in the third case, $P_{1}$ loses by our induction hypothesis.

\item Suppose $\ell \in \{2,3,\ldots,L-1\}$ and $t=1$. If $\ell \equiv 2 \bmod 4$, then $P_{2}$ removes weight $(K+1)-(\ell-1)$ from the edge $EF$ in the second round, leaving $P_{1}$ with
\begin{align}
{}&w_{2}(AB)=2^{f(\ell-1)+1}2^{R-f(\ell-1)}m+2,\ w_{2}(BC)=1,\nonumber\\
{}&w_{2}(CD)=2^{f(\ell-1)+1}2^{R-f(\ell-1)}m+\ell,\ w_{2}(EF)=(\ell-1),\nonumber
\end{align}
which is of the same form as \eqref{eq:H_{1}_losing_2,4_case_1} with $r=2$ and $s=1$, and by what we have already proved in \S\ref{subsec:proof_eq_H_{1}_losing_2_1,2}, we conclude that $P_{1}$ loses. If $\ell \equiv i \bmod 4$ for some $i \in \{0,1,3\}$, then $P_{2}$ removes $(K+1)-(\ell+3)$ from the edge $EF$ in the second round (note that $\ell \leqslant L-1$ and $L \leqslant K-3$ together imply that $\ell+3\leqslant K-1$), leaving $P_{1}$ with
\begin{align}
{}&w_{2}(AB)=2^{f(\ell+3)+1}2^{R-f(\ell+3)}m+2,\ w_{2}(BC)=1,\nonumber\\
{}&w_{2}(CD)=2^{f(\ell+3)+1}2^{R-f(\ell+3)}m+\ell,\ w_{2}(EF)=(\ell+3),\nonumber
\end{align}
which is of the same form as \eqref{eq:H_{1}_losing_2,4_case_2} with $r=2$ and $s=1$, and by what we have already proved in \S\ref{subsec:proof_eq_H_{1}_losing_2_1,2}, we know that $P_{1}$ loses. 

\item If $\ell\in\{0,1,\ldots,L-1\}$ and $t=0$, $P_{2}$, after the first round, is left with a galaxy graph consisting of the edges $AB$, $CD$ and $EF$, where $w_{1}(AB)=2^{R+1}m+2$, $w_{1}(CD)=2^{R+1}m+\ell$ and $w_{1}(EF)=(K+1)$. Since $\ell\leqslant(L-1)$ and $L\leqslant (K-3)$, we conclude, by Lemma~\ref{lem:H_{1}_winning_galaxy}, that $P_{2}$ wins.
\end{enumerate}

The second possibility is that $w_{1}(CD)=2^{R+1}n+\ell_{1}$ for some $n < m$ and some $\ell_{1}\in\{0,1,\ldots,2^{R+1}-1\}$. In this case, the configuration after the first round is of the same form as in Lemma~\ref{prop:H_{1}_winning_1}, with $m_{1}=n$, $m_{2}=m$ and $\ell_{2}=2$, and as $n < m \implies m_{1}\neq m_{2}$, we know that $P_{2}$ wins by Lemma~\ref{prop:H_{1}_winning_1}.

\item Suppose $P_{1}$ removes a positive integer weight from $AB$, and a non-negative integer weight from $BC$, in the first round. The first possibility is where $w_{1}(AB)=2^{R+1}m+\ell$ for some $\ell \in \{0,1\}$, and $w_{1}(BC)=t\in\{0,1,\ldots,K-L-1\}$. We divide this into a few subcases:
\begin{enumerate}
\item If $t\in \{1,2,\ldots,K-L-1\}$, then $P_{2}$ removes weight $(K+1)-(\ell+t+L)$ from the edge $EF$ in the second round. Note that $t \leqslant K-L-1$ and $\ell \leqslant 1$ together imply that $(\ell+t+L) \leqslant K$. This leaves $P_{1}$ with
\begin{align}
{}&w_{2}(AB)=2^{f(\ell+t+L)+1}2^{R-f(\ell+t+L)}m+\ell,\ w_{2}(BC)=t,\nonumber\\
{}&w_{2}(CD)=2^{f(\ell+t+L)+1}2^{R-f(\ell+t+L)}m+L,\ w_{2}(EF)=(\ell+t+L),\nonumber
\end{align}
which is of the same form as \eqref{eq:H_{1}_losing_0,1,3} with either $r=0$ when $\ell=0$ or $r=1$ when $\ell=1$, and by what we have already proved in \S\ref{subsec:proof_eq_H_{1}_losing_0} and \S\ref{subsec:H_{1}_losing_1_proof}, we conclude that $P_{1}$ loses.

\item If $t=0$, then $P_{2}$, after the first round, is left with a galaxy graph consisting of the edges $AB$, $CD$ and $EF$, where $w_{1}(AB)=2^{R+1}m+\ell$, $w_{1}(CD)=2^{R+1}m+L$ and $w_{1}(EF)=(K+1)$. Since $\ell\leqslant 1$ and $L\leqslant(K-3)$, we conclude, by Lemma~\ref{lem:H_{1}_winning_galaxy}, that $P_{2}$ wins.
\end{enumerate}

The second possibility is that $w_{1}(AB)=2^{R+1}n+\ell_{1}$ for some $n < m$ and some $\ell_{1}\in\{0,1,\ldots,2^{R+1}-1\}$. In this case, the configuration after the first round is of the same form as in Lemma~\ref{prop:H_{1}_winning_1}, with $m_{1}=n$, $m_{2}=m$ and $\ell_{2}=L$, and as $n < m \implies m_{1}\neq m_{2}$, we know that $P_{2}$ wins by Lemma~\ref{prop:H_{1}_winning_1}.

\item Suppose $P_{1}$ removes a positive integer weight from $EF$ in the first round. If $w_{1}(EF)=k\in\{K-L+1,\ldots,K\}$, $P_{2}$ removes weight $(K+1-k)$ from $CD$ in the second round, leaving $P_{1}$ with
\begin{align}
{}&w_{2}(AB)=2^{f(k)+1}2^{R-f(k)}m+2,\ w_{2}(BC)=K-L-1,\nonumber\\
{}&w_{2}(CD)=2^{f(k)+1}2^{R-f(k)}m+L+k-K-1,\ w_{2}(EF)=k,\nonumber
\end{align}
which is of the form given by \eqref{eq:H_{1}_losing_0,1,3} with $r=0$ when $k=K-L+1$, of the form given by \eqref{eq:H_{1}_losing_0,1,3} with $r=1$ when $k=K-L+2$, and of the form given by \eqref{eq:H_{1}_losing_2,4_case_3} with $r=2$ and $w_{2}(EF)=k\leqslant K$ when $k\in\{K-L+3,\ldots,K\}$. That $P_{1}$ loses in the first two cases follows from what we have already proved in \S\ref{subsec:proof_eq_H_{1}_losing_0} and \S\ref{subsec:H_{1}_losing_1_proof}, and that she loses in the third case follows from our induction hypothesis. If $k\in\{1,2,\ldots,K-L\}$, then the configuration after the first round is of the same form as in Lemma~\ref{lem:H_{1}_winning_2}, with $\ell_{1}=2$, $\ell_{2}=L$ and $w_{1}(BC)=(K-L-1)>w_{1}(EF)-\min\{\ell_{1},\ell_{2}\}$, so that $P_{2}$ wins by Lemma~\ref{lem:H_{1}_winning_2}. If $k=0$, $P_{2}$ wins by Remark~\ref{rem:EF_edgeweight_0}.

\item Suppose $P_{1}$ removes a positive integer weight from $BC$ in the first round, without disturbing the edge-weights of $AB$ and $CD$, so that $w_{1}(BC)=t \in\{0,1,\ldots,K-L-2\}$.
\begin{enumerate}
\item If $t \in \{2,3,\ldots,K-L-2\}$, $P_{2}$ removes weight $(K+1)-(2+t+L)$ from the edge $EF$ in the second round. Note that $(2+t+L)\leqslant 2+(K-L-2)+L=K$. This leaves $P_{1}$ with
\begin{align}
{}&w_{2}(AB)=2^{f(2+t+L)+1}2^{R-f(2+t+L)}m+2,\ w_{2}(BC)=t,\nonumber\\
{}&w_{2}(CD)=2^{f(2+t+L)+1}2^{R-f(2+t+L)}m+L,\ w_{2}(EF)=(2+t+L),\nonumber
\end{align}
which is of the same form as \eqref{eq:H_{1}_losing_2,4_case_3} with $r=2$ and $w_{2}(BC)=t\geqslant 2$, and as $w_{2}(EF)=(2+t+L)\leqslant K$, we conclude that $P_{1}$ loses by our induction hypothesis.

\item Suppose $t=1$. Here, if $L \equiv 2 \bmod 4$, then $P_{2}$ removes weight $(K+1)-(L-1)$ from the edge $EF$ in the second round, leaving $P_{1}$ with
\begin{align}
{}&w_{2}(AB)=2^{f(L-1)+1}2^{R-f(L-1)}m+2,\ w_{2}(BC)=1,\nonumber\\
{}&w_{2}(CD)=2^{f(L-1)+1}2^{R-f(L-1)}m+L,\ w_{2}(EF)=(L-1),\nonumber
\end{align}
which is of the form given by \eqref{eq:H_{1}_losing_2,4_case_1} with $r=2$ and $w_{1}(BC)=s=1$, and by what we have already proved in \S\ref{subsec:proof_eq_H_{1}_losing_2_1,2}, we conclude that $P_{1}$ loses. If, on the other hand, $L \equiv i \bmod 4$ for some $i \in \{0,1,3\}$, then $P_{2}$ removes weight $(K+1)-(L+3)$ from the edge $EF$ in the second round (recall that $L\leqslant (K-3) \implies (K+1)-(L+3)>0$), leaving $P_{1}$ with
\begin{align}
{}&w_{2}(AB)=2^{f(L+3)+1}2^{R-f(L+3)}m+2,\ w_{2}(BC)=1,\nonumber\\
{}&w_{2}(CD)=2^{f(L+3)+1}2^{R-f(L+3)}m+L,\ w_{2}(EF)=(L+3),\nonumber
\end{align}
which is of the form given by \eqref{eq:H_{1}_losing_2,4_case_2} with $r=2$ and $w_{1}(BC)=s=1$, and by what we have already proved in \S\ref{subsec:proof_eq_H_{1}_losing_2_1,2}, we know that $P_{1}$ loses.

\item Finally, if $t=0$, $P_{2}$ is left with a galaxy graph at the end of the first round, consisting of the edges $AB$, $CD$ and $EF$, where $w_{1}(AB)=2^{R+1}m+2$, $w_{1}(CD)=2^{R+1}m+L$ and $w_{1}(EF)=(K+1)$. Since $L\leqslant(K-3)$, we conclude, by Lemma~\ref{lem:H_{1}_winning_galaxy}, that $P_{2}$ wins.
\end{enumerate}
\end{enumerate}

This concludes the proof of our claim that the configuration in \eqref{eq:H_{1}_losing_2_3_inductive} is losing. This also concludes the inductive proof of our claim that any configuration on $H_{1}$ of the form given by \eqref{eq:H_{1}_losing_2,4_case_3}, with $r=2$, is losing.

\section{Appendix}\label{sec:appendix}
We begin this section with the proofs that have been ommitted from the main body of the paper.
\begin{proof}[Proof of Theorem~\ref{thm:galaxy}]
Let us interpret each component of $G$, i.e.\ each constituent star-graph in $G$, as a heap, and the sum $p_{i}$ of edge-weights assigned to all edges of the $i$-th star graph as the total number of chips the $i$-th heap contains at the beginning of the game. A player selecting a vertex inside the $i$-th star graph and removing non-negative-integer-valued edge-weights from each of the edges incident on that vertex, such that the total edge-weight removed, say $w$, is a strictly positive integer, in the game of Graph Nim played on $G$, is equivalent to the player selecting the $i$-th heap and removing a total of $w$ chips from it in the corresponding game of Nim. The proof now follows immediately from Theorem~\ref{thm:Nim}.
\end{proof}

\begin{proof}[Proof of Proposition~\ref{prop:characterizing_graphs}]
From the definition of $\mathcal{H}_{G}$ above, it is evident that every singleton subset of $E(G)$ is a hyperedge in $\mathcal{H}_{G}$, i.e.\ $\{e\}\in\mathcal{H}_{G}$ for each $e\in E(G)$ (this is because, if $e=\{u,v\}$ for some $u,v\in V(G)$, then $e\in N(u)$ and $e\in N(v)$, and therefore, $\{e\}$ is an element of both $2^{N(u)}$ and $2^{N(v)}$). Evidently, then, if $G$ has at least two edges, we may consider the subset $S=\{e_{1},e_{2}\}$ of $E(G)$ (which, note, is a subset of the vertices of $\mathcal{H}_{G}$), and the subhypergraph induced on $S$ would then be 
\begin{equation}\label{transversal_free_cases}
\mathcal{H}_{G,S}=
\begin{cases}
\{\{e_{1}\},\{e_{2}\}\} &\text{if }e_{1} \text{ and }e_{2} \text{ share no common vertex,}\\
\{\{e_{1}\},\{e_{2}\},\{e_{1},e_{2}\}\} &\text{if }e_{1}\text{ and }e_{2}\text{ share a common vertex}.
\end{cases}
\end{equation}
Note that, in the first of these two scenarios, the subhypergraph $\mathcal{H}_{G,S}$ is transversal-free. Therefore, the only way that
\begin{enumerate*}
\item the graph $G$ has two distinct edges, $e_{1}$ and $e_{2}$, that do \emph{not} share a common vertex,
\item the hypergraph $\mathcal{H}_{G}$ is transversal-free,
\item and no \emph{proper} non-empty subhypergraph of $\mathcal{H}_{G}$ is transversal-free, 
\end{enumerate*}
is if $\mathcal{H}_{G}=\{\{e_{1}\},\{e_{2}\}\}$, which is equivalent to saying that $G$ consists of two vertex-disjoint edges only (thus, $E(G)$ can be written as $\{\{u_{1},u_{2}\},\{v_{1},v_{2}\}\}$ where $V(G)=\{u_{1},u_{2},v_{1},v_{2}\}$). 

The only other possible scenario (with $G$ possessing more than one edge) is where, given \emph{any} two distinct edges $e_{1}$ and $e_{2}$ of $G$, they share a common vertex. This is possible if and only if 
\begin{enumerate}
\item \label{star} either $G$ is a star graph, i.e.\ $V(G)=\{u_{0},u_{1},\ldots,u_{k}\}$ with $E(G)=\{\{u_{0},u_{i}\}:i\in\{1,2,\ldots,k\}\}$,
\item \label{triangle} or $G$ is a triangle, i.e.\ $V(G)=\{u_{1},u_{2},u_{3}\}$ with $E(G)=\{\{u_{1},u_{2}\},\{u_{2},u_{3}\},\{u_{3},u_{1}\}\}$.
\end{enumerate}
In the former situation, note that $N(u_{0})=E(G)$, so that $E(G)\in\mathcal{H}_{G}$, and this serves as a transversal for $\mathcal{H}_{G}$ -- this means that $\mathcal{H}_{G}$ no longer remains transversal-free. In the latter situation, if we set $e_{1}=\{u_{i},u_{j}\}$ and $e_{2}=\{u_{j},u_{k}\}$ for $\{i,j,k\}=\{1,2,3\}$, then, by the second scenario described in \eqref{transversal_free_cases}, we see that $\{e_{1},e_{2}\}$ is a transversal for $\mathcal{H}_{G,S}$ -- thus, no proper subhypergraph of $\mathcal{H}_{G}$ is transversal-free. Moreover, the only transversal for $\mathcal{H}_{G}$ is $E(G)$ itself, but this does not appear among the hyperedges of $\mathcal{H}_{G}$, making $\mathcal{H}_{G}$ transversal-free. Thus, of \eqref{star} and \eqref{triangle}, only the latter yields a minimal transversal-free hypergraph.

The previous two paragraphs allow us to conclude that the only two simple, finite graphs $G$, with more than a single edge, corresponding to which the hypergraphs $\mathcal{H}_{G}$ are minimal transversal-free are 
\begin{enumerate}[label=(\alph*)]
\item $G$ with $V(G)=\{u_{1},u_{2},v_{1},v_{2}\}$ and $E(G)=\{\{u_{1},u_{2}\},\{v_{1},v_{2}\}\}$,
\item and $G$ with $V(G)=\{u_{1},u_{2},u_{3}\}$ and $E(G)=\{\{u_{1},u_{2}\},\{u_{2},u_{3}\},\{u_{3},u_{1}\}\}$.
\end{enumerate}
Note that if $G$ consists of a single edge, say $e$, and a finite number (which could be $0$) of isolated vertices, then, for the sake of both the Graph Nim game and the Hypergraph Nim game, we may ignore all the isolated vertices, and $\mathcal{H}_{G}=\{e\}$, which is \emph{not} transversal-free.

We now come to the characterization of finite, simple graphs $G$ corresponding to which the hypergraphs $\mathcal{H}_{G}$ are SG-decreasing. From Theorem 4 of \cite{boros2024impartial}, we know that $\mathcal{H}_{G}$ is SG-decreasing only if $\mathcal{H}_{G}$ is totally transversal -- in other words, \emph{no} non-empty induced subhypergraph of $\mathcal{H}_{G}$, including $\mathcal{H}_{G}$ itself, is transversal-free. Right away, we see that, if $G$ contains at least two distinct edges, then the first scenario in \eqref{transversal_free_cases} cannot be allowed, i.e.\ \emph{every} pair of distinct edges, $e_{1},e_{2}\in E(G)$, must share a common vertex. Once again, this admits only two possibilities, namely \eqref{star} and \eqref{triangle}. As argued above, if our $G$ is the triangle graph described in \eqref{triangle}, then the corresponding $\mathcal{H}_{G}$ is transversal-free and hence cannot be totally transversal. If $G$ is the star graph described in \eqref{star}, then we have already argued that $E(G)=N(u_{0})\in\mathcal{H}_{G}$ serves as the transversal for $\mathcal{H}_{G}$. It is also immediate that, in case of the star graph described in \eqref{star}, if we consider 
\begin{equation}
S=\left\{\left\{u_{0},u_{i_{1}}\right\},\left\{u_{0},u_{i_{2}}\right\},\ldots,\left\{u_{0},u_{i_{r}}\right\}\right\} \text{ for some }r\in\{1,2,\ldots,k\},\nonumber
\end{equation}  
then the hyperedge $S$ itself belongs to the induced subhypergraph $\mathcal{H}_{G,S}$ and serves as its transversal. Thus, $G$ described by \eqref{star} yields a totally transversal hypergraph $\mathcal{H}_{G}$. Note, by Theorem 6 of \cite{boros2024impartial}, that when $k=3$ in \eqref{star}, the corresponding $\mathcal{H}_{G}$ is SG-decreasing as well.

For $k\geqslant 4$, we now argue why $\mathcal{H}_{G}$ is SG-decreasing for $G$ of the form described in \eqref{star}. For $G$ as in \eqref{star}, $\mathcal{H}_{G}=2^{E(G)}\setminus\emptyset$ (i.e.\ all non-empty subsets of $E(G)$ are present as hyperedges in $\mathcal{H}_{G}$). Note that, in this case, a typical \emph{position} for the Hypergraph Nim played on $\mathcal{H}_{G}$ is given by $\overline{x}=(x_{1},x_{2},\ldots,x_{k})$, where $x_{i}$ is the current / remaining weight (or number of chips) at the vertex $\{u_{0},u_{i}\}$ of $\mathcal{H}_{G}$. A single legal move, when possible, involves choosing any of the hyperedges of $\mathcal{H}_{G}$ and removing some positive integer weight from each of the vertices forming the chosen hyperedge -- in other words, it involves choosing a subset $S$ of $E(G)$ and then removing a positive integer from the edge-weight of each of the edges belonging to $S$. Consequently, we may go from position $\overline{x}=(x_{1},x_{2},\ldots,x_{k})$ to $\overline{y}=(y_{1},y_{2},\ldots,y_{k})$ in a single legal move if and only if $\sum_{i=1}^{k}x_{i}>\sum_{i=1}^{k}y_{i}$ and $x_{i}\geqslant y_{i}$ for each $i\in\{1,2,\ldots,k\}$. If we let $\SG(\overline{x})$ indicate the Sprague-Grundy function value at position $\overline{x}$, we now show, inductively, that $\SG(\overline{x})=\sum_{i=1}^{k}x_{i}$.

To begin with, we consider the terminal positions (from which there can be no legal move), i.e.\ when every hyperedge of $\mathcal{H}_{G}$ has at least one vertex with $0$ remaining weight. Note that, for $G$ as in \eqref{star}, since every non-empty subset of $E(G)$ is a hyperedge in $\mathcal{H}_{G}$, the only terminal position is when every vertex of $\mathcal{H}_{G}$ (equivalently, every edge of $G$) has remaining weight $0$. If we denote this position by $\overline{0}$, then by definition, we have $\SG(\overline{0})=0$, which satisfies the claim we stated at the end of the previous paragraph. Suppose we have shown, for some $w\in\mathbb{N}_{0}$, that $\SG(\overline{x})=\sum_{i=1}^{k}x_{i}$ for each position $\overline{x}$ with $\sum_{i=1}^{k}x_{i}\leqslant w$. Let us now consider a position $\overline{x}$ such that $\sum_{i=1}^{k}x_{i}=w+1$. We know that
\begin{equation}
\SG(\overline{x})=\mex\left\{\SG(\overline{y}):\text{ there exists a single legal move from }\overline{x}\text{ to }\overline{y}\right\},\label{SG_recursive_defn}
\end{equation}
where $\mex(S)$, for any subset $S$ of $\mathbb{N}_{0}$, is the smallest non-negative integer that is not contained in $S$. Note that, according to the rules of the Hypergraph Nim game, for each $0\leqslant w'\leqslant w$, a hyperedge $S\subset E(G)$ of $\mathcal{H}_{G}$ can be chosen such that 
\begin{equation}
|S|\leqslant (w-w') \text{ and }\sum_{t: \{u_{0},u_{t}\}\in S}x_{t}\geqslant (w-w'),\nonumber
\end{equation}
and a suitable positive-integer-valued weight can be removed from each element of $S$ so that the total weight removed equals $(w-w')$, in a single legal move, yielding a new position $\overline{y}$ with $\sum_{i=1}^{k}y_{i}=w'$. Therefore, the set described in \eqref{SG_recursive_defn} actually equals $\{0,1,\ldots,w\}$, and therefore, its $\mex$ value equals $(w+1)$ -- in other words, we have shown that $\SG(\overline{x})=w+1$, as desired. 

Now that we have established, for $G$ as described in \eqref{star}, that $\SG(\overline{x})=\sum_{i=1}^{k}x_{i}$ when $x_{i}$ is the weight on the edge $\{u_{0},u_{i}\}$ for all $i\in\{1,2,\ldots,k\}$ and the Hypergraph Nim is played on $\mathcal{H}_{G}$, we know that every single legal move strictly lowers the SG-value here, proving that $\mathcal{H}_{G}$ is, indeed, SG-decreasing. This completes the proof of the fact that the only simple, finite graph $G$ with more than one edge for which $\mathcal{H}_{G}$ is SG-decreasing is a star graph of the form given by \eqref{star}. 
\end{proof}

\begin{proof}[Proof of Proposition~\ref{prop:hypergraph_to_graph}]
Suppose $\mathcal{H}$ is a hypergraph, defined on the ground set $V$, satisfying the properties described in the statement of Proposition~\ref{prop:hypergraph_to_graph}. We enumerate the (finitely many) maximal hyperedges of $\mathcal{H}$, as $H_{1},\ldots,H_{m}$ for some $m\in\mathbb{N}$. Note that all non-empty subsets of $H_{i}$ will be hyperedges of $\mathcal{H}$ for all $i\in\{1,2,\ldots,m\}$ -- in fact, we can write $\mathcal{H}=\bigcup_{i=1}^{m}\left(2^{H_{i}}\setminus\emptyset\right)$. We now begin constructing the corresponding graph $G$. The vertex set for $G$ consists of 
\begin{enumerate}
\item \label{construction_1} vertices $v_{1},\ldots,v_{m}$, where we associate $v_{i}$ with the maximal hyperedge $H_{i}$ for each $i\in\{1,2,\ldots,m\}$,
\item \label{construction_2} vertices $v_{i,1},\ldots,v_{i,n_{i}}$, where $n_{i}$ is the number of elements of the ground set $V$ that belong to the maximal hyperedge $H_{i}$ \emph{only} (i.e.\ to no other maximal hyperedge of $\mathcal{H}$), for each $i\in\{1,2,\ldots,m\}$. Note that $n_{i}$ can be $0$.
\end{enumerate}
We draw the reader's attention to a special case: if $H_{i}$ consists of a single element from $V$, then we introduce a vertex $v_{i}$ corresponding to $H_{i}$, and we introduce another vertex $v_{i,1}$ correpsonding to the element of $V$ that $H_{i}$ is made up of (since this element appears in no other maximal hyperedge of $\mathcal{H}$). We draw an edge between $v_{i}$ and $v_{j}$ if $H_{i}$ and $H_{j}$ have an element in common (note that $H_{i}$ and $H_{j}$ can have at most one element in common). We also draw an edge between $v_{i}$ and $v_{i,t}$ for each $t\in\{1,2,\ldots,n_{i}\}$, for each $i\in\{1,2,\ldots,m\}$. 

We now verify that, indeed, $\mathcal{H}=\mathcal{H}_{G}$ for $G$ as constructed above. For each $i\in\{1,2,\ldots,m\}$, the edges incident on $v_{i}$ in $G$ are 
\begin{enumerate}[label=(\alph*)]
\item $\{v_{i},v_{i,t}\}$ for each $t\in\{1,2,\ldots,n_{i}\}$,
\item $\{v_{i},v_{j}\}$ for each $j\in\{1,2,\ldots,m\}\setminus\{i\}$ such that $|H_{i}\cap H_{j}|=1$.
\end{enumerate}
The ground set for $\mathcal{H}_{G}$ is the set of all edges of $G$, and we can write
\begin{equation}
\mathcal{H}_{G}=\bigcup_{i=1}^{m}\left(2^{\left\{\{v_{i},v_{i,t}\}:t\in\{1,2,\ldots,n_{i}\}\right\}\bigcup\left\{\{v_{i},v_{j}\}:j\in\{1,2,\ldots,m\}\setminus\{i\}\text{ and }|H_{i}\cap H_{j}|=1\right\}}\setminus\emptyset\right),\nonumber
\end{equation} 
with its maximal hyperedges being 
\begin{equation}
\left\{\{v_{i},v_{i,t}\}:t\in\{1,2,\ldots,n_{i}\}\right\}\bigcup\left\{\{v_{i},v_{j}\}:j\in\{1,2,\ldots,m\}\setminus\{i\}\text{ and }|H_{i}\cap H_{j}|=1\right\} \text{ for all }i\in\{1,2,\ldots,m\}.\label{maximal_hyperedge}
\end{equation}
We can draw a bijection between $H_{i}$ of $\mathcal{H}$ and the $i$-th hyperedge of $\mathcal{H}_{G}$, as described in \eqref{maximal_hyperedge}, since 
\begin{enumerate}[label=(\alph*)]
\item there are precisely $n_{i}$ many elements in $H_{i}$ that belong to no other maximal hyperedge of $\mathcal{H}$,
\item there are precisely $|\{\{v_{i},v_{j}\}:j\in\{1,2,\ldots,m\}\setminus\{i\}\text{ and }|H_{i}\cap H_{j}|=1\}|$ many elements in $H_{i}$ that belong to one other maximal hyperedge of $\mathcal{H}$,
\end{enumerate}
and this is also evidently true for $\mathcal{H}_{G}$.

Proving the converse, i.e.\ that the hypergraph $\mathcal{H}_{G}$ obtained from a simple, finite graph $G$ via \eqref{hypergame} satisfies the properties described in the statement of Proposition~\ref{prop:hypergraph_to_graph}, is more straightforward. Corresponding to each vertex $v$ of $G$ with degree at least $2$, we obtain a maximal hyperedge $H_{v}=N(v)$ (where $N(v)$ is as defined in \eqref{N(v)_defn}), and in case of vertices $u$ and $v$ such that $\{u,v\}\in E(G)$ and no other edge of $G$ is incident on $u$ nor on $v$ (so that $u$ and $v$ form a connected component of $G$), there is a single maximal hyperedge, namely $\{\{u,v\}\}$, that accounts for both $u$ and $v$ (this is the special case that has been mentioned above, right after \eqref{construction_2}) -- we may indicate this maximal hyperedge by either $H_{u}$ or $H_{v}$. By \eqref{hypergame}, it is immediate that every non-empty subset of a maximal hyperedge is also a hyperedge of $\mathcal{H}_{G}$, thus satisfying \eqref{E1}. Next, if $u$ and $v$ are two distinct vertices of $G$, each with degree at least $2$, then 
\begin{enumerate*}
\item either $H_{u}\cap H_{v}=\emptyset$, which happens if $\{u,v\}\notin E(G)$ (i.e.\ $u$ and $v$ are not adjacent), or
\item $H_{u}\cap H_{v}=\{u,v\}$.
\end{enumerate*}
No other edge of $G$ can be an element in $H_{u}\cap H_{v}$ since such an edge has to be incident on \emph{both} $u$ and $v$, and we have assumed $G$ to be a simple graph. If $\{u,v\}\in E(G)$ and no other edge of $G$ is incident on any of $u$ and $v$, then $H_{u}=H_{v}=\{\{u,v\}\}$, so that their intersection consists of a single edge of $G$. This validates \eqref{E2}. Finally, for any edge $\{u,v\}$ of $G$, one of three possibilities is true: 
\begin{enumerate}[label=(\alph*)]
\item no other edge of $G$ is incident on either of $u$ and $v$, in which case $\{u,v\}$ belongs to a single maximal hyperedge, namely $H_{u}=H_{v}$,
\item no other edge of $G$ is incident on $u$ but at least one other edge of $G$ is incident on $v$, in which case $\{u,v\}$ belongs to a single maximal hyperedge, namely $H_{v}$,
\item there is at least one other edge of $G$ that is incident on $u$, and at least one other edge of $G$ that is incident on $v$, in which case $\{u,v\}$ belongs to each of two maximal hyperedges, namely $H_{u}$ and $H_{v}$.
\end{enumerate}
This validates \eqref{E3}. This completes the proof of Proposition~\ref{prop:hypergraph_to_graph}.
\end{proof}

\subsection{Proofs of our results regarding the graphs $F_{2}$, $G_{2}$ and $G_{3}$}\label{sec:proof_4_vertices}
Before we begin with the proofs of our results, we urge the reader to recall Theorem~\ref{triangle_losing}. During the analysis of a game of Graph Nim played on a given graph $G=(V,E)$, we denote by $w_{i}(e)$, for each $e \in E$, the edge-weight remaining on the edge $e$ \emph{after} the $i$-th round of the game has been completed, for each $i \in \mathbb{N}$.

\begin{proof}[Proof of Theorem~\ref{thm:F_{2}}]
In order to establish Theorem~\ref{thm:F_{2}}, we begin by showing that whenever an initial weight configuration $\left(w_{0}(AB),w_{0}(BC),w_{0}(CD),w_{0}(DB)\right)$, on $F_{2}$, satisfies the criteria:
\begin{equation}
w_{0}(BC)=w_{0}(DB) \text{ and } w_{0}(CD)=w_{0}(AB)+w_{0}(BC),\label{F_{2}_losing_eq}
\end{equation}
it is losing. The proof of this claim happens via induction on $w_{0}(CD)$.

\textbf{Proof of the base case:} Since the induction happens with respect to the edge-weight $w_{0}(CD)$, the base case consists of $w_{0}(AB)=w_{0}(BC)=w_{0}(DB)=1$ and $w_{0}(CD)=2$. We consider the first round of the game played on this initial configuration:
\begin{enumerate}
\item Suppose $P_{1}$ selects the vertex $B$ in the first round, and removes at least one of the edges $AB$, $BC$ and $DB$. 
\begin{enumerate}
\item  If $P_{1}$ removes only $AB$, then in the second round, $P_{2}$ removes weight $1$ from $CD$ (by choosing either $C$ or $D$), leaving behind a triangle with all three edge-weights equal, so that $P_{1}$ loses by Theorem~\ref{triangle_losing}.
\item If $P_{1}$ removes only $BC$ (analogously, $P_{1}$ removes only $DB$), then $P_{2}$ selects $D$ and removes weight $1$ from each of $DB$ and $CD$ in the second round (analogously, $P_{2}$ selects $C$ and removes weight $1$ from each of $BC$ and $CD$ in the second round), leaving $P_{1}$ with $AB$ and $CD$ where $w_{0}(AB)=w_{0}(CD)=1$, so that $P_{1}$ loses by Theorem~\ref{thm:galaxy}.
\item If $P_{1}$ removes both $AB$ and $BC$ (which is analogous to her removing both $AB$ and $DB$), then $P_{2}$ selects $D$ and removes both $DB$ and $CD$ in the second round (analogously, she selects $C$ and removes both $BC$ and $CD$), thus winning the game immediately.
\item If $P_{1}$ removes both $BC$ and $DB$, then $P_{2}$ removes weight $1$ from $CD$ in the second round, so that $P_{1}$ loses by Theorem~\ref{thm:galaxy}.
\item If $P_{1}$ removes all three of $AB$, $BC$ and $DB$, then $P_{2}$ removes $CD$ in the second round, thus winning the game immediately.
\end{enumerate}

\item Suppose $P_{1}$ selects the vertex $C$ in the first round (which is analogous to her selecting the vertex $D$). If $P_{1}$ removes only $BC$, then $P_{2}$'s response is the same as the corresponding case discussed above.
\begin{enumerate}
\item If $P_{1}$ removes only weight $1$ from $CD$, and none from $BC$, then $P_{2}$ removes the edge $AB$ in the second round, and $P_{1}$ loses by Theorem~\ref{triangle_losing}. 
\item If $P_{1}$ removes the entire edge $CD$, but leaves $BC$ intact, then $P_{2}$ selects the vertex $B$ and removes all of $AB$, $DB$ and $BC$ in the second round, thus winning the game immediately.
\item If $P_{1}$ removes weight $1$ from each of $BC$ and $CD$, then $P_{2}$ removes the edge $DB$ in the second round, and $P_{1}$ loses by Theorem~\ref{thm:galaxy}.
\item If $P_{1}$ removes both $BC$ and $CD$, then $P_{2}$ selects the vertex $B$ and removes both $AB$ and $DB$ in the second round, thus winning the game immediately.
\end{enumerate}

\item Suppose $P_{1}$ selects $A$, and removes $AB$ in the first round. Then $P_{2}$ removes weight $1$ from $CD$ in the second round, and wins by Theorem~\ref{triangle_losing}.
\end{enumerate}

This shows us that no matter the move made by $P_{1}$ in the first round, $P_{2}$ possesses a winning strategy, thus completing the proof of the base case for the inductive argument employed in proving that any configuration on $F_{2}$ that is of the form \eqref{F_{2}_losing_eq} is losing. This completes the proof of the base case for our inductive argument.

\textbf{Inductive step of the proof:} Suppose we have shown, for some $K \in \mathbb{N}$ with $K \geqslant 2$, that whenever $\left(w_{0}(AB),w_{0}(BC),w_{0}(CD),w_{0}(DB)\right)$ satisfies \eqref{F_{2}_losing_eq} along with the constraint that $w_{0}(CD)=k$ for some $k \leqslant K$, it is a losing configuration. We now consider the configuration
\begin{equation}
w_{0}(BC)=w_{0}(DB)=i, \quad w_{0}(AB)=(K+1-i) \quad \text{and} \quad w_{0}(CD)=(K+1),\label{F_{2}_losing_eq_inductive}
\end{equation}
for \emph{any} $i \in \mathbb{N}$ with $i \leqslant K$. The first round of the game of Graph Nim played on this initial configuration can unfold in one of the following ways:
\begin{enumerate}
\item Suppose $P_{1}$ selects a vertex such that the edge-weights of the edges belonging to some subset of $\{AB, DB, BC\}$ are modified, while all else remain intact. We then have 
\begin{equation}
w_{1}(BC)=j_{1}, \quad w_{1}(DB)=j_{2} \quad \text{and} \quad w_{1}(AB)=j_{3},\nonumber
\end{equation}
where $\max\{j_{1},j_{2}\} \leqslant i$ and $j_{3} \leqslant (K+1-i)$, with at least one of these three inequalities being strict, or equivalently, either $j_{1}+j_{3} \leqslant K$ or $j_{2}+j_{3} \leqslant K$. Without loss of generality, we assume that $j_{1} \leqslant j_{2}$ -- which, in turn, ensures that we must have $j_{1}+j_{3} \leqslant K$. Then $P_{2}$ chooses the vertex $D$ and removes weight $(j_{2}-j_{1})$ from $DB$ and weight $(K+1)-(j_{1}+j_{3})$ from $CD$ in the second round, leaving $P_{1}$ with
\begin{equation}
w_{2}(BC)=w_{2}(DB)=j_{1}, \quad w_{2}(AB)=j_{3} \quad \text{and} \quad w_{2}(CD)=j_{1}+j_{3},\nonumber
\end{equation}
and as $j_{1}+j_{3} \leqslant K$, we know that $P_{1}$ loses by our induction hypothesis.

\item Suppose $P_{1}$ chooses a vertex such that the edge-weights of $CD$ and $BC$ (analogously, $CD$ and $DB$) are modified, while all else remain intact. We then have
\begin{equation}
w_{1}(BC)=j, \quad w_{1}(DB)=i, \quad w_{1}(AB)=(K+1-i) \quad \text{and} \quad w_{1}(CD)=k,\nonumber
\end{equation}
for some $j \leqslant i$ and some $k \leqslant K$ (it suffices to consider the case where $k \leqslant K$, since the case of $k=(K+1)$ is already contained in the previous scenario). There are a few subcases to consider here:
\begin{enumerate}
\item When $(K+1)-k \geqslant i-j$, then $P_{2}$ selects the vertex $B$ and removes weight $(i-j)$ from $DB$ and $(K+1-i)-(k-j)$ from $AB$ in the second round, leaving $P_{1}$ with
\begin{equation}
w_{2}(BC)=w_{2}(DB)=j, \quad w_{2}(AB)=k-j \quad \text{and} \quad w_{2}(CD)=k,\nonumber
\end{equation}
and as $k \leqslant K$, we know that $P_{1}$ loses by our induction hypothesis.

\item When $(K+1)-k < i-j$, then $P_{2}$ selects the vertex $D$ and removes weight $(i-j)$ from $DB$ and weight $k-[\{(K+1)-i\}+j]$ from $CD$ in the second round, leaving $P_{1}$ with
\begin{equation}
w_{2}(BC)=w_{2}(DB)=j, \quad w_{2}(AB)=(K+1)-i \quad \text{and} \quad w_{2}(CD)=(K+1)-i+j,\nonumber
\end{equation}
and since $(K+1)-i+j < k \leqslant K$, we know that $P_{1}$ loses by our induction hypothesis.
\end{enumerate}
\end{enumerate}

This completes the proof of our claim that the initial configuration in \eqref{F_{2}_losing_eq_inductive} is losing, and thereby completes the inductive proof of our claim that the initial configuration in \eqref{F_{2}_losing_eq} is losing. We are now left to show that any initial configuration on $F_{2}$ that does \emph{not} satisfy \eqref{F_{2}_losing_eq} is winning, which we argue as follows.

Any initial configuration $\left(w_{0}(AB),w_{0}(BC),w_{0}(CD),w_{0}(DB)\right)$ on $F_{2}$ not satisfying \eqref{F_{2}_losing_eq} must either be such that $w_{0}(BC) \neq w_{0}(DB)$ or $w_{0}(CD) \neq w_{0}(BC)+w_{0}(AB)$. We can, thus, consider the following two possibilities:
\begin{enumerate}
\item Suppose $w_{0}(BC) \neq w_{0}(DB)$, and without loss of generality, we assume that $w_{0}(BC) > w_{0}(DB)$. Let us set
\begin{equation}
w_{0}(DB)=i,\quad w_{0}(BC)=j,\quad w_{0}(CD)=k \quad \text{and} \quad w_{0}(AB)=\ell.\nonumber
\end{equation}
If $k \geqslant i+\ell$, then $P_{1}$ selects the vertex $C$ and removes weight $(j-i)$ from $BC$ and weight $k-(i+\ell)$ from $CD$ in the first round, leaving $P_{2}$ with 
\begin{equation}
w_{1}(DB)=w_{1}(BC)=i, \quad w_{1}(AB)=\ell \quad \text{and} \quad w_{1}(CD)=i+\ell,\nonumber
\end{equation}
which is of the form given by \eqref{F_{2}_losing_eq}, so that $P_{2}$ loses. If, on the other hand, $k < i+\ell$, then $P_{1}$ selects the vertex $B$ and removes weight $(j-i)$ from $BC$ and weight $\ell-(k-i)$ from $AB$ in the first round, leaving $P_{2}$ with
\begin{equation}
w_{1}(DB)=w_{1}(BC)=i, \quad w_{1}(AB)=k-i \quad \text{and} \quad w_{1}(CD)=k,\nonumber
\end{equation}
which is, once again, of the form given by \eqref{F_{2}_losing_eq}, so that $P_{2}$ loses. 

\item Suppose, now, that $w_{0}(BC)=w_{0}(DB)$, but $w_{0}(AB)+w_{0}(BC) \neq w_{0}(CD)$. Once again, letting
\begin{equation}
w_{0}(DB)=w_{0}(BC)=i, \quad w_{0}(CD)=k \quad \text{and} \quad w_{0}(AB)=\ell,\nonumber
\end{equation}
we consider two subcases: 
\begin{enumerate*}
\item when $k > i+\ell$, $P_{1}$ removes weight $k-(i+\ell)$ from $CD$ in the first round,
\item whereas when $k < i+\ell$, $P_{1}$ removes weight $\ell-(k-i)$ from $AB$ in the first round,
\end{enumerate*}
so that, in either scenario, $P_{2}$ is left with a configuration of the form given by \eqref{F_{2}_losing_eq}, leading to her defeat.
\end{enumerate}

This completes the proof of our claim that any initial weight configuration on $F_{2}$ that does not satisfy \eqref{F_{2}_losing_eq} is winning. This also concludes the proof of Theorem~\ref{thm:F_{2}}.
\end{proof}

\begin{proof}[Proof of Theorem~\ref{thm:G_{2},G_{3}}]
In order to establish our claim regarding the graph $G_{2}$, we consider the following cases:
\begin{enumerate}
\item If $w_{0}(BE) \geqslant w_{0}(AC)+w_{0}(AD)$, then $P_{1}$ selects the vertex $B$, removes the entire edge $AB$, and removes weight $w_{0}(BE)-\{w_{0}(AC)+w_{0}(AD)\}$ from the edge $BE$, so that $P_{2}$ is left with the galaxy graph consisting of the edges $BE$, $AC$ and $AD$, with $w_{1}(BE)=w_{1}(AC)+w_{1}(AD)=w_{0}(AC)+w_{0}(AD)$. Therefore, $P_{2}$ loses by Theorem~\ref{thm:galaxy}. 

\item If $w_{0}(BE) < w_{0}(AC)+w_{0}(AD)$, then $P_{1}$ selects the vertex $A$, removes the entire edge $AB$, and removes from at least one of $AC$ and $AD$ sufficient weight such that $w_{1}(AC)+w_{1}(AD)=w_{0}(BE)=w_{1}(BE)$ (the choice of weights to be removed from $AC$ and / or $AD$ for this equality to hold need not be unique), so that, once again, $P_{2}$ loses by Theorem~\ref{thm:galaxy}. 
\end{enumerate}
This completes the proof of our claim that \emph{all} initial weight configurations on $G_{2}$ are winning. We now come to the graph $G_{3}$, and we consider the following cases:
\begin{enumerate}
\item Suppose $w_{0}(DE) \geqslant w_{0}(AB)+w_{0}(BC)$ (an analogous situation would be where $w_{0}(AB) \geqslant w_{0}(CD)+w_{0}(DE)$). In this case, $P_{1}$ selects the vertex $D$, removes the entire edge $CD$, and removes weight $w_{0}(DE)-\{w_{0}(AB)+w_{0}(BC)\}$ from the edge $DE$ in the first round, leaving $P_{2}$ with a galaxy graph consisting of the edges $AB$, $BC$ and $DE$, where $w_{1}(DE)=w_{0}(AB)+w_{0}(BC)=w_{1}(AB)+w_{1}(BC)$. Therefore, $P_{2}$ loses by Theorem~\ref{thm:galaxy}.

\item Suppose $w_{0}(AB)+w_{0}(BC)>w_{0}(DE)\geqslant w_{0}(AB)$ (an analogous situation would be where $w_{0}(CD)+w_{0}(DE)>w_{0}(AB)\geqslant w_{0}(DE)$). In this case, $P_{1}$ selects the vertex $C$, removes the entire edge $CD$, and removes weight $w_{0}(AB)+w_{0}(BC)-w_{0}(DE)$ from the edge $BC$. Note, here, that the inequalities $w_{0}(AB)+w_{0}(BC)>w_{0}(DE)\geqslant w_{0}(AB)$ ensures that the weight removed from $BC$ is strictly positive and bounded above by $w_{0}(BC)$. This leaves $P_{2}$ with a galaxy graph consisting of the edges $AB$, $BC$ and $DE$, where $w_{1}(BC)=w_{0}(BC)-\{w_{0}(AB)+w_{0}(BC)-w_{0}(DE)\}=w_{0}(DE)-w_{0}(AB)=w_{1}(DE)-w_{1}(AB)$, and $P_{2}$ loses by Theorem~\ref{thm:galaxy}.

\item The last remaining case to consider is where $w_{0}(AB)+w_{0}(BC)>w_{0}(AB)>w_{0}(DE)$. But here, if $w_{0}(AB)\geqslant w_{0}(CD)+w_{0}(DE)$, then this becomes analogous to the first scenario considered above, and if $w_{0}(AB)<w_{0}(CD)+w_{0}(DE)$, then this becomes analogous to the second scenario considered above. Either way, $P_{1}$ wins.
\end{enumerate}
This completes the proof of our claim that \emph{all} initial weight configurations on $G_{3}$ are winning.
\end{proof}

The rest of \S\ref{sec:appendix} has been dedicated to all those details whose inclusion in the main body of the paper would make reading and comprehension of the proof techniques cumbersome and become a hindrance to lucidity.
\subsection{Details omitted from the proof of Step~\eqref{thm:G_{4}_losing_1} meant for proving Theorem~\ref{thm:main_G_{4}}}\label{subsec:appendix_G_{4}_losing_1}
We prove that a configuration satisfying the hypothesis of \eqref{G_{4}_losing_cond_1} is losing when $m=1$, $k=1$, $i=1$ and $\ell \in \mathbb{N}_{0}$, i.e.\ we show that the following configuration is losing on $G_{4}$:
\begin{equation}
w_{0}(AB)=w_{0}(BC)=2+2\ell, \quad w_{0}(CA)=3+2\ell \quad \text{and} \quad w_{0}(DE)=1.\label{G_{4}_losing_1_K=M=1}
\end{equation}
To begin with, we consider the configuration obtained from \eqref{G_{4}_losing_1_K=M=1} by setting $\ell=0$, i.e.\ where $w_{0}(AB)=w_{0}(BC)=2$, $w_{0}(CA)=3$ and $w_{0}(DE)=1$. In the first round of the game played on this initial weight configuration, one of the following transpires:
\begin{enumerate}
\item Suppose $P_{1}$ removes edge-weights from at most two of the edges among $AB$, $BC$ and $CA$, leaving $P_{2}$ with $\min\{w_{1}(AB),w_{1}(BC),w_{1}(CA)\} \in \{0,1,2\}$. Note that this minimum equals $2$ if and only if $P_{1}$ removes weight $1$ from the edge $CA$ in the first round, and leaves all else undisturbed, so that in this case, it suffices for $P_{2}$ to simply remove the edge $DE$ in the second round, leaving $P_{1}$ with $AB$, $BC$ and $CA$, with $w_{2}(AB)=w_{2}(BC)=w_{2}(CA)=2$. Therefore, $P_{1}$ loses by Theorem~\ref{triangle_losing}.

If $\min\{w_{1}(AB),w_{1}(BC),w_{1}(CA)\}=1$, then $P_{2}$ removes all edges, out of $AB$, $BC$ and $CA$, except for the one whose edge-weight after the first round of the game equals $1$. This leaves $P_{1}$ with two disjoint edges, one of which is $DE$, each having edge-weight $1$, and $P_{1}$ loses by Theorem~\ref{thm:galaxy}.

If $\min\{w_{1}(AB),w_{1}(BC),w_{1}(CA)\}=0$, then at least one of the edges out of $AB$, $BC$ and $CA$ has been completely removed in the first round. Note that this already leaves $P_{2}$ with a galaxy graph whose connected components are two star graphs: one of which consists of one or two edges out of $AB$, $BC$ and $CA$, and the other consists of only the edge $DE$. Since $P_{1}$ can alter the edge-weights of at most two of the edges $AB$, $BC$ and $CA$ in the first round, the former of these two components has at least one ray with edge-weight strictly greater than $1$. $P_{2}$, therefore, can remove the requisite weight from the former component in the second round, so that $P_{1}$ is left with two disjoint edges (one of which is $DE$), each with edge-weight $1$, and she loses by Theorem~\ref{thm:galaxy}.

\item Suppose $P_{1}$ removes the edge $DE$ in the first round. Then $P_{2}$ removes weight $1$ from the edge $CA$ in the second round, leaving $P_{1}$ with $AB$, $BC$ and $CA$ where $w_{2}(AB)=w_{2}(BC)=w_{2}(CA)=2$. Consequently, $P_{1}$ loses by Theorem~\ref{triangle_losing}.
\end{enumerate}
This completes the proof of our claim that the configuration in \eqref{G_{4}_losing_1_K=M=1} is losing for $\ell=0$.

Suppose we have proved that the configuration in \eqref{G_{4}_losing_1_K=M=1} is losing whenever $\ell \leqslant L$, for some $L \in \mathbb{N}_{0}$. We now consider the configuration
\begin{equation}
w_{0}(AB)=w_{0}(BC)=2+2(L+1),\quad w_{0}(CA)=3+2(L+1) \quad \text{and} \quad w_{0}(DE)=1.\label{G_{4}_losing_1_K=M=1_inductive}
\end{equation}
The first round of the game played on this initial weight configuration unfolds in one of the following ways:
\begin{enumerate}
\item Suppose $P_{1}$ removes edge-weights from at most two of the edges $AB$, $BC$ and $CA$ in the first round, such that 
\begin{equation}
\min\{w_{1}(AB),w_{1}(BC),w_{1}(CA)\}=2(i+1) \text{ for some } i \in \mathbb{N}_{0}.\nonumber
\end{equation}
Note that we must have $2(i+1) \leqslant 2+2(L+1)=2(L+2) \implies i \leqslant (L+1)$. If $i=(L+1)$, then $P_{1}$ must have removed weight $1$ from the edge $CA$, and left all else unchanged, in the first round. However, this means that $w_{1}(AB)=w_{1}(BC)=w_{1}(CA)=2(L+2)$, and $P_{2}$ simply removes the edge $DE$ in the second round so as to make $P_{1}$ lose by Theorem~\ref{triangle_losing}.

Suppose $i \leqslant L$. In this case, $3+2i\leqslant 3+2L < 4+2L=2+2(L+1)$. Therefore, $P_{2}$, in the second round, can select an appropriate vertex out of $A$, $B$ and $C$, and reduce the edge-weights of the two edges incident on it such that the edge-weights of $AB$, $BC$ and $CA$, \emph{in some order}, become $(2+2i)$, $(2+2i)$ and $(3+2i)$. As an example, if $P_{1}$ selects the vertex $C$ and removes weights from $BC$ and $CA$ in the first round in such a manner that 
\begin{equation}
w_{1}(CA)=2(i+1)=2+2i \quad \text{and} \quad 2+2i \leqslant w_{1}(BC) \leqslant w_{0}(BC),\nonumber
\end{equation} 
then in the second round, $P_{2}$ selects the vertex $B$, removes edge-weight $w_{1}(BC)-(2+2i)$ from $BC$, and removes edge-weight $w_{1}(AB)-(3+2i)=w_{0}(AB)-(3+2i)=\{2+2(L+1)\}-(3+2i)$ from $AB$, so that $P_{1}$ is left with
\begin{equation}
w_{2}(BC)=w_{2}(CA)=2+2i, \quad w_{2}(AB)=3+2i \quad \text{and} \quad w_{2}(DE)=1,\nonumber
\end{equation}
which is of the same form as \eqref{G_{4}_losing_1_K=M=1}, with $\ell=i \leqslant L$. Consequently, $P_{1}$ loses by our induction hypothesis. 

If, on the other hand, $\min\{w_{1}(AB),w_{1}(BC),w_{1}(CA)\}=0$, then, after the first round, $P_{2}$ is left with a galaxy graph with two components, one of which is the edge $DE$ with edge-weight $1$, while the other has at least one ray whose edge-weight is at least $2+2(L+1)$. It is evident from Theorem~\ref{thm:galaxy} that $P_{2}$ wins.  

\item Suppose $P_{1}$ removes edge-weights from at most two of the edges $AB$, $BC$ and $CA$ in the first round, such that 
\begin{equation}
\min\{w_{1}(AB),w_{1}(BC),w_{1}(CA)\}=2(i+1)+1 \text{ for some } i \in \mathbb{N}_{0}.\nonumber
\end{equation}
Note that we must have $2(i+1)+1 < 2+2(L+1) \implies i \leqslant L$, which, in turn, implies that $2+2i\leqslant 2+2L < 2+2(L+1)$. In this case, $P_{2}$ selects an appropriate vertex out of $AB$, $BC$ and $CA$, and reduces the edge-weights of the two edges incident on it such that the edge-weights of $AB$, $BC$ and $CA$, at the end of the second round, become $(2+2i)$, $(2+2i)$ and $2(i+1)+1=(3+2i)$ \emph{in some order}. For instance, if $P_{1}$ selects the vertex $B$ and removes weights from $AB$ and $BC$ in the first round in such a manner that
\begin{equation}
w_{1}(AB)=2(i+1)+1 \quad \text{and} \quad 2(i+1)+1 \leqslant w_{1}(BC) \leqslant w_{0}(BC),\nonumber
\end{equation}
then $P_{2}$, in the second round, selects the vertex $C$ and removes weight $w_{1}(BC)-(2+2i)$ from the edge $BC$ and weight $w_{1}(CA)-(2+2i)=w_{0}(CA)-(2+2i)=3+2(L+1)-(2+2i)$ from the edge $CA$. This leaves $P_{1}$ with
\begin{equation}
w_{2}(AB)=2(i+1)+1=3+2i, \quad w_{2}(BC)=w_{2}(CA)=2+2i \quad \text{and} \quad w_{2}(DE)=1,\nonumber
\end{equation}
and as $i \leqslant L$, hence $P_{1}$ loses by our induction hypothesis.

On the other hand, if $\min\{w_{1}(AB),w_{1}(BC),w_{1}(CA)\}=1$, then $P_{2}$ removes, by choosing an appropriate vertex out of $A$, $B$ and $C$ in the second round, two of the edges out of $AB$, $BC$ and $CA$, leaving behind only the third that has edge-weight $1$. This leaves $P_{1}$ with two disjoint edges (one of which is $DE$), each with edge-weight $1$, and she loses by Theorem~\ref{triangle_losing}.

\item Finally, if $P_{1}$ removes the edge $DE$ in the first round, then $P_{2}$ removes weight $1$ from $CA$ in the second, making $P_{1}$ lose by Theorem~\ref{triangle_losing}.
\end{enumerate}
This concludes the proof of the fact that the configuration in \eqref{G_{4}_losing_1_K=M=1_inductive} is losing, thus completing the inductive proof of the claim that the configuration in \eqref{G_{4}_losing_1_K=M=1} is losing for all $\ell \in \mathbb{N}_{0}$. This establishes the base case of the inductive argument employed in proving Step~\eqref{thm:G_{4}_losing_1}.

\subsection{Details omitted from the proof of Step~\eqref{thm:G_{4}_losing_2} meant for proving Theorem~\ref{thm:main_G_{4}}}\label{subsec:appendix_G_{4}_losing_2}
The base case for the inductive argument employed for proving Step~\eqref{thm:G_{4}_losing_2} corresponds to $w_{0}(DE)=4$, two of $w_{0}(AB)$, $w_{0}(BC)$ and $w_{0}(CA)$ being equal to $1$ each, and the third being equal to $2$ (this is because we must ensure that $w_{0}(AB)$, $w_{0}(BC)$ and $w_{0}(CA)$ are not all equal). Without loss of generality, let us consider $w_{0}(AB)=w_{0}(BC)=1$, $w_{0}(CA)=2$ and $w_{0}(DE)=4$.

The first round of the game played on this initial configuration can unfold in one of the following ways:
\begin{enumerate}
\item Suppose $P_{1}$ removes, in the first round, at least one of the edges $AB$, $BC$ and $CA$. Then $P_{2}$ is left with a galaxy graph with two components: one of which is the edge $DE$ with $w_{1}(DE)=4$, and the other having the sum of its edge-weights strictly less than $4$. Therefore, the corresponding configuration is unbalanced, and $P_{2}$ wins by Theorem~\ref{thm:galaxy}.  

\item Suppose $P_{1}$ removes weight $1$ from the edge $CA$ in the first round, and leaves all else unchanged. Then $P_{2}$ removes the edge $DE$ in the second round, and wins by Theorem~\ref{triangle_losing}.

\item Suppose $P_{1}$ removes some edge-weight from $DE$ in the first round. If $w_{1}(DE)=3$, then $P_{2}$ removes the edge $AB$ in the second round, leaving $P_{1}$ with a galaxy graph comprising two components: one of which is the edge $DE$ with $w_{2}(DE)=3$, and the other is a star graph consisting of the edges $BC$ and $CA$, with $w_{2}(BC)+w_{2}(CA)=3$. Consequently, $P_{1}$ loses by Theorem~\ref{thm:galaxy}. If $w_{1}(DE)=2$, then $P_{2}$ removes edges $AB$ and $BC$ in the second round, and if $w_{1}(DE)=1$, then $P_{2}$ removes edges $BC$ and $CA$ in the second round. In either of these cases, once again, $P_{1}$ loses by Theorem~\ref{thm:galaxy}. If $w_{1}(DE)=0$, then $P_{2}$ removes weight $1$ from $CA$ in the second round, and $P_{1}$ loses by Theorem~\ref{triangle_losing}.
\end{enumerate}

This completes the proof of the base case for the inductive argument employed in proving Step~\eqref{thm:G_{4}_losing_2}.

\subsection{Proof of Lemma~\ref{lem:H_{1}_winning_galaxy}}\label{subsec:proof_lem_H_{1}_winning_galaxy}
We set $R=f(k)$, and since each of $\ell_{1}$ and $\ell_{2}$ is in $\left\{0,1,\ldots,2^{R+1}-1\right\}$, we can write the base-$2$ representations of $\ell_{1}$ and $\ell_{2}$ as $\ell_{1}=(a_{R}a_{R-1}\ldots a_{1}a_{0})_{2}$ and $\ell_{2}=(b_{R}b_{R-1}\ldots b_{1}b_{0})_{2}$. We now find $c_{i} \in \{0,1\}$, for each $i \in \{0,1,\ldots,R\}$, such that $a_{i}+b_{i}+c_{i}$ is even. Evidently, the triple
\begin{equation}
\left(2^{R+1}m+\ell_{1}, 2^{R+1}m+\ell_{2}, \sum_{i=0}^{R}c_{i}2^{i}\right)\nonumber
\end{equation}
is balanced, and therefore, losing by Theorem~\ref{thm:Nim}. Note that, for each $i \in \{0,1,\ldots,R\}$,
\begin{enumerate}
\item either at most one of $a_{i}$ and $b_{i}$ equals $1$, in which case we have $c_{i}=a_{i}+b_{i}$, 
\item or else $a_{i}=b_{i}=1$, in which case $c_{i}=0<a_{i}+b_{i}$.
\end{enumerate}
Combining these observations, we can write the inequality:
\begin{equation}
\sum_{i=0}^{R}c_{i}2^{i}\leqslant\sum_{i=0}^{R}(a_{i}+b_{i})2^{i}=\ell_{1}+\ell_{2}<k.\nonumber
\end{equation}
Consequently, the triple $(2^{R+1}m+\ell_{1},2^{R+1}m+\ell_{2},k)$ must be unbalanced, and therefore, winning by Theorem~\ref{thm:Nim}. This completes the proof of Lemma~\ref{lem:H_{1}_winning_galaxy}.

\subsection{Proof that any configuration of the form given by \eqref{eq:H_{1}_losing_0,1,3}, with $r=1$, is losing}\label{subsec:H_{1}_losing_1_proof}
The proof that any configuration on $H_{1}$ that is of the form given by \eqref{eq:H_{1}_losing_0,1,3}, with $r=1$, is inductive, and we begin by establishing the base case, obtained by setting $s=1$ and $k=3$ in \eqref{eq:H_{1}_losing_0,1,3}. Consider the first round of the game played on this initial configuration, i.e.\ on $w_{0}(AB)=w_{0}(CD)=4m+1$, $w_{0}(BC)=1$ and $w_{0}(EF)=3$, for any $m \in \mathbb{N}$:
\begin{enumerate}
\item Suppose $P_{1}$ removes a positive integer weight from $AB$ and a non-negative integer weight from $BC$ in the first round (an analogous situation would be where $P_{1}$ removes a positive integer weight from $CD$ and a non-negative integer weight from $BC$). We consider a few subcases:
\begin{enumerate}
\item Suppose $w_{1}(AB)=4m$ and $w_{1}(BC)=1$. Then $P_{2}$ removes weight $1$ from $EF$ in the second round, leaving $P_{1}$ with $w_{2}(AB)=4m$, $w_{2}(BC)=1$, $w_{2}(CD)=4m+1$ and $w_{2}(EF)=2$, which is of the same form as \eqref{eq:H_{1}_losing_0,1,3} with $r=0$, $s=1$ and $k=2$, and by what we have already proved in \S\ref{subsec:proof_eq_H_{1}_losing_0}, we conclude that $P_{1}$ loses.

\item Suppose $w_{1}(AB)=4m$ and $w_{1}(BC)=0$. This leaves $P_{2}$ with a galaxy graph comprising the edges $AB$, $CD$ and $EF$, where $w_{1}(AB)=4m$, $w_{1}(CD)=4m+1$ and $w_{1}(EF)=3$, and it is evident, from Lemma~\ref{lem:H_{1}_winning_galaxy}, that this configuration is winning, allowing $P_{2}$ to defeat $P_{1}$.

\item Suppose $w_{1}(AB)=4n+\ell_{1}$ for some $n < m$ and some $\ell_{1}\in\{0,1,2,3\}$. Then the resulting configuration is of the same form as in Lemma~\ref{prop:H_{1}_winning_1}, with $m_{1}=n$, $m_{2}=m$ and $\ell_{2}=1$, and as $n<m \implies m_{1} \neq m_{2}$, hence $P_{2}$ wins. 
\end{enumerate}

\item Suppose $P_{1}$ removes a positive integer weight from $EF$ in the first round, so that $w_{1}(EF)=t\in \{0,1,2\}$. If $t=2$, $P_{2}$ removes weight $1$ from $CD$ in the second round, leaving $P_{1}$ with $w_{2}(AB)=4m+1$, $w_{2}(BC)=1$, $w_{2}(CD)=4m$ and $w_{2}(EF)=2$, which is of the form given by \eqref{eq:H_{1}_losing_0,1,3} with $r=0$, $s=1$ and $k=2$, so that $P_{1}$ loses by what we have already proved. If $t=1$, then $P_{2}$ removes weight $1$ from each of $BC$ and $CD$ in the second round, leaving $P_{1}$ with a galaxy graph comprising the edges $AB$, $CD$ and $EF$, where $w_{2}(AB)=4m+1$, $w_{2}(CD)=4m$ and $w_{2}(EF)=1$, so that $P_{1}$ loses by Theorem~\ref{thm:galaxy}. Finally, if $t=0$, $P_{2}$ wins by Remark~\ref{rem:EF_edgeweight_0}.

\item If $P_{1}$ removes $BC$ in the first round, without disturbing the edge-weights of $AB$ and $CD$, $P_{2}$ is left with a galaxy graph comprising the edges $AB$, $CD$ and $EF$, with $w_{1}(AB)=w_{1}(CD)=4m+1$ and $w_{1}(EF)=3$, and $P_{2}$ wins by Lemma~\ref{lem:H_{1}_winning_galaxy}. 
\end{enumerate}

Suppose, for some $K \in \mathbb{N}$, we have shown that any configuration of the form given by \eqref{eq:H_{1}_losing_0,1,3}, with $r=1$ and $k \leqslant K$, is losing. We now consider the first round of the game played on the initial weight configuration
\begin{align}
w_{0}(AB)=2^{R+1}m+1,\ w_{0}(BC)=K-L,\ w_{0}(CD)=2^{R+1}m+L,\ w_{0}(EF)=K+1,\label{eq:H_{1}_losing_1_inductive}
\end{align}
where $R=f(K+1)$ and $L \in \{1,2,\ldots,K-1\}$:
\begin{enumerate}
\item Suppose $P_{1}$ removes a positive integer weight from $CD$ and a non-negative integer weight from $BC$ in the first round. We subdivide this case into a few subcases, as follows:
\begin{enumerate}
\item Suppose $w_{1}(CD)=2^{R+1}m+\ell$ for some $\ell \in \{0,1,\ldots,L-1\}$, and $w_{1}(BC)=t \in \{1,2,\ldots,K-L\}$. Note that $t+\ell+1\leqslant(K-L)+(L-1)+1=K$. Here, $P_{2}$ removes weight $(K-t-\ell)$ from $EF$ in the second round, leaving $P_{1}$ with
\begin{align}
{}&w_{2}(AB)=2^{f(t+\ell+1)+1}2^{R-f(t+\ell+1)}m+1,\ w_{2}(BC)=t,\nonumber\\
{}&w_{2}(CD)=2^{f(t+\ell+1)+1}2^{R-f(t+\ell+1)}m+\ell,\ w_{2}(EF)=t+\ell+1.\nonumber
\end{align}
When $\ell\geqslant 1$, this configuration is of the form given by \eqref{eq:H_{1}_losing_0,1,3} with $r=1$ and $w_{2}(EF)=t+\ell+1\leqslant K$, so that $P_{1}$ loses by our induction hypothesis; on the other hand, when $\ell=0$, this configuration is of the form given by \eqref{eq:H_{1}_losing_0,1,3} with $r=0$, and by what we have already proved in \S\ref{subsec:proof_eq_H_{1}_losing_0}, we conclude that $P_{1}$ loses.

\item Suppose $w_{1}(CD)=2^{R+1}m+\ell$ for some $\ell \in \{0,1,\ldots,L-1\}$ and $w_{1}(BC)=0$. This leaves $P_{2}$ with a galaxy graph comprising the edges $AB$, $CD$ and $EF$, where $w_{1}(AB)=2^{R+1}m+1$, $w_{1}(CD)=2^{R+1}m+\ell$ and $w_{1}(EF)=(K+1)$. Since $\ell+1\leqslant L\leqslant(K-1)<(K+1)$, it is immediate, from Lemma~\ref{lem:H_{1}_winning_galaxy}, that $P_{2}$ wins.

\item Suppose $w_{1}(CD)=2^{R+1}n+\ell_{1}$ for some $n < m$ and some $\ell_{1} \in \{0,1,\ldots,2^{R+1}-1\}$. The resulting configuration is of the form given by Lemma~\ref{prop:H_{1}_winning_1}, with $m_{1}=n$, $m_{2}=m$ and $\ell_{2}=1$, and as $n<m \implies m_{1}\neq m_{2}$, hence $P_{2}$ wins.
\end{enumerate}

\item Suppose $P_{1}$ removes a positive integer weight from $AB$ and a non-negative integer weight from $BC$ in the first round. Once again, we subdivide our analysis as follows:
\begin{enumerate}
\item If $w_{1}(AB)=2^{R+1}m$ and $w_{1}(BC)=t\in \{1,2,\ldots,K-L\}$, $P_{2}$ removes $(K+1)-(t+L)$ from $EF$ in the second round. Note that $(t+L)\leqslant(K-L)+L=K$. This leaves $P_{1}$ with 
\begin{align}
{}&w_{2}(AB)=2^{f(t+L)+1}2^{R-f(t+L)}m,\ w_{2}(BC)=t,\nonumber\\
{}&w_{2}(CD)=2^{f(t+L)+1}2^{R-f(t+L)}m+L,\ w_{2}(EF)=t+L,\nonumber
\end{align}
which is of the same form as \eqref{eq:H_{1}_losing_0,1,3} with $r=0$, and by what we have already proved in \S\ref{subsec:proof_eq_H_{1}_losing_0}, we conclude that $P_{1}$ loses.

\item If $w_{1}(AB)=2^{R+1}m$ and $w_{1}(BC)=0$, $P_{2}$ is left with a galaxy graph comprising the edges $AB$, $CD$ and $EF$, where $w_{1}(AB)=2^{R+1}m$, $w_{1}(CD)=2^{R+1}m+L$ and $w_{1}(EF)=(K+1)$, and as $L\leqslant (K-1)<(K+1)$, $P_{2}$ wins by Lemma~\ref{lem:H_{1}_winning_galaxy}.

\item Suppose $w_{1}(AB)=2^{R+1}n+\ell_{1}$ for some $n < m$ and some $\ell_{1} \in \{0,1,\ldots,2^{R+1}-1\}$. The resulting configuration is of the form given by Lemma~\ref{prop:H_{1}_winning_1}, with $m_{1}=n$, $m_{2}=m$ and $\ell_{2}=L$, and as $n<m \implies m_{1}\neq m_{2}$, hence $P_{2}$ wins.
\end{enumerate}

\item Suppose $P_{1}$ removes a positive integer weight from $EF$, so that $w_{1}(EF)=k\leqslant K$. When $k \in \{K-L+1,\ldots,K\}$, $P_{2}$ removes weight $L-\{k-(K-L)-1\}$ from the edge $CD$ in the second round, leaving $P_{1}$ with 
\begin{align}
{}&w_{2}(AB)=2^{f(k)+1}2^{R-f(k)}m+1,\ w_{2}(BC)=(K-L),\nonumber\\
{}&w_{2}(CD)=2^{f(k)+1}2^{R-f(k)}m+k-(K-L)-1,\ w_{2}(EF)=k.\nonumber
\end{align}
When $k \geqslant K-L+2$, this configuration is of the form given by \eqref{eq:H_{1}_losing_0,1,3} with $r=1$, and $P_{1}$ loses by our induction hypothesis since $k \leqslant K$; on the other hand, when $k=K-L+1$, it is of the form given by \eqref{eq:H_{1}_losing_0,1,3} with $r=0$, and by what we have already proved in \S\ref{subsec:proof_eq_H_{1}_losing_0}, we conclude that $P_{1}$ loses. If $k\in \{1,2,\ldots,K-L-1\}$, the configuration after the first round is of the form stated in Lemma~\ref{lem:H_{1}_winning_2}, with $w_{1}(BC)>w_{1}(EF)$, and $P_{2}$ wins by Lemma~\ref{lem:H_{1}_winning_2}. If $k=K-L$, the configuration after the first round is of the form stated in Lemma~\ref{lem:H_{1}_winning_2}, with $\ell_{1}=1$, $\ell_{2}=L$, and $w_{1}(BC)=(K-L)>(K-L)-1=w_{1}(EF)-\min\{\ell_{1},\ell_{2}\}$. Finally, if $k=0$, $P_{2}$ wins by Remark~\ref{rem:EF_edgeweight_0}.

\item Suppose $P_{1}$ removes a positive integer weight from $BC$ in the first round, leaving the edge-weights of $AB$ and $CD$ undisturbed. As long as $w_{1}(BC)=t \in \{1,2,\ldots,K-L-1\}$, $P_{2}$ removes $(K-L-t)$ from the edge $EF$ in the second round, leaving $P_{1}$ with a configuration that is of the same form as \eqref{eq:H_{1}_losing_0,1,3} with $r=1$ and $w_{2}(EF)=L+t+1\leqslant K$, so that $P_{1}$ loses by our induction hypothesis. 

If $w_{1}(BC)=0$, then $P_{2}$ is left with a galaxy graph consisting of the edges $AB$, $CD$ and $EF$, where $w_{1}(AB)=2^{R+1}m+1$, $w_{1}(CD)=2^{R+1}m+L$ and $w_{1}(EF)=(K+1)$. Since $L\leqslant (K-1)$, we conclude, by Lemma~\ref{lem:H_{1}_winning_galaxy}, that $P_{2}$ wins.
\end{enumerate}

This concludes the proof of our claim that the configuration in \eqref{eq:H_{1}_losing_1_inductive} is losing, thus completing the inductive proof of our claim that any configuration on $H_{1}$ of the form given by \eqref{eq:H_{1}_losing_0,1,3} with $r=1$, is losing.

\subsection{Details omitted from our proof that any configuration that is either of the form \eqref{eq:H_{1}_losing_2,4_case_1} or of the form \eqref{eq:H_{1}_losing_2,4_case_2}, with $r=2$, is losing}\label{subsec:H_{1}_losing_2_1,2_base_cases}
We include here the details omitted from the proof provided in \S\ref{subsec:proof_eq_H_{1}_losing_2_1,2}. To begin with, we consider the base case for the inductive argument used to prove that \eqref{eq:H_{1}_losing_2,4_case_1}, for $r=2$, is losing, which is obtained by setting $k=1$ in \eqref{eq:H_{1}_losing_2,4_case_1} (note that since $r=2$, we must have $s=1$). This yields the configuration (for any $m \in \mathbb{N}$):
\begin{align}
w_{0}(AB)=w_{0}(CD)=2m+2=2(m+1),\ w_{0}(BC)=w_{0}(EF)=1,\nonumber
\end{align}
which is of the same form as \eqref{eq:H_{1}_losing_0,1,3} with $r=0$, and by what we have already proved in \S\ref{subsec:proof_eq_H_{1}_losing_0}, we know it must be losing. The base case for the inductive argument used to prove that \eqref{eq:H_{1}_losing_2,4_case_2}, with $r=2$, is losing, is obtained by setting $k=6$ in \eqref{eq:H_{1}_losing_2,4_case_2} (once again, since $r=2$, we must have $s=1$), which yields the configuration (for any $m \in \mathbb{N}$):
\begin{align}
w_{0}(AB)=8m+2,\ w_{0}(BC)=1,\ w_{0}(CD)=8m+3,\ w_{0}(EF)=6.\label{eq:base_case_H_{1}_losing_2_2}
\end{align}
We consider the first round of the game played on this initial configuration:
\begin{enumerate}
\item Suppose $P_{1}$ removes a positive integer weight from $CD$ and a non-negative integer weight from $BC$ in the first round. We consider a few subcases of this case:
\begin{enumerate}
\item Suppose $w_{1}(CD)=8m+2$ and $w_{1}(BC)=1$. Then $P_{2}$ removes weight $5$ from $EF$, so that $P_{1}$ is left with
\begin{equation}
w_{2}(AB)=w_{2}(CD)=8m+2=2(4m+1),\ w_{2}(BC)=w_{2}(EF)=1,\nonumber
\end{equation}
which is of the same form as \eqref{eq:H_{1}_losing_0,1,3} with $r=0$, and by what we have already proved in \S\ref{subsec:proof_eq_H_{1}_losing_0}, we conclude that $P_{1}$ loses. Suppose $w_{1}(CD)=8m+i$ for $i \in\{0,1\}$, and $w_{1}(BC)=1$. Then $P_{2}$ removes weight $(3-i)$ from $EF$ in the second round, so that $P_{1}$ is left with 
\begin{align}
w_{2}(AB)=8m+2,\ w_{2}(BC)=1,\ w_{2}(CD)=8m+i,\ w_{2}(EF)=3+i,\nonumber
\end{align}
which is of the form given by \eqref{eq:H_{1}_losing_0,1,3} with either $r=1$ (when $i=1$) or $r=0$ (when $i=0$), and by what we have already proved in \S\ref{subsec:proof_eq_H_{1}_losing_0} and \S\ref{subsec:H_{1}_losing_1_proof}, we know that $P_{1}$ loses. 

\item Suppose $w_{1}(CD)=8m+i$ for some $i \in \{0,1,2\}$ and $w_{1}(BC)=0$. This leaves $P_{2}$ with a galaxy graph consisting of the edges $AB$, $CD$ and $EF$, where $w_{1}(AB)=8m+2$, $w_{1}(CD)=8m+i$ and $w_{1}(EF)=6$, so that $P_{2}$ wins by Lemma~\ref{lem:H_{1}_winning_galaxy}.

\item \label{item_1} If $w_{1}(CD)=8n+\ell_{1}$ for some $n < m$ and $\ell_{1} \in \{0,1,\ldots,7\}$, then the configuration after the first round is of the form given by Lemma~\ref{prop:H_{1}_winning_1} with $m_{1}=n$, $m_{2}=m$ and $\ell_{2}=2$. Since $n<m \implies m_{1}\neq m_{2}$, hence $P_{2}$ wins.
\end{enumerate}

\item Suppose $P_{1}$ removes a positive integer weight from $AB$ and a non-negative integer weight from $BC$ in the first round. Once again, we consider a few subcases:
\begin{enumerate}
\item Suppose $w_{1}(AB)=8m+i$ for $i \in \{0,1\}$ and $w_{1}(BC)=1$. Then $P_{2}$ removes weight $(2-i)$ from $EF$ in the second round, leaving $P_{1}$ with
\begin{align}
w_{2}(AB)=8m+i,\ w_{2}(BC)=1,\ w_{2}(CD)=8m+3,\ w_{2}(EF)=4+i,\nonumber
\end{align}
which is of the form given by \eqref{eq:H_{1}_losing_0,1,3} with either $r=1$ (when $i=1$) or $r=0$ (when $i=0$), and by what we have already proved in \S\ref{subsec:proof_eq_H_{1}_losing_0} and \S\ref{subsec:H_{1}_losing_1_proof}, we know that $P_{1}$ loses.

\item Suppose $w_{1}(AB)=8m+i$ for $i \in \{0,1\}$ and $w_{1}(BC)=0$. This leaves $P_{2}$ with a galaxy graph consisting of the edges $AB$, $CD$ and $EF$, where $w_{1}(AB)=8m+i$, $w_{1}(CD)=8m+3$ and $w_{1}(EF)=6$, and $P_{2}$ wins by Lemma~\ref{lem:H_{1}_winning_galaxy}. 

\item Finally, suppose $w_{1}(AB)=8n+\ell_{1}$ for some $n < m$ and $\ell_{1} \in \{0,1,\ldots,7\}$, and as argued above in \eqref{item_1}, this leads to $P_{2}$ winning due to Lemma~\ref{prop:H_{1}_winning_1}.
\end{enumerate}

\item Suppose $P_{1}$ removes a positive integer weight from $EF$ in the first round, so that $w_{1}(EF)=k\leqslant 5$. For $k \in \{4,5\}$, $P_{2}$ removes weight $(6-k)$ from the edge $AB$ in the second round, leaving $P_{1}$ with 
\begin{align}
w_{2}(AB)=8m+(k-4),\ w_{2}(BC)=1,\ w_{2}(CD)=8m+3,\ w_{2}(EF)=k,\nonumber
\end{align}
which is of the form given by \eqref{eq:H_{1}_losing_0,1,3} with $r=1$ if $k=5$, and of the form given by \eqref{eq:H_{1}_losing_0,1,3} with $r=0$ if $k=4$, and by what we have already proved in \S\ref{subsec:proof_eq_H_{1}_losing_0} and \S\ref{subsec:H_{1}_losing_1_proof}, we know that $P_{1}$ loses. If $k \in \{2,3\}$, then the configuration after the first round is of the form given by Lemma~\ref{prop:H_{1}_winning_1}, with $k \in \{\ell_{1},\ell_{2}\}$ since $\ell_{1}=2$ and $\ell_{2}=3$, and hence, $P_{2}$ wins. If $k=1$, then $P_{2}$ removes the edge $BC$ in the second round, leaving $P_{1}$ with a galaxy graph consisting of the edges $AB$, $CD$ and $EF$, where $w_{2}(AB)=8m+2$, $w_{2}(CD)=8m+3$ and $w_{2}(EF)=1$, so that $P_{1}$ loses by Theorem~\ref{thm:galaxy}. Finally, if $k=0$, $P_{2}$ wins by Remark~\ref{rem:EF_edgeweight_0}.

\item Suppose $P_{1}$ removes the edge $BC$ in the first round, leaving $P_{2}$ with a galaxy graph consisting of the edges $AB$, $CD$ and $EF$, where $w_{1}(AB)=8m+2$, $w_{1}(CD)=8m+3$ and $w_{1}(EF)=6$, so that $P_{2}$ wins by Lemma~\ref{lem:H_{1}_winning_galaxy}. 
\end{enumerate}

This completes the proof of our claim that the configuration in \eqref{eq:base_case_H_{1}_losing_2_2} is losing on $H_{1}$.

\subsection{Details omitted from the proof of our claim that any configuration that is of the form \eqref{eq:H_{1}_losing_2,4_case_3}, with $r=2$, is losing}\label{subsec:H_{1}_losing_2_3_base_case}
An inductive argument has been employed, in \S\ref{subsec:proof_eq_H_{1}_losing_2_3}, in proving that a configuration of the form given by \eqref{eq:H_{1}_losing_2,4_case_3}, with $r=2$, is losing on $H_{1}$. The base case for this induction corresponds to $k=6$ and $s=2$, resulting in the configuration: $w_{0}(AB)=w_{0}(CD)=8m+2$, $w_{0}(BC)=2$ and $w_{0}(EF)=6$.
\begin{enumerate}
\item Suppose $P_{1}$ removes a positive integer weight from $CD$ and a non-negative integer weight from $BC$ in the first round (an analogous situation arises if $P_{1}$ removes a positive integer weight from $AB$ and a non-negative integer weight from $BC$). We consider a few subcases of this case:
\begin{enumerate}
\item Suppose $w_{1}(CD)=8m+\ell$ for some $\ell \in \{0,1\}$, and $w_{1}(BC)=t\in \{1,2\}$. Then $P_{2}$ removes weight $6-(t+\ell+2)$ from the edge $EF$ in the second round, so that $P_{1}$ is left with
\begin{align}
w_{2}(AB)=8m+\ell,\ w_{2}(BC)=t,\ w_{2}(CD)=8m+2,\ w_{2}(EF)=t+\ell+2,\nonumber
\end{align}
which is of the form \eqref{eq:H_{1}_losing_0,1,3} with either $r=0$ (when $\ell=0$) or $r=1$ (when $\ell=1$), and by what we have already proved in \S\ref{subsec:proof_eq_H_{1}_losing_0} and \S\ref{subsec:H_{1}_losing_1_proof}, we conclude that $P_{1}$ loses.

\item Suppose $w_{1}(CD)=8m+\ell$ for some $\ell \in \{0,1\}$, and $w_{1}(BC)=0$. This leaves $P_{2}$ with a galaxy graph consisting of the edges $AB$, $CD$ and $EF$, where $w_{1}(AB)=8m+\ell$, $w_{1}(CD)=8m+2$ and $w_{1}(EF)=6$, and $P_{2}$ wins by Lemma~\ref{lem:H_{1}_winning_galaxy}. 

\item Suppose $w_{1}(CD)=8n+\ell_{1}$ for some $n < m$ and $\ell_{1} \in \{0,1,\ldots,7\}$. The resulting configuration, after the first round, is of the form given by Lemma~\ref{prop:H_{1}_winning_1} with $m_{1}=n$, $m_{2}=m$ and $\ell_{2}=2$. Since $n<m \implies m_{1}\neq m_{2}$, hence $P_{2}$ wins by Lemma~\ref{prop:H_{1}_winning_1}.
\end{enumerate}

\item Suppose $P_{1}$ removes a positive integer weight from $EF$ in the first round, so that $w_{1}(EF)=k\leqslant 5$. If $k \in \{4,5\}$, then $P_{2}$ removes weight $(6-k)$ from $AB$ in the second round, leaving $P_{1}$ with
\begin{equation}
w_{2}(AB)=8m+(k-4),\ w_{2}(BC)=2,\ w_{2}(CD)=8m+2,\ w_{2}(EF)=k,\nonumber
\end{equation}
which is of the form given by \eqref{eq:H_{1}_losing_0,1,3} with either $r=0$ (when $k=4$) or $r=1$ (when $k=5$). When $k=3$, $P_{2}$ removes weight $2$ from $AB$ and weight $1$ from $BC$, leaving $P_{1}$ with 
\begin{equation}
w_{2}(AB)=8m,\ w_{2}(BC)=1,\ w_{2}(CD)=8m+2,\ w_{2}(EF)=3,\nonumber
\end{equation}
which is of the form given by \eqref{eq:H_{1}_losing_0,1,3} with $r=0$. By what we have already proved in \S\ref{subsec:proof_eq_H_{1}_losing_0} and \S\ref{subsec:H_{1}_losing_1_proof}, we conclude that $P_{1}$ loses in each of these cases. When $k=2$, the resulting configuration is of the form stated in Lemma~\ref{prop:H_{1}_winning_1}, with $m_{1}=m_{2}=2m$ and $\ell_{1}=\ell_{2}=2$, so that $k \in \{\ell_{1},\ell_{2}\}$, and hence, $P_{2}$ wins by Lemma~\ref{prop:H_{1}_winning_1}. If $k=1$, then $w_{1}(BC)>w_{1}(EF)$, so that $P_{2}$ wins by Lemma~\ref{lem:H_{1}_winning_2}. Finally, if $k=0$, $P_{2}$ wins by Remark~\ref{rem:EF_edgeweight_0}.

\item Suppose $P_{1}$ removes a positive integer weight from $BC$ in the first round, without disturbing the edge-weights of $AB$ and $CD$. If $w_{1}(BC)=1$, then $P_{2}$ removes weight $5$ from $EF$ in the second round, leaving $P_{1}$ with $w_{2}(AB)=w_{2}(CD)=8m+2=2(4m+1)$ and $w_{2}(BC)=w_{2}(EF)=1$, which is of the form given by \eqref{eq:H_{1}_losing_0,1,3} with $r=0$, and by what we have already proved in \S\ref{subsec:proof_eq_H_{1}_losing_0}, we know that $P_{1}$ loses. If $w_{1}(BC)=0$, then $P_{2}$ is left with a galaxy graph consisting of the edges $AB$, $CD$ and $EF$, where $w_{1}(AB)=w_{1}(CD)=8m+2$ and $w_{1}(EF)=6$, and $P_{2}$ wins by Lemma~\ref{lem:H_{1}_winning_galaxy}. 
\end{enumerate}

This completes the proof of the base case for the inductive argument employed in proving that a configuration of the form \eqref{eq:H_{1}_losing_2,4_case_3}, with $r=2$, is losing on $H_{1}$.

\subsection{Proof that any configuration of the form given by \eqref{eq:H_{1}_losing_0,1,3} with $r=3$ is losing}\label{subsec:H_{1}_losing_3_proof}
We employ a sequence of inductive arguments in our proof for establishing that a configuration on $H_{1}$ that is of the form given by \eqref{eq:H_{1}_losing_0,1,3} with $r=3$ is losing. First, we show that any configuration that is of the form given by \eqref{eq:H_{1}_losing_0,1,3} with $r=3$ and $k=7$ is losing, i.e.\ for any $m \in \mathbb{N}_{0}$, we focus on the configuration
\begin{align}
w_{0}(AB)=w_{0}(CD)=8m+3,\ w_{0}(BC)=1,\ w_{0}(EF)=7.\label{eq:H_{1}_losing_3_base_case}
\end{align}
and consider the first round of the game played on this configuration:
\begin{enumerate}
\item Suppose $P_{1}$ removes a positive integer weight from $CD$, and a non-negative integer weight from $BC$, in the first round (this is analogous to $P_{1}$ removing a positive integer weight from $AB$ and a non-negative integer weight from $BC$ in the first round). The following subcases are possible:
\begin{enumerate}
\item Suppose $w_{1}(CD)=8m+i$ for some $i \in \{0,1,2\}$ and $w_{1}(BC)=1$. Then $P_{2}$ removes weight $(3-i)$ from the edge $EF$ in the second round, leaving $P_{1}$ with
\begin{equation}
w_{2}(AB)=8m+3,\ w_{2}(BC)=1,\ w_{2}(CD)=8m+i,\ w_{2}(EF)=4+i,\nonumber
\end{equation}
which is either of the form given by \eqref{eq:H_{1}_losing_0,1,3} with $r=0$ (when $i=0$) or $r=1$ (when $i=1$), or of the form given by \eqref{eq:H_{1}_losing_2,4_case_2} with $r=2$, and by what we have already proved in \S\ref{subsec:proof_eq_H_{1}_losing_0}, \S\ref{subsec:H_{1}_losing_1_proof} and \S\ref{subsec:proof_eq_H_{1}_losing_2_1,2}, we conclude that $P_{1}$ loses.

\item Suppose $w_{1}(CD)=8m+i$ for some $i \in\{0,1,2\}$, and $w_{1}(BC)=0$. This leaves $P_{2}$ with a galaxy graph consisting of the edges $AB$, $CD$ and $EF$, where $w_{1}(AB)=8m+3$, $w_{1}(CD)=8m+i$ and $w_{1}(EF)=7$, and $P_{2}$ wins by Lemma~\ref{lem:H_{1}_winning_galaxy}.

\item Finally, suppose $w_{1}(CD)=8n+\ell_{1}$ for some $n < m$ and some $\ell_{1} \in \{0,1,\ldots,7\}$. The resulting configuration, after the first round, is of the form given by Lemma~\ref{prop:H_{1}_winning_1} with $m_{1}=n$, $m_{2}=m$ and $\ell_{2}=3$. Since $n<m \implies m_{1}\neq m_{2}$, hence $P_{2}$ wins by Lemma~\ref{prop:H_{1}_winning_1}. 
\end{enumerate}

\item Suppose $P_{1}$ removes a positive integer weight from $EF$ in the first round, so that $w_{1}(EF)=k \leqslant 6$. If $k\in \{4,5,6\}$, then $P_{2}$ removes weight $(7-k)$ from $CD$ in the second round, leaving $P_{1}$ with
\begin{equation}
w_{2}(AB)=8m+3,\ w_{2}(BC)=1,\ w_{2}(CD)=8m+(k-4),\ w_{2}(EF)=k,\nonumber
\end{equation}
which is either of the form given by \eqref{eq:H_{1}_losing_0,1,3} with $r=0$ (when $k=4$) or $r=1$ (when $k=5$), or of the form given by \eqref{eq:H_{1}_losing_2,4_case_2} with $r=2$ (when $k=6$), and by what we have already proved in \S\ref{subsec:proof_eq_H_{1}_losing_0}, \S\ref{subsec:H_{1}_losing_1_proof} and \S\ref{subsec:proof_eq_H_{1}_losing_2_1,2}, we conclude that $P_{1}$ loses. If $k\in\{1,2,3\}$, the configuration after the first round is of the form stated in Lemma~\ref{prop:H_{1}_winning_1} with $\ell_{1}=\ell_{2}=3$ and $w_{1}(EF)=k\leqslant \min\{\ell_{1},\ell_{2}\}$, and $P_{2}$ wins by Lemma~\ref{prop:H_{1}_winning_1}. When $k=0$, $P_{2}$ wins by Remark~\ref{rem:EF_edgeweight_0}. 

\item Suppose $P_{1}$ removes the edge $BC$ in the first round, but leaves the edge-weights of $AB$ and $CD$ undisturbed. Then, after the first round, $P_{2}$ is left with a galaxy graph, consisting of the edges $AB$, $CD$ and $EF$, with $w_{1}(AB)=w_{1}(CD)=8m+3$ and $w_{1}(EF)=7$, and $P_{2}$ wins by Lemma~\ref{lem:H_{1}_winning_galaxy}. 
\end{enumerate}
This completes the proof of our claim that the configuration in \eqref{eq:H_{1}_losing_3_base_case} is losing.

Suppose, for some $K \in \mathbb{N}$, with $K \geqslant 7$, we have proved that any configuration that is of the form given by \eqref{eq:H_{1}_losing_0,1,3} with $r=3$ and $w_{0}(EF)=k \leqslant K$, is losing on $H_{1}$. We now consider a configuration of the form
\begin{equation}
w_{0}(AB)=2^{R+1}m+3,\ w_{2}(BC)=K-L-2,\ w_{2}(CD)=2^{R+1}m+L,\ w_{2}(EF)=K+1,\label{eq:H_{1}_losing_3_inductive}
\end{equation}
where $R=f(K+1)$ and $L \in \{3,4,\ldots,K-3\}$. The first round of the game played on this configuration may unfold as follows:
\begin{enumerate}
\item Suppose $P_{1}$ removes a positive integer weight from $CD$, and a non-negative integer weight from $BC$, in the first round. This can be subdivided into the following cases:
\begin{enumerate}
\item If $w_{1}(CD)=2^{R+1}m+\ell$ for some $\ell \in \{0,1,\ldots,L-1\}$ and $w_{1}(BC)=t$ for some $t \in \{1,2,\ldots,K-L-2\}$, $P_{2}$ removes weight $(K+1)-(3+t+\ell)$ from $EF$ in the second round. Note that $\ell \leqslant L-1$ and $t \leqslant K-L-2$ together imply $(3+t+\ell)\leqslant K$. This leaves $P_{1}$ with
\begin{align}
{}&w_{2}(AB)=2^{f(3+t+\ell)+1}2^{R-f(3+t+\ell)}m+3,\ w_{2}(BC)=t,\nonumber\\
{}&w_{2}(CD)=2^{f(3+t+\ell)+1}2^{R-f(3+t+\ell)}m+\ell,\ w_{2}(EF)=(3+t+\ell),\nonumber
\end{align}
which is of the same form as \eqref{eq:H_{1}_losing_0,1,3} with $r=3$ and $w_{2}(EF)=(3+t+\ell)\leqslant K$. Consequently, $P_{1}$ loses by our induction hypothesis.

\item Suppose $w_{1}(CD)=2^{R+1}m+\ell$ for some $\ell \in \{0,1,\ldots,L-1\}$, and $w_{1}(BC)=0$. This leaves $P_{2}$ with a galaxy graph comprising the edges $AB$, $CD$ and $EF$, where $w_{1}(AB)=2^{R+1}m+3$, $w_{1}(CD)=2^{R+1}m+\ell$ and $w_{1}(EF)=(K+1)$. Since $L \leqslant (K-3)$ and $\ell\leqslant (L-1)$, $P_{2}$ wins by Lemma~\ref{lem:H_{1}_winning_galaxy}.

\item If $w_{1}(CD)=2^{R+1}n+\ell_{1}$ for some $n < m$ and some $\ell_{1}\in\{0,1,\ldots,2^{R+1}-1\}$, the configuration after the first round is of the form given by Lemma~\ref{prop:H_{1}_winning_1} with $m_{1}=n$, $m_{2}=m$ and $\ell_{2}=3$. Since $n<m \implies m_{1}\neq m_{2}$, hence $P_{2}$ wins by Lemma~\ref{prop:H_{1}_winning_1}.
\end{enumerate}

\item Suppose $P_{1}$ removes a positive integer weight from $AB$, and a non-negative integer weight from $BC$, in the first round. Once again, we consider the possible cases:
\begin{enumerate}
\item Suppose $w_{1}(AB)=2^{R+1}m+\ell$ for some $\ell \in \{0,1,2\}$, and $w_{1}(BC)=t \in\{1,2,\ldots,K-L-2\}$. If \begin{enumerate*}
\item $\ell\in\{0,1\}$, 
\item or if $\ell=2$ and $t\geqslant 2$, 
\item or if $\ell=2$, $t=1$ and $L \equiv i \bmod 4$ for some $i \in \{0,1,3\}$,
\end{enumerate*}
then $P_{2}$ removes weight $(K+1)-(\ell+t+L)$ from $EF$ in the second round. Note that $\ell \leqslant 2$ and $t \leqslant K-L-2$ together imply $(\ell+t+L) \leqslant K$. This leaves $P_{1}$ with
\begin{align}
{}&w_{2}(AB)=2^{f(\ell+t+L)+1}2^{R-f(\ell+t+L)}m+\ell,\ w_{2}(BC)=t,\nonumber\\
{}&w_{2}(CD)=2^{f(\ell+t+L)+1}2^{R-f(\ell+t+L)}m+L,\ w_{2}(EF)=(\ell+t+L),\nonumber
\end{align} 
which is of the form given by \eqref{eq:H_{1}_losing_0,1,3} with either $r=0$ (when $\ell=0$) or $r=1$ (when $\ell=1$), or of the form given by \eqref{eq:H_{1}_losing_2,4_case_3} with $r=2$ when $\ell=2$ and $t \geqslant 2$, or of the form given by \eqref{eq:H_{1}_losing_2,4_case_2} wth $r=2$ when $\ell=2$, $t=1$ and $L \equiv i \bmod 4$ for some $i \in \{0,1,3\}$. By what we have already proved in \S\ref{subsec:proof_eq_H_{1}_losing_0}, \S\ref{subsec:H_{1}_losing_1_proof}, \S\ref{subsec:proof_eq_H_{1}_losing_2_1,2} and \S\ref{subsec:proof_eq_H_{1}_losing_2_3}, we conclude that in each of these cases, $P_{1}$ loses. If $\ell=2$, $t=1$ and $L \equiv 2 \bmod 4$, $P_{2}$ removes weight $(K+1)-(L-1)$ from $EF$ in the second round, so that $P_{1}$ is left with
\begin{align}
{}&w_{2}(AB)=2^{f(L-1)+1}2^{R-f(L-1)}m+2,\ w_{2}(BC)=1,\nonumber\\
{}&w_{2}(CD)=2^{f(L-1)+1}2^{R-f(L-1)}m+L,\ w_{2}(EF)=(L-1),\nonumber
\end{align}
which is of the form given by \eqref{eq:H_{1}_losing_2,4_case_1} with $r=2$, and by what we have already proved in \S\ref{subsec:proof_eq_H_{1}_losing_2_1,2}, we conclude that $P_{1}$ loses.

\item \label{recalled_later_1} Suppose $w_{1}(AB)=2^{R+1}m+\ell$ for some $\ell \in \{0,1,2\}$, and $w_{1}(BC)=0$. This leaves $P_{2}$ with a galaxy graph, consisting of the edges $AB$, $CD$ and $EF$, where $w_{1}(AB)=2^{R+1}m+\ell$, $w_{1}(CD)=2^{R+1}m+L$ and $w_{1}(EF)=(K+1)$. Since $L\leqslant (K-3)$ and $\ell\leqslant 2$, $P_{2}$ wins by Lemma~\ref{lem:H_{1}_winning_galaxy}.

\item If $w_{1}(AB)=2^{R+1}n+\ell_{1}$ for some $n < m$ and some $\ell_{1}\in\{0,1,\ldots,2^{R+1}-1\}$, the configuration after the first round is of the form given by Lemma~\ref{prop:H_{1}_winning_1} with $m_{1}=n$, $m_{2}=m$ and $\ell_{2}=L$. Since $n<m \implies m_{1}\neq m_{2}$, hence $P_{2}$ wins by Lemma~\ref{prop:H_{1}_winning_1}.
\end{enumerate}

\item Suppose $P_{1}$ removes a positive integer weight from $EF$ in the first round, so that $w_{1}(EF)=k\leqslant K$. For $k \in \{K-L+1,K-L+2,\ldots,K\}$, $P_{2}$ removes weight $(K+1-k)$ from the edge $CD$ in the second round, leaving $P_{1}$ with 
\begin{align}
{}&w_{2}(AB)=2^{f(k)+1}2^{R-f(k)}m+3,\ w_{2}(BC)=(K-L-2),\nonumber\\
{}&w_{2}(CD)=2^{f(k)+1}2^{R-f(k)}m+\{k+L-(K+1)\},\ w_{2}(EF)=k,\nonumber
\end{align}
which is of the form given by \eqref{eq:H_{1}_losing_0,1,3} with $r=3$ and $w_{2}(EF)\leqslant K$ if $k \geqslant K-L+4$, of the form given by \eqref{eq:H_{1}_losing_2,4_case_3} with $r=2$ when $k=K-L+3$ and $L<K-3$, of the form given by \eqref{eq:H_{1}_losing_2,4_case_2} with $r=2$ when $k=K-L+3$ and $L=K-3$, of the form given by \eqref{eq:H_{1}_losing_0,1,3} with $r=1$ when $k=K-L+2$, and of the form given by \eqref{eq:H_{1}_losing_0,1,3} with $r=0$ when $k=K-L+1$. In the first of these cases, $P_{1}$ loses by our induction hypothesis, and in the remaining cases, $P_{1}$ loses by what we have already established in \S\ref{subsec:proof_eq_H_{1}_losing_2_3}, \S\ref{subsec:proof_eq_H_{1}_losing_2_1,2}, \S\ref{subsec:H_{1}_losing_1_proof} and \S\ref{subsec:proof_eq_H_{1}_losing_0} respectively. 

For $k \in \{4,5,\ldots,K-L\}$, the configuration at the end of the first round is of the same form as that mentioned in Lemma~\ref{lem:H_{1}_winning_2}, with $\ell_{1}=3$, $\ell_{2}=L$, so that $w_{2}(EF)=k>3=\min\{\ell_{1},\ell_{2}\}$ and $k-\min\{\ell_{1},\ell_{2}\}\leqslant K-L-3 < K-L-2=w_{2}(BC)$. Consequently, $P_{2}$ wins by Lemma~\ref{lem:H_{1}_winning_2}. For $k\in\{1,2,3\}$, the configuration at the end of the first round is of the same form as that mentioned in Lemma~\ref{prop:H_{1}_winning_1}, with $\ell_{1}=3$, $\ell_{2}=L$, so that $\min\{\ell_{1},\ell_{2}\}\geqslant k=w_{1}(EF)$, and $P_{2}$ wins by Lemma~\ref{prop:H_{1}_winning_1}. If $k=0$, $P_{2}$ wins by Remark~\ref{rem:EF_edgeweight_0}.

\item Finally, suppose $P_{1}$ removes a positive integer weight from $BC$ in the first round, without disturbing the edge-weights of $AB$ and $CD$. Then $w_{1}(BC)=t \in \{0,1,\ldots,K-L-3\}$. For $t\geqslant 1$, $P_{2}$ simply removes weight $(K+1)-(3+t+L)$ from the edge $EF$ in the second round, leaving $P_{1}$ with
\begin{align}
{}&w_{2}(AB)=2^{f(3+t+L)+1}2^{R-f(3+t+L)}m+3,\ w_{2}(BC)=t,\nonumber\\
{}&w_{2}(CD)=2^{f(3+t+L)+1}2^{R-f(3+t+L)}m+L,\ w_{2}(EF)=(3+t+L),\nonumber
\end{align}
which is of the same form as \eqref{eq:H_{1}_losing_0,1,3} with $r=3$ and $w_{2}(EF)=(3+t+L)\leqslant K$, so that $P_{1}$ loses by our induction hypothesis. If $t=0$, then, at the end of the first round, $P_{2}$ is left with a galaxy graph consisting of the edges $AB$, $CD$ and $EF$, where $w_{1}(AB)=2^{R+1}m+3$, $w_{1}(CD)=2^{R+1}m+L$ and $w_{1}(EF)=(K+1)$. Since $L\leqslant (K-3)$, it is evident from Lemma~\ref{lem:H_{1}_winning_galaxy} that $P_{2}$ wins.
\end{enumerate}
This completes the proof of our claim that the configuration in \eqref{eq:H_{1}_losing_3_inductive} is losing on $H_{1}$, and also concludes our inductive proof of our claim that any configuration that is of the form given by \eqref{eq:H_{1}_losing_0,1,3}, with $r=3$, is losing.

\subsection{Proof that a configuration that is either of the form given by \eqref{eq:H_{1}_losing_2,4_case_1} or of the form given by \eqref{eq:H_{1}_losing_2,4_case_2} is losing for $r=4$ and $s=1$}\label{subsec:H_{1}_losing_4_1,2_proof}
Our argument is, as above, inductive. The base case corresponding to \eqref{eq:H_{1}_losing_2,4_case_2} with $r=4$ and $s=1$ is obtained by setting $k=12$, which yields the configuration
\begin{equation}
w_{0}(AB)=16m+4,\ w_{0}(BC)=1,\ w_{0}(CD)=16m+7,\ w_{0}(EF)=12.\label{eq:H_{1}_losing_4_1_base_case}
\end{equation}
We consider the first round of the game played on this configuration:
\begin{enumerate}
\item Suppose $P_{1}$ removes a positive integer weight from $CD$, and a non-negative integer weight from $BC$, in the first round. This case can be subdivided into the following scenarios:
\begin{enumerate}
\item Suppose $w_{1}(CD)=16m+\ell$ for some $\ell \in \{0,1,\ldots,6\}$, and $w_{1}(BC)=1$. If $\ell \in \{0,1,2,3\}$, then $P_{2}$ removes weight $(7-\ell)$ from the edge $EF$ in the second round, leaving $P_{1}$ with
\begin{equation}
w_{2}(AB)=16m+4,\ w_{2}(BC)=1,\ w_{2}(CD)=16m+\ell,\ w_{2}(EF)=(5+\ell),\nonumber
\end{equation}
which is of the form given by \eqref{eq:H_{1}_losing_0,1,3} with $r=0$ when $\ell=0$, of the form given by \eqref{eq:H_{1}_losing_0,1,3} with $r=1$ when $\ell=1$, of the form given by \eqref{eq:H_{1}_losing_2,4_case_2} with $r=2$ when $\ell=2$, and of the form given by \eqref{eq:H_{1}_losing_0,1,3} with $r=3$ when $\ell=3$. By what we have already proved in \S\ref{subsec:proof_eq_H_{1}_losing_0}, \S\ref{subsec:H_{1}_losing_1_proof}, \S\ref{subsec:proof_eq_H_{1}_losing_2_1,2} and \S\ref{subsec:H_{1}_losing_3_proof}, we conclude that $P_{1}$ loses in each of these cases. On the other hand, if $\ell \in \{4,5,6\}$, $P_{2}$ removes weight $(15-\ell)$ from the edge $EF$ in the second round, leaving $P_{1}$ with the configuration
\begin{align}
{}&w_{2}(AB)=16m+4=4(4m+1),\ w_{2}(BC)=1,\nonumber\\
{}&w_{2}(CD)=16m+\ell=4(4m+1)+(\ell-4),\ w_{2}(EF)=(\ell-3),\nonumber
\end{align}
which is of the form given by \eqref{eq:H_{1}_losing_0,1,3} with $r=0$, and by what we have already proved in \S\ref{subsec:proof_eq_H_{1}_losing_0}, we conclude that $P_{1}$ loses.

\item If $w_{1}(CD)=16m+\ell$ for some $\ell \in \{0,1,\ldots,6\}$, and $w_{1}(BC)=0$, then $P_{2}$, after the first round, is left with a galaxy graph, comprising the edges $AB$, $CD$ and $EF$, such that $w_{1}(AB)=16m+4$, $w_{1}(CD)=16m+\ell$ and $w_{1}(EF)=12$, and it is evident from Lemma~\ref{lem:H_{1}_winning_galaxy} that $P_{2}$ wins.

\item If $w_{1}(CD)=16n+\ell_{1}$ for some $n < m$ and some $\ell_{1}\in\{0,1,\ldots,15\}$, the configuration, after the first round, is of the form given by Lemma~\ref{prop:H_{1}_winning_1} with $m_{1}=n$, $m_{2}=m$ and $\ell_{2}=4$. Since $n<m \implies m_{1}\neq m_{2}$, hence $P_{2}$ wins by Lemma~\ref{prop:H_{1}_winning_1}.
\end{enumerate}

\item Suppose $P_{1}$ removes a positive integer weight from $AB$, and a non-negative integer weight from $BC$, in the first round. Once again, we consider all possible subcases:
\begin{enumerate}
\item Suppose $w_{1}(AB)=16m+\ell$ for some $\ell \in \{0,1,2,3\}$, and $w_{1}(BC)=1$. Then $P_{2}$ removes weight $(4-\ell)$ from the edge $EF$ in the second round, leaving $P_{1}$ with
\begin{align}
w_{2}(AB)=16m+\ell,\ w_{2}(BC)=1,\ w_{2}(CD)=16m+7,\ w_{2}(EF)=8+\ell,\nonumber
\end{align}
which is of the form given by \eqref{eq:H_{1}_losing_0,1,3} with $r=0$ when $\ell=0$, of the form given by \eqref{eq:H_{1}_losing_0,1,3} with $r=1$ when $\ell=1$, of the form given by \eqref{eq:H_{1}_losing_2,4_case_2} with $r=2$ when $\ell=2$, and of the form given by \eqref{eq:H_{1}_losing_0,1,3} with $r=3$ when $\ell=3$. By what we have already proved in \S\ref{subsec:proof_eq_H_{1}_losing_0}, \S\ref{subsec:H_{1}_losing_1_proof}, \S\ref{subsec:proof_eq_H_{1}_losing_2_1,2} and \S\ref{subsec:H_{1}_losing_3_proof}, we know that $P_{1}$ loses.

\item If $w_{1}(AB)=16m+\ell$ for some $\ell \in \{0,1,2,3\}$, and $w_{1}(BC)=0$, $P_{2}$ is left, at the end of the first round, with a galaxy graph consisting of the edges $AB$, $CD$ and $EF$, where $w_{1}(AB)=16m+\ell$, $w_{1}(CD)=16m+7$ and $w_{1}(EF)=12$, and $P_{2}$ wins by Lemma~\ref{lem:H_{1}_winning_galaxy}.

\item If $w_{1}(AB)=16n+\ell_{1}$ for some $n < m$ and some $\ell_{1}\in\{0,1,\ldots,15\}$, the configuration after the first round is of the form given by Lemma~\ref{prop:H_{1}_winning_1} with $m_{1}=n$, $m_{2}=m$ and $\ell_{2}=7$. Since $n<m \implies m_{1}\neq m_{2}$, hence $P_{2}$ wins by Lemma~\ref{prop:H_{1}_winning_1}.
\end{enumerate}

\item Suppose $P_{1}$ removes a positive integer weight from $EF$ in the first round, so that $w_{1}(EF)=k\leqslant 11$. If $k \in \{8,9,10,11\}$, $P_{2}$ removes weight $(12-k)$ from $AB$ in the second round, so that $P_{1}$ is left with
\begin{align}
w_{2}(AB)=16m+(k-8),\ w_{2}(BC)=1,\ w_{2}(CD)=16m+7,\ w_{2}(EF)=k,\nonumber
\end{align}
which is of the form given by \eqref{eq:H_{1}_losing_0,1,3} with $r=0$ when $k=8$, of the form given by \eqref{eq:H_{1}_losing_0,1,3} with $r=1$ when $k=9$, of the form given by \eqref{eq:H_{1}_losing_2,4_case_2} with $r=2$ when $k=10$, and of the form given by \eqref{eq:H_{1}_losing_0,1,3} with $r=3$ when $k=11$, and by what we have already proved in \S\ref{subsec:proof_eq_H_{1}_losing_0}, \S\ref{subsec:H_{1}_losing_1_proof}, \S\ref{subsec:proof_eq_H_{1}_losing_2_1,2} and \S\ref{subsec:H_{1}_losing_3_proof}, we know that $P_{1}$ loses. If $k \in \{1,2,3,4,7\}$, the configuration at the end of the first round is of the same form as described in Lemma~\ref{prop:H_{1}_winning_1}, with $\ell_{1}=4$ and $\ell_{2}=7$, so that we either have $\min\{\ell_{1},\ell_{2}\}\geqslant k$, or we have $k \in \{\ell_{1},\ell_{2}\}$, and in each of these cases, $P_{2}$ wins by Lemma~\ref{prop:H_{1}_winning_1}. If $k=0$, $P_{2}$ wins by Remark~\ref{rem:EF_edgeweight_0}.

If $k \in \{5,6\}$, $P_{2}$ removes weight $(12-k)$ from the edge $CD$ in the second round, leaving $P_{1}$ with
\begin{equation}
w_{2}(AB)=16m+4,\ w_{2}(BC)=1,\ w_{2}(CD)=16m+(k-5),\ w_{2}(EF)=k,\nonumber
\end{equation}
which is of the form given by \eqref{eq:H_{1}_losing_0,1,3} with either $r=0$ (when $k=5$) or $r=1$ (when $k=6$), and by what we have already proved in \S\ref{subsec:proof_eq_H_{1}_losing_0} and \S\ref{subsec:H_{1}_losing_1_proof}, we know that $P_{1}$ loses.

\item Finally, suppose $P_{1}$ removes the edge $BC$, without disturbing the edge-weights of $AB$ and $CD$, in the first round. This leaves $P_{2}$ with a galaxy graph, consisting of the edges $AB$, $CD$ and $EF$, with $w_{2}(AB)=16m+4$, $w_{2}(CD)=16m+7$ and $w_{2}(EF)=12$, and $P_{2}$ wins by Lemma~\ref{lem:H_{1}_winning_galaxy}.
\end{enumerate}

This completes the proof of our claim that the configuration in \eqref{eq:H_{1}_losing_4_1_base_case} is losing on $H_{1}$.

Next, setting $r=4$, $s=1$ and one of $k=1$, $k=2$ and $k=3$ in \eqref{eq:H_{1}_losing_2,4_case_1}, we obtain, respectively:
\begin{align}
{}&w_{0}(AB)=w_{0}(CD)=2m+4=2(m+2),\ w_{0}(BC)=w_{0}(EF)=1,\nonumber\\
{}&w_{0}(AB)=4m+4=4(m+1),\ w_{0}(BC)=1,\ w_{0}(CD)=4m+5=4(m+1)+1,\ w_{0}(EF)=2,\nonumber\\
{}&w_{0}(AB)=4m+4=4(m+1),\ w_{0}(BC)=1,\ w_{0}(CD)=4m+6=4(m+1)+2,\ w_{0}(EF)=3,\nonumber
\end{align}
each of which is losing since it is of the form given by \eqref{eq:H_{1}_losing_0,1,3} with $r=0$. This establishes the base case for the inductive proof of our claim that a configuration of the form given by \eqref{eq:H_{1}_losing_2,4_case_1} with $r=4$ and $s=1$ is losing.

We now come to the inductive steps of our proof. Suppose, for some $K \in \mathbb{N}$, we have shown that any configuration on $H_{1}$ that is either of the form \eqref{eq:H_{1}_losing_2,4_case_1} or of the form \eqref{eq:H_{1}_losing_2,4_case_2}, with $r=4$ and $s=1$, and satisfying the inequality $k \leqslant K$, is losing. 

\subsubsection{When $K \geqslant 8$ and $K \equiv i \bmod 8$ for some $i \in \{0,1,2\}$}\label{subsubsec:H_{1}_losing_4_0,1,2} For any $m \in \mathbb{N}_{0}$, we focus on the configuration
\begin{equation}
w_{0}(AB)=2^{R+1}m+4,\ w_{0}(BC)=1,\ w_{0}(CD)=2^{R+1}m+(K+4),\ w_{0}(EF)=(K+1),\label{eq:H_{1}_losing_4_2_inductive}
\end{equation}
where $R=f(K+1)$ -- this is of the form given by \eqref{eq:H_{1}_losing_2,4_case_1} with $r=4$ and $s=1$. We consider the first round of the game played on this configuration:
\begin{enumerate}
\item Suppose $P_{1}$ removes a positive integer weight from $CD$, and a non-negative integer weight from $BC$, in the first round. We subdivide this into the following cases:
\begin{enumerate}
\item If $w_{1}(CD)=2^{R+1}m+(K+1)$, then, irrespective of what the value of $w_{1}(BC)$ is, the configuration after the first round is of the form stated in Lemma~\ref{prop:H_{1}_winning_1}, with $\ell_{1}=4$, $\ell_{2}=(K+1)$ and $w_{1}(EF)=(K+1) \in \{\ell_{1},\ell_{2}\}$, so that $P_{2}$ wins by Lemma~\ref{prop:H_{1}_winning_1}. 

\item Suppose $w_{1}(CD)=2^{R+1}m+L$ where $L \in \{7,8,\ldots,K-5\}$, $L \equiv j \bmod 8$ for some $j \in \{0,1,2,3,7\}$, and $w_{1}(BC)=1$. Here, $P_{2}$ removes weight $(K+1)-(L+5)$ from the edge $EF$ in the second round, leaving $P_{1}$ with
\begin{align}
{}&w_{2}(AB)=2^{f(L+5)+1}2^{R-f(L+5)}m+4,\ w_{2}(BC)=1,\nonumber\\
{}&w_{2}(CD)=2^{f(L+5)+1}2^{R-f(L+5)}m+L,\ w_{2}(EF)=L+5,\nonumber
\end{align}
where $(L+5) \geqslant 12$ and $(L+5) \equiv j' \bmod 8$ for some $j' \in \{0,4,5,6,7\}$. Therefore, this configuration is of the form given by \eqref{eq:H_{1}_losing_2,4_case_2} with $r=4$, $s=1$ and $w_{2}(EF)=(L+5)\leqslant K$, so that $P_{1}$ loses by our induction hypothesis. 

\item Suppose $w_{1}(CD)=2^{R+1}m+L$ where $L \in \{1,2,\ldots,K+3\}$, $L \equiv j \bmod 8$ for some $j \in \{4,5,6\}$, and $w_{1}(BC)=1$. Recall that $K \equiv i \bmod 8$, for some $i \in \{0,1,2\}$, for the entirety of \S\ref{subsubsec:H_{1}_losing_4_0,1,2}. The following cases are, consequently, taken care of under this scenario:  
\begin{enumerate}
\item $i=0$ and $L \in \{K-4,K-3,K-2\}$, 
\item $i=1$ and $L \in \{K-4,K-3,K+3\}$,
\item $i=2$ and $L \in \{K-4,K+2,K+3\}$.
\end{enumerate} 
Here, $P_{2}$ removes weight $(K+1)-(L-3)$ from the edge $EF$ in the second round. Note that since $L \leqslant (K+3)$, we have $(L-3) \leqslant K$. The configuration after the second round is
\begin{align}
{}&w_{2}(AB)=2^{f(L-3)+1}2^{R-f(L-3)}m+4,\ w_{2}(BC)=1,\nonumber\\
{}&w_{2}(CD)=2^{f(L-3)+1}2^{R-f(L-3)}m+L,\ w_{2}(EF)=L-3,\nonumber
\end{align}
and as $(L-3) \equiv j'\bmod 8$ for some $j' \in \{1,2,3\}$, this is of the form given by \eqref{eq:H_{1}_losing_2,4_case_1} with $r=4$ and $s=1$. Since $w_{2}(EF)=(L-3)\leqslant K$, $P_{1}$ loses by our induction hypothesis.

\item Suppose $w_{1}(CD)=2^{R+1}m+L$ where $L \in \{K-1,K\}$, and $w_{1}(BC)=1$. Then $P_{2}$ removes weight $(L-K+4)$ from the edge $AB$ in the second round, leaving $P_{1}$ with
\begin{equation}
w_{2}(AB)=2^{R+1}m+(K-L),\ w_{2}(BC)=1,\ w_{2}(CD)=2^{R+1}m+L,\ w_{2}(EF)=(K+1),\nonumber
\end{equation}
which is of the form given by \eqref{eq:H_{1}_losing_0,1,3} with either $r=0$ (when $L=K$) or $r=1$ (when $L=K-1$), and by what we have already proved in \S\ref{subsec:proof_eq_H_{1}_losing_0} and \S\ref{subsec:H_{1}_losing_1_proof}, we know that $P_{1}$ loses. Likewise, when $K \equiv i \bmod 8$ for $i \in \{1,2\}$, $w_{1}(CD)=2^{R+1}m+(K-2)$ and $w_{1}(BC)=1$, $P_{2}$ removes weight $2$ from $AB$ in the second round, leaving $P_{1}$ with 
\begin{equation}
w_{2}(AB)=2^{R+1}m+2,\ w_{2}(BC)=1,\ w_{2}(CD)=2^{R+1}m+(K-2),\ w_{2}(EF)=(K+1),\nonumber
\end{equation}
which is of the form given by \eqref{eq:H_{1}_losing_2,4_case_2} with $r=2$ and $s=1$, and by what we have already proved in \S\ref{subsec:proof_eq_H_{1}_losing_2_1,2}, we conclude that $P_{1}$ loses. When $K \equiv 2 \bmod 8$, $w_{1}(CD)=2^{R+1}m+(K-3)$ and $w_{1}(BC)=1$, $P_{2}$ removes weight $1$ from $AB$, leaving $P_{1}$ with
\begin{align}
{}&w_{2}(AB)=2^{R+1}m+3,\ w_{2}(BC)=1,\ w_{2}(CD)=2^{R+1}m+(K-3),\ w_{2}(EF)=(K+1),\nonumber
\end{align}
which is of the same form as \eqref{eq:H_{1}_losing_0,1,3} with $r=3$, and by what we have already proved in \S\ref{subsec:H_{1}_losing_3_proof}, we know that $P_{1}$ loses. If $K \equiv i \bmod 8$ for $i \in \{0,1\}$, $w_{1}(CD)=2^{R+1}m+(K+2)$, and $w_{1}(BC)=1$,  
\begin{enumerate}
\item $P_{2}$ removes $BC$ and weight $1$ from $AB$ in the second round when $K \equiv 0 \bmod 8$, 
\item $P_{2}$ removes $BC$ and weight $3$ from $AB$ in the second round when $K \equiv 1 \bmod 8$.
\end{enumerate}
Consequently, $P_{1}$ is left with a galaxy graph, consisting of the edges $AB$, $CD$ and $EF$, with
\begin{align}
w_{2}(AB)=2^{R+1}m+3,\ w_{2}(CD)=2^{R+1}m+(8n+2),\ w_{2}(EF)=(8n+1)\nonumber
\end{align}
when $K=8n$ for some $n \in \mathbb{N}$, or 
\begin{align}
w_{2}(AB)=2^{R+1}m+1,\ w_{2}(CD)=2^{R+1}m+(8n+3),\ w_{2}(EF)=(8n+2)\nonumber
\end{align}
when $K=(8n+1)$ for some $n \in \mathbb{N}$. In each case, the triple is balanced, and $P_{1}$ loses by Theorem~\ref{thm:galaxy}. Likewise, when $K \equiv 0 \bmod 8$, $w_{1}(CD)=2^{R+1}m+(K+3)$ and $w_{1}(BC)=1$, $P_{2}$ removes $BC$ and weight $2$ from $AB$ in the second round, leaving $P_{1}$ with a galaxy graph consisting of the edges $AB$, $CD$ and $EF$, where (once again, writing $K=8n$ for some $n \in \mathbb{N}$)
\begin{equation}
w_{2}(AB)=2^{R+1}m+2,\ w_{2}(CD)=2^{R+1}m+(8n+3),\ w_{2}(EF)=(8n+1),\nonumber
\end{equation}
and as this triple is balanced, $P_{1}$ loses by Theorem~\ref{thm:galaxy}.

\item Suppose $w_{1}(CD)=2^{R+1}m+L$ for some $L \in \{0,1,\ldots,6\}$, and $w_{1}(BC)=1$. If $L \in \{0,1,2,3\}$, then $P_{2}$ removes weight $(K+1)-(L+5)$ from $EF$ in the second round, leaving $P_{1}$ with
\begin{align}
w_{2}(AB)=2^{R+1}m+4,\ w_{2}(BC)=1,\ w_{2}(CD)=2^{R+1}m+L,\ w_{2}(EF)=(L+5),\nonumber
\end{align}
which is of the form given by \eqref{eq:H_{1}_losing_0,1,3} with $r=0$ when $L=0$, of the form given by \eqref{eq:H_{1}_losing_0,1,3} with $r=1$ when $L=1$, of the form given by \eqref{eq:H_{1}_losing_2,4_case_2} with $r=2$ when $L=2$, and of the form given by \eqref{eq:H_{1}_losing_0,1,3} with $r=3$ when $L=3$. By what we have already proved in \S\ref{subsec:proof_eq_H_{1}_losing_0}, \S\ref{subsec:H_{1}_losing_1_proof}, \S\ref{subsec:proof_eq_H_{1}_losing_2_1,2} and \S\ref{subsec:H_{1}_losing_3_proof}, we conclude that $P_{1}$ loses. If $L \in \{4,5,6\}$, $P_{2}$ removes weight $(K+1)-(L-3)$ from the edge $EF$ in the second round, leaving $P_{1}$ with
\begin{align}
{}&w_{2}(AB)=2^{R+1}m+4=4\left(2^{R-1}m+1\right),\ w_{2}(BC)=1,\nonumber\\
{}&w_{2}(CD)=2^{R+1}m+L=4\left(2^{R-1}m+1\right)+(L-4),\ w_{2}(EF)=(L-3),\nonumber
\end{align}
which is of the form given by \eqref{eq:H_{1}_losing_0,1,3} with $r=0$, and by what we have already proved in \S\ref{subsec:proof_eq_H_{1}_losing_0}, we know that $P_{1}$ loses.

\item If $w_{1}(CD)=2^{R+1}m+L$ for some $L \in \{0,1,\ldots,K-4\}$ and $w_{1}(BC)=0$, then $P_{2}$, after the first round, is left with a galaxy graph comprising the edges $AB$, $CD$ and $EF$, where $w_{1}(AB)=2^{R+1}m+4$, $w_{1}(CD)=2^{R+1}m+L$ and $w_{1}(EF)=(K+1)$, and $P_{2}$ wins by Lemma~\ref{lem:H_{1}_winning_galaxy}. 

Since $K \equiv i \bmod 8$ for some $i \in \{0,1,2\}$, we write $K=8(n+1)+i$ for some $n\in\mathbb{N}_{0}$. When $w_{1}(CD)=2^{R+1}m+L$ for $L=(K-3)=(8n+5+i)$, and $w_{1}(BC)=0$, $P_{2}$, after the first round, is left with a galaxy graph consisting of the edges $AB$, $CD$ and $EF$, where $w_{1}(AB)=2^{R+1}m+4$, $w_{1}(CD)=2^{R+1}m+(8n+5+i)$ and $w_{1}(EF)=(8n+9+i)$. Since the triple $\left(2^{R+1}m+4,2^{R+1}m+(8n+5+i),(8n+i+1)\right)$ is balanced, the triple $\left(2^{R+1}m+4,2^{R+1}m+(8n+5+i),(8n+9+i)\right)$ is not, and $P_{2}$ wins by Theorem~\ref{thm:galaxy}. Likewise, when $w_{1}(CD)=2^{R+1}m+L$ and $w_{1}(BC)=0$
\begin{enumerate}
\item with $L=(K-2)$ and $K \equiv i \bmod 8$ for $i \in \{0,1\}$, 
\item or with $L=(K-1)$ and $K \equiv 0 \bmod 8$,
\item or with $L=(K+2)$ and $K \equiv 2 \bmod 8$,
\item or with $L=(K+3)$ and $K \equiv i \bmod 8$ for $i \in \{1,2\}$,
\end{enumerate}
the triple $\left(2^{R+1}m+4,2^{R+1}m+L,(L-4)\right)$ is balanced, and hence, the triple $\left(2^{R+1}m+4,2^{R+1}m+L,(K+1)\right)$ not, allowing $P_{2}$ to win by Theorem~\ref{thm:galaxy}.

When $w_{1}(CD)=2^{R+1}m+L$ with $L=K$ and $w_{1}(BC)=0$, $P_{2}$, after the first round, is left with a galaxy graph comprising the edges $AB$, $CD$ and $EF$, with $w_{1}(AB)=2^{R+1}m+4$, $w_{1}(CD)=2^{R+1}m+K$ and $w_{1}(EF)=(K+1)$. Since the triple $\left(2^{R+1}m+1,2^{R+1}m+K,K+1\right)$ is balanced when $K\equiv i\bmod 8$ for $i\in\{0,2\}$, and the triple $\left(2^{R+1}m+3,2^{R+1}m+K,K+1\right)$ is balanced when $K\equiv 1 \bmod 8$, we conclude that $P_{2}$ wins by Theorem~\ref{thm:galaxy}. Very similar arguments work when $w_{1}(CD)=2^{R+1}m+L$ and $w_{1}(BC)=0$ with
\begin{enumerate}
\item $L=(K-2)$ and $K \equiv 2 \bmod 8$,
\item $L=(K-1)$ and $K \equiv i \bmod 8$ for $i \in \{1,2\}$,
\item $L=(K+2)$ and $K \equiv i \bmod 8$ for $i \in \{0,1\}$,
\item $L=(K+3)$ and $K \equiv 0 \bmod 8$.
\end{enumerate}

\item If $w_{1}(CD)=2^{R+1}n+\ell_{1}$ for some $n < m$ and $\ell_{1} \in \{0,1,\ldots,2^{R+1}-1\}$, the configuration, after the first round, is of the form given by Lemma~\ref{prop:H_{1}_winning_1} with $m_{1}=n$, $m_{2}=m$ and $\ell_{2}=4$. Since $n<m \implies m_{1}\neq m_{2}$, hence $P_{2}$ wins by Lemma~\ref{prop:H_{1}_winning_1}.
\end{enumerate}

\item Suppose $P_{1}$ removes a positive integer weight from $AB$, and a non-negative integer weight from $BC$, in the first round. We divide the analysis into the following cases:
\begin{enumerate}
\item Suppose $w_{1}(AB)=2^{R+1}m+\ell$ for some $\ell \in \{0,1,2,3\}$. Writing $(K+1)=(a_{R}a_{R-1}\ldots a_{1}a_{0})_{2}$, we see that $a_{2}=0$ since $K \equiv i \bmod 8$ for some $i \in \{0,1,2\}$. Moreover, writing $(K+4)=(c_{R}c_{R-1}\ldots c_{1}c_{0})_{2}$, we see that $c_{i}=a_{i}$ for all $i \geqslant 3$, and $c_{2}=1$. On the other hand, writing $\ell=(b_{R}b_{R-1}\ldots b_{1}b_{0})_{2}$, we have $b_{i}=0$ for all $i \geqslant 2$. Consequently, $P_{2}$ wins by the second criterion stated in Theorem~\ref{thm:H_{1}_winning_1}.

\item If $w_{1}(AB)=2^{R+1}n+\ell_{1}$ for some $n < m$ and $\ell_{1} \in \{0,1,\ldots,2^{R+1}-1\}$, the configuration, after the first round, is of the form given by Lemma~\ref{prop:H_{1}_winning_1} with $m_{1}=n$, $m_{2}=m$ and $\ell_{2}=(K+4)$. Since $n<m \implies m_{1}\neq m_{2}$, hence $P_{2}$ wins by Lemma~\ref{prop:H_{1}_winning_1}.
\end{enumerate}

\item Suppose $P_{1}$ removes a positive integer weight from $EF$ in the first round, so that $w_{1}(EF)=k \leqslant K$. If $k \in \{12,13,\ldots,K\}$ and $k \equiv j \bmod 8$ for some $j \in \{0,4,5,6,7\}$, then $P_{2}$ removes weight $(K+4)-(k-5)$ from the edge $CD$ in the second round, leaving $P_{1}$ with
\begin{align}
{}&w_{2}(AB)=2^{f(k)+1}2^{R-f(k)}m+4,\ w_{2}(BC)=1,\nonumber\\
{}&w_{2}(CD)=2^{f(k)+1}2^{R-f(k)}m+(k-5),\ w_{2}(EF)=k,\nonumber
\end{align}
which is of the form given by \eqref{eq:H_{1}_losing_2,4_case_2} with $r=4$, $s=1$ and $w_{2}(EF)=k\leqslant K$, so that $P_{1}$ loses by our induction hypothesis. If $k \in \{1,2,\ldots,K\}$ and $k \equiv j \bmod 8$ for some $j \in \{1,2,3\}$, then $P_{2}$ removes weight $(K+4)-(k+3)$ from the edge $CD$ in the second round, leaving $P_{1}$ with
\begin{align}
{}&w_{2}(AB)=2^{f(k)+1}2^{R-f(k)}m+4,\ w_{2}(BC)=1,\nonumber\\
{}&w_{2}(CD)=2^{f(k)+1}2^{R-f(k)}m+(k+3),\ w_{2}(EF)=k,\nonumber
\end{align}
which is of the form given by \eqref{eq:H_{1}_losing_2,4_case_1} with $r=4$, $s=1$ and $w_{2}(EF)=k\leqslant K$, so that $P_{1}$ loses by our induction hypothesis. Note that this last case covers $k \in \{1,2,3,9,10,11\}$. If $k=4$, the configuration after the first round is of the form stated in Lemma~\ref{lem:H_{1}_winning_2} with $\ell_{1}=4$ and $\ell_{2}=(K+4)$, so that $k \in \{\ell_{1},\ell_{2}\}$ and $P_{2}$ wins by Lemma~\ref{lem:H_{1}_winning_2}. If $k \in \{5,6,7,8\}$, $P_{2}$ removes weight $(K+4)-(k-5)$ from the edge $CD$ in the second round, leaving $P_{1}$ with
\begin{align}
{}&w_{2}(AB)=2^{R+1}m+4,\ w_{2}(BC)=1,\ w_{2}(CD)=2^{R+1}m+(k-5),\ w_{2}(EF)=k,\nonumber
\end{align}
which is of the same form as \eqref{eq:H_{1}_losing_0,1,3} with $r=0$ when $k=5$, of the form given by \eqref{eq:H_{1}_losing_0,1,3} with $r=1$ when $k=6$, of the form given by \eqref{eq:H_{1}_losing_2,4_case_2} with $r=2$ when $k=7$, and of the form given by \eqref{eq:H_{1}_losing_0,1,3} with $r=3$ when $k=8$. By what we have already proved in \S\ref{subsec:proof_eq_H_{1}_losing_0}, \S\ref{subsec:H_{1}_losing_1_proof}, \S\ref{subsec:proof_eq_H_{1}_losing_2_1,2} and \S\ref{subsec:H_{1}_losing_3_proof}, we know that $P_{1}$ loses. If $k=0$, $P_{2}$ wins by Remark~\ref{rem:EF_edgeweight_0}.

\item Finally, suppose $P_{1}$ removes the edge $BC$ in the first round, without disturbing the edge-weights of $AB$ and $CD$. This leaves $P_{2}$ with a galaxy graph consisting of the edges $AB$, $CD$ and $EF$, where $w_{2}(AB)=2^{R+1}m+4$, $w_{2}(CD)=2^{R+1}m+(K+4)$ and $w_{2}(EF)=(K+1)$. Writing $K=8n+i$ for some $n\in\mathbb{N}$ and $i\in\{0,1,2\}$, we see that the triple $\left(2^{R+1}m+4,2^{R+1}m+8n+4+i,8n+i\right)$ is balanaced, so that $P_{2}$ wins by Theorem~\ref{thm:galaxy}.
\end{enumerate}

This completes the proof of our claim that the configuration in \eqref{eq:H_{1}_losing_4_2_inductive} is losing on $H_{1}$.

\subsubsection{When $K \geqslant 12$ and $K \equiv i \bmod 8$ for some $i \in \{3,4,5,6,7\}$}\label{subsubsec:H_{1}_losing_4_3,4,5,6,7} For $m \in \mathbb{N}_{0}$, we focus on the configuration
\begin{align}
w_{0}(AB)=2^{R+1}m+4,\ w_{0}(BC)=1,\ w_{0}(CD)=2^{R+1}m+(K-4),\ w_{0}(EF)=(K+1),\label{eq:H_{1}_losing_4_1_inductive}
\end{align}
where $R=f(K+1)$. The first round of the game played on this configuration unfolds as follows:
\begin{enumerate}
\item Suppose $P_{1}$ removes a positive integer weight from $CD$, and a non-negative integer weight from $BC$, in the first round. The following possibilities can arise:
\begin{enumerate}
\item Let $w_{1}(CD)=2^{R+1}m+L$ for some $L \in \{0,1,\ldots,K-5\}$, with $L \equiv j \bmod 8$ for some $j \in \{4,5,6\}$ and $w_{1}(BC)=1$. In this case, $P_{2}$ removes weight $(K+1)-(L-3)$ from the edge $EF$ in the second round, leaving $P_{1}$ with
\begin{align}
{}&w_{2}(AB)=2^{f(L-3)+1}2^{R-f(L-3)}m+4,\ w_{2}(BC)=1,\nonumber\\
{}&w_{2}(CD)=2^{f(L-3)+1}2^{R-f(L-3)}m+L,\ w_{2}(EF)=(L-3),\nonumber
\end{align}
which is of the form given by \eqref{eq:H_{1}_losing_2,4_case_1} with $r=4$, $s=1$ and $w_{2}(EF)=(L-3)\leqslant (K-8)<K$, so that $P_{1}$ loses by our induction hypothesis.

\item Let $w_{1}(CD)=2^{R+1}m+L$ for some $L \in \{7,8,\ldots,K-5\}$, with $L \equiv j \bmod 8$ for some $j \in \{0,1,2,3,7\}$, and $w_{1}(BC)=1$. In this case, $P_{2}$ removes weight $(K+1)-(L+5)$ from the edge $EF$ in the second round, leaving $P_{1}$ with
\begin{align}
{}&w_{2}(AB)=2^{f(L+5)+1}2^{R-f(L+5)}m+4,\ w_{2}(BC)=1,\nonumber\\
{}&w_{2}(CD)=2^{f(L+5)+1}2^{R-f(L+5)}m+L,\ w_{2}(EF)=(L+5),\nonumber
\end{align}
which is of the form given by \eqref{eq:H_{1}_losing_2,4_case_2} with $r=4$, $s=1$ and $w_{2}(EF)=(L+5) \leqslant K$, so that $P_{1}$ loses by our induction hypothesis.

\item Let $w_{1}(CD)=2^{R+1}m+L$ for $L \in \{0,1,2,3\}$, and $w_{1}(BC)=1$. In this case, $P_{2}$ removes weight $(K+1)-(L+5)$ from the edge $EF$ in the second round, leaving $P_{1}$ with
\begin{equation}
w_{2}(AB)=2^{R+1}m+4,\ w_{2}(BC)=1,\ w_{2}(CD)=2^{R+1}m+L,\ w_{2}(EF)=(L+5),\nonumber
\end{equation}
which is of the form given by \eqref{eq:H_{1}_losing_0,1,3} with $r=0$ if $L=0$, of the form given by \eqref{eq:H_{1}_losing_0,1,3} with $r=1$ if $L=1$, of the form given by \eqref{eq:H_{1}_losing_2,4_case_2} with $r=2$ if $L=2$, and of the form \eqref{eq:H_{1}_losing_0,1,3} with $r=3$ when $L=3$. By what we have already proved in \S\ref{subsec:proof_eq_H_{1}_losing_0}, \S\ref{subsec:H_{1}_losing_1_proof}, \S\ref{subsec:proof_eq_H_{1}_losing_2_1,2} and \S\ref{subsec:H_{1}_losing_3_proof}, we know that $P_{1}$ loses in each of these cases. 

\item If $w_{1}(CD)=2^{R+1}m+L$ for some $L \in \{0,1,\ldots,K-5\}$ and $w_{1}(BC)=0$, the configuration at the end of the first round is that on a galaxy graph consisting of the edges $AB$, $CD$ and $EF$, where $w_{1}(AB)=2^{R+1}m+4$, $w_{1}(CD)=2^{R+1}m+L$ and $w_{1}(EF)=(K+1)$, and $P_{2}$ wins by Lemma~\ref{lem:H_{1}_winning_galaxy}.

\item Finally, if $w_{1}(CD)=2^{R+1}n+\ell_{1}$ for some $n < m$ and $\ell_{1} \in \{0,1,\ldots,2^{R+1}-1\}$, the resulting configuration, after the first round, is of the form given by Lemma~\ref{prop:H_{1}_winning_1} with $m_{1}=n$, $m_{2}=m$ and $\ell_{2}=4$. Since $n<m \implies m_{1}\neq m_{2}$, hence $P_{2}$ wins by Lemma~\ref{prop:H_{1}_winning_1}.
\end{enumerate}

\item Suppose $P_{1}$ removes a positive integer weight from $AB$, and a non-negative integer weight from $BC$, in the first round. We consider the possible subcases:
\begin{enumerate}
\item Suppose $w_{1}(AB)=2^{R+1}m+L$ for some $L \in \{0,1,2,3\}$ and $w_{1}(BC)=1$. If $L \in \{0,1,3\}$, or if $L=2$ and $K \equiv i \bmod 8$ for some $i \in \{3,4,5,7\}$, $P_{2}$ removes weight $(4-L)$ from the edge $EF$ in the second round, leaving $P_{1}$ with
\begin{align}
w_{2}(AB)=2^{R+1}m+L,\ w_{2}(BC)=1,\ w_{2}(CD)=2^{R+1}m+(K-4),\ w_{2}(EF)=(K+L-3),\nonumber
\end{align}
which is of the form \eqref{eq:H_{1}_losing_0,1,3} with $r=0$ when $L=0$, of the form \eqref{eq:H_{1}_losing_0,1,3} with $r=1$ when $L=1$, of the form \eqref{eq:H_{1}_losing_2,4_case_2} with $r=2$ when $L=2$, and of the form \eqref{eq:H_{1}_losing_0,1,3} with $r=3$ when $L=3$. By what we have already proved in \S\ref{subsec:proof_eq_H_{1}_losing_0}, \S\ref{subsec:H_{1}_losing_1_proof}, \S\ref{subsec:proof_eq_H_{1}_losing_2_1,2} and \S\ref{subsec:H_{1}_losing_3_proof}, we know that $P_{1}$ loses. If $L=2$ and $K \equiv 6 \bmod 7$, then $P_{2}$ removes weight $6$ from $EF$ in the second round, leaving $P_{1}$ with
\begin{align}
w_{2}(AB)=2^{R+1}m+2,\ w_{2}(BC)=1,\ w_{2}(CD)=2^{R+1}m+(K-4),\ w_{2}(EF)=(K-5),\nonumber
\end{align}
which is of the form \eqref{eq:H_{1}_losing_2,4_case_1} with $r=2$, and by what we have already proved in \S\ref{subsec:proof_eq_H_{1}_losing_2_1,2}, we conclude that $P_{1}$ loses.

\item If $w_{1}(AB)=2^{R+1}m+L$ for some $L \in \{0,1,2,3\}$ and $w_{1}(BC)=0$, $P_{2}$ is left with a galaxy graph consisting of the edges $AB$, $CD$ and $EF$, where $w_{1}(AB)=2^{R+1}m+L$, $w_{1}(CD)=2^{R+1}m+(K-4)$ and $w_{1}(EF)=(K+1)$ and she wins by Lemma~\ref{lem:H_{1}_winning_galaxy}.

\item Finally, if $w_{1}(AB)=2^{R+1}n+\ell_{1}$ for some $n < m$ and $\ell_{1} \in \{0,1,\ldots,2^{R+1}-1\}$, the resulting configuration, after the first round, is of the form given by Lemma~\ref{prop:H_{1}_winning_1} with $m_{1}=n$, $m_{2}=m$ and $\ell_{2}=(K-4)$. Since $n<m \implies m_{1}\neq m_{2}$, hence $P_{2}$ wins by Lemma~\ref{prop:H_{1}_winning_1}.
\end{enumerate}

\item Suppose $P_{1}$ removes a positive integer weight from $EF$ in the first round, so that $w_{1}(EF)=k \leqslant K$. To begin with, note that if $k\in \{0,1,2,3,4,(K-4)\}$, then the configuration after the first round is of the form stated in Lemma~\ref{prop:H_{1}_winning_1}, with $\ell_{1}=4$, $\ell_{2}=(K-4)$ and $w_{2}(EF)=k$ either satisfying $\min\{\ell_{1},\ell_{2}\}\geqslant k$ or satisfying $k \in \{\ell_{1},\ell_{2}\}$. Consequently, $P_{2}$ wins by Lemma~\ref{prop:H_{1}_winning_1}.

If $k \leqslant K-8$ and $k \equiv j \bmod 8$ for some $j \in \{1,2,3\}$, then $P_{2}$ removes weight $(K-4)-(k+3)$ from $CD$ in the second round, leaving $P_{1}$ with
\begin{align}
{}&w_{2}(AB)=2^{f(k)+1}2^{R-f(k)}m+4,\ w_{2}(BC)=1,\nonumber\\
{}&w_{2}(CD)=2^{f(k)+1}2^{R-f(k)}m+(k+3),\ w_{2}(EF)=k,\nonumber
\end{align}
which is of the form \eqref{eq:H_{1}_losing_2,4_case_1} with $r=4$, $s=1$ and $w_{2}(EF)=k\leqslant K$, so that $P_{1}$ loses by our induction hypothesis. Note that this case covers $k \in \{9,10,11\}$. If $k \geqslant 12$ and $k \equiv j \bmod 8$ for some $j \in \{0,4,5,6,7\}$, $P_{2}$ removes weight $(K-4)-(k-5)$ from $CD$ in the second round, leaving $P_{1}$ with
\begin{align}
{}&w_{2}(AB)=2^{f(k)+1}2^{R-f(k)}m+4,\ w_{2}(BC)=1,\nonumber\\
{}&w_{2}(CD)=2^{f(k)+1}2^{R-f(k)}m+(k-5),\ w_{2}(EF)=k,\nonumber
\end{align}
which is of the form \eqref{eq:H_{1}_losing_2,4_case_2} with $r=4$, $s=1$ and $w_{2}(EF)=k\leqslant K$, and hence, $P_{1}$ loses by our induction hypothesis. Note that this case covers the following:
\begin{enumerate}
\item when $K \equiv 3 \bmod 8$ and $k \in \{K-7,K-6,K-5,K-3\}$,
\item when $K \equiv 4 \bmod 8$ and $k \in \{K-7,K-6,K-5,K\}$,
\item when $K \equiv 5 \bmod 8$ and $k \in \{K-7,K-6,K-5,K-1,K\}$,
\item when $K \equiv 6 \bmod 8$ and $k \in \{K-7,K-6,K-2,K-1,K\}$,
\item when $K \equiv 7 \bmod 8$ and $k \in \{K-7,K-3,K-2,K-1,K\}$.
\end{enumerate}

If $k=(K-6)$ and $K \equiv 7 \bmod 8$, we write $K=(8n+7)$ for some $n \in \mathbb{N}$, so that $w_{1}(AB)=2^{R+1}m+4$, $w_{1}(CD)=2^{R+1}m+(8n+3)$ and $w_{1}(EF)=(8n+1)$, and $P_{2}$ wins by Theorem~\ref{thm:H_{1}_winning_1}. If $k=(K-5)$ and $K \equiv 6 \bmod 8$, we write $K=(8n+6)$ for some $n \in \mathbb{N}$, so that $w_{1}(AB)=2^{R+1}m+4$, $w_{1}(CD)=2^{R+1}m+(8n+2)$ and $w_{1}(EF)=(8n+1)$. If $k=(K-5)$ and $K \equiv 7 \bmod 8$, writing $K=(8n+7)$ for some $n \in \mathbb{N}$, we have $w_{1}(AB)=2^{R+1}m+4$, $w_{1}(CD)=2^{R+1}m+(8n+3)$ and $w_{1}(EF)=(8n+2)$. In each of these cases, $P_{2}$ wins by Theorem~\ref{thm:H_{1}_winning_1}. If $k\in \{K-3,K-2,K-1,K\}$, $P_{2}$ removes weight $(K+1-k)$ from the edge $AB$ in the second round, so that $P_{1}$ is left with
\begin{equation}
w_{2}(AB)=2^{R+1}m+k+3-K,\ w_{2}(BC)=1,\ w_{2}(CD)=2^{R+1}m+(K-4),\ w_{2}(EF)=k,\nonumber
\end{equation}
which is of the form \eqref{eq:H_{1}_losing_0,1,3} with either $r=0$ (when $k=K-3$) or $r=1$ (when $k=K-2$) or $r=3$ (when $k=K$), or of the form \eqref{eq:H_{1}_losing_2,4_case_2} with $r=2$ (when $k=K-2$). By what we have already proved in \S\ref{subsec:proof_eq_H_{1}_losing_0}, \S\ref{subsec:H_{1}_losing_1_proof}, \S\ref{subsec:H_{1}_losing_3_proof} and \S\ref{subsec:proof_eq_H_{1}_losing_2_1,2}, we conclude that $P_{1}$ loses. 

If $k \in \{5,6,7,8\}$, $P_{2}$ removes weight $(K+1-k)$ from the edge $CD$ in the second round, leaving $P_{1}$ with
\begin{equation}
w_{2}(AB)=2^{R+1}m+4,\ w_{2}(BC)=1,\ w_{2}(CD)=2^{R+1}m+(k-5),\ w_{2}(EF)=k,\nonumber
\end{equation}
which is of the form \eqref{eq:H_{1}_losing_0,1,3} with either $r=0$ (when $k=5$) or $r=1$ (when $k=6$) or $r=3$ (when $k=8$), or of the form \eqref{eq:H_{1}_losing_2,4_case_2} with $r=2$ (when $k=7$). By what we have already proved in \S\ref{subsec:proof_eq_H_{1}_losing_0}, \S\ref{subsec:H_{1}_losing_1_proof}, \S\ref{subsec:H_{1}_losing_3_proof} and \S\ref{subsec:proof_eq_H_{1}_losing_2_1,2}, we conclude that $P_{1}$ loses. Finally, if $k=0$, $P_{1}$ loses by Remark~\ref{rem:EF_edgeweight_0}.

\item Finally, if $P_{1}$ removes the edge $BC$ in the first round, without disturbing the edge-weights of $AB$ and $CD$, $P_{2}$ is left with a galaxy graph consisting of the edges $AB$, $CD$ and $EF$, with $w_{1}(AB)=2^{R+1}m+4$, $w_{1}(CD)=2^{R+1}m+(K-4)$ and $w_{1}(EF)=(K+1)$, and $P_{2}$ wins by Lemma~\ref{lem:H_{1}_winning_galaxy}.
\end{enumerate}

This completes the proof of our claim that \eqref{eq:H_{1}_losing_4_1_inductive} is losing on $H_{1}$. Together with our proof that the configuration in \eqref{eq:H_{1}_losing_4_2_inductive} is losing on $H_{1}$, this completes our inductive proof of the claim that any configuration on $H_{1}$ that is either of the form given by \eqref{eq:H_{1}_losing_2,4_case_1} or of the form given by \eqref{eq:H_{1}_losing_2,4_case_2} with $r=4$ and $s=1$ is losing.

\subsection{Proof that a configuration on $H_{1}$ satisfying either \eqref{eq:H_{1}_losing_2,4_case_1} or \eqref{eq:H_{1}_losing_2,4_case_2}, with $r=4$ and $s=2$, is losing}\label{subsec:H_{1}_losing_4_3,4_proof}
The base case corresponding to \eqref{eq:H_{1}_losing_2,4_case_2}, with $r=4$ and $s=2$, is obtained by setting $k=12$, yielding \begin{equation}
w_{0}(AB)=16m+4,\ w_{0}(BC)=2,\ w_{0}(CD)=16m+6,\ w_{0}(EF)=12.\label{eq:H_{1}_losing_4_3_base_case}
\end{equation}
\begin{enumerate}
\item Suppose $P_{1}$ removes a positive integer weight from $CD$, and a non-negative integer weight from $BC$, in the first round of the game played on the initial configuration in \eqref{eq:H_{1}_losing_4_3_base_case}:
\begin{enumerate}
\item If $w_{1}(CD)=16m+\ell$ for some $\ell \in \{0,1,2,3\}$ and $w_{1}(BC)=t \in \{1,2\}$, $P_{2}$ removes weight $(8-\ell-t)$ from the edge $EF$ in the second round, leaving $P_{1}$ with
\begin{equation}
w_{2}(AB)=16m+4,\ w_{2}(BC)=t,\ w_{2}(CD)=16m+\ell,\ w_{2}(EF)=\ell+t+4,\nonumber
\end{equation}
which is of the form \eqref{eq:H_{1}_losing_0,1,3} with either $r=0$ (when $\ell=0$) or $r=1$ (when $\ell=1$) or $r=3$ (when $\ell=3$), or of the form \eqref{eq:H_{1}_losing_2,4_case_2} with $r=2$ and $s=1$ (when $\ell=2$ and $t=1$), or of the form \eqref{eq:H_{1}_losing_2,4_case_3} with $r=2$ and $s=2$ (when $\ell=2$ and $t=2$). By what we have already proved in \S\ref{subsec:proof_eq_H_{1}_losing_0}, \S\ref{subsec:H_{1}_losing_1_proof}, \S\ref{subsec:H_{1}_losing_3_proof}, \S\ref{subsec:proof_eq_H_{1}_losing_2_1,2} and \S\ref{subsec:proof_eq_H_{1}_losing_2_3}, we conclude that $P_{1}$ loses.

\item If $w_{1}(CD)=2^{R+1}m+\ell$ for some $\ell \in \{4,5\}$ and $w_{1}(BC)=t \in \{1,2\}$, then $P_{2}$ removes weight $(16-\ell-t)$ from the edge $EF$ in the second round, leaving $P_{1}$ with
\begin{align}
{}&w_{2}(AB)=4(4m+1),\ w_{2}(BC)=t,\ w_{2}(CD)=4(4m+1)+\ell-4,\ w_{2}(EF)=t+\ell-4,\nonumber
\end{align}
which is of the same form as \eqref{eq:H_{1}_losing_0,1,3} with $r=0$, and by what we have already proved in \S\ref{subsec:proof_eq_H_{1}_losing_0}, we know that $P_{1}$ loses.

\item If $w_{1}(CD)=16m+\ell$ for some $\ell \in \{0,1,\ldots,5\}$ and $w_{1}(BC)=0$, $P_{2}$ is left with a galaxy graph consisting of the edges $AB$, $CD$ and $EF$, where $w_{1}(AB)=16m+4$, $w_{1}(CD)=16m+\ell$ and $w_{1}(EF)=12$, and $P_{2}$ wins by Lemma~\ref{lem:H_{1}_winning_galaxy}.

\item Finally, if $w_{1}(CD)=16n+\ell_{1}$ for some $n < m$ and $\ell_{1} \in \{0,1,\ldots,15\}$, the configuration after the first round is of the form given by Lemma~\ref{prop:H_{1}_winning_1} with $m_{1}=n$, $m_{2}=m$ and $\ell_{2}=4$. Since $n<m \implies m_{1}\neq m_{2}$, $P_{2}$ wins by Lemma~\ref{prop:H_{1}_winning_1}.
\end{enumerate}

\item Suppose $P_{1}$ removes a positive integer weight from $AB$ and a non-negative integer weight from $BC$ in the first round. The following subcases are possible:
\begin{enumerate}
\item Suppose $w_{1}(AB)=16m+\ell$ for some $\ell \in \{0,1,2,3\}$ and $w_{1}(BC)=t\in\{1,2\}$. If $\ell \in \{0,1,3\}$, or if $\ell=2$ and $t=2$, $P_{2}$ removes $(6-\ell-t)$ from $EF$ in the second round, leaving $P_{1}$ with
\begin{align}
w_{2}(AB)=16m+\ell,\ w_{2}(BC)=t,\ w_{2}(CD)=16m+6,\ w_{2}(EF)=\ell+t+6,\nonumber
\end{align}
which is of the form \eqref{eq:H_{1}_losing_0,1,3} with either $r=0$ (when $\ell=0$) or $r=1$ (when $\ell=1$) or $r=3$ (when $\ell=3$), or of the form \eqref{eq:H_{1}_losing_2,4_case_3} with $r=2$ and $s=2$ (when $\ell=2$ and $t=2$). If $\ell=2$ and $t=1$, $P_{2}$ removes weight $7$ from the edge $EF$ in the first round, leaving $P_{1}$ with
\begin{align}
w_{2}(AB)=16m+2,\ w_{2}(BC)=1,\ w_{2}(CD)=16m+6,\ w_{2}(EF)=5,\nonumber
\end{align}
which is of the form \eqref{eq:H_{1}_losing_2,4_case_1} with $r=2$. In each of the cases described above, by what we have already proved in \S\ref{subsec:proof_eq_H_{1}_losing_0}, \S\ref{subsec:H_{1}_losing_1_proof}, \S\ref{subsec:H_{1}_losing_3_proof}, \S\ref{subsec:proof_eq_H_{1}_losing_2_1,2} and \S\ref{subsec:proof_eq_H_{1}_losing_2_3}, we conclude that $P_{1}$ loses.

\item If $w_{1}(AB)=16m+\ell$ for some $\ell \in \{0,1,2,3\}$ and $w_{1}(BC)=0$, $P_{2}$, after the first round, is left with a galaxy graph consisting of the edges $AB$, $CD$ and $EF$, where $w_{1}(AB)=16m+\ell$, $w_{1}(CD)=16m+6$ and $w_{1}(EF)=12$, and $P_{2}$ wins by Lemma~\ref{lem:H_{1}_winning_galaxy}.

\item Finally, if $w_{1}(AB)=16n+\ell_{1}$ for some $n < m$ and $\ell_{1} \in \{0,1,\ldots,15\}$, the configuration after the first round is of the form given by Lemma~\ref{prop:H_{1}_winning_1} with $m_{1}=n$, $m_{2}=m$ and $\ell_{2}=6$. Since $n<m \implies m_{1}\neq m_{2}$, hence $P_{2}$ wins by Lemma~\ref{prop:H_{1}_winning_1}.
\end{enumerate}

\item Suppose $P_{1}$ removes a positive integer weight from $EF$ in the first round, so that $w_{1}(EF)=k \leqslant 11$. If $k \in \{8,9,10,11\}$, then $P_{2}$ removes weight $(12-k)$ from $AB$ in the second round, leaving $P_{1}$ with 
\begin{equation}
w_{2}(AB)=16m+(k-8),\ w_{2}(BC)=2,\ w_{2}(CD)=16m+6,\ w_{2}(EF)=k,\nonumber
\end{equation}
which is of the form given by \eqref{eq:H_{1}_losing_0,1,3} with either $r=0$ (when $k=8$) or $r=1$ (when $k=9$) or $r=3$ (when $k=11$), or of the form given by \eqref{eq:H_{1}_losing_2,4_case_3} with $r=2$ and $s=2$ (when $k=10$). If $k=7$, then $P_{1}$ removes weight $4$ from $AB$ and weight $1$ from $BC$ in the first round, leaving $P_{1}$ with
\begin{equation}
w_{2}(AB)=16m,\ w_{2}(BC)=1,\ w_{2}(CD)=16m+6,\ w_{2}(EF)=7,\nonumber
\end{equation}
which is of the form \eqref{eq:H_{1}_losing_0,1,3} with $r=2$. By what we have already proved in \S\ref{subsec:proof_eq_H_{1}_losing_0}, \S\ref{subsec:H_{1}_losing_1_proof}, \S\ref{subsec:H_{1}_losing_3_proof} and \S\ref{subsec:proof_eq_H_{1}_losing_2_3}, we conclude that $P_{1}$ loses. If $k \in \{1,2,3,4,6\}$, the configuration after the first round is of the form stated in Lemma~\ref{prop:H_{1}_winning_1}, with $\ell_{1}=4$, $\ell_{2}=6$, and $w_{1}(EF)=k$ satisfying either $\min\{\ell_{1},\ell_{2}\}\geqslant k$ or $k \in \{\ell_{1},\ell_{2}\}$, so that $P_{2}$ wins by Lemma~\ref{prop:H_{1}_winning_1}. Finally, if $k=5$, then $P_{2}$ removes weight $6$ from the edge $CD$ and weight $1$ from the edge $BC$ in the second round, leaving $P_{1}$ with
\begin{equation}
w_{2}(AB)=16m+4,\ w_{2}(BC)=1,\ w_{2}(CD)=16m,\ w_{2}(EF)=5,\nonumber
\end{equation}
which is of the form \eqref{eq:H_{1}_losing_0,1,3} with $r=0$, and by what we have already proved in \S\ref{subsec:proof_eq_H_{1}_losing_0}, we know that $P_{1}$ loses. If $k=0$, $P_{2}$ wins by Remark~\ref{rem:EF_edgeweight_0}.

\item Suppose $P_{1}$ removes a positive integer weight from $BC$ in the first round, without disturbing the edge-weights of $AB$ and $CD$, so that $w_{1}(BC)=t \in \{0,1\}$. If $t=1$, $P_{2}$ removes weight $9$ from the edge $EF$ in the second round, leaving $P_{1}$ with
\begin{equation}
w_{2}(AB)=4(4m+1),\ w_{2}(BC)=1,\ w_{2}(CD)=4(4m+1)+2,\ w_{2}(EF)=3,\nonumber
\end{equation}
which is of the form \eqref{eq:H_{1}_losing_0,1,3} with $r=0$, and by what we have already proved in \S\ref{subsec:proof_eq_H_{1}_losing_0}, we conclude that $P_{1}$ loses. If $t=0$, $P_{2}$ is left, after the first round, with a galaxy graph consisting of the edges $AB$, $CD$ and $EF$, where $w_{1}(AB)=16m+4$, $w_{1}(CD)=16m+6$ and $w_{1}(EF)=12$, and $P_{2}$ wins by Lemma~\ref{lem:H_{1}_winning_galaxy}.
\end{enumerate}
This concludes the proof of our claim that the configuration in \eqref{eq:H_{1}_losing_4_3_base_case} is losing. 

The base cases corresponding to \eqref{eq:H_{1}_losing_2,4_case_1} with $r=4$ and $s=2$ are obtained by setting $k=2$ and $k=3$, which yield, respectively, the configurations
\begin{align}
{}&w_{0}(AB)=4m+4=4(m+1),\ w_{0}(BC)=2,\ w_{0}(CD)=4m+4=4(m+1),\ w_{0}(EF)=2,\nonumber\\
{}&w_{0}(AB)=4m+4=4(m+1),\ w_{0}(BC)=2,\ w_{0}(CD)=4m+5=4(m+1)+1,\ w_{0}(EF)=3,\nonumber
\end{align}
each of which is of the form \eqref{eq:H_{1}_losing_0,1,3} with $r=0$, and by what we have proved in \S\ref{subsec:proof_eq_H_{1}_losing_0}, we know that $P_{1}$ loses. 

Suppose, for some $K \in \mathbb{N}$, we have proved that any configuration that is either of the form \eqref{eq:H_{1}_losing_2,4_case_1} or of the form \eqref{eq:H_{1}_losing_2,4_case_2} with $r=4$ and $s=2$, and with $w_{0}(EF)=k\leqslant K$, is losing.

\subsubsection{When $K \geqslant 9$ and $K \equiv i \bmod 8$ for some $i \in \{1,2\}$}\label{subsubsec:H_{1}_losing_4_4}
For $m \in \mathbb{N}_{0}$, we consider the configuration
\begin{align}
w_{0}(AB)=2^{R+1}m+4,\ w_{0}(BC)=2,\ w_{0}(CD)=2^{R+1}m+(K+3),\ w_{0}(EF)=(K+1),\label{eq:H_{1}_losing_4_4_inductive}
\end{align}
where $R=f(K+1)$. We consider the first round of the game played on this initial configuration:
\begin{enumerate}
\item Suppose $P_{1}$ removes a positive integer weight from $CD$ and a non-negative integer weight from $BC$ in the first round. This can be divided into the following subcases:
\begin{enumerate}
\item Suppose $w_{1}(CD)=2^{R+1}m+(K+2)$. If $K\equiv 1 \bmod 8$, so that $K=8n+1$ for some $n \in \mathbb{N}$, $P_{2}$ removes weight $3$ from $AB$, and weight $w_{1}(BC)$ from $BC$, in the second round, leaving $P_{1}$ with the galaxy graph consisting of $AB$, $CD$ and $EF$, where $w_{2}(AB)=2^{R+1}m+1$, $w_{2}(CD)=2^{R+1}m+(8n+3)$ and $w_{2}(EF)=(8n+2)$, and $P_{1}$ loses by Theorem~\ref{thm:galaxy}. If $K \equiv 2 \bmod 8$ and $w_{1}(BC)=1$, $P_{2}$ removes weight $2$ from $EF$ in the second round, leaving $P_{1}$ with
\begin{align}
w_{2}(AB)=2^{R+1}m+4,\ w_{2}(BC)=1,\ w_{2}(CD)=2^{R+1}m+K+2,\ w_{2}(EF)=K-1,\nonumber
\end{align}
which is of the form \eqref{eq:H_{1}_losing_2,4_case_1} with $r=4$ and $s=1$, and by what we have proved in \S\ref{subsec:H_{1}_losing_4_1,2_proof}, we know that $P_{1}$ loses. If $w_{1}(BC)=2$, then $P_{2}$ removes weight $1$ from $EF$ in the second round, so that 
\begin{equation}
w_{2}(AB)=2^{R+1}m+4,\ w_{2}(BC)=2,\ w_{2}(CD)=2^{R+1}m+(K+2),\ w_{2}(EF)=K,\nonumber
\end{equation}
which is of the form \eqref{eq:H_{1}_losing_2,4_case_1} with $r=4$ and $s=2$, so that $P_{1}$ loses by our induction hypothesis.

Suppose $w_{1}(CD)=2^{R+1}m+(K+2)$ with $K \equiv 2 \bmod 8$, and $w_{1}(BC)=0$. Writing $K=8n+2$ for some $n \in \mathbb{N}$, $P_{2}$, after the first round, is left with a galaxy graph consisting of the edges $AB$, $CD$ and $EF$, where $w_{1}(AB)=2^{R+1}m+4$, $w_{1}(CD)=2^{R+1}m+(8n+4)$ and $w_{1}(EF)=(8n+3)$, and she wins by Theorem~\ref{thm:galaxy}.

\item Suppose $w_{1}(CD)=2^{R+1}m+(K+1)$. Then, irrespective of what $w_{1}(BC)$ is, the configuration after the first round is of the form mentioned in Lemma~\ref{prop:H_{1}_winning_1}, with $\ell_{1}=4$, $\ell_{2}=(K+1)$ and $w_{1}(EF)=(K+1)$. Since $w_{1}(EF)\in\{\ell_{1},\ell_{2}\}$, $P_{2}$ wins by Lemma~\ref{prop:H_{1}_winning_1}.

\item If $w_{1}(CD)=2^{R+1}m+K$ and $w_{1}(BC)=t \in \{1,2\}$, $P_{2}$ removes weight $4$ from $AB$ and weight $\{w_{1}(BC)-1\}$ from $BC$ in the second round, leaving $P_{1}$ with
\begin{equation}
w_{2}(AB)=2^{R+1}m,\ w_{2}(BC)=1,\ w_{2}(CD)=2^{R+1}m+K,\ w_{2}(EF)=(K+1),\nonumber
\end{equation}
which is of the form given by \eqref{eq:H_{1}_losing_0,1,3} with $r=0$, and by what we have already proved in \S\ref{subsec:proof_eq_H_{1}_losing_0}, we know that $P_{1}$ loses. If $w_{1}(CD)=2^{R+1}m+K$ and $w_{1}(BC)=0$, $P_{2}$, after the first round, is left with a galaxy graph consisting of the edges $AB$, $CD$ and $EF$, with $w_{1}(AB)=2^{R+1}m+4$, $w_{1}(CD)=2^{R+1}m+K$ and $w_{1}(EF)=(K+1)$. Writing $K=8n+i$ for some $n\in\mathbb{N}$ and $i\in\{1,2\}$, we see that the triple $(2^{R+1}m+3,2^{R+1}m+K,K+1)$ is balanced when $i=1$, and the triple $(2^{R+1}m+1,2^{R+1}m+K,K+1)$ is balanced when $i=2$, so that $P_{2}$ wins by Theorem~\ref{thm:galaxy}. 

\item If $w_{1}(CD)=2^{R+1}m+L$ for some $L \in \{4,5,\ldots,K-6\}$, with $L \equiv j \bmod 8$ for some $j \in \{0,1,2,3,6,7\}$, and $w_{1}(BC)=2$, $P_{2}$ removes weight $(K+1)-(L+6)$ from the edge $EF$ in the second round, leaving $P_{1}$ with
\begin{align}
{}&w_{2}(AB)=2^{f(L+6)+1}2^{R-f(L+6)}m+4,\ w_{2}(BC)=2,\nonumber\\
{}&w_{2}(CD)=2^{f(L+6)+1}2^{R-f(L+6)}m+L,\ w_{2}(EF)=(L+6),\nonumber
\end{align}
which is of the form \eqref{eq:H_{1}_losing_2,4_case_2} with $r=4$ and $s=2$, and since $w_{2}(EF)=(L+6)\leqslant K$, $P_{1}$ loses by our induction hypothesis.

\item If $w_{1}(CD)=2^{R+1}m+L$ for some $L \in \{4,5,\ldots,K+2\}$, with $L \equiv j \bmod 8$ for some $j \in \{4,5\}$, and $w_{1}(BC)=2$, $P_{2}$ removes weight $(K+1)-(L-2)$ from the edge $EF$ in the second round, leaving $P_{1}$ with
\begin{align}
{}&w_{2}(AB)=2^{f(L-2)+1}2^{R-f(L-2)}m+4,\ w_{2}(BC)=2,\nonumber\\
{}&w_{2}(CD)=2^{f(L-2)+1}2^{R-f(L-2)}m+L,\ w_{2}(EF)=(L-2),\nonumber
\end{align}
which is of the form \eqref{eq:H_{1}_losing_2,4_case_1} with $r=4$ and $s=2$, and since $w_{2}(EF)=(L-2)\leqslant K$, $P_{1}$ loses by our induction hypothesis. This case covers $L=(K-5)$ since $K \equiv i \bmod 8$ for $i \in \{1,2\}$.

\item If $w_{1}(CD)=2^{R+1}m+L$ for some $L \in \{K-4,K-3,K-2,K-1\}$, and $w_{1}(BC)=2$, $P_{2}$ removes weight $(L+5-K)$ from the edge $AB$ in the second round, leaving $P_{1}$ with
\begin{align}
w_{2}(AB)=2^{R+1}m+(K-L-1),\ w_{2}(BC)=2,\ w_{2}(CD)=2^{R+1}m+L,\ w_{2}(EF)=(K+1),\nonumber
\end{align}
which is of the form given by \eqref{eq:H_{1}_losing_0,1,3} with either $r=0$ (when $L=K-1$) or $r=1$ (when $L=K-2$) or $r=3$ (when $L=K-4$), or of the form \eqref{eq:H_{1}_losing_2,4_case_3} with $r=2$ and $s=2$ (when $L=K-3$). By what we have already proved in \S\ref{subsec:proof_eq_H_{1}_losing_0}, \S\ref{subsec:H_{1}_losing_1_proof}, \S\ref{subsec:H_{1}_losing_3_proof} and \S\ref{subsec:proof_eq_H_{1}_losing_2_3}, we know that $P_{1}$ loses. 

\item If $w_{1}(CD)=2^{R+1}m+L$ for some $L \in \{4,5,\ldots,K-5\}$ with $L \equiv j \bmod 8$ for some $j \in \{0,1,2,3,7\}$, and $w_{1}(BC)=1$, $P_{2}$ removes weight $(K+1)-(L+5)$ from $EF$ in the second round, leaving $P_{1}$ with
\begin{align}
{}&w_{2}(AB)=2^{f(L+5)+1}2^{R-f(L+5)}m+4,\ w_{2}(BC)=1,\nonumber\\
{}&w_{2}(CD)=2^{f(L+5)+1}2^{R-f(L+5)}m+L,\ w_{2}(EF)=(L+5),\nonumber
\end{align}
which is of the form \eqref{eq:H_{1}_losing_2,4_case_2} with $r=4$ and $s=1$, and by what we have already proved in \S\ref{subsec:H_{1}_losing_4_1,2_proof}, we know that $P_{1}$ loses.

\item If $w_{1}(CD)=2^{R+1}m+L$ for some $L \in \{4,5,\ldots,K+2\}$, with $L \equiv j \bmod 8$ for some $j \in \{4,5,6\}$, and $w_{1}(BC)=1$, $P_{2}$ removes weight $(K+1)-(L-3)$ from the edge $EF$ in the second round, leaving $P_{1}$ with
\begin{align}
{}&w_{2}(AB)=2^{f(L-3)+1}2^{R-f(L-3)}m+4,\ w_{2}(BC)=1,\nonumber\\
{}&w_{2}(CD)=2^{f(L-3)+1}2^{R-f(L-3)}m+L,\ w_{2}(EF)=(L-3),\nonumber
\end{align}
which is of the form \eqref{eq:H_{1}_losing_2,4_case_1} with $r=4$ and $s=1$, and by what we have already proved in \S\ref{subsec:H_{1}_losing_4_1,2_proof}, we know that $P_{1}$ loses. This case covers $L=(K-4)$, since $K \equiv i \bmod 8$ for some $i \in \{1,2\}$.

\item If $w_{1}(CD)=2^{R+1}m+L$ for $L \in \{K-3,K-2,K-1\}$, and $w_{1}(BC)=1$, $P_{2}$ removes weight $(L+4-K)$ from the edge $AB$ in the second round, leaving $P_{1}$ with
\begin{align}
w_{2}(AB)=2^{R+1}m+(K-L),\ w_{2}(BC)=1,\ w_{2}(CD)=2^{R+1}m+L,\ w_{2}(EF)=(K+1),\nonumber
\end{align}
which is of the form \eqref{eq:H_{1}_losing_0,1,3} with either $r=3$ (when $L=K-3$) or $r=1$ (when $L=K-1$), or of the form \eqref{eq:H_{1}_losing_2,4_case_2} with $r=2$ and $s=1$ (when $L=K-2$). By what we have already proved in \S\ref{subsec:H_{1}_losing_1_proof}, \S\ref{subsec:H_{1}_losing_3_proof} and \S\ref{subsec:proof_eq_H_{1}_losing_2_1,2}, we conclude that $P_{1}$ loses.

\item If $w_{1}(CD)=2^{R+1}m+L$ for $L \in \{0,1,2,3\}$, and $w_{1}(BC)=t\in\{1,2\}$, $P_{2}$ removes weight $(K+1)-(4+t+L)$ from the edge $EF$ in the second round, leaving $P_{1}$ with
\begin{equation}
w_{2}(AB)=2^{R+1}m+4,\ w_{2}(BC)=t,\ w_{2}(CD)=2^{R+1}m+L,\ w_{2}(EF)=(4+t+L),\nonumber
\end{equation}
which is of the form \eqref{eq:H_{1}_losing_0,1,3} with either $r=3$ (when $L=3$) or $r=1$ (when $L=1$) or $r=0$ (when $L=0$), or of the form \eqref{eq:H_{1}_losing_2,4_case_2} with $r=2$ and $s=1$ (when $L=2$ and $t=1$), or of the form \eqref{eq:H_{1}_losing_2,4_case_3} with $r=2$ and $s=2$ (when $L=2$ and $t=2$). By what we have already proved in \S\ref{subsec:proof_eq_H_{1}_losing_0}, \S\ref{subsec:H_{1}_losing_1_proof}, \S\ref{subsec:H_{1}_losing_3_proof}, \S\ref{subsec:proof_eq_H_{1}_losing_2_1,2} and \S\ref{subsec:proof_eq_H_{1}_losing_2_3}, we conclude that $P_{1}$ loses.

\item If $w_{1}(CD)=2^{R+1}m+L$ for $L \in \{0,1,\ldots,K-1\}$, and $w_{1}(BC)=0$, $P_{2}$, after the first round, is left with a galaxy graph consisting of the edges $AB$, $CD$ and $EF$, where $w_{1}(AB)=2^{R+1}m+4$, $w_{1}(CD)=2^{R+1}m+L$ and $w_{1}(EF)=(K+1)$. When $L\leqslant(K-4)$, $P_{2}$ wins by Lemma~\ref{lem:H_{1}_winning_galaxy}. When $L=(K-3)$, writing $K=8(n+1)+i$ for some $n\in\mathbb{N}_{0}$ and $i\in\{1,2\}$, we see that the triple $\left(2^{R+1}m+4,2^{R+1}m+8n+5+i,8n+1+i\right)$ is balanced, and hence, $P_{2}$ wins by Theorem~\ref{thm:galaxy}. When $L=(K-2)$ and $K \equiv 1 \bmod 8$, writing $K=8(n+1)+1$ for some $n \in \mathbb{N}_{0}$, we see that the triple $\left(2^{R+1}m+4,2^{R+1}m+8n+7,8n+3\right)$ is balanced, whereas if $K\equiv 2\bmod 8$, writing $K=8(n+1)+2$ for some $n\in\mathbb{N}_{0}$, we see that the triple $\left(2^{R+1}m+3,2^{R+1}m+8n+8,8n+11\right)$ is balanced -- consequently, $P_{2}$ wins by Theorem~\ref{thm:galaxy}. When $L=(K-1)$, writing $K=8(n+1)+i$ for some $n\in\mathbb{N}_{0}$ and $i\in\{1,2\}$, we see that the triple $\left(2^{R+1}m+2,2^{R+1}m+8n+7+i,8n+9+i\right)$ is balanced, so that $P_{2}$ wins by Theorem~\ref{thm:galaxy}.

\item Finally, if $w_{1}(CD)=2^{R+1}n+\ell_{1}$ for some $n<m$ and $\ell_{1}\in\{0,1,\ldots,2^{R+1}-1\}$, the configuration after the first round is of the form given by Lemma~\ref{prop:H_{1}_winning_1} with $m_{1}=n$, $m_{2}=m$ and $\ell_{2}=4$. Since $n<m \implies m_{1}\neq m_{2}$, $P_{2}$ wins by Lemma~\ref{prop:H_{1}_winning_1}.
\end{enumerate}

\item Suppose $P_{1}$ removes a positive integer weight from $AB$ and a non-negative integer weight from $BC$ in the first round. We have the following possibilities:
\begin{enumerate}
\item If $w_{1}(AB)=2^{R+1}m+\ell$ for $\ell \in \{0,1,2,3\}$ and $w_{1}(BC)=t\in \{1,2\}$, $P_{2}$ removes weight $(2+\ell+t)$ from $CD$ in the second round, leaving $P_{1}$ with 
\begin{align}
{}&w_{2}(AB)=2^{R+1}m+\ell,\ w_{2}(BC)=t,\nonumber\\
{}&w_{2}(CD)=2^{R+1}m+(K+1)-\ell-t,\ w_{2}(EF)=(K+1),\nonumber
\end{align}
which is of the form \eqref{eq:H_{1}_losing_0,1,3} with either $r=0$ (when $\ell=0$) or $r=1$ (when $\ell=1$) or $r=3$ (when $\ell=3$), or of the form \eqref{eq:H_{1}_losing_2,4_case_2} with $r=2$ and $s=1$ (when $\ell=2$ and $t=1$), or of the form \eqref{eq:H_{1}_losing_2,4_case_3} with $r=2$ and $s=2$ (when $\ell=2$ and $t=2$). 

\item If $w_{1}(AB)=2^{R+1}m+\ell$ for some $\ell \in \{0,1,2,3\}$ and $w_{1}(BC)=0$, $P_{2}$, at the end of the first round, is left with a galaxy graph consisting of the edges $AB$, $CD$ and $EF$, with $w_{1}(AB)=2^{R+1}m+\ell$, $w_{1}(CD)=2^{R+1}m+(K+3)$ and $w_{1}(EF)=(K+1)$. Writing $K=8n+i$ for $i \in \{1,2\}$ and some $n \in \mathbb{N}$, the resulting triple is $\left(2^{R+1}m+\ell,2^{R+1}m+8n+i+3,8n+i+1\right)$, and this is winning by Part~\eqref{A2} of Theorem~\ref{thm:H_{1}_winning_1}. 

\item Finally, if $w_{1}(AB)=2^{R+1}n+\ell_{1}$ for some $n<m$ and $\ell_{1}\in\{0,1,\ldots,2^{R+1}-1\}$, the configuration after the first round is of the form given by Lemma~\ref{prop:H_{1}_winning_1} with $m_{1}=n$, $m_{2}=m$ and $\ell_{2}=K+3$. Since $n<m \implies m_{1}\neq m_{2}$, $P_{2}$ wins by Lemma~\ref{prop:H_{1}_winning_1}.
\end{enumerate}

\item Suppose $P_{1}$ removes a positive integer weight from $EF$ in the first round, so that $w_{1}(EF)=k \leqslant K$. If $k \in \{12,13,\ldots,K\}$ and $k \equiv j \bmod 8$ for some $j \in \{0,1,4,5,6,7\}$, then $P_{2}$ removes weight $(K+9-k)$ from the edge $CD$ in the second round, leaving $P_{1}$ with
\begin{align}
{}&w_{2}(AB)=2^{f(k)+1}2^{R-f(k)}m+4,\ w_{2}(BC)=2,\nonumber\\
{}&w_{2}(CD)=2^{f(k)+1}2^{R-f(k)}m+(k-6),\ w_{2}(EF)=k,\nonumber
\end{align}
which is of the form \eqref{eq:H_{1}_losing_2,4_case_2} with $r=4$, $s=2$ and $k \leqslant K$, so that $P_{1}$ loses by our induction hypothesis. If $k \in \{10,11,\ldots,K\}$ and $k \equiv j \bmod 8$ for some $j \in \{2,3\}$, then $P_{2}$ removes weight $(K+1-k)$ from the edge $CD$ in the second round, leaving $P_{1}$ with
\begin{align}
{}&w_{2}(AB)=2^{f(k)+1}2^{R-f(k)}m+4,\ w_{2}(BC)=2,\nonumber\\
{}&w_{2}(CD)=2^{f(k)+1}2^{R-f(k)}m+(k+2),\ w_{2}(EF)=k,\nonumber
\end{align}
which is of the form \eqref{eq:H_{1}_losing_2,4_case_1} with $r=4$, $s=2$ and $k \leqslant K$, so that $P_{1}$ loses by our induction hypothesis.  

If $k \in \{1,2,3,4\}$, the configuration after the first round is of the form mentioned in Lemma~\ref{prop:H_{1}_winning_1}, with $\ell_{1}=4$, $\ell_{2}=(K+3)$ and $w_{1}(EF)=k \leqslant \min\{\ell_{1},\ell_{2}\}$. Thus, $P_{2}$ wins by Lemma~\ref{prop:H_{1}_winning_1}. If $k=5$, then the configuration after the first round is of the form mentioned in Lemma~\ref{lem:H_{1}_winning_2}, with $\ell_{1}=4$, $\ell_{2}=(K+3)$, $w_{1}(EF)=k=5>\min\{\ell_{1},\ell_{2}\}$ and $w_{1}(BC)=2>k-\min\{\ell_{1},\ell_{2}\}$, so that $P_{2}$ wins by Lemma~\ref{lem:H_{1}_winning_2}. If $k \in \{6,7,8,9\}$, then $P_{2}$ removes weight $(K+9-k)$ from the edge $CD$ in the second round, leaving $P_{1}$ with
\begin{align}
w_{2}(AB)=2^{R+1}m+4,\ w_{2}(BC)=2,\ w_{2}(CD)=2^{R+1}m+(k-6),\ w_{2}(EF)=k,\nonumber
\end{align}
which is of the form \eqref{eq:H_{1}_losing_0,1,3} with either $r=0$ (when $k=6$) or $r=1$ (when $k=7$) or $r=3$ (when $k=9$), or of the form \eqref{eq:H_{1}_losing_2,4_case_3} with $r=2$ and $s=2$ (when $k=8$). By what we have already proved in \S\ref{subsec:proof_eq_H_{1}_losing_0}, \S\ref{subsec:H_{1}_losing_1_proof}, \S\ref{subsec:H_{1}_losing_3_proof} and \S\ref{subsec:proof_eq_H_{1}_losing_2_3}, we know that $P_{1}$ loses. If $k=0$, $P_{2}$ wins by Remark~\ref{rem:EF_edgeweight_0}.

\item Suppose $P_{1}$ removes a positive integer weight from the edge $BC$ in the first round, without disturbing the edge-weights of $AB$ and $CD$, so that $w_{1}(BC)=t \in \{0,1\}$. If $t=1$, $P_{2}$ removes weight $1$ from $EF$ in the second round, leaving $P_{1}$ with
\begin{equation}
w_{2}(AB)=2^{R+1}m+4,\ w_{2}(BC)=1,\ w_{2}(CD)=2^{R+1}m+(K+3),\ w_{2}(EF)=K,\nonumber
\end{equation}
which is of the form \eqref{eq:H_{1}_losing_2,4_case_1} with $r=4$ and $s=1$, and by what we have already proved in \S\ref{subsec:H_{1}_losing_4_1,2_proof}, we conclude that $P_{1}$ loses. If $t=0$, $P_{2}$ is left with a galaxy graph after the first round, consisting of the edges $AB$, $CD$ and $EF$, with $w_{1}(AB)=2^{R+1}m+4$, $w_{1}(CD)=2^{R+1}m+(K+3)$ and $w_{1}(EF)=(K+1)$. Writing $K=8n+i$ for some $n\in\mathbb{N}$ and $i\in\{1,2\}$, we see that the triple $\left(2^{R+1}m+4,2^{R+1}m+8n+3+i,8n-1+i\right)$ is balanced, so that $P_{2}$ wins by Theorem~\ref{thm:galaxy}.
\end{enumerate}
This concludes the proof of our claim that the configuration in \eqref{eq:H_{1}_losing_4_4_inductive} is losing on $H_{1}$. 

\subsubsection{When $K \geqslant 12$ and $K \equiv i \bmod 8$ for some $i \in \{0,3,4,5,6,7\}$} For $m \in \mathbb{N}_{0}$, we consider
\begin{equation}
w_{0}(AB)=2^{R+1}m+4,\ w_{0}(BC)=2,\ w_{0}(CD)=2^{R+1}m+(K-5),\ w_{0}(EF)=(K+1),\label{eq:H_{1}_losing_4_3_inductive}
\end{equation}
where $R=f(K+1)$. The first round of the game played on this initial configuration unfolds as follows:
\begin{enumerate}
\item Suppose $P_{1}$ removes a positive integer weight from $CD$, and a non-negative integer weight from $BC$, in the first round. The following possibilities may arise:
\begin{enumerate}
\item If $w_{1}(CD)=2^{R+1}m+L$ for $L \in \{4,5,\ldots,K-6\}$, with $L \equiv j \bmod 8$ for $j \in \{0,1,2,3,6,7\}$, and $w_{1}(BC)=2$, $P_{2}$ removes weight $(K-L-5)$ from $EF$ in the second round, leaving $P_{1}$ with
\begin{align}
{}&w_{2}(AB)=2^{f(L+6)+1}2^{R-f(L+6)}m+4,\ w_{2}(BC)=2,\nonumber\\
{}&w_{2}(CD)=2^{f(L+6)+1}2^{R-f(L+6)}m+L,\ w_{2}(EF)=(L+6),\nonumber
\end{align}
which is of the form \eqref{eq:H_{1}_losing_2,4_case_2} with $r=4$, $s=2$ and $w_{2}(EF)=(L+6)\leqslant K$, so that $P_{1}$ loses by our induction hypothesis.

\item If $w_{1}(CD)=2^{R+1}m+L$ for $L \in \{4,5,\ldots,K-6\}$ with $L \equiv j \bmod 8$ for some $j \in \{4,5\}$, and $w_{1}(BC)=2$, $P_{2}$ removes weight $(K+3-L)$ from $EF$ in the second round, leaving $P_{1}$ with
\begin{align}
{}&w_{2}(AB)=2^{f(L-2)+1}2^{R-f(L-2)}m+4,\ w_{2}(BC)=2,\nonumber\\
{}&w_{2}(CD)=2^{f(L-2)+1}2^{R-f(L-2)}m+L,\ w_{2}(EF)=(L-2),\nonumber
\end{align}
which is of the form \eqref{eq:H_{1}_losing_2,4_case_1} with $r=4$, $s=2$ and $w_{2}(EF)=(L-2)\leqslant (K-8)<K$, so that $P_{1}$ loses by our induction hypothesis.

\item If $w_{1}(CD)=2^{R+1}m+L$ for $L \in \{4,5,\ldots,K-6\}$, with $L \equiv j \bmod 8$ for some $j \in \{0,1,2,3,7\}$, and $w_{1}(BC)=1$, $P_{2}$ removes weight $(K-4-L)$ from $EF$ in the second round, leaving $P_{1}$ with
\begin{align}
{}&w_{2}(AB)=2^{f(L+5)+1}2^{R-f(L+5)}m+4,\ w_{2}(BC)=1,\nonumber\\
{}&w_{2}(CD)=2^{f(L+5)+1}2^{R-f(L+5)}m+L,\ w_{2}(EF)=(L+5),\nonumber
\end{align}
which is of the form \eqref{eq:H_{1}_losing_2,4_case_2} with $r=4$ and $s=1$, and $P_{1}$ loses by what we have proved in \S\ref{subsec:H_{1}_losing_4_1,2_proof}. 

\item If $w_{1}(CD)=2^{R+1}m+L$ for $L \in \{4,5,\ldots,K-6\}$, with $L \equiv j \bmod 8$ for some $j \in \{4,5,6\}$, and $w_{1}(BC)=1$, $P_{2}$ removes weight $(K+4-L)$ from $EF$ in the second round, leaving $P_{1}$ with
\begin{align}
{}&w_{2}(AB)=2^{f(L-3)+1}2^{R-f(L-3)}m+4,\ w_{2}(BC)=1,\nonumber\\
{}&w_{2}(CD)=2^{f(L-3)+1}2^{R-f(L-3)}m+L,\ w_{2}(EF)=(L-3),\nonumber
\end{align}
which is of the form \eqref{eq:H_{1}_losing_2,4_case_1} with $r=4$ and $s=1$, and $P_{1}$ loses by what we have proved in \S\ref{subsec:H_{1}_losing_4_1,2_proof}.

\item If $w_{1}(CD)=2^{R+1}m+L$ for $L \in \{0,1,2,3\}$ and $w_{1}(BC)=t\in\{1,2\}$, $P_{2}$ removes weight $(K-3-t-L)$ from $EF$ in the second round, leaving $P_{1}$ with
\begin{equation}
w_{2}(AB)=2^{R+1}m+4,\ w_{2}(BC)=t,\ w_{2}(CD)=2^{R+1}m+L,\ w_{2}(EF)=4+t+L,\nonumber
\end{equation}
which is of the form \eqref{eq:H_{1}_losing_0,1,3} with either $r=0$ (when $L=0$) or $r=1$ (when $L=1$) or $r=3$ (when $L=3$), or of the form \eqref{eq:H_{1}_losing_2,4_case_2} with $r=2$ and $s=1$ (when $L=2$ and $t=1$), or of the form \eqref{eq:H_{1}_losing_2,4_case_3} with $r=2$ and $s=2$ (when $L=2$ and $t=2$). By what we have already proved in \S\ref{subsec:proof_eq_H_{1}_losing_0}, \S\ref{subsec:H_{1}_losing_1_proof}, \S\ref{subsec:H_{1}_losing_3_proof}, \S\ref{subsec:proof_eq_H_{1}_losing_2_1,2} and \S\ref{subsec:proof_eq_H_{1}_losing_2_3}, we conclude that $P_{1}$ loses.

\item If $w_{1}(CD)=2^{R+1}m+L$ for some $L \in \{0,1,\ldots,K-6\}$, and $w_{1}(BC)=0$, $P_{2}$, after the first round, is left with a galaxy graph consisting of the edges $AB$, $CD$ and $EF$, where $w_{1}(AB)=2^{R+1}m+4$, $w_{1}(CD)=2^{R+1}m+L$ and $w_{1}(EF)=(K+1)$, and she wins by Lemma~\ref{lem:H_{1}_winning_galaxy}.

\item Finally, if $w_{1}(CD)=2^{R+1}n+\ell_{1}$ for some $n<m$ and $\ell_{1}\in\{0,1,\ldots,2^{R+1}-1\}$, the configuration after the first round is of the form given by Lemma~\ref{prop:H_{1}_winning_1} with $m_{1}=n$, $m_{2}=m$ and $\ell_{2}=4$. Since $n<m \implies m_{1}\neq m_{2}$, $P_{2}$ wins by Lemma~\ref{prop:H_{1}_winning_1}.
\end{enumerate}

\item Suppose $P_{1}$ removes a positive integer weight from $AB$, and a non-negative integer weight from $BC$, in the first round. The following possibilities may arise:
\begin{enumerate}
\item If either $w_{1}(AB)=2^{R+1}m+\ell$ for some $\ell \in \{0,1,3\}$ and $w_{1}(BC)=t\in\{1,2\}$, or $w_{1}(AB)=2^{R+1}m+2$ and $w_{1}(BC)=t=2$, or $w_{1}(AB)=2^{R+1}m+2$ and $w_{1}(BC)=t=1$ and $K \equiv i \bmod 8$ for $i \in \{0,4,5,6\}$, $P_{2}$ removes weight $(6-\ell-t)$ from $EF$ in the second round, leaving $P_{1}$ with
\begin{equation}
w_{2}(AB)=2^{R+1}m+\ell,\ w_{2}(BC)=t,\ w_{2}(CD)=2^{R+1}m+(K-5),\ w_{2}(EF)=K+\ell+t-5,\nonumber
\end{equation}
which is of the form \eqref{eq:H_{1}_losing_0,1,3} with either $r=0$ (when $\ell=0$) or $r=1$ (when $\ell=1$) or $r=3$ (when $\ell=3$), or of the form \eqref{eq:H_{1}_losing_2,4_case_2} with $r=2$ and $s=1$ (when $\ell=2$ and $t=1$, or of the form \eqref{eq:H_{1}_losing_2,4_case_3} with $r=2$ and $s=2$ (when $\ell=2$ and $t=2$). By what we have proved in \S\ref{subsec:proof_eq_H_{1}_losing_0}, \S\ref{subsec:H_{1}_losing_1_proof}, \S\ref{subsec:H_{1}_losing_3_proof}, \S\ref{subsec:proof_eq_H_{1}_losing_2_1,2} and \S\ref{subsec:proof_eq_H_{1}_losing_2_3}, we know that $P_{1}$ loses.

\item If $w_{1}(AB)=2^{R+1}m+2$ and $w_{1}(BC)=t=1$, where $K \equiv i \bmod 8$ for $i \in \{3,7\}$, $P_{2}$ removes weight $7$ from $EF$ in the second round, leaving $P_{1}$ with 
\begin{equation}
w_{2}(AB)=2^{R+1}m+2,\ w_{2}(BC)=1,\ w_{2}(CD)=2^{R+1}m+(K-5),\ w_{2}(EF)=(K-6),\nonumber
\end{equation}
which is of the form \eqref{eq:H_{1}_losing_2,4_case_1} with $r=2$ and $s=1$, and by what we have already proved in \S\ref{subsec:proof_eq_H_{1}_losing_2_1,2}, we know that $P_{1}$ loses.

\item If $w_{1}(AB)=2^{R+1}m+\ell$ for some $\ell \in \{0,1,2,3\}$, and $w_{1}(BC)=0$, $P_{2}$, after the first round, is left with a galaxy graph consisting of the edges $AB$, $CD$ and $EF$, where $w_{1}(AB)=2^{R+1}m+\ell$, $w_{1}(CD)=2^{R+1}m+(K-5)$ and $w_{1}(EF)=(K+1)$, and she wins by Lemma~\ref{lem:H_{1}_winning_galaxy}.

\item Finally, if $w_{1}(AB)=2^{R+1}n+\ell_{1}$ for some $n<m$ and $\ell_{1}\in\{0,1,\ldots,2^{R+1}-1\}$, the configuration after the first round is of the form given by Lemma~\ref{prop:H_{1}_winning_1} with $m_{1}=n$, $m_{2}=m$ and $\ell_{2}=K-5$. Since $n<m \implies m_{1}\neq m_{2}$, $P_{2}$ wins by Lemma~\ref{prop:H_{1}_winning_1}.
\end{enumerate}

\item Suppose $P_{1}$ removes a positive integer weight from $EF$ in the first round, so that $w_{1}(EF)=k \leqslant K$. If $k \in \{12,13,\ldots,K\}$ and $k \equiv j \bmod 8$ for some $j \in \{0,1,4,5,6,7\}$, $P_{2}$ removes weight $(K+1-k)$ from the edge $CD$ in the second round, leaving $P_{1}$ with
\begin{align}
{}&w_{2}(AB)=2^{f(k)+1}2^{R-f(k)}m+4,\ w_{2}(BC)=2,\nonumber\\
{}&w_{2}(CD)=2^{f(k)+1}2^{R-f(k)}m+(k-6),\ w_{2}(EF)=k,\nonumber
\end{align}
which is of the form \eqref{eq:H_{1}_losing_2,4_case_2} with $r=4$, $s=2$ and $w_{2}(EF)=k \leqslant K$, so that $P_{1}$ loses by our induction hypothesis. Note that this case covers
\begin{enumerate*}
\item $k=(K-7)$,
\item $k=(K-6)$ and $K \equiv i \bmod 8$ for $i \in \{3,4,5,6,7\}$.
\end{enumerate*}
If $k=(K-6)$ and $K \equiv 0 \bmod 8$, writing $K=8(n+1)$ for some $n\in\mathbb{N}_{0}$, $P_{2}$ removes weight $3$ from $AB$, and the entire $BC$, in the second round, leaving $P_{1}$ with a galaxy graph consisting of the edges $AB$, $CD$ and $EF$, with $w_{2}(AB)=2^{R+1}m+1$, $w_{2}(CD)=2^{R+1}m+8n+3$ and $w_{2}(EF)=8n+2$, and $P_{1}$ loses by Theorem~\ref{thm:galaxy}.

If $k\in\{10,11,\ldots,K-8\}$ and $k \equiv j \bmod 8$ for some $j \in \{2,3\}$, $P_{2}$ removes weight $(K-k-7)$ from the edge $CD$ in the second round, leaving $P_{1}$ with 
\begin{align}
{}&w_{2}(AB)=2^{f(k)+1}2^{R-f(k)}m+4,\ w_{2}(BC)=2,\nonumber\\
{}&w_{2}(CD)=2^{f(k)+1}2^{R-f(k)}m+(k+2),\ w_{2}(EF)=k,\nonumber
\end{align}
which is of the form \eqref{eq:H_{1}_losing_2,4_case_1} with $r=4$, $s=2$ and $w_{2}(EF)=k \leqslant K$, so that $P_{1}$ loses by our induction hypothesis. 

If $k \in \{1,2,3,4,K-5\}$, then the configuration after the first round is of the form mentioned in Lemma~\ref{prop:H_{1}_winning_1}, with $\ell_{1}=4$, $\ell_{2}=(K-5)$ and $w_{1}(EF)=k$, so that either $\min\{\ell_{1},\ell_{2}\}\geqslant k$ or $k \in \{\ell_{1},\ell_{2}\}$. Consequently, $P_{2}$ wins by Lemma~\ref{prop:H_{1}_winning_1}. If $k\in \{K-3,\ldots,K\}$, then $P_{2}$ removes $(K+1-k)$ from $AB$ in the second round, leaving $P_{1}$ with
\begin{equation}
w_{2}(AB)=2^{R+1}m+k-(K-3),\ w_{2}(BC)=2,\ w_{2}(CD)=2^{R+1}m+(K-5),\ w_{2}(EF)=k,\nonumber
\end{equation}
which is of the form \eqref{eq:H_{1}_losing_0,1,3} with either $r=0$ (when $k=K-3$) or $r=1$ (when $k=K-2$) or $r=3$ (when $k=K$), or of the form \eqref{eq:H_{1}_losing_2,4_case_3} with $r=2$ and $s=2$ (when $k=K-1$). If $k=(K-4)$, $P_{2}$ removes weight $4$ from $AB$ and weight $1$ from $BC$ in the second round, leaving $P_{1}$ with
\begin{equation}
w_{2}(AB)=2^{R+1}m,\ w_{2}(BC)=1,\ w_{2}(CD)=2^{R+1}m+(K-5),\ w_{2}(EF)=(K-4),\nonumber
\end{equation}
which is of the form \eqref{eq:H_{1}_losing_0,1,3} with $r=0$. By what we have already proved in \S\ref{subsec:proof_eq_H_{1}_losing_0}, \S\ref{subsec:H_{1}_losing_1_proof}, \S\ref{subsec:H_{1}_losing_3_proof} and \S\ref{subsec:proof_eq_H_{1}_losing_2_3}, we know that $P_{1}$ loses. If $k\in\{5,6,8\}$, $P_{2}$ removes weight $(K-k)$ from $CD$ and weight $1$ from $BC$ in the second round, leaving $P_{1}$ with
\begin{equation}
w_{2}(AB)=2^{R+1}m+4,\ w_{2}(BC)=1,\ w_{2}(CD)=2^{R+1}m+(k-5),\ w_{2}(EF)=k,\nonumber
\end{equation}
which is of the form \eqref{eq:H_{1}_losing_0,1,3} with either $r=0$ (when $k=5$) or $r=1$ (when $k=6$) or $r=3$ (when $k=8$); if $k=7$, $P_{2}$ removes weight $(K-6)$ from $CD$ in the second round, leaving $P_{1}$ with
\begin{equation}
w_{2}(AB)=2^{R+1}m+4,\ w_{2}(BC)=2,\ w_{2}(CD)=2^{R+1}m+1,\ w_{2}(EF)=7,\nonumber
\end{equation}
which is of the form \eqref{eq:H_{1}_losing_0,1,3} with $r=1$; if $k=9$, $P_{2}$ removes weight $(K-8)$ from $CD$ in the second round, leaving $P_{1}$ with
\begin{equation}
w_{2}(AB)=2^{R+1}m+4,\ w_{2}(BC)=2,\ w_{2}(CD)=2^{R+1}m+3,\ w_{2}(EF)=9,\nonumber
\end{equation}
which is of the form \eqref{eq:H_{1}_losing_0,1,3} with $r=3$. By what we have already proved in \S\ref{subsec:proof_eq_H_{1}_losing_0}, \S\ref{subsec:H_{1}_losing_1_proof} and \S\ref{subsec:H_{1}_losing_3_proof}, we conclude that $P_{1}$ loses. If $k=0$, $P_{2}$ wins by Remark~\ref{rem:EF_edgeweight_0}.

\item Suppose $P_{1}$ removes a positive integer weight from the edge $BC$ in the first round, without disturbing the edge-weights of $AB$ and $CD$. If $w_{1}(BC)=1$, and $K \equiv i \bmod 8$ for some $i \in \{0,4,5,6,7\}$, then $P_{2}$ removes weight $1$ from $EF$ in the second round, leaving $P_{1}$ with
\begin{align}
{}&w_{2}(AB)=2^{R+1}m+4,\ w_{2}(BC)=1,\ w_{2}(CD)=2^{R+1}m+(K-5),\ w_{2}(EF)=K,\nonumber
\end{align}
which is of the form \eqref{eq:H_{1}_losing_2,4_case_2} with $r=4$ and $s=1$, and by what we have already proved in \S\ref{subsec:H_{1}_losing_4_1,2_proof}, we conclude that $P_{1}$ loses. If $w_{1}(BC)=1$ and $K \equiv 1 \bmod 8$, $P_{2}$ removes weight $9$ from $EF$ in the second round, leaving $P_{1}$ with
\begin{align}
{}&w_{2}(AB)=2^{R+1}m+4,\ w_{2}(BC)=1,\ w_{2}(CD)=2^{R+1}m+(K-5),\ w_{2}(EF)=(K-8),\nonumber
\end{align}
which is of the form \eqref{eq:H_{1}_losing_2,4_case_1} with $r=4$ and $s=1$, and by what we have already proved in \S\ref{subsec:H_{1}_losing_4_1,2_proof}, we conclude that $P_{1}$ loses. If $w_{1}(BC)=0$, then $P_{2}$, after the first round, is left with a galaxy graph consisting of the edges $AB$, $CD$ and $EF$, where $w_{1}(AB)=2^{R+1}m+4$, $w_{1}(CD)=2^{R+1}m+(K-5)$ and $w_{1}(EF)=(K+1)$, and she wins by Lemma~\ref{lem:H_{1}_winning_galaxy}.
\end{enumerate}
This completes the proof of our claim that the configuration in \eqref{eq:H_{1}_losing_4_3_inductive} is losing, and thereby, the proof of our claim that aconfiguration that is either of the form \eqref{eq:H_{1}_losing_2,4_case_1} or of the form \eqref{eq:H_{1}_losing_2,4_case_2}, with $r=4$ and $s=2$, is losing. 

\subsection{Proof that configurations satisfying \eqref{eq:H_{1}_losing_2,4_case_1} or \eqref{eq:H_{1}_losing_2,4_case_2}, with $r=4$ and $s=3$, are losing}\label{subsec:H_{1}_losing_4_5,6_proof}
The base case corresponding to \eqref{eq:H_{1}_losing_2,4_case_2} with $r=4$ and $s=3$ is obtained by setting $k=12$, which yields the configuration
\begin{equation}
w_{0}(AB)=16m+4,\ w_{0}(BC)=3,\ w_{0}(CD)=16m+5,\ w_{0}(EF)=12.\label{eq:H_{1}_losing_4_5_base_case}
\end{equation}
The first round of the game played on this configuration unfolds as follows:
\begin{enumerate}
\item Suppose $P_{1}$ removes a positive integer weight from $CD$, and a non-negative integer weight from $BC$, in the first round. This can be further divided into the following possibilities:
\begin{enumerate}
\item If $w_{1}(CD)=16m+\ell$ for $\ell \in \{0,1,2,3\}$ and $w_{1}(BC)=t \in \{1,2,3\}$, $P_{2}$ removes weight $(8-\ell-t)$ from the edge $EF$ in the second round, leaving $P_{1}$ with
\begin{equation}
w_{2}(AB)=16m+4,\ w_{2}(BC)=t,\ w_{2}(CD)=16m+\ell,\ w_{2}(EF)=4+\ell+t,\nonumber
\end{equation}
which is of the form \eqref{eq:H_{1}_losing_0,1,3} with either $r=0$ (when $\ell=0$) or $r=1$ (when $\ell=1$) or $r=3$ (when $\ell=3$), or of the form \eqref{eq:H_{1}_losing_2,4_case_2} with $r=2$ and $s=1$ (when $\ell=2$ and $t=1$), or of the form \eqref{eq:H_{1}_losing_2,4_case_3} with $r=2$ and $s\geqslant 2$ (when $\ell=2$ and $t\geqslant 2$). By what we have already proved in \S\ref{subsec:proof_eq_H_{1}_losing_0}, \S\ref{subsec:H_{1}_losing_1_proof}, \S\ref{subsec:H_{1}_losing_3_proof}, \S\ref{subsec:proof_eq_H_{1}_losing_2_1,2} and \S\ref{subsec:proof_eq_H_{1}_losing_2_3}, we know that $P_{1}$ loses.

\item If $w_{1}(CD)=16m+4$ and $w_{1}(BC)=t \in \{1,2,3\}$, $P_{2}$ removes weight $(12-t)$ from the edge $EF$ in the second round, leaving $P_{1}$ with
\begin{equation}
w_{2}(AB)=4(4m+1),\ w_{2}(BC)=t,\ w_{2}(CD)=4(4m+1),\ w_{2}(EF)=t,\nonumber
\end{equation}
which is of the form \eqref{eq:H_{1}_losing_0,1,3} with $r=0$, so that $P_{1}$ loses by what we have already proved in \S\ref{subsec:proof_eq_H_{1}_losing_0}. 

\item If $w_{1}(CD)=16m+\ell$ for some $\ell \in \{0,1,2,3,4\}$ and $w_{1}(BC)=0$, $P_{2}$, at the end of the first round, is left with a galaxy graph consisting of the edges $AB$, $CD$ and $EF$, where $w_{1}(AB)=16m+4$, $w_{1}(CD)=16m+\ell$ and $w_{1}(EF)=12$, and she wins by Lemma~\ref{lem:H_{1}_winning_galaxy}.

\item Finally, if $w_{1}(CD)=16n+\ell_{1}$ for some $n<m$ and $\ell_{1}\in\{0,1,\ldots,15\}$, the configuration after the first round is of the form given by Lemma~\ref{prop:H_{1}_winning_1} with $m_{1}=n$, $m_{2}=m$ and $\ell_{2}=4$. Since $n<m \implies m_{1}\neq m_{2}$, $P_{2}$ wins by Lemma~\ref{prop:H_{1}_winning_1}.
\end{enumerate} 

\item Suppose $P_{1}$ removes a positive integer weight from $AB$, and a non-negative integer weight from $BC$, in the first round. We again consider the various possibilities separately:
\begin{enumerate}
\item If $w_{1}(AB)=16m+\ell$ for $\ell \in \{0,1,2,3\}$ and $w_{1}(BC)=t\in\{1,2,3\}$, $P_{2}$ removes weight $(7-\ell-t)$ from the edge $EF$ in the second round, leaving $P_{1}$ with
\begin{align}
w_{2}(AB)=16m+\ell,\ w_{2}(BC)=t,\ w_{2}(CD)=16m+5,\ w_{2}(EF)=5+\ell+t,\nonumber
\end{align}
which is of the form \eqref{eq:H_{1}_losing_0,1,3} with either $r=0$ (when $\ell=0$) or $r=1$ (when $\ell=1$) or $r=3$ (when $\ell=3$), or of the form \eqref{eq:H_{1}_losing_2,4_case_2} with $r=2$ and $s=1$ (when $\ell=2$ and $t=1$), or of the form \eqref{eq:H_{1}_losing_2,4_case_3} with $r=2$ and $s\geqslant 2$ (when $\ell=2$ and $t\geqslant 2$). By what we have already proved in \S\ref{subsec:proof_eq_H_{1}_losing_0}, \S\ref{subsec:H_{1}_losing_1_proof}, \S\ref{subsec:H_{1}_losing_3_proof}, \S\ref{subsec:proof_eq_H_{1}_losing_2_1,2} and \S\ref{subsec:proof_eq_H_{1}_losing_2_3}, we conclude that $P_{1}$ loses. 

\item If $w_{1}(AB)=16m+\ell$ for $\ell \in \{0,1,2,3\}$ and $w_{1}(BC)=0$, $P_{2}$, at the end of the first round, is left with a galaxy graph consisting of the edges $AB$, $CD$ and $EF$, where $w_{1}(AB)=16m+\ell$, $w_{1}(CD)=16m+5$ and $w_{2}(EF)=12$, and she wins by Lemma~\ref{lem:H_{1}_winning_galaxy}. 
\end{enumerate}

\item Suppose $P_{1}$ removes a positive integer weight from $EF$ in the first round, so that $w_{1}(EF)=k \leqslant 11$. If $k \in \{8,9,10,11\}$, $P_{2}$ removes weight $(12-k)$ from $AB$ in the second round, leaving $P_{1}$ with 
\begin{equation}
w_{2}(AB)=16m+(k-8),\ w_{2}(BC)=3,\ w_{2}(CD)=16m+5,\ w_{2}(EF)=k,\nonumber
\end{equation}
which is of the form \eqref{eq:H_{1}_losing_0,1,3} with either $r=0$ (when $k=8$) or $r=1$ (when $k=9$) or $r=3$ (when $k=11$), or of the form \eqref{eq:H_{1}_losing_2,4_case_3} with $r=2$ and $s=3$ (when $k=10$). If $k\in\{6,7\}$, $P_{2}$ removes weight $4$ from $AB$ and weight $(8-k)$ from $BC$ in the second round, leaving $P_{1}$ with 
\begin{equation}
w_{2}(AB)=16m,\ w_{2}(BC)=(k-5),\ w_{2}(CD)=16m+5,\ w_{2}(EF)=k,\nonumber
\end{equation}
which is of the form \eqref{eq:H_{1}_losing_0,1,3} with $r=0$. By what we have proved in \S\ref{subsec:proof_eq_H_{1}_losing_0}, \S\ref{subsec:H_{1}_losing_1_proof}, \S\ref{subsec:H_{1}_losing_3_proof} and \S\ref{subsec:proof_eq_H_{1}_losing_2_3}, we conclude that $P_{1}$ loses. If $k \in \{1,2,3,4,5\}$, the configuration after the first round is of the form mentioned in Lemma~\ref{prop:H_{1}_winning_1}, with $\ell_{1}=4$, $\ell_{2}=5$, and $w_{1}(EF)=k$ satisfying either $\min\{\ell_{1},\ell_{2}\}\geqslant k$ or $k \in\{\ell_{1},\ell_{2}\}$, so that $P_{2}$ wins by Lemma~\ref{prop:H_{1}_winning_1}. If $k=0$, $P_{2}$ wins by Remark~\ref{rem:EF_edgeweight_0}.

\item Finally, suppose $P_{1}$ removes a positive integer weight from $BC$ in the first round, without disturbing the edge-weights of $AB$ and $CD$, so that $w_{1}(BC)=t\in\{0,1,2\}$. For $t\in\{1,2\}$, $P_{2}$ removes weight $(11-t)$ from $EF$ in the second round, leaving $P_{1}$ with
\begin{equation}
w_{2}(AB)=4(4m+1),\ w_{2}(BC)=t,\ w_{2}(CD)=4(4m+1)+1,\ w_{2}(EF)=t+1,\nonumber
\end{equation}
which is of the form \eqref{eq:H_{1}_losing_0,1,3} with $r=0$. By what we have proved in \S\ref{subsec:proof_eq_H_{1}_losing_0}, we conclude that $P_{1}$ loses. If $t=0$, $P_{2}$ is left with a galaxy graph at the end of the first round, consisting of the edges $AB$, $CD$ and $EF$, where $w_{1}(AB)=16m+4$, $w_{1}(CD)=16m+5$ and $w_{1}(EF)=12$, and she wins by Lemma~\ref{lem:H_{1}_winning_galaxy}.
\end{enumerate}
This proves our claim that the configuration in \eqref{eq:H_{1}_losing_4_5_base_case} is losing on $H_{1}$. The base case corresponding to \eqref{eq:H_{1}_losing_2,4_case_1} with $r=4$ and $s=3$ is obtained by setting $k=3$, which yields the configuration 
\begin{equation}
w_{0}(AB)=16m+4=4(4m+1),\ w_{0}(BC)=3,\ w_{0}(CD)=16m+4=4(4m+1),\ w_{0}(EF)=3,\nonumber
\end{equation}
which is losing since it is of the form \eqref{eq:H_{1}_losing_0,1,3} with $r=0$.

Suppose, for some $K \in \mathbb{N}$, we have shown that a configuration that is either of the form \eqref{eq:H_{1}_losing_2,4_case_1} or of the form \eqref{eq:H_{1}_losing_2,4_case_2}, with $r=4$ and $s=3$, is losing as long as $w_{0}(EF)=k \leqslant K$.

\subsubsection{When $K \geqslant 10$ and $K \equiv 2 \bmod 8$} For $m \in \mathbb{N}_{0}$, we consider the configuration
\begin{equation}
w_{0}(AB)=2^{R+1}m+4,\ w_{0}(BC)=3,\ w_{0}(CD)=2^{R+1}m+K+2,\ w_{0}(EF)=K+1,\label{eq:H_{1}_losing_4_6_inductive}
\end{equation}
where $R=f(K+1)$. The first round of the game played on this initial configuration unfolds as follows:
\begin{enumerate}
\item Suppose $P_{1}$ removes a positive integer weight from $CD$ and a non-negative integer weight from $BC$ in the first round. The following possibilities may arise:
\begin{enumerate}
\item If $w_{1}(CD)=2^{R+1}m+(K+1)$, the configuration, after the first round, is of the form mentioned in Lemma~\ref{prop:H_{1}_winning_1}, with $\ell_{1}=4$, $\ell_{2}=(K+1)$ and $w_{1}(EF)=k=(K+1)\in \{\ell_{1},\ell_{2}\}$, so that $P_{2}$ wins by Lemma~\ref{prop:H_{1}_winning_1}. 

\item If $w_{1}(CD)=2^{R+1}m+K$ and $w_{1}(BC)=t \in \{1,2,3\}$, $P_{2}$ removes weight $4$ from $AB$, and weight $(t-1)$ from $BC$, in the second round, leaving $P_{1}$ with
\begin{equation}
w_{2}(AB)=2^{R+1}m,\ w_{2}(BC)=1,\ w_{2}(CD)=2^{R+1}m+K,\ w_{2}(EF)=K+1,\nonumber
\end{equation}
which is of the form \eqref{eq:H_{1}_losing_0,1,3} with $r=0$, and $P_{1}$ loses by what we have already proved in \S\ref{subsec:proof_eq_H_{1}_losing_0}. If $w_{1}(CD)=2^{R+1}m+K$ and $w_{1}(BC)=0$, writing $K=8n+2$ for some $n\in\mathbb{N}$, we see that $P_{2}$ is left, at the end of the first round, with a galaxy graph consisting of the edges $AB$, $CD$ and $EF$, where $w_{1}(AB)=2^{R+1}m+4$, $w_{1}(CD)=2^{R+1}m+8n+2$ and $w_{1}(EF)=8n+3$, and since the triple $(2^{R+1}m+1,2^{R+1}m+8n+2,8n+3)$ is balanced, $P_{2}$ wins by Theorem~\ref{thm:galaxy}.

\item Suppose $w_{1}(CD)=2^{R+1}m+(K-1)$. If $w_{1}(BC)=t \in \{2,3\}$, $P_{2}$ removes weight $(t-2)$ from the edge $BC$ and weight $4$ from the edge $AB$ in the second round, leaving $P_{1}$ with
\begin{equation}
w_{2}(AB)=2^{R+1}m,\ w_{2}(BC)=2,\ w_{2}(CD)=2^{R+1}m+(K-1),\ w_{2}(EF)=(K+1),\nonumber
\end{equation}
which is of the form \eqref{eq:H_{1}_losing_0,1,3} with $r=0$. If $w_{1}(BC)=1$, $P_{2}$ removes weight $3$ from $AB$ in the second round, leaving $P_{1}$ with
\begin{equation}
w_{2}(AB)=2^{R+1}m+1,\ w_{2}(BC)=1,\ w_{2}(CD)=2^{R+1}m+(K-1),\ w_{2}(EF)=(K+1),\nonumber
\end{equation}
which is of the form \eqref{eq:H_{1}_losing_0,1,3} with $r=1$. By what we have already proved in \S\ref{subsec:proof_eq_H_{1}_losing_0} and \S\ref{subsec:H_{1}_losing_1_proof}, we know that $P_{1}$ loses. If $w_{1}(BC)=0$, writing $K=8n+2$ for some $n\in\mathbb{N}$, $P_{2}$, at the end of the first round, is left with a galaxy graph consisting of the edges $AB$, $CD$ and $EF$, where $w_{1}(AB)=2^{R+1}m+4$, $w_{1}(CD)=2^{R+1}m+8n+1$ and $w_{1}(EF)=8n+3$, and since the triple $(2^{R+1}m+2,2^{R+1}m+8n+1,8n+3)$ is balanced, $P_{2}$ wins by Theorem~\ref{thm:galaxy}. 

\item Suppose $w_{1}(CD)=2^{R+1}m+(K-2)$. If $w_{1}(BC)=t \in \{1,2,3\}$, $P_{2}$ removes weight $(t+1)$ from the edge $AB$ in the second round, leaving $P_{1}$ with 
\begin{equation}
w_{2}(AB)=2^{R+1}m+(3-t),\ w_{2}(BC)=t,\ w_{2}(CD)=2^{R+1}m+(K-2),\ w_{2}(EF)=(K+1),\nonumber
\end{equation}
which is of the form \eqref{eq:H_{1}_losing_0,1,3} with either $r=0$ (when $t=3$) or $r=1$ (when $t=2$), or of the form \eqref{eq:H_{1}_losing_2,4_case_2} with $r=2$ and $s=1$ (when $t=1$). By what we have already proved in \S\ref{subsec:proof_eq_H_{1}_losing_0}, \S\ref{subsec:H_{1}_losing_1_proof} and \S\ref{subsec:proof_eq_H_{1}_losing_2_1,2}, we conclude that $P_{1}$ loses. If $w_{1}(BC)=0$, writing $K=8n+2$ for some $n\in\mathbb{N}$, $P_{2}$, at the end of the first round, is left with a galaxy graph consisting of the edges $AB$, $CD$ and $EF$, where $w_{1}(AB)=2^{R+1}m+4$, $w_{1}(CD)=2^{R+1}m+8n$ and $w_{1}(EF)=8n+3$, and since the triple $(2^{R+1}m+3,2^{R+1}m+8n,8n+3)$ is balanced, $P_{2}$ wins by Theorem~\ref{thm:galaxy}.

\item Suppose $w_{1}(CD)=2^{R+1}m+(K-3)$. If $w_{1}(BC)=t \in \{1,2,3\}$, then $P_{2}$ removes weight $t$ from the edge $AB$ in the second round, leaving $P_{1}$ with 
\begin{equation}
w_{2}(AB)=2^{R+1}m+(4-t),\ w_{2}(BC)=t,\ w_{2}(CD)=2^{R+1}m+(K-3),\ w_{2}(EF)=(K+1),\nonumber
\end{equation} 
which is of the form \eqref{eq:H_{1}_losing_0,1,3} with either $r=1$ (when $t=3$) or $r=3$ (when $t=1$), or of the form \eqref{eq:H_{1}_losing_2,4_case_3} with $r=2$ and $s=2$ (when $t=2$). By what we have already proved in \S\ref{subsec:H_{1}_losing_1_proof}, \S\ref{subsec:H_{1}_losing_3_proof} and \S\ref{subsec:proof_eq_H_{1}_losing_2_3}, we conclude that $P_{1}$ loses. If $w_{1}(BC)=0$, writing $K=8(n+1)+2$ for some $n\in\mathbb{N}_{0}$, $P_{2}$, at the end of the first round, is left with a galaxy graph consisting of the edges $AB$, $CD$ and $EF$, where $w_{1}(AB)=2^{R+1}m+4$, $w_{1}(CD)=2^{R+1}m+8n+7$ and $w_{1}(EF)=8n+11$, and since the triple $(2^{R+1}m+4,2^{R+1}m+8n+7,8n+3)$ is balanced, $P_{2}$ wins by Theorem~\ref{thm:galaxy}.

\item If $w_{1}(CD)=2^{R+1}m+L$ for some $L \in \{0,1,\ldots,K-4\}$, and $w_{1}(BC)=0$, then, at the end of the first round, $P_{2}$ is left with a galaxy graph consisting of the edges $AB$, $CD$ and $EF$, where $w_{1}(AB)=2^{R+1}m+4$, $w_{1}(CD)=2^{R+1}m+L$ and $w_{1}(EF)=(K+1)$, and she wins by Lemma~\ref{lem:H_{1}_winning_galaxy}.

\item Suppose $w_{1}(CD)=2^{R+1}m+(K-4)$. If $w_{1}(BC)=t \in \{2,3\}$, $P_{2}$ removes weight $(t-1)$ from the edge $AB$ in the second round, leaving $P_{1}$ with
\begin{equation}
w_{2}(AB)=2^{R+1}m+(5-t),\ w_{2}(BC)=t,\ w_{2}(CD)=2^{R+1}m+(K-4),\ w_{2}(EF)=(K+1),\nonumber
\end{equation}
which is of the form \eqref{eq:H_{1}_losing_2,4_case_3} with $r=2$ and $s=3$ (when $t=3$) or of the form \eqref{eq:H_{1}_losing_0,1,3} with $r=3$ (when $t=2$). If $w_{1}(BC)=1$, $P_{2}$ removes $8$ from $EF$ in the second round, leaving $P_{1}$ with
\begin{align}
w_{2}(AB)=2^{R+1}m+4,\ w_{2}(BC)=1,\ w_{2}(CD)=2^{R+1}m+(K-4),\ w_{2}(EF)=(K-7),\nonumber
\end{align}
which is of the form \eqref{eq:H_{1}_losing_2,4_case_1} with $r=4$ and $s=1$. By what we have already proved in \S\ref{subsec:proof_eq_H_{1}_losing_2_3}, \S\ref{subsec:H_{1}_losing_3_proof} and \S\ref{subsec:H_{1}_losing_4_1,2_proof}, we conclude that $P_{1}$ loses in each of the above three cases.

\item Suppose $w_{1}(CD)=2^{R+1}m+(K-5)$. If $w_{1}(BC)=3$, $P_{2}$ removes weight $1$ from the edge $AB$ in the second round, leaving $P_{1}$ with
\begin{equation}
w_{2}(AB)=2^{R+1}m+3,\ w_{2}(BC)=3,\ w_{2}(CD)=2^{R+1}m+(K-5),\ w_{2}(EF)=(K+1),\nonumber
\end{equation}
which is of the form \eqref{eq:H_{1}_losing_0,1,3} with $r=3$. If $w_{1}(BC)=2$, $P_{2}$ removes weight $8$ from the edge $EF$ in the second round, leaving $P_{1}$ with
\begin{equation}
w_{2}(AB)=2^{R+1}m+4,\ w_{2}(BC)=2,\ w_{2}(CD)=2^{R+1}m+(K-5),\ w_{2}(EF)=(K-7),\nonumber
\end{equation}
which is of the form \eqref{eq:H_{1}_losing_2,4_case_1} with $r=4$ and $s=2$. If $w_{1}(BC)=1$, $P_{2}$ removes weight $9$ from the edge $EF$ in the second round, leaving $P_{1}$ with
\begin{equation}
w_{2}(AB)=2^{R+1}m+4,\ w_{2}(BC)=1,\ w_{2}(CD)=2^{R+1}m+(K-5),\ w_{2}(EF)=(K-8),\nonumber
\end{equation}
which is of the form \eqref{eq:H_{1}_losing_2,4_case_1} with $r=4$ and $s=1$. By what we have already proved in \S\ref{subsec:H_{1}_losing_3_proof}, \S\ref{subsec:H_{1}_losing_4_1,2_proof} and \S\ref{subsec:H_{1}_losing_4_3,4_proof}, we conclude that $P_{1}$ loses in each of these cases. 

\item If $w_{1}(CD)=2^{R+1}m+L$ for $L \in \{4,5,\ldots,K-6\}$, with $L\equiv 4 \bmod 8$, and $w_{1}(BC)=3$, $P_{2}$ removes weight $(K+2-L)$ from $EF$ in the second round, leaving $P_{1}$ with
\begin{align}
{}&w_{2}(AB)=2^{f(L-1)+1}2^{R-f(L-1)}m+4,\ w_{2}(BC)=3,\nonumber\\
{}&w_{2}(CD)=2^{f(L-1)+1}2^{R-f(L-1)}m+L,\ w_{2}(EF)=L-1,\nonumber
\end{align}
which is of the form \eqref{eq:H_{1}_losing_2,4_case_1} with $r=4$ and $s=3$, and since $w_{2}(EF)=(L-1)<K$, $P_{1}$ loses by our induction hypothesis. This case covers $w_{1}(CD)=2^{R+1}m+(K-6)$ and $w_{1}(BC)=3$.

\item If $w_{1}(CD)=2^{R+1}m+L$ for $L \in \{4,5,\ldots,K-7\}$, with $L \equiv i \bmod 8$ for $i \in \{0,1,2,3,5,6,7\}$, and $w_{1}(BC)=3$, $P_{2}$ removes weight $(K-6-L)$ from $EF$ in the second round, leaving $P_{1}$ with
\begin{align}
{}&w_{2}(AB)=2^{f(L+7)+1}2^{R-f(L+7)}m+4,\ w_{2}(BC)=3,\nonumber\\
{}&w_{2}(CD)=2^{f(L+7)+1}2^{R-f(L+7)}m+L,\ w_{2}(EF)=L+7,\nonumber
\end{align}
which is of the form \eqref{eq:H_{1}_losing_2,4_case_2} with $r=4$ and $s=3$, and since $w_{2}(EF)=(L+7)\leqslant K$, $P_{1}$ loses by our induction hypothesis. 

\item If $w_{1}(CD)=2^{R+1}m+L$ for $L \in \{4,5,\ldots,K-6\}$, with $L \equiv j \bmod 8$ for $j \in \{4,5\}$, and $w_{1}(BC)=2$, $P_{2}$ removes weight $(K+3-L)$ from $EF$ in the second round, leaving $P_{1}$ with 
\begin{align}
{}&w_{2}(AB)=2^{f(L-2)+1}2^{R-f(L-2)}m+4,\ w_{2}(BC)=2,\nonumber\\
{}&w_{2}(CD)=2^{f(L-2)+1}2^{R-f(L-2)}m+L,\ w_{2}(EF)=(L-2),\nonumber
\end{align}
which is of the form \eqref{eq:H_{1}_losing_2,4_case_1} with $r=4$ and $s=2$, and by what we have already proved in \S\ref{subsec:H_{1}_losing_4_3,4_proof}, we conclude that $P_{1}$ loses. 

\item If $w_{1}(CD)=2^{R+1}m+L$ for $L \in \{4,5,\ldots,K-6\}$, with $L \equiv j \bmod 8$ for $j \in \{0,1,2,3,6,7\}$, and $w_{1}(BC)=2$, $P_{2}$ removes weight $(K-5-L)$ from $EF$ in the second round, leaving $P_{1}$ with
\begin{align}
{}&w_{2}(AB)=2^{f(L+6)+1}2^{R-f(L+6)}m+4,\ w_{2}(BC)=2,\nonumber\\
{}&w_{2}(CD)=2^{f(L+6)+1}2^{R-f(L+6)}m+L,\ w_{2}(EF)=(L+6),\nonumber
\end{align}
which is of the form \eqref{eq:H_{1}_losing_2,4_case_2} with $r=4$ and $s=2$, and by what we have already proved in \S\ref{subsec:H_{1}_losing_4_3,4_proof}, we know that $P_{1}$ loses.

\item If $w_{1}(CD)=2^{R+1}m+L$ for $L \in \{4,5,\ldots,K-6\}$, with $L \equiv j \bmod 8$ for $j \in \{4,5,6\}$, and $w_{1}(BC)=1$, $P_{2}$ removes weight $(K+4-L)$ from $EF$ in the second round, leaving $P_{1}$ with
\begin{align}
{}&w_{2}(AB)=2^{f(L-3)+1}2^{R-f(L-3)}m+4,\ w_{2}(BC)=1,\nonumber\\
{}&w_{2}(CD)=2^{f(L-3)+1}2^{R-f(L-3)}m+L,\ w_{2}(EF)=(L-3),\nonumber
\end{align}
which is of the form \eqref{eq:H_{1}_losing_2,4_case_1} with $r=4$ and $s=1$, and by what we have already proved in \S\ref{subsec:H_{1}_losing_4_1,2_proof}, we know that $P_{1}$ loses.

\item If $w_{1}(CD)=2^{R+1}m+L$ for $L \in \{4,5,\ldots,K-6\}$, with $L \equiv j \bmod 8$ for $j \in \{0,1,2,3,7\}$, and $w_{1}(BC)=1$, $P_{2}$ removes weight $(K-4-L)$ from $EF$ in the second round, leaving $P_{1}$ with 
\begin{align}
{}&w_{2}(AB)=2^{f(L+5)+1}2^{R-f(L+5)}m+4,\ w_{2}(BC)=1,\nonumber\\
{}&w_{2}(CD)=2^{f(L+5)+1}2^{R-f(L+5)}m+L,\ w_{2}(EF)=(L+5),\nonumber
\end{align}
which is of the form \eqref{eq:H_{1}_losing_2,4_case_2} with $r=4$ and $s=1$, and by what we have already proved in \S\ref{subsec:H_{1}_losing_4_1,2_proof}, we know that $P_{1}$ loses.

\item If $w_{1}(CD)=2^{R+1}m+L$ for $L \in \{0,1,2,3\}$ and $w_{1}(BC)=t\in\{1,2,3\}$, $P_{2}$ removes weight $(K+1)-(4+t+L)$ from $EF$ in the second round, leaving $P_{1}$ with
\begin{equation}
w_{2}(AB)=2^{R+1}m+4,\ w_{2}(BC)=t,\ w_{2}(CD)=2^{R+1}m+L,\ w_{2}(EF)=4+t+L,\nonumber
\end{equation}
which is of the form \eqref{eq:H_{1}_losing_0,1,3} with either $r=0$ (when $L=0$) or $r=1$ (when $L=1$) or $r=3$ (when $L=3$), or of the form \eqref{eq:H_{1}_losing_2,4_case_2} with $r=2$ and $s=1$ (when $L=2$ and $t=1$), or of the form \eqref{eq:H_{1}_losing_2,4_case_3} with $r=2$ and $s=t\geqslant 2$ (when $L=2$ and $t \geqslant 2$). By what we have already proved in \S\ref{subsec:proof_eq_H_{1}_losing_0}, \S\ref{subsec:H_{1}_losing_1_proof}, \S\ref{subsec:H_{1}_losing_3_proof}, \S\ref{subsec:proof_eq_H_{1}_losing_2_1,2} and \S\ref{subsec:proof_eq_H_{1}_losing_2_3}, we know that $P_{1}$ loses in each of these cases. 

\item Finally, if $w_{1}(CD)=2^{R+1}n+\ell_{1}$ for some $n<m$ and $\ell_{1}\in\{0,1,\ldots,2^{R+1}-1\}$, the configuration after the first round is of the form given by Lemma~\ref{prop:H_{1}_winning_1} with $m_{1}=n$, $m_{2}=m$ and $\ell_{2}=4$. Since $n<m \implies m_{1}\neq m_{2}$, $P_{2}$ wins by Lemma~\ref{prop:H_{1}_winning_1}.
\end{enumerate}

\item Suppose $P_{1}$ removes a positive integer weight from $AB$, and a non-negative integer weight from $BC$, in the first round. Once again, we consider the various possibilities:
\begin{enumerate}
\item If $w_{1}(AB)=2^{R+1}m+\ell$ for $\ell \in \{0,1,2,3\}$, and $w_{1}(BC)=t\in\{1,2,3\}$, $P_{2}$ removes weight $(\ell+t+1)$ from $CD$ in the second round, leaving $P_{1}$ with
\begin{align}
w_{2}(AB)=2^{R+1}m+\ell,\ w_{2}(BC)=t,\ w_{2}(CD)=2^{R+1}m+K+1-\ell-t,\ w_{2}(EF)=K+1,\nonumber
\end{align}
which is of the form \eqref{eq:H_{1}_losing_0,1,3} with $r=0$ (when $\ell=0$) or $r=1$ (when $\ell=1$) or $r=3$ (when $\ell=3$), of the form \eqref{eq:H_{1}_losing_2,4_case_2} with $r=2$ and $s=1$ (when $\ell=2$ and $t=1$), or of the form \eqref{eq:H_{1}_losing_2,4_case_3} with $r=2$ and $s\geqslant 2$ (when $\ell=2$ and $t \geqslant 2$). By what we have already proved in \S\ref{subsec:proof_eq_H_{1}_losing_0}, \S\ref{subsec:H_{1}_losing_1_proof}, \S\ref{subsec:H_{1}_losing_3_proof}, \S\ref{subsec:proof_eq_H_{1}_losing_2_1,2} and \S\ref{subsec:proof_eq_H_{1}_losing_2_3}, we know that $P_{1}$ loses in each of these cases. 

\item If $w_{1}(AB)=2^{R+1}m+\ell$ for $\ell \in \{0,1,2,3\}$, and $w_{1}(BC)=0$, $P_{2}$, after the first round, is left with a galaxy graph consisting of $AB$, $CD$ and $EF$, where $w_{1}(AB)=2^{R+1}m+\ell$, $w_{1}(CD)=2^{R+1}m+(K+2)$ and $w_{1}(EF)=(K+1)$. Writing $K=8n+2$ for some $n \in \mathbb{N}$, since the triple $\left(2^{R+1}m+\ell,2^{R+1}m+8n+3-\ell,8n+3\right)$ is balanced, $P_{2}$ wins by Theorem~\ref{thm:galaxy}.

\item Finally, if $w_{1}(AB)=2^{R+1}n+\ell_{1}$ for some $n<m$ and $\ell_{1}\in\{0,1,\ldots,2^{R+1}-1\}$, the configuration after the first round is of the form given by Lemma~\ref{prop:H_{1}_winning_1} with $m_{1}=n$, $m_{2}=m$ and $\ell_{2}=K+2$. Since $n<m \implies m_{1}\neq m_{2}$, $P_{2}$ wins by Lemma~\ref{prop:H_{1}_winning_1}.
\end{enumerate}

\item Suppose $P_{1}$ removes a positive integer weight from $EF$ in the first round, so that $w_{1}(EF)=k \leqslant K$. If $k=0$, $P_{2}$ wins by Remark~\ref{rem:EF_edgeweight_0}. If $k\in\{1,2,3,4\}$, the configuration after the first round is of the form mentioned in Lemma~\ref{prop:H_{1}_winning_1}, with $\ell_{1}=4$ and $\ell_{2}=(K+2)$, so that $w_{1}(EF)=k$ satisfies the inequality $k \leqslant \min\{\ell_{1},\ell_{2}\}$. Consequently, $P_{2}$ wins by Lemma~\ref{prop:H_{1}_winning_1}. If $k \in \{5,6,\ldots,K\}$ and $k \equiv 3 \bmod 8$, then $P_{2}$ removes weight $(K+1-k)$ from $CD$ in the second round, leaving $P_{1}$ with
\begin{equation}
w_{2}(AB)=2^{R+1}m+4,\ w_{2}(BC)=3,\ w_{2}(CD)=2^{R+1}m+(k+1),\ w_{2}(EF)=k,\nonumber
\end{equation}
which is of the form \eqref{eq:H_{1}_losing_2,4_case_1} with $r=4$, $s=3$ and $w_{2}(EF)=k\leqslant K$, so that $P_{1}$ loses by our induction hypothesis. If $k \in\{7,8,\ldots,K\}$ and $k \equiv j \bmod 8$ for some $j \in \{0,1,2,4,5,6,7\}$, then $P_{2}$ removes weight $(K+9-k)$ from $CD$ in the second round, leaving $P_{1}$ with
\begin{equation}
w_{2}(AB)=2^{R+1}m+4,\ w_{2}(BC)=3,\ w_{2}(CD)=2^{R+1}m+(k-7),\ w_{2}(EF)=k,\label{intermediate_eq_1}
\end{equation}
which, when $k\geqslant 12$, is of the form \eqref{eq:H_{1}_losing_2,4_case_2} with $r=4$, $s=3$ and $w_{2}(EF)=k\leqslant K$, so that $P_{1}$ loses by our induction hypothesis. If $k \in \{7,8,9,10\}$, the configuration in \eqref{intermediate_eq_1} is of the form \eqref{eq:H_{1}_losing_0,1,3} with either $r=0$ (when $k=7$) or $r=1$ (when $k=8$) or $r=3$ (when $k=10$), or of the form \eqref{eq:H_{1}_losing_2,4_case_3} with $r=2$ and $s=3$ (when $k=9$). If $k\in\{5,6\}$, $P_{2}$ removes weight $(K+2)$ from the edge $CD$ and weight $(7-k)$ from the edge $BC$ in the second round, leaving $P_{1}$ with
\begin{equation}
w_{2}(AB)=2^{R+1}m+4,\ w_{2}(BC)=k-4,\ w_{2}(CD)=2^{R+1}m,\ w_{2}(EF)=k,\nonumber
\end{equation}
which is of the form \eqref{eq:H_{1}_losing_0,1,3} with $r=0$. By what we have already proved in \S\ref{subsec:proof_eq_H_{1}_losing_0}, \S\ref{subsec:H_{1}_losing_1_proof}, \S\ref{subsec:H_{1}_losing_3_proof} and \S\ref{subsec:proof_eq_H_{1}_losing_2_3}, we know that $P_{1}$ loses in each of these cases.

\item Suppose $P_{1}$ removes a positive integer weight from $BC$ in the first round, without disturbing the edge-weights of $AB$ and $CD$, so that $w_{1}(BC)=t\in\{0,1,2\}$. If $t\in\{1,2\}$, $P_{2}$ removes weight $(3-t)$ from $EF$ in the second round, leaving $P_{1}$ with
\begin{equation}
w_{2}(AB)=2^{R+1}m+4,\ w_{2}(BC)=t,\ w_{2}(CD)=2^{R+1}m+(K+2),\ w_{2}(EF)=(K-2+t),\nonumber
\end{equation}
which is of the form \eqref{eq:H_{1}_losing_2,4_case_1} with either $r=4$ and $s=2$ (when $t=2$) or $r=4$ and $s=1$ (when $t=1$). By what we have already proved in \S\ref{subsec:H_{1}_losing_4_1,2_proof} and \S\ref{subsec:H_{1}_losing_4_3,4_proof}, we know that $P_{1}$ loses. If $t=0$, $P_{2}$, at the end of the first round, is left with a galaxy graph consisting of the edges $AB$, $CD$ and $EF$, where $w_{1}(AB)=2^{R+1}m+4$, $w_{1}(CD)=2^{R+1}m+(K+2)$ and $w_{1}(EF)=(K+1)$. Writing $K=8n+2$ for some $n \in \mathbb{N}$, since the triple $(2^{R+1}m+4,2^{R+1}m+8n+4,8n)$ is balanced, $P_{2}$ wins by Theorem~\ref{thm:galaxy}. 
\end{enumerate}

This completes the proof of our claim that the configuration in \eqref{eq:H_{1}_losing_4_6_inductive} is losing on $H_{1}$.

\subsubsection{When $K\geqslant 12$ and $K \equiv i \bmod 8$ for some $i \in \{0,1,3,4,5,6,7\}$} For $m \in \mathbb{N}_{0}$, we focus on
\begin{equation}
w_{0}(AB)=2^{R+1}m+4,\ w_{0}(BC)=3,\ w_{0}(CD)=2^{R+1}m+(K-6),\ w_{0}(EF)=(K+1),\label{eq:H_{1}_losing_4_5_inductive}
\end{equation}
where $R=f(K+1)$. The first round of the game played on this initial configuration unfolds as follows:
\begin{enumerate}
\item Suppose $P_{1}$ removes a positive integer weight from $CD$, and a non-negative integer weight from $BC$, in the first round. The following cases are possible:
\begin{enumerate}
\item If $w_{1}(CD)=2^{R+1}m+L$ for $L \in \{4,5,\ldots,K-7\}$, with $L \equiv j \bmod 8$ for $j \in \{0,1,2,3,5,6,7\}$, and $w_{1}(BC)=3$, $P_{2}$ removes weight $(K-6-L)$ from $EF$ in the second round, leaving $P_{1}$ with 
\begin{align}
{}&w_{2}(AB)=2^{f(L+7)+1}2^{R-f(L+7)}m+4,\ w_{2}(BC)=3,\nonumber\\
{}&w_{2}(CD)=2^{f(L+7)+1}2^{R-f(L+7)}m+L,\ w_{2}(EF)=(L+7),\nonumber
\end{align}
which is of the form \eqref{eq:H_{1}_losing_2,4_case_2} with $r=4$, $s=3$ and $w_{2}(EF)=(L+7) \leqslant K$, so that $P_{1}$ loses by our induction hypothesis. 

\item If $w_{1}(CD)=2^{R+1}m+L$ for $L \in \{4,5,\ldots,K-7\}$, with $L \equiv 4 \bmod 8$, and $w_{1}(BC)=3$, $P_{2}$ removes weight $(K+2-L)$ from $EF$ in the second round, leaving $P_{1}$ with
\begin{align}
{}&w_{2}(AB)=2^{f(L-1)+1}2^{R-f(L-1)}m+4,\ w_{2}(BC)=3,\nonumber\\
{}&w_{2}(CD)=2^{f(L-1)+1}2^{R-f(L-1)}m+L,\ w_{2}(EF)=(L-1),\nonumber
\end{align}
which is of the form \eqref{eq:H_{1}_losing_2,4_case_1} with $r=4$, $s=3$ and $w_{2}(EF)=(L-1)\leqslant (K-8)<K$, so that $P_{1}$ loses by our induction hypothesis.

\item If $w_{1}(CD)=2^{R+1}m+L$ for $L \in \{4,5,\ldots,K-7\}$, with $L \equiv j \bmod 8$ for $j \in \{0,1,2,3,6,7\}$, and $w_{1}(BC)=2$, $P_{2}$ removes weight $(K-5-L)$ from $EF$ in the second round, leaving $P_{1}$ with
\begin{align}
{}&w_{2}(AB)=2^{f(L+6)+1}2^{R-f(L+6)}m+4,\ w_{2}(BC)=2,\nonumber\\
{}&w_{2}(CD)=2^{f(L+6)+1}2^{R-f(L+6)}m+L,\ w_{2}(EF)=(L+6),\nonumber
\end{align}
which is of the form \eqref{eq:H_{1}_losing_2,4_case_2} with $r=4$ and $s=2$, and by what we have already proved in \S\ref{subsec:H_{1}_losing_4_3,4_proof}, we conclude that $P_{1}$ loses.

\item If $w_{1}(CD)=2^{R+1}m+L$ for $L \in \{4,5,\ldots,K-7\}$, with $L \equiv j \bmod 8$ for $j \in \{4,5\}$, and $w_{1}(BC)=2$, $P_{2}$ removes weight $(K+3-L)$ from $EF$ in the second round, leaving $P_{1}$ with
\begin{align}
{}&w_{2}(AB)=2^{f(L-2)+1}2^{R-f(L-2)}m+4,\ w_{2}(BC)=2,\nonumber\\
{}&w_{2}(CD)=2^{f(L-2)+1}2^{R-f(L-2)}m+L,\ w_{2}(EF)=(L-2),\nonumber
\end{align}
which is of the form \eqref{eq:H_{1}_losing_2,4_case_1} with $r=4$ and $s=2$, and by what we have already proved in \S\ref{subsec:H_{1}_losing_4_3,4_proof}, we conclude that $P_{1}$ loses. 

\item If $w_{1}(CD)=2^{R+1}m+L$ for $L \in \{4,5,\ldots,K-7\}$, with $L \equiv j \bmod 8$ for $j \in \{0,1,2,3,7\}$, and $w_{1}(BC)=1$, $P_{2}$ removes weight $(K-4-L)$ from $EF$ in the second round, leaving $P_{1}$ with
\begin{align}
{}&w_{2}(AB)=2^{f(L+5)+1}2^{R-f(L+5)}m+4,\ w_{2}(BC)=1,\nonumber\\
{}&w_{2}(CD)=2^{f(L+5)+1}2^{R-f(L+5)}m+L,\ w_{2}(EF)=(L+5),\nonumber
\end{align}
which is of the form \eqref{eq:H_{1}_losing_2,4_case_2} with $r=4$ and $s=1$, and by what we have already proved in \S\ref{subsec:H_{1}_losing_4_1,2_proof}, we conclude that $P_{1}$ loses.

\item If $w_{1}(CD)=2^{R+1}m+L$ for $L \in \{4,5,\ldots,K-7\}$, with $L \equiv j \bmod 8$ for $j \in \{4,5,6\}$, and $w_{1}(BC)=1$, $P_{2}$ removes weight $(K+4-L)$ from $EF$ in the second round, leaving $P_{1}$ with
\begin{align}
{}&w_{2}(AB)=2^{f(L-3)+1}2^{R-f(L-3)}m+4,\ w_{2}(BC)=1,\nonumber\\
{}&w_{2}(CD)=2^{f(L-3)+1}2^{R-f(L-3)}m+L,\ w_{2}(EF)=(L-3),\nonumber
\end{align}
which is of the form \eqref{eq:H_{1}_losing_2,4_case_1} with $r=4$ and $s=1$, and by what we have already proved in \S\ref{subsec:H_{1}_losing_4_1,2_proof}, we conclude that $P_{1}$ loses.

\item If $w_{1}(CD)=2^{R+1}m+L$ for $L \in \{0,1,2,3\}$, and $w_{1}(BC)=t\in\{1,2,3\}$, $P_{2}$ removes weight $(K-3-t-L)$ from $EF$ in the second round, leaving $P_{1}$ with
\begin{equation}
w_{2}(AB)=2^{R+1}m+4,\ w_{2}(BC)=t,\ w_{2}(CD)=2^{R+1}m+L,\ w_{2}(EF)=(4+t+L),\nonumber
\end{equation}
which is of the form \eqref{eq:H_{1}_losing_0,1,3} with either $r=0$ (when $L=0$) or $r=1$ (when $L=1$) or $r=3$ (when $L=3$), or of the form \eqref{eq:H_{1}_losing_2,4_case_2} with $r=2$ and $s=1$ (when $L=2$ and $t=1$), or of the form \eqref{eq:H_{1}_losing_2,4_case_3} with $r=2$ and $s\geqslant 2$ (when $L=2$ and $t \geqslant 2$). By what we have already proved in \S\ref{subsec:proof_eq_H_{1}_losing_0}, \S\ref{subsec:H_{1}_losing_1_proof}, \S\ref{subsec:H_{1}_losing_3_proof}, \S\ref{subsec:proof_eq_H_{1}_losing_2_1,2} and \S\ref{subsec:proof_eq_H_{1}_losing_2_3}, we conclude that $P_{1}$ loses in each of these cases.

\item If $w_{1}(CD)=2^{R+1}m+L$ for $L \in \{0,1,\ldots,K-7\}$, and $w_{1}(BC)=0$, $P_{2}$, after the first round, is left with a galaxy graph, consisting of the edges $AB$, $CD$ and $EF$,with $w_{1}(AB)=2^{R+1}m+4$, $w_{1}(CD)=2^{R+1}m+L$ and $w_{1}(EF)=(K+1)$, and $P_{2}$ wins by Lemma~\ref{lem:H_{1}_winning_galaxy}.

\item Finally, if $w_{1}(CD)=2^{R+1}n+\ell_{1}$ for some $n<m$ and $\ell_{1}\in\{0,1,\ldots,2^{R+1}-1\}$, the configuration after the first round is of the form given by Lemma~\ref{prop:H_{1}_winning_1} with $m_{1}=n$, $m_{2}=m$ and $\ell_{2}=4$. Since $n<m \implies m_{1}\neq m_{2}$, $P_{2}$ wins by Lemma~\ref{prop:H_{1}_winning_1}.
\end{enumerate}

\item Suppose $P_{1}$ removes a positive integer weight from $AB$, and a non-negative integer weight from $BC$, in the first round. The following scenarios are possible:
\begin{enumerate}
\item Suppose $w_{1}(AB)=2^{R+1}m+\ell$ for some $\ell \in \{0,1,2,3\}$ and $w_{1}(BC)=t\in\{1,2,3\}$. As long as
\begin{enumerate*}
\item $\ell \in \{0,1,3\}$,
\item or $\ell=2$ and $t\in\{2,3\}$,
\item or $\ell=2$, $t=1$ and $K \equiv i \bmod 8$ for some $i \in \{1,3,5,6,7\}$,
\end{enumerate*}
$P_{2}$ removes weight $(7-\ell-t)$ from $EF$ in the second round, leaving $P_{1}$ with
\begin{align}
w_{2}(AB)=2^{R+1}m+\ell,\ w_{2}(BC)=t,\ w_{2}(CD)=2^{R+1}m+K-6,\ w_{2}(EF)=K+\ell+t-6,\nonumber
\end{align}
which is of the form \eqref{eq:H_{1}_losing_0,1,3} with $r=0$ (when $\ell=0$) or $r=1$ (when $\ell=1$) or $r=3$ (when $\ell=3$), or of the form \eqref{eq:H_{1}_losing_2,4_case_2} with $r=2$ and $s=1$ (when $\ell=2$, $t=1$ and $K \equiv i \bmod 8$ for some $i \in \{1,3,5,6,7\}$), or of the form \eqref{eq:H_{1}_losing_2,4_case_3} with $r=2$ and $s\geqslant 2$ (when $\ell=2$ and $t\in\{2,3\}$). By what we have already proved in \S\ref{subsec:proof_eq_H_{1}_losing_0}, \S\ref{subsec:H_{1}_losing_1_proof}, \S\ref{subsec:H_{1}_losing_3_proof}, \S\ref{subsec:proof_eq_H_{1}_losing_2_1,2} and \S\ref{subsec:proof_eq_H_{1}_losing_2_3}, we conclude that $P_{1}$ loses. 

If $\ell=2$, $t=1$ and $K \equiv i \bmod 8$ for some $i \in\{0,4\}$, $P_{2}$ removes weight $8$ from $EF$ in the second round, leaving $P_{1}$ with
\begin{equation}
w_{2}(AB)=2^{R+1}m+2,\ w_{2}(BC)=1,\ w_{2}(CD)=2^{R+1}m+K-6,\ w_{2}(EF)=K-7,\nonumber
\end{equation}
which is of the form \eqref{eq:H_{1}_losing_2,4_case_1} with $r=2$ and $s=1$, and by what we have already proved in \S\ref{subsec:proof_eq_H_{1}_losing_2_1,2}, we conclude that $P_{1}$ loses. 

\item If $w_{1}(AB)=2^{R+1}m+\ell$ for $\ell\in\{0,1,2,3\}$, and $w_{1}(BC)=0$, $P_{2}$, at the end of the first round, is left with a galaxy graph consisting of the edges $AB$, $CD$ and $EF$, with $w_{1}(AB)=2^{R+1}m+\ell$, $w_{1}(CD)=2^{R+1}m+(K-6)$ and $w_{1}(EF)=(K+1)$, and she wins by Lemma~\ref{lem:H_{1}_winning_galaxy}.

\item Finally, if $w_{1}(AB)=2^{R+1}n+\ell_{1}$ for some $n<m$ and $\ell_{1}\in\{0,1,\ldots,2^{R+1}-1\}$, the configuration after the first round is of the form given by Lemma~\ref{prop:H_{1}_winning_1} with $m_{1}=n$, $m_{2}=m$ and $\ell_{2}=K-6$. Since $n<m \implies m_{1}\neq m_{2}$, $P_{2}$ wins by Lemma~\ref{prop:H_{1}_winning_1}.
\end{enumerate}

\item Suppose $P_{1}$ removes a positive integer weight from the edge $EF$ in the first round, so that $w_{1}(EF)=k\leqslant K$. If $k=0$, $P_{2}$ wins by Remark~\ref{rem:EF_edgeweight_0}. If $k \in \{1,2,3,4,K-6\}$, the configuration after the first round is of the form mentioned in Lemma~\ref{prop:H_{1}_winning_1}, with $\ell_{1}=4$, $\ell_{2}=(K-6)$ and $w_{1}(EF)=k$ satisfying either $\min\{\ell_{1},\ell_{2}\}\geqslant k$ or $k \in \{\ell_{1},\ell_{2}\}$. Therefore, $P_{2}$ wins by Lemma~\ref{prop:H_{1}_winning_1}. If $k \in \{K-3,K-2,K-1,K\}$, $P_{2}$ removes $(K+1-k)$ from $AB$ in the second round, leaving $P_{1}$ with 
\begin{equation}
w_{2}(AB)=2^{R+1}m+3+k-K,\ w_{2}(BC)=3,\ w_{2}(CD)=2^{R+1}m+K-6,\ w_{2}(EF)=k,\nonumber
\end{equation}
which is of the form \eqref{eq:H_{1}_losing_0,1,3} with either $r=0$ (when $k=K-3$) or $r=1$ (when $k=K-2$) or $r=3$ (when $k=K$), or of the form \eqref{eq:H_{1}_losing_2,4_case_3} with $r=2$ and $s=3$ (when $k=K-1$). If $k \in \{K-5,K-4\}$, $P_{2}$ removes weight $4$ from $AB$ and weight $(K-k-3)$ from $BC$ in the second round, leaving $P_{1}$ with
\begin{equation}
w_{2}(AB)=2^{R+1}m,\ w_{2}(BC)=k+6-K,\ w_{2}(CD)=2^{R+1}m+K-6,\ w_{2}(EF)=k,\nonumber
\end{equation}
which is of the form \eqref{eq:H_{1}_losing_0,1,3} with $r=0$. By what we have already proved in \S\ref{subsec:proof_eq_H_{1}_losing_0}, \S\ref{subsec:H_{1}_losing_1_proof}, \S\ref{subsec:H_{1}_losing_3_proof} and \S\ref{subsec:proof_eq_H_{1}_losing_2_3}, we conclude that $P_{1}$ loses in each of these cases.

If $k \in \{7,8,9,10,12,13,\ldots,K-7\}$ with $k \equiv j \bmod 8$ for $j \in \{0,1,2,4,5,6,7\}$, $P_{2}$ removes weight $(K+1-k)$ from $CD$ in the second round, leaving $P_{1}$ with
\begin{equation}
w_{2}(AB)=2^{R+1}m+4,\ w_{2}(BC)=3,\ w_{2}(CD)=2^{R+1}m+(k-7),\ w_{2}(EF)=k,\label{intermediate_eq_2}
\end{equation}
which, for $k\geqslant 12$, is of the form \eqref{eq:H_{1}_losing_2,4_case_2} with $r=4$ and $s=3$, and as $w_{2}(EF)=k \leqslant K$, $P_{1}$ loses by our induction hypothesis. The configuration in \eqref{intermediate_eq_2} is of the form \eqref{eq:H_{1}_losing_0,1,3} with either $r=0$ (when $k=7$) or $r=1$ (when $k=8$) or $r=3$ (when $k=10$), or of the form \eqref{eq:H_{1}_losing_2,4_case_3} with $r=2$ and $s=3$ (when $k=9$). Note that the case of $k=(K-7)$ is taken care of here. If $k \in \{11,12,\ldots,K-8\}$ with $k \equiv 3 \bmod 8$, $P_{2}$ removes weight $(K-7-k)$ from $CD$ in the second round, leaving $P_{1}$ with 
\begin{equation}
w_{2}(AB)=2^{R+1}m+4,\ w_{2}(BC)=3,\ w_{2}(CD)=2^{R+1}m+(k+1),\ w_{2}(EF)=k,\nonumber
\end{equation}
which is of the form \eqref{eq:H_{1}_losing_2,4_case_1} with $r=4$ and $s=3$, and as $w_{2}(EF)=k \leqslant K$, $P_{1}$ loses by our induction hypothesis. If $k\in\{5,6\}$, $P_{2}$ removes weight $(K-6)$ from $CD$, and weight $(7-k)$ from $BC$, in the second round, leaving $P_{1}$ with
\begin{equation}
w_{2}(AB)=2^{R+1}m+4,\ w_{2}(BC)=k-4,\ w_{2}(CD)=2^{R+1}m,\ w_{2}(EF)=k,\nonumber
\end{equation}
which is of the form \eqref{eq:H_{1}_losing_0,1,3} with $r=0$. By what we have already proved in \S\ref{subsec:proof_eq_H_{1}_losing_0}, \S\ref{subsec:H_{1}_losing_1_proof}, \S\ref{subsec:H_{1}_losing_3_proof} and \S\ref{subsec:proof_eq_H_{1}_losing_2_3}, we conclude that $P_{1}$ loses in each of these cases.

\item Finally, suppose $P_{1}$ removes a positive integer weight from $BC$ in the first round, without disturbing the edge-weights of $AB$ and $CD$, so that $w_{1}(BC)=t\in\{0,1,2\}$. If $t=2$ and $K\equiv i \bmod 8$ for $i \in \{0,1,4,5,6,7\}$, $P_{2}$ removes weight $1$ from $EF$ in the second round, so that $P_{1}$ is left with
\begin{equation}
w_{2}(AB)=2^{R+1}m+4,\ w_{2}(BC)=2,\ w_{2}(CD)=2^{R+1}m+(K-6),\ w_{2}(EF)=K,\nonumber
\end{equation}
which is of the form \eqref{eq:H_{1}_losing_2,4_case_2} with $r=4$ and $s=2$. If $t=2$ and $K \equiv 3 \bmod 8$, $P_{2}$ removes weight $9$ from $EF$ in the second round, so that $P_{1}$ is left with
\begin{equation}
w_{2}(AB)=2^{R+1}m+4,\ w_{2}(BC)=2,\ w_{2}(CD)=2^{R+1}m+(K-6),\ w_{2}(EF)=(K-8),\nonumber
\end{equation}
which is of the form \eqref{eq:H_{1}_losing_2,4_case_1} with $r=4$ and $s=1$. By what we have already proved in \S\ref{subsec:H_{1}_losing_4_3,4_proof}, we conclude that $P_{1}$ loses in each of these cases.

If $t=1$ and $K\equiv i \bmod 8$ for $i \in \{0,1,5,6,7\}$, $P_{2}$ removes weight $2$ from $EF$ in the second round, leaving $P_{1}$ with
\begin{equation}
w_{2}(AB)=2^{R+1}m+4,\ w_{2}(BC)=1,\ w_{2}(CD)=2^{R+1}m+(K-6),\ w_{2}(EF)=(K-1),\nonumber
\end{equation}
which is of the form \eqref{eq:H_{1}_losing_2,4_case_2} with $r=4$ and $s=1$. If $t=1$ and $K \equiv i \bmod 8$ for $i \in \{3,4\}$, $P_{2}$ removes weight $10$ from $EF$ in the second round, so that $P_{1}$ is left with
\begin{equation}
w_{2}(AB)=2^{R+1}m+4,\ w_{2}(BC)=2,\ w_{2}(CD)=2^{R+1}m+(K-6),\ w_{2}(EF)=(K-9),\nonumber
\end{equation}
which is of the form \eqref{eq:H_{1}_losing_2,4_case_1} with $r=4$ and $s=1$. By what we have already proved in \S\ref{subsec:H_{1}_losing_4_1,2_proof}, we conclude that $P_{1}$ loses.

If $t=0$, $P_{2}$ is left with a galaxy graph at the end of the first round, consisting of edges $AB$, $CD$ and $EF$, with $w_{1}(AB)=2^{R+1}m+4$, $w_{1}(CD)=2^{R+1}m+(K-6)$ and $w_{1}(EF)=(K+1)$, and she wins by Lemma~\ref{lem:H_{1}_winning_galaxy}. 
\end{enumerate}
This completes the proof of our claim that the configuration in \eqref{eq:H_{1}_losing_4_5_inductive} is losing, thereby proving our claim that any configuration that is either of the form \eqref{eq:H_{1}_losing_2,4_case_1} or of the form \eqref{eq:H_{1}_losing_2,4_case_2}, with $r=4$ and $s=3$, is losing.

\subsection{Proof that any configuration of the form \eqref{eq:H_{1}_losing_2,4_case_3}, with $r=4$ and $s\geqslant 4$, is losing}\label{subsec:H_{1}_losing_4_7_proof}
The base case corresponding to \eqref{eq:H_{1}_losing_2,4_case_3} with $r=4$ is obtained by setting $s=4$ and $k=12$, which yields the configuration
\begin{equation}
w_{0}(AB)=w_{0}(CD)=16m+4,\ w_{0}(BC)=4,\ w_{0}(EF)=12.\label{eq:H_{1}_losing_4_7_base_case}
\end{equation}
We consider the first round of the game played on this initial configuration:
\begin{enumerate}
\item Suppose $P_{1}$ removes a positive integer weight from $CD$, and a non-negative integer weight from $BC$, in the first round (analogously, $P_{1}$ removes a positive integer weight from $AB$, and a non-negative integer weight from $BC$, in the first round).
\begin{enumerate}
\item If $w_{1}(CD)=16m+\ell$ for some $\ell \in \{0,1,2,3\}$, and $w_{1}(BC)=t\in\{1,2,3,4\}$, $P_{2}$ removes weight $(8-\ell-t)$ from the edge $EF$ in the second round, leaving $P_{1}$ with
\begin{equation}
w_{2}(AB)=16m+4,\ w_{2}(BC)=t,\ w_{2}(CD)=16m+\ell,\ w_{2}(EF)=4+\ell+t,\nonumber
\end{equation}
which is of the form \eqref{eq:H_{1}_losing_0,1,3} with either $r=0$ (when $\ell=0$) or $r=1$ (when $\ell=1$) or $r=3$ (when $\ell=3$), or of the form \eqref{eq:H_{1}_losing_2,4_case_2} with $r=2$ and $s=1$ (when $\ell=2$ and $t=1$), or of the form \eqref{eq:H_{1}_losing_2,4_case_3} with $r=2$ and $s\geqslant 2$ (when $\ell=2$ and $t\in\{2,3\}$). By what we have proved in \S\ref{subsec:proof_eq_H_{1}_losing_0}, \S\ref{subsec:H_{1}_losing_1_proof}, \S\ref{subsec:H_{1}_losing_3_proof}, \S\ref{subsec:proof_eq_H_{1}_losing_2_1,2} and \S\ref{subsec:proof_eq_H_{1}_losing_2_3}, we conclude that $P_{1}$ loses in each of these cases.

\item If $w_{1}(CD)=16m+\ell$ for $\ell \in \{0,1,2,3\}$, and $w_{1}(BC)=0$, $P_{2}$, at the end of the first round, is left with a galaxy graph that consists of the edges $AB$, $CD$ and $EF$, with $w_{1}(AB)=16m+4$, $w_{1}(CD)=16m+\ell$ and $w_{1}(EF)=12$, and she wins by Lemma~\ref{lem:H_{1}_winning_galaxy}.
\end{enumerate}

\item Suppose $P_{1}$ removes a positive integer weight from $EF$ in the second round, so that $w_{1}(EF)=k\leqslant 11$. If $k \in \{8,9,10,11\}$, $P_{2}$ removes weight $(12-k)$ from $CD$ in the second round, leaving $P_{1}$ with
\begin{equation}
w_{2}(AB)=16m+4,\ w_{2}(BC)=4,\ w_{2}(CD)=16m+(k-8),\ w_{2}(EF)=k,\nonumber
\end{equation}
which is of the form \eqref{eq:H_{1}_losing_0,1,3} with either $r=0$ (when $k=8$) or $r=1$ (when $k=9$) or $r=3$ (when $k=11$), or of the form \eqref{eq:H_{1}_losing_2,4_case_3} with $r=2$ and $s=4$ (when $k=10$). If $k \in \{5,6,7\}$, $P_{2}$ removes weight $4$ from $CD$ and weight $(8-k)$ from $BC$ in the second round, so that $P_{1}$ is left with
\begin{equation}
w_{2}(AB)=16m+4,\ w_{2}(BC)=(k-4),\ w_{2}(CD)=16m,\ w_{2}(EF)=k,\nonumber
\end{equation}
which is of the form \eqref{eq:H_{1}_losing_0,1,3} with $r=0$. By what we have proved already in \S\ref{subsec:proof_eq_H_{1}_losing_0}, \S\ref{subsec:H_{1}_losing_1_proof}, \S\ref{subsec:H_{1}_losing_3_proof} and \S\ref{subsec:proof_eq_H_{1}_losing_2_3}, we conclude that $P_{1}$ loses in each of these cases. If $k \in \{1,2,3,4\}$, the configuration at the end of the first round is of the form stated in Lemma~\ref{prop:H_{1}_winning_1}, with $\ell_{1}=\ell_{2}=4$ and $w_{1}(EF)=k$ satisfying $\min\{\ell_{1},\ell_{2}\}\geqslant k$, so that $P_{2}$ wins by Lemma~\ref{prop:H_{1}_winning_1}. If $k=0$, $P_{2}$ wins by Remark~\ref{rem:EF_edgeweight_0}.

\item Suppose $P_{1}$ removes a positive integer weight from $BC$ in the first round, without disturbing the edge-weights of $AB$ and $CD$, so that $w_{1}(BC)=t\in\{0,1,2,3\}$. For $t\in\{1,2,3\}$, $P_{2}$ removes weight $(12-t)$ from the edge $EF$ in the second round, so that $P_{1}$ is left with
\begin{equation}
w_{2}(AB)=w_{2}(CD)=16m+4=4(4m+1),\ w_{2}(BC)=w_{2}(EF)=t,\nonumber
\end{equation}
which is of the form \eqref{eq:H_{1}_losing_0,1,3} with $r=0$. By what we have proved in \S\ref{subsec:proof_eq_H_{1}_losing_0}, we conclude that $P_{1}$ loses. For $t=0$, $P_{2}$ removes $EF$ in the second round, and $P_{1}$ loses by Theorem~\ref{thm:galaxy}.
\end{enumerate}

This completes the proof of our claim that the configuration in \eqref{eq:H_{1}_losing_4_7_base_case} is losing on $H_{1}$. 

Suppose we have proved, for some $K \in \mathbb{N}$ with $K \geqslant 12$, that any configuration on $H_{1}$ that is of the form \eqref{eq:H_{1}_losing_2,4_case_3} with $r=4$ and $k\leqslant K$, is losing. We consider the first round of the game played on the configuration
\begin{equation}
w_{0}(AB)=2^{R+1}m+4,\ w_{0}(BC)=K-L-3,\ w_{0}(CD)=2^{R+1}m+L,\ w_{0}(EF)=K+1,\label{eq:H_{1}_losing_4_7_inductive}
\end{equation}
where $R=f(K+1)$, $L\in\{4,5,\ldots,K-7\}$ and $m\in\mathbb{N}_{0}$:
\begin{enumerate}
\item Suppose $P_{1}$ removes a positive integer weight from $CD$, and a non-negative integer weight from $BC$, in the first round. The following cases may arise:
\begin{enumerate}
\item If $w_{1}(CD)=2^{R+1}m+\ell$ for $\ell\in\{4,5,\ldots,L-1\}$ and $w_{1}(BC)=t\in\{4,5,\ldots,K-L-3\}$, $P_{2}$ removes weight $(K-3-\ell-t)$ from $EF$ in the second round, leaving $P_{1}$ with
\begin{align}
{}&w_{2}(AB)=2^{f(\ell+t+4)+1}2^{R-f(\ell+t+4)}m+4,\ w_{2}(BC)=t,\nonumber\\
{}&w_{2}(CD)=2^{f(\ell+t+4)+1}2^{R-f(\ell+t+4)}m+\ell,\ w_{2}(EF)=\ell+t+4,\nonumber
\end{align}
which is of the form \eqref{eq:H_{1}_losing_2,4_case_3} with $r=4$, $s\geqslant 4$ and $w_{2}(EF)=\ell+t+4\leqslant K$, so that $P_{1}$ loses by our induction hypothesis.

\item If $w_{1}(CD)=2^{R+1}m+\ell$ for $\ell \in \{4,5,\ldots,L-1\}$ with $\ell \equiv 4 \bmod 8$, and $w_{1}(BC)=3$, $P_{2}$ removes weight $(K+2-\ell)$ from $EF$ in the second round, leaving $P_{1}$ with
\begin{align}
{}&w_{2}(AB)=2^{f(\ell-1)+1}2^{R-f(\ell-1)}m+4,\ w_{2}(BC)=3,\nonumber\\
{}&w_{2}(CD)=2^{f(\ell-1)+1}2^{R-f(\ell-1)}m+\ell,\ w_{2}(EF)=\ell-1,\nonumber
\end{align}
which is of the form \eqref{eq:H_{1}_losing_2,4_case_1} with $r=4$ and $s=3$. If $\ell\equiv i \bmod 8$ for $i \in \{0,1,2,3,5,6,7\}$, $P_{2}$ removes weight $(K-6-\ell)$ from $EF$ in the second round, leaving $P_{1}$ with
\begin{align}
{}&w_{2}(AB)=2^{f(\ell+7)+1}2^{R-f(\ell+7)}m+4,\ w_{2}(BC)=3,\nonumber\\
{}&w_{2}(CD)=2^{f(\ell+7)+1}2^{R-f(\ell+7)}m+\ell,\ w_{2}(EF)=\ell+7,\nonumber
\end{align}
which is of the form \eqref{eq:H_{1}_losing_2,4_case_2} with $r=4$ and $s=3$. By what we have already proved in \S\ref{subsec:H_{1}_losing_4_5,6_proof}, we know that $P_{2}$ loses in each of these cases.

\item Suppose $w_{1}(CD)=2^{R+1}m+\ell$ for $\ell \in \{4,5,\ldots,L-1\}$ and $w_{1}(BC)=2$. If $\ell \equiv i\bmod 8$ for $i \in \{4,5\}$, $P_{2}$ removes weight $(K+3-\ell)$ from $EF$ in the second round, leaving $P_{1}$ with
\begin{align}
{}&w_{2}(AB)=2^{f(\ell-2)+1}2^{R-f(\ell-2)}m+4,\ w_{2}(BC)=2,\nonumber\\
{}&w_{2}(CD)=2^{f(\ell-2)+1}2^{R-f(\ell-2)}m+\ell,\ w_{2}(EF)=\ell-2,\nonumber
\end{align}
which is of the form \eqref{eq:H_{1}_losing_2,4_case_1} with $r=4$ and $s=2$. If $\ell \equiv i \bmod 8$ for $i \in \{0,1,2,3,6,7\}$, $P_{2}$ removes weight $(K-5-\ell)$ from $EF$ in the second round, leaving $P_{1}$ with
\begin{align}
{}&w_{2}(AB)=2^{f(\ell+6)+1}2^{R-f(\ell+6)}m+4,\ w_{2}(BC)=2,\nonumber\\
{}&w_{2}(CD)=2^{f(\ell+6)+1}2^{R-f(\ell+6)}m+\ell,\ w_{2}(EF)=\ell+6,\nonumber
\end{align}
which is of the form \eqref{eq:H_{1}_losing_2,4_case_2} with $r=4$ and $s=2$. By what we have already proved in \S\ref{subsec:H_{1}_losing_4_3,4_proof}, we know that $P_{2}$ loses in each of these cases.

\item Suppose $w_{1}(CD)=2^{R+1}m+\ell$ for some $\ell \in \{4,5,\ldots,L-1\}$ and $w_{1}(BC)=1$. If $\ell \equiv i\bmod 8$ for $i \in \{4,5,6\}$, $P_{2}$ removes weight $(K+4-\ell)$ from $EF$ in the second round, leaving $P_{1}$ with
\begin{align}
{}&w_{2}(AB)=2^{f(\ell-3)+1}2^{R-f(\ell-3)}m+4,\ w_{2}(BC)=1,\nonumber\\
{}&w_{2}(CD)=2^{f(\ell-3)+1}2^{R-f(\ell-3)}m+\ell,\ w_{2}(EF)=\ell-3,\nonumber
\end{align}
which is of the form \eqref{eq:H_{1}_losing_2,4_case_1} with $r=4$ and $s=1$. If $\ell \equiv i \bmod 8$ for $i \in \{0,1,2,3,7\}$, $P_{2}$ removes weight $(K-4-\ell)$ from $EF$ in the second round, leaving $P_{1}$ with
\begin{align}
{}&w_{2}(AB)=2^{f(\ell+5)+1}2^{R-f(\ell+5)}m+4,\ w_{2}(BC)=1,\nonumber\\
{}&w_{2}(CD)=2^{f(\ell+5)+1}2^{R-f(\ell+5)}m+\ell,\ w_{2}(EF)=\ell+5,\nonumber
\end{align}
which is of the form \eqref{eq:H_{1}_losing_2,4_case_2} with $r=4$ and $s=1$. By what we have already proved in \S\ref{subsec:H_{1}_losing_4_1,2_proof}, we know that $P_{2}$ loses in each of these cases.

\item If $w_{1}(CD)=2^{R+1}m+\ell$ for $\ell \in \{0,1,2,3\}$ and $w_{1}(BC)=t\in\{1,2,\ldots,K-L-3\}$, $P_{2}$ removes weight $(K-3-t-\ell)$ from $EF$ in the second round, leaving $P_{1}$ with
\begin{equation}
w_{2}(AB)=2^{R+1}m+4,\ w_{2}(BC)=t,\ w_{2}(CD)=2^{R+1}m+\ell,\ w_{2}(EF)=4+t+\ell,\nonumber
\end{equation}
which is of the form \eqref{eq:H_{1}_losing_0,1,3} with either $r=0$ (when $\ell=0$) or $r=1$ (when $\ell=1$) or $r=3$ (when $\ell=3$), or of the form \eqref{eq:H_{1}_losing_2,4_case_2} with $r=2$ and $s=1$ (when $\ell=2$ and $t=1$), or of the form \eqref{eq:H_{1}_losing_2,4_case_3} with $r=2$ and $s\geqslant 2$ (when $\ell=2$ and $t\in \{2,3,\ldots,K-L-3\}$). By what we have already proved in \S\ref{subsec:proof_eq_H_{1}_losing_0}, \S\ref{subsec:H_{1}_losing_1_proof}, \S\ref{subsec:H_{1}_losing_3_proof}, \S\ref{subsec:proof_eq_H_{1}_losing_2_1,2} and \S\ref{subsec:proof_eq_H_{1}_losing_2_3}, we know that $P_{2}$ loses in each of these cases.

\item If $w_{1}(CD)=2^{R+1}m+\ell$ for $\ell \in \{0,1,\ldots,L-1\}$ and $w_{1}(BC)=0$, $P_{2}$, after the first round, is left with a galaxy graph consisting of the edges $AB$, $CD$ and $EF$, with $w_{1}(AB)=2^{R+1}m+4$, $w_{1}(CD)=2^{R+1}m+\ell$ and $w_{1}(EF)=(K+1)$, and she wins by Lemma~\ref{lem:H_{1}_winning_galaxy}.

\item Finally, if $w_{1}(CD)=2^{R+1}n+\ell_{1}$ for some $n<m$ and $\ell_{1}\in\{0,1,\ldots,2^{R+1}-1\}$, the configuration after the first round is of the form given by Lemma~\ref{prop:H_{1}_winning_1} with $m_{1}=n$, $m_{2}=m$ and $\ell_{2}=4$. Since $n<m \implies m_{1}\neq m_{2}$, $P_{2}$ wins by Lemma~\ref{prop:H_{1}_winning_1}.
\end{enumerate}

\item Suppose $P_{1}$ removes a positive integer weight from $AB$, and a non-negative integer weight from $BC$, in the first round. The following cases may arise:
\begin{enumerate}
\item Suppose $w_{1}(AB)=2^{R+1}m+\ell$ for some $\ell \in \{0,1,2,3\}$ and $w_{1}(BC)=t\in\{1,2,\ldots,K-L-3\}$. As long as 
\begin{enumerate*}
\item either $\ell \in \{0,1,3\}$,
\item or $\ell=2$ and $t \in \{2,3,\ldots,K-L-3\}$,
\item or $\ell=2$, $t=1$ and $L\equiv j \bmod 4$ for some $j \in \{0,1,3\}$,
\end{enumerate*}
$P_{2}$ removes weight $(K+1)-(L+t+\ell)$ from $EF$ in the second round, leaving $P_{1}$ with
\begin{align}
{}&w_{2}(AB)=2^{f(L+t+\ell)+1}2^{R-f(L+t+\ell)}m+\ell,\ w_{2}(BC)=t,\nonumber\\
{}&w_{2}(CD)=2^{f(L+t+\ell)+1}2^{R-f(L+t+\ell)}m+L,\ w_{2}(EF)=L+t+\ell,\nonumber
\end{align}
which is of the form \eqref{eq:H_{1}_losing_0,1,3} with either $r=0$ (when $\ell=0$) or $r=1$ (when $\ell=1$) or $r=3$ (when $\ell=3$), or of the form \eqref{eq:H_{1}_losing_2,4_case_2} with $r=2$ and $s=1$ (when $\ell=2$, $t=1$ and $L \equiv j \bmod 4$ for some $j \in \{0,1,3\}$), or of the form \eqref{eq:H_{1}_losing_2,4_case_3} with $r=2$ and $s\geqslant 2$ (when $\ell=2$ and $t\in\{2,3,\ldots,K-L-3\}$). By what we have already proved in \S\ref{subsec:proof_eq_H_{1}_losing_0}, \S\ref{subsec:H_{1}_losing_1_proof}, \S\ref{subsec:H_{1}_losing_3_proof}, \S\ref{subsec:proof_eq_H_{1}_losing_2_1,2} and \S\ref{subsec:proof_eq_H_{1}_losing_2_3}, we know that $P_{1}$ loses. If $\ell=2$, $t=1$ and $L\equiv 2 \bmod 4$, $P_{2}$ removes weight $(K+2-L)$ from $EF$ in the second round, leaving $P_{1}$ with
\begin{align}
{}&w_{2}(AB)=2^{f(L-1)+1}2^{R-f(L-1)}m+2,\ w_{2}(BC)=1,\nonumber\\
{}&w_{2}(CD)=2^{f(L-1)+1}2^{R-f(L-1)}m+L,\ w_{2}(EF)=(L-1),\nonumber
\end{align}
which is of the form \eqref{eq:H_{1}_losing_2,4_case_1} with $r=2$ and $s=1$, and by what we have proved in \S\ref{subsec:proof_eq_H_{1}_losing_2_1,2}, we know that $P_{1}$ loses.

\item If $w_{1}(AB)=2^{R+1}m+\ell$ for $\ell \in \{0,1,2,3\}$, and $w_{1}(BC)=0$, $P_{2}$, at the end of the first round, is left with a galaxy graph consisting of the edges $AB$, $CD$ and $EF$, where $w_{1}(AB)=2^{R+1}m+\ell$, $w_{1}(CD)=2^{R+1}m+L$ and $w_{1}(EF)=(K+1)$, and she wins by Lemma~\ref{lem:H_{1}_winning_galaxy}.

\item Finally, if $w_{1}(AB)=2^{R+1}n+\ell_{1}$ for some $n<m$ and $\ell_{1}\in\{0,1,\ldots,2^{R+1}-1\}$, the configuration after the first round is of the form given by Lemma~\ref{prop:H_{1}_winning_1} with $m_{1}=n$, $m_{2}=m$ and $\ell_{2}=L$. Since $n<m \implies m_{1}\neq m_{2}$, $P_{2}$ wins by Lemma~\ref{prop:H_{1}_winning_1}.
\end{enumerate}

\item Suppose $P_{1}$ removes a positive integer weight from $EF$ in the first round, so that $w_{1}(EF)=k\leqslant K$. For $k \in \{K-L+1,\ldots,K\}$, $P_{2}$ removes $(K+1-k)$ from $CD$ in the second round, leaving $P_{1}$ with
\begin{align}
{}&w_{2}(AB)=2^{f(k)+1}2^{R-f(k)}m+4,\ w_{2}(BC)=(K-L-3),\nonumber\\
{}&w_{2}(CD)=2^{f(k)+1}2^{R-f(k)}m+(k+L-K-1),\ w_{2}(EF)=k,\nonumber
\end{align}
which is of the form \eqref{eq:H_{1}_losing_2,4_case_3} with $r=4$ and $s\geqslant 4$ when $k\in\{K-L+5,\ldots,K\}$, of the form \eqref{eq:H_{1}_losing_0,1,3} with either $r=3$ (when $k=(K-L+4)$) or $r=1$ (when $k=(K-L+2)$) or $r=0$ (when $k=(K-L+1)$), or of the form \eqref{eq:H_{1}_losing_2,4_case_3} with $r=2$ and $s\geqslant 4$ when $k=(K-L+3)$. In the first case, since $w_{2}(EF)=k\leqslant K$, $P_{1}$ loses by our induction hypothesis; in each of the latter cases, $P_{1}$ loses by what we have already proved in \S\ref{subsec:proof_eq_H_{1}_losing_0}, \S\ref{subsec:H_{1}_losing_1_proof}, \S\ref{subsec:H_{1}_losing_3_proof} and \S\ref{subsec:proof_eq_H_{1}_losing_2_3}.

For $k \in \{5,6,\ldots,K-L\}$, $P_{2}$ removes weight $L$ from $CD$ and weight $(K-L+1-k)$ from $BC$ in the second round, leaving $P_{1}$ with 
\begin{align}
w_{2}(AB)=2^{f(k)+1}2^{R-f(k)}m+4,\ w_{2}(BC)=(k-4),\ w_{2}(CD)=2^{f(k)+1}2^{R-f(k)}m,\ w_{2}(EF)=k,\nonumber
\end{align}
which is of the form \eqref{eq:H_{1}_losing_0,1,3} with $r=0$, and by what we have already proved in \S\ref{subsec:proof_eq_H_{1}_losing_0}, we conclude that $P_{1}$ loses. For $k \in \{1,2,3,4\}$, the configuration after the first round is of the form mentioned in Lemma~\ref{prop:H_{1}_winning_1}, with $\ell_{1}=4$, $\ell_{2}=L$ and $w_{1}(EF)=k$ satisfying $\min\{\ell_{1},\ell_{2}\}\geqslant k$, so that $P_{2}$ wins by Lemma~\ref{prop:H_{1}_winning_1}. If $k=0$, $P_{2}$ wins by Remark~\ref{rem:EF_edgeweight_0}.

\item Suppose $P_{1}$ removes a positive integer weight from $BC$ in the first round, without disturbing the edge-weights of $AB$ and $CD$, so that $w_{1}(BC)=t\in\{0,1,\ldots,K-L-4\}$. If $t\in\{4,5,\ldots,K-L-4\}$, $P_{2}$ removes weight $(K-L-3-t)$ from $EF$ in the second round, leaving $P_{1}$ with
\begin{align}
{}&w_{2}(AB)=2^{f(L+t+4)+1}2^{R-f(L+t+4)}m+4,\ w_{2}(BC)=t,\nonumber\\
{}&w_{2}(CD)=2^{f(L+t+4)+1}2^{R-f(L+t+4)}m+L,\ w_{2}(EF)=(L+t+4),\nonumber
\end{align}
which is of the form \eqref{eq:H_{1}_losing_2,4_case_3} with $r=4$ and $s\geqslant 4$, and as $w_{2}(EF)=(L+t+4)\leqslant K$, $P_{1}$ loses by our induction hypothesis. 

If $t=3$ and $L\equiv 4 \bmod 8$, $P_{2}$ removes $(K+2-L)$ from $EF$ in the second round, leaving $P_{1}$ with
\begin{align}
{}&w_{2}(AB)=2^{f(L-1)+1}2^{R-f(L-1)}m+4,\ w_{2}(BC)=3,\nonumber\\
{}&w_{2}(CD)=2^{f(L-1)+1}2^{R-f(L-1)}m+L,\ w_{2}(EF)=(L-1),\nonumber
\end{align}
which is of the form \eqref{eq:H_{1}_losing_2,4_case_1} with $r=4$ and $s=3$. If $t=3$ and $L \equiv j \bmod 8$ for $j \in \{0,1,2,3,5,6,7\}$, $P_{2}$ removes $(K-6-L)$ from $EF$ in the second round, leaving $P_{1}$ with
\begin{align}
{}&w_{2}(AB)=2^{f(L+7)+1}2^{R-f(L+7)}m+4,\ w_{2}(BC)=3,\nonumber\\
{}&w_{2}(CD)=2^{f(L+7)+1}2^{R-f(L+7)}m+L,\ w_{2}(EF)=(L+7),\nonumber
\end{align}
which is of the form \eqref{eq:H_{1}_losing_2,4_case_2} with $r=4$ and $s=3$. By what we have already proved in \S\ref{subsec:H_{1}_losing_4_5,6_proof}, we conclude that $P_{1}$ loses in each of these cases.

If $t=2$ and $L \equiv j \bmod 8$ for $j \in \{4,5\}$, $P_{2}$ removes weight $(K+3-L)$ from $EF$ in the second round, leaving $P_{1}$ with
\begin{align}
{}&w_{2}(AB)=2^{f(L-2)+1}2^{R-f(L-2)}m+4,\ w_{2}(BC)=2,\nonumber\\
{}&w_{2}(CD)=2^{f(L-2)+1}2^{R-f(L-2)}m+L,\ w_{2}(EF)=(L-2),\nonumber
\end{align}
which is of the form \eqref{eq:H_{1}_losing_2,4_case_1} with $r=4$ and $s=2$. If $t=2$ and $L \equiv j \bmod 8$ for $j \in \{0,1,2,3,6,7\}$, $P_{2}$ removes weight $(K-5-L)$ from $EF$ in the second round, leaving $P_{1}$ with
\begin{align}
{}&w_{2}(AB)=2^{f(L+6)+1}2^{R-f(L+6)}m+4,\ w_{2}(BC)=2,\nonumber\\
{}&w_{2}(CD)=2^{f(L+6)+1}2^{R-f(L+6)}m+L,\ w_{2}(EF)=(L+6),\nonumber
\end{align}
which is of the form \eqref{eq:H_{1}_losing_2,4_case_2} with $r=4$ and $s=2$. By what we have already proved in \S\ref{subsec:H_{1}_losing_4_3,4_proof}, we conclude that $P_{1}$ loses in each of these cases.

If $t=1$ and $L \equiv j \bmod 8$ for $j \in \{4,5,6\}$, $P_{2}$ removes weight $(K+4-L)$ from $EF$ in the second round, leaving $P_{1}$ with
\begin{align}
{}&w_{2}(AB)=2^{f(L-3)+1}2^{R-f(L-3)}m+4,\ w_{2}(BC)=1,\nonumber\\
{}&w_{2}(CD)=2^{f(L-3)+1}2^{R-f(L-3)}m+L,\ w_{2}(EF)=(L-3),\nonumber
\end{align}
which is of the form \eqref{eq:H_{1}_losing_2,4_case_1} with $r=4$ and $s=1$. If $t=1$ and $L \equiv j \bmod 8$ for $j \in \{0,1,2,3,7\}$, $P_{2}$ removes weight $(K-4-L)$ from $EF$ in the second round, leaving $P_{1}$ with
\begin{align}
{}&w_{2}(AB)=2^{f(L+5)+1}2^{R-f(L+5)}m+4,\ w_{2}(BC)=1,\nonumber\\
{}&w_{2}(CD)=2^{f(L+5)+1}2^{R-f(L+5)}m+L,\ w_{2}(EF)=(L+5),\nonumber
\end{align}
which is of the form \eqref{eq:H_{1}_losing_2,4_case_2} with $r=4$ and $s=1$. By what we have already proved in \S\ref{subsec:H_{1}_losing_4_1,2_proof}, we conclude that $P_{1}$ loses in each of these cases.

If $t=0$, $P_{2}$, at the end of the first round, is left with a galaxy graph consisting of the edges $AB$, $CD$ and $EF$, where $w_{1}(AB)=2^{R+1}m+4$, $w_{1}(CD)=2^{R+1}m+L$ and $w_{1}(EF)=(K+1)$, and she wins by Lemma~\ref{lem:H_{1}_winning_galaxy}.
\end{enumerate}
This completes the proof of our claim that the configuration in \eqref{eq:H_{1}_losing_4_7_inductive} is losing, and thereby also completes our inductive proof of the claim that any configuration on $H_{1}$ that is of the form \eqref{eq:H_{1}_losing_2,4_case_3} with $r=4$, is losing.

\medskip
\noindent \textbf{Acknowledgements} The authors SK, MP, SR and SS would like to acknowledge the research grants NSF DMS 2124222, SERB CRG/2021/006785, ANRF/ARG/2025/011715/QSS and IITK/MATH/2021295, respectively, for partially supporting their research.

\bibliographystyle{ecta}
\bibliography{Graph_nim_bibliography}

@book{williams2017combinatorial,
  title={Combinatorial Games on Graphs},
  author={Williams, Trevor K},
  year={2017},
  publisher={Utah State University}
}

@book{brualdi1977introductory,
  title={Introductory combinatorics},
  author={Brualdi, Richard A},
  year={1977},
  publisher={Pearson Education India}
}

@article{bouton1901nim,
  title={Nim, a game with a complete mathematical theory},
  author={Bouton, Charles L},
  journal={Annals of mathematics},
  volume={3},
  number={1/4},
  pages={35--39},
  year={1901},
  publisher={JSTOR}
}

@article{moore1910generalization,
  title={A generalization of the game called nim},
  author={Moore, Ellakim H},
  journal={Annals of Mathematics},
  volume={11},
  number={3},
  pages={93--94},
  year={1910},
  publisher={JSTOR}
}

@article{sprague1935mathematische,
  title={{\"U}ber mathematische kampfspiele},
  author={Sprague, Richard},
  journal={Tohoku Mathematical Journal, First Series},
  volume={41},
  pages={438--444},
  year={1935},
  publisher={Mathematical Institute, Tohoku University}
}

@article{gale1974curious,
  title={A curious Nim-type game},
  author={Gale, David},
  journal={The American Mathematical Monthly},
  volume={81},
  number={8},
  pages={876--879},
  year={1974},
  publisher={Taylor \& Francis}
}

@article{grundy1939mathematics,
  title={Mathematics and games},
  author={Grundy, Patrick M},
  journal={Eureka},
  volume={2},
  pages={6--8},
  year={1939}
}

@article{robin1989poisoned,
  title={A poisoned chocolate problem, Problem corner},
  author={Robin, AC},
  journal={The Mathematical Gazette},
  volume={73},
  number={466},
  pages={341--343},
  year={1989}
}

@article{lv2018greedy,
  title={Greedy Nim k Game},
  author={Lv, Xinzhong and Xu, Rongxing and Zhu, Xuding},
  journal={Journal of Combinatorial Optimization},
  volume={35},
  number={4},
  pages={1241--1249},
  year={2018},
  publisher={Springer}
}

@article{low2016notes,
  title={Notes on the combinatorial game: graph Nim.},
  author={Low, Richard M and Chan, Wai Hong},
  journal={Electron. J. Graph Theory Appl.},
  volume={4},
  number={2},
  pages={190--205},
  year={2016}
}

@article{brown2019graphnim,
  title={Graph NIM},
  author={Brown, David E. and Williams, Trevor K.},
  journal={Congressus Numerantium},
  volume={233},
  number={},
  pages={151--164},
  year={2019}
}

@article{duchene2016building,
  title={Building nim},
  author={Duch{\^e}ne, Eric and Dufour, Matthieu and Heubach, Silvia and Larsson, Urban},
  journal={International Journal of Game Theory},
  volume={45},
  number={4},
  pages={859--873},
  year={2016},
  publisher={Springer}
}

@article{van2022nim,
  title={Nim variants},
  author={Van den Bergh, Mark and Kosters, Walter and Spieksma, Flora},
  journal={ICGA Journal},
  volume={44},
  number={1},
  pages={2--17},
  year={2022},
  publisher={SAGE Publications Sage UK: London, England}
}

@article{XU20181,
title = {Bounded Greedy Nim},
journal = {Theoretical Computer Science},
volume = {746},
pages = {1-5},
year = {2018},
issn = {0304-3975},
doi = {https://doi.org/10.1016/j.tcs.2018.06.015},
url = {https://www.sciencedirect.com/science/article/pii/S0304397518304316},
author = {Rongxing Xu and Xuding Zhu}
}

@article{gurvich2015slow,
  title={Slow $ k $-Nim},
  author={Gurvich, Vladimir and Ho, Nhan Bao},
  journal={arXiv preprint arXiv:1508.05777},
  year={2015}
}

@article{schwartz1971some,
  title={Some extensions of Nim},
  author={Schwartz, Benjamin L},
  journal={Mathematics Magazine},
  volume={44},
  number={5},
  pages={252--257},
  year={1971},
  publisher={Taylor \& Francis}
}

@article{albert2004nim,
  title={Nim restrictions},
  author={Albert, Michael H and Nowakowski, Richard J},
  journal={Integers},
  volume={4},
  pages={G1},
  year={2004},
  publisher={Walter de Gruyter}
}

@article{li1978n,
  title={N-person nim and n-person moore's games},
  author={Li, S-YR},
  journal={International Journal of Game Theory},
  volume={7},
  number={1},
  pages={31--36},
  year={1978},
  publisher={Springer}
}

@inproceedings{calkin2010computing,
  title={Computing strategies for graphical Nim},
  author={Calkin, Neil J and James, Kevin and Janoski, Janine E and Leggett, Sarah and Richards, Bryce and Sitaraman, Nathan and Thomas, Stephanie M},
  booktitle={Proceedings of the Forty-First Southeastern International Conference on Combinatorics, Graph Theory and Computing},
  volume={202},
  pages={171--185},
  year={2010}
}

@article{fukuyama2003nim,
  title={A Nim game played on graphs},
  author={Fukuyama, Masahiko},
  journal={Theoretical Computer Science},
  volume={304},
  number={1-3},
  pages={387--399},
  year={2003 A},
  publisher={Elsevier}
}

@article{fukuyamanim2003,
  title={A Nim game played on graphs II},
  author={Fukuyama, Masahiko},
  journal={Theoretical computer science},
  volume={304},
  number={1-3},
  pages={401--419},
  year={2003 B},
  publisher={Elsevier}
}

@book{erickson2011game,
  title={The game of Nim on graphs},
  author={Erickson, Lindsay Anne},
  year={2011},
  publisher={North Dakota State University}
}

@book{conway2000numbers,
  title={On numbers and games},
  author={Conway, John H},
  year={2000},
  publisher={AK Peters/CRC Press}
}

@book{ferguson2020course,
  title={A course in game theory},
  author={Ferguson, Thomas S},
  year={2020},
  publisher={World Scientific}
}

@article{boros2024impartial,
  title={Impartial games with decreasing Sprague--Grundy function and their hypergraph compound},
  author={Boros, Endre and Gurvich, Vladimir and Ho, Nhan Bao and Makino, Kazuhisa and Mursic, Peter},
  journal={International Journal of Game Theory},
  volume={53},
  number={4},
  pages={1119--1144},
  year={2024},
  publisher={Springer}
}

@article{boros2019sprague,
  title={Sprague-Grundy function of matroids and related hypergraphs},
  author={Boros, Endre and Gurvich, Vladimir and Ho, Nhan Bao and Makino, Kazuhisa and Mursic, Peter},
  journal={Theoretical computer science},
  volume={799},
  pages={40--58},
  year={2019},
  publisher={Elsevier}
}

@article{boros2017tetris,
  title={Tetris hypergraphs and combinations of impartial games},
  author={Boros, Endre and Gurvich, Vladimir and Ho, Nhan Bao and Makino, Kazuhisa and Mursic, Peter},
  journal={arXiv preprint arXiv:1701.02819},
  year={2017}
}

@article{boros2018sprague,
  title={On the Sprague--Grundyfunction of Exact k-Nim},
  author={Boros, Endre and Gurvich, Vladimir and Ho, Nhan Bao and Makino, Kazuhisa and Mursic, Peter},
  journal={Discrete Applied Mathematics},
  volume={239},
  pages={1--14},
  year={2018},
  publisher={Elsevier}
}

@article{boros2019symmetric,
  title={Sprague--Grundy function of symmetric hypergraphs},
  author={Boros, Endre and Gurvich, Vladimir and Ho, Nhan Bao and Makino, Kazuhisa and Mursic, Peter},
  journal={Journal of Combinatorial Theory, Series A},
  volume={165},
  pages={176--186},
  year={2019},
  publisher={Elsevier}
}

@article{boros2019compound,
  title={On the Sprague-Grundy function of compound games},
  author={Boros, Endre and Gurvich, Vladimir and Kitrossky, Levi and Makino, Kazuhisa},
  journal={arXiv preprint arXiv:1903.08138},
  year={2019}
}

@article{smith1966graphs,
  title={Graphs and composite games},
  author={Smith, Cedric AB},
  journal={Journal of Combinatorial Theory},
  volume={1},
  number={1},
  pages={51--81},
  year={1966},
  publisher={Elsevier}
}

@article{boros2024computing,
  title={Computing remoteness functions of Moore, Wythoff, and Euclid’s games: E. Boros et al.},
  author={Boros, Endre and Gurvich, Vladimir and Makino, Kazuhisa and Vyalyi, Michael},
  journal={International Journal of Game Theory},
  volume={53},
  number={4},
  pages={1315--1333},
  year={2024},
  publisher={Springer}
}

@article{gurvich2024remoteness,
  title={On remoteness functions of k-NIM with k+ 1 piles in normal and in mis{\`e}re versions},
  author={Gurvich, Vladimir and Maximchuk, Vladislav and Miheenkov, Georgy and Naumova, Mariya},
  journal={Games},
  volume={15},
  number={6},
  pages={37},
  year={2024},
  publisher={MDPI}
}

@article{gurvich2024screw,
  title={Screw discrete dynamical systems and their applications to exact slow NIM},
  author={Gurvich, Vladimir and Naumova, Mariya},
  journal={Discrete Applied Mathematics},
  volume={358},
  pages={382--394},
  year={2024},
  publisher={Elsevier}
}
\end{document}